\RequirePackage[thmmarks]{ntheorem}
\makeatletter
\renewtheoremstyle{plain} 
  {\item[\hskip\labelsep \theorem@headerfont ##1\ \textup{##2}\theorem@separator]} 
  {\item[\hskip\labelsep \theorem@headerfont ##1\ \textup{##2}\ (##3)\theorem@separator]}
\makeatother

\let\latexdocument\document
\let\latexarabic\arabic

\documentclass[]{biometrika}

\let\document\latexdocument
\let\arabic\latexarabic

\def\rm{}

\usepackage{times, bm}
\usepackage{mypackage}
\usepackage{mymacros2, myenv}
\usepackage{mymacros_brackets}
\usepackage{convex_opti}
\usepackage{mymacros_CCA_prelim}
\usepackage{appendix}
\usepackage[ruled,vlined]{algorithm2e}
\makeatletter
\renewcommand{\algocf@captiontext}[2]{#1\algocf@typo. \AlCapFnt{}#2} 
\def\@algocf@capt@plain{top}
\renewcommand{\algocf@makecaption}[2]{%
  \addtolength{\hsize}{\algomargin}%
  \sbox\@tempboxa{\algocf@captiontext{#1}{#2}}%
  \ifdim\wd\@tempboxa >\hsize
    \hskip .5\algomargin%
    \parbox[t]{\hsize}{\algocf@captiontext{#1}{#2}}
  \else%
    \global\@minipagefalse%
    \hbox to\hsize{\box\@tempboxa}
  \fi%
  \addtolength{\hsize}{-\algomargin}%
}
\makeatother

\begin{document}



\markboth{N. Laha, N. Huey, B. Coull, \and R. Mukherjee}{Biometrika style}

\title{On Statistical Inference with High Dimensional Sparse CCA}

\author{N. Laha, N. Huey, B. Coull, \and R. Mukherjee}
\affil{Harvard University, 677 Huntington Ave, Boston, MA 02115, \email{nlaha@hsph.harvard.edu},
\email{nhuey@g.harvard.edu},
\email{bcoull@hsph.harvard.edu}, \email{ram521@mail.harvard.edu}}

\maketitle

\begin{abstract}
We consider asymptotically exact inference on the leading canonical correlation directions and strengths between two high dimensional vectors under sparsity restrictions. In this regard, our main contribution is the development of a loss function, based on which, one can operationalize a one-step bias-correction  on  reasonable  initial estimators. Our analytic results in this regard are adaptive over suitable structural restrictions of the high dimensional nuisance parameters, which, in this set-up, correspond to the covariance matrices of the variables of interest.  We further supplement the theoretical guarantees behind our procedures with extensive numerical studies.
\end{abstract}


\section{Introduction}\label{sec:introduction}
 \textcolor{black}{ Statistical analyses of 
biomedical applications require methods which can handle complex data structures. In particular, to understand the  relationship between potentially high dimensional variables,
formal and systematic Exploratory Data Analysis (EDA)  is often an important first step. Key examples in this regard include, but are not limited to, eQTL mapping studies
\citep{witten2009,chen2012structured},  
epigenetic studies \citep{holm2010molecular,sofer2012multivariate,hu2017adaptive,hu2016integration}, and in general  studies involving integration of multiple biological data such as genetic markers, gene
expressions, and disease phenotypes \citep{kang2013sparse,lin2013group}. 
Of critical relevance in each of these examples is that of understanding relationships between possibly high dimensional variables of interest. 
}
In this regard,  linear relationships are the simplest, most intuitive, and lend themselves to easy interpretations. Subsequently, a large volume of statistical literature has been devoted to exploring linear relationships through variants of the classical statistical toolbox of  Canonical Correlation Analysis (CCA) \citep{hotelling1992relations}. Our focus in this paper pertains to 
some fundamental inferential questions in the context of high dimensional CCA.

To  formally set up the inferential questions in the CCA framework, we   consider i.i.d. data $(X_i,Y_i)_{i=1}^n\sim \mathbb{P}$ on two random vectors $X\in\RR^p$ and $Y\in\RR^q$ with joint covariance matrix
\[\Sigma=\begin{bmatrix}
\Sx & \Sxy\\
\Syx & \Sy\\
\end{bmatrix}.\]
The first canonical correlation $\rho_0$ is defined as the maximum possible correlation between two linear combinations of $X$ and $Y$. More specifically, consider the following optimization problem:
\begin{maxi}[4]
    {\alpha\in\RR^{p},\beta\in\RR^{q}}{\alpha^T\Sxy \beta}
     {\label{opt: sparse canonical correlation analysis}}{}
     \addConstraint{\alpha^T\Sx \alpha=\beta^T\Sy \beta}{= 1}
       \end{maxi}
The  maximum value attained in \eqref{opt: sparse canonical correlation analysis} is $\rhk$, and the solutions to \eqref{opt: sparse canonical correlation analysis} are commonly referred as the first canonical directions, which we will denote by  $\alpha_0$ and $\beta_0$, respectively.
\textcolor{black}{
This paper considers inference on $\alk$, $\bk$, and associated quantities of interest. In most scientific problems, the first canonical correlation coefficient is of prime interest as it summarizes the ``maximum linear association" between $X$ and $Y$ and thereby motivating our choice of inferential target.}

Early developments in the theory and applications of CCA have now been well documented in statistical literature and we refer the interested reader to \cite{anderson2003,anderson1962introduction} and references therein for further details. These classical results have been thereafter heavily used to provide statistical inference (i.e. asymptotically valid hypotheses tests, confidence intervals and P-values) across a vast canvas of disciplines such as psychology, agriculture, oceanography and others. However, modern surge in interests for CCA, often being motivated by data from high throughput biological experiments, requires re-thinking several aspects of the traditional theory and methods. In particular, in most modern data examples, the number of samples is typically comparable to or much smaller than the number of variables in the study -- rendering the classical CCA inconsistent and inadequate without further structural assumptions \cite{cai2018rate,ma2020subspace,bao2019}. A natural structural constraint that has gained popularity in this regard, is that of sparsity i.e. the phenomenon of an (unknown) few collection of variables being related to each other rather than contributions to the associations from the whole collection of high dimensional components.  The framework of Sparse Canonical Correlation Analysis (SCCA) \citep{witten2009} has thereafter been developed to target such low dimensional structures,  and to subsequently provide consistent estimation in the context of high dimensional CCA. Although such structured CCA problems have witnessed a renewed enthusiasm from both theoretical and applied communities, most papers have heavily focused on key aspects of estimation (in suitable norms) and relevant scalable algorithms -- se e.g. \cite{gao2013,gao2015,gao2017,ma2020subspace,mai}. However, asymptotically valid inference  is yet to be explored systematically in the context of SCCA. 
\textcolor{black}{In particular, 
none of the existing estimation methods for SCCA lend themselves to uncertainty quantification, i.e. inference on $\alpha_i$ $(i=1,\ldots,p)$, $\beta_j$ $(j=1,\ldots,q)$, or $\rho$. This is indeed not surprising,  since being based on penalized procedures, existing estimators are asymptotically biased, super-efficient for estimating $0$ coordinates, and not tractable in terms of estimating underlying asymptotic distribution \cite{leeb2005model,leeb2006can,leeb2008sparse,potscher2009distribution}. Therefore, construction of asymptotically valid confidence intervals for $\alpha_i$, $\beta_j$'s or $\rho$ is not straightforward. In absence of such intervals, bootstrap or permutation tests are typically used in practice \citep{witten2009}. However, these methods are often empirically justified and even then might suffer from subtle pathological issues that underlie standard re-sampling techniques in penalized estimation framework \cite{chatterjee2010asymptotic,chatterjee2011bootstrapping,chatterjee2013rates}.} This paper is motivated by taking a step in resolving these fundamental issues with inference on SCCA.

\subsection{Main contribution}
\textcolor{black}{
 The main results of this paper is the construction of asymptotically valid confidence intervals for $\sqrt{\rho_0}\alpha_0$ and $\sqrt{\rho_0}\beta_0$. Our method is based on a one-step bias-correction performed on preliminary estimators of the canonical directions. The resulting bias-corrected estimators have an asymptotic linear influence function type expansion (see e.g. \cite{tsiatis2007semiparametric} for asymptotic influence function expansions) with $\sqrt{n}$-scaling (see Theorem \ref{thm: for alpha} and Proposition \ref{corollary: main theorem}) under suitable sparsity conditions on the truth. This representation is subsequently exploited to build  confidence intervals for a variety of relevant lower dimensional functions of the top canonical directions; see Corollary~\ref{corollary: main theorem single} and Corollary~\ref{cor: projection matrix} and the discussions that follow. Finally,  we will show that the entire de-biased vector is asymptotically equivalent to a high dimensional Gaussian vector in a suitably uniform sense; see Proposition~\ref{corollary: main theorem}, which enables the control of familywise error rate.
}

\textcolor{black}{
 The bias correction procedure  crucially relies on a novel representation of $\sqrt{\rho_0}\alpha_0$ and $\sqrt{\rho_0}\beta_0$  as the  unique maximizers (up to a sign flip) of a smooth objective (see Lemma~\ref{lemma: objective function}), which may be of independent interest. The uniqueness criteria is indispensable here since otherwise  a crucial  local convexity convexity property (see Lemma~\ref{lemma: Hessian positive definite}), which we  fundamentally exploit to deal with high dimensionality of the problem, is not guaranteed. We also discuss why the commonly used representations of the top canonical correlations is difficult to work with owing to either the  lack of such local convexity properties, or  the flexibility of its form to offer a non-cumbersome derivation of the one-step bias correction. We elaborate on these subtleties in  Section \ref{subsec:debbiasing_method} for details.} 

Further, we pay special attention to adapt to underlying sparsity structures of the marginal precision matrices ($\Sigma_x^{-1},\Sigma_y^{-1}$) of the high dimensional variables ($X,Y$) under study -- which serve as high dimensional nuisance parameters in the problem. Consequently, our construction of asymptotically valid confidence intervals for top canonical correlation strength and directions are agnostic over \textcolor{black}{the structures (e.g. sparsity of the precision matrices of $X$ and $Y$) of these complex nuisance parameters.} The de-biasing procedure can be implemented using our R package \texttt{de.bias.CCA} available at \emph{https://github.com/nilanjanalaha/de.bias.CCA}.

Finally, we supplement our methods for inference with suitable constructions of initial estimators of canonical correlation directions as well as nuisance parameters under suitable sparsity assumptions. The construction of these estimators, although motivated by existing ideas, requires careful modifications to tackle inference on the first canonical correlation strength and directions -- while treating remaining directions as nuisance parameters.


\section{Mathematical Formalism}\label{sec:math_formalism}


In this section we collect some assumptions and notation that will be used throughout the rest of the paper.


\subsection{Structural Assumptions}
Throughout this paper, we will assume that $X$ and $Y$ are centered sub-Gaussian random vectors \footnote{ see \cite{vershynin2010} for more details.} with joint covariance matrix $\Sigma$ as described above. 
We will let $\Sxy$ to have a fixed rank $r\geq 1$  \citep[implying that apart from $\rho_0$, there are $r-1$ additional canonical correlations][]{anderson2003}. Since the cross-covariance matrix $\Sxy$ has rank $r$, it
it can be shown  that \citep[cf.][]{gao2013, gao2017} 
\begin{equation}\label{eq: decomposition: sample covariance matrix}
\Sxy=\Sx U\Lambda V^T\Sy,
\end{equation}
where  $U=[u_1\ldots u_r]$ and $V=[v_1\ldots v_r]$ are $p\times r$ and $q\times r$ dimensional matrices  satisfying $U^T\Sx U=I_r$ and $V^T\Sy V=I_r$, respectively. The $\Lambda$ in \eqref{eq: decomposition: sample covariance matrix} is a diagonal matrix, whose diagonal entries are the canonical correlations, i.e.
\[\rhk=\Lambda_1\geq \Lambda_2\geq \ldots \geq\Lambda_r>0.\]
In this regard, the matrices $U$ and $V$ need not be unique unless the canonical correlations, i.e. the $\Lambda_i$'s, are all unique.  Indeed, we will at the least require uniqueness of  $\alk$ and $\bk$, since otherwise they are not even identifiable. To that end, we will make the following assumption that is common in the literature since it grants uniqueness of $\alk$ and $\bk$ up to a sign flip \citep[cf.][]{gao2013, gao2017, mai}.
\begin{assumption}[\bf Eigengap Assumption]
\label{assump: eigengap assumptions}
There exists  $\e_0\in(0,1)$ so that $\rhk-\Lambda_1>\e_0$  for all $n$.
\end{assumption}
Note that  Assumption
 \ref{assump: eigengap assumptions} also implies that $\rhk$ stays bounded away from zero.
  We will further assume that  $\Sx$ and $\Sy$ are positive definite and bounded in operator norm.   
\begin{assumption}[\bf Bounded eigenvalue Assumption]
\label{assump: bounded eigenvalue}
There exists $M>0$ such that the eigenvalues of $\Sx$ and $\Sy$ are bounded below by $M^{-1}$ and bounded above by $M$.
\end{assumption}
This regularity assumption is also common in the literature of SCCA \citep{gao2017, gao2015, mai, laha2021}.



%


\subsection{Notation}

We will denote the set of all positive integers by $\NN$.
For a matrix $A$, we denote its $j$th column by $A_j$.  
Also, let $\Lambda_{max}(A)$ and $\Lambda_{min}(A)$  denote the largest and smallest eigenvalue of $A$, respectively.
 We denote the gradient of a function $f$ by $\dot f$ or $\grad f$, where we reserve the notation $\grad^2 f$ for the hessian. The $i$th element of any vector $v$ is denoted by $v_i$. We use the notation $\|\cdot\|_p$ to denote the usual $l_p$ norm of a vector for any $p\in\NN$. For a matrix $A\in\RR^{p\times q}$, $\norm{A}_F$  and $\norm{A}_{op}$ will denote  the Frobenius and the  operator norm, respectively. We denote by $|A|_\infty$ the elementwise supremum of $A$.
Throughout the paper, $C$ will be used to denote a positive constant whose value may change from line to line.
 
 The results in this paper are mostly asymptotic (in $n$) in nature and thus require some standard asymptotic notations.  If $a_n$ and $b_n$ are two sequences of real numbers then $a_n \gg b_n$ (and $a_n \ll b_n$) implies that ${a_n}/{b_n} \rightarrow \infty$ (and ${a_n}/{b_n} \rightarrow 0$) as $n \rightarrow \infty$, respectively. Similarly $a_n \gtrsim b_n$ (and $a_n \lesssim b_n$) implies that $\liminf_{n \rightarrow \infty} {{a_n}/{b_n}} = C$ for some $C \in (0,\infty]$ (and $\limsup_{n \rightarrow \infty} {{a_n}/{b_n}} =C$ for some $C \in [0,\infty)$). Alternatively, $a_n = o(b_n)$ will also imply $a_n \ll b_n$ and $a_n=O(b_n)$ will imply that $\limsup_{n \rightarrow \infty} \ a_n / b_n = C$ for some $C \in [0,\infty)$).
 
 We will denote the set of the  indices of the non-zero rows in $U$ and $V$ by $S_U$ and $S_V$, respectively. We let $s_U$ and $s_V$ be the cardinalities of $S_U$ and $S_V$ and use $s=s_U+s_V$ to denote the total sparsity.  We further denote by $s_x$ and $s_y$ the number of nonzero elements of $\alk$ and $\bk$, respectively. The supports of $\alk$ and $\bk$ will be similarly be denoted by $S_x$ and $S_y$, respectively. We will discuss the precise requirements on these sparsities, and the necessities of such assumptions in detail in  Section~\ref{subsec: choices of ha, hb, hf}.
 
 Our method requires initial estimators of $\alk$, $\bk$, and $\rho_0$. We let $\ha$ and $\hb$ be the initial estimators of $\alpha_0$ and  $\beta_0$,  respectively. Also, we denote the empirical estimates of $\Sx$, $\Sy$, and $\Sxy$, by $\hSx$, $\hSy$, and $\hSxy$, respectively. The  estimate $\hro$ of $\rhk$ is 
 \begin{equation}\label{def: hro}
   \hro=\frac{\ha^T\hSxy\hb}{(\ha^T\hSx\ha)^{1/2}(\hb^T\hSy\hb)^{1/2}}.  
 \end{equation}
 The quantity $\hro$ may not be positive for any  $\ha$ and $\hb$. Therefore, mostly we will use $|\hro|$ as an estimate of $\rhk$.  Finally, for the sake of simplicity, we  let $\lambda$ denote the term 
\begin{equation}\label{def:lambda}
\lambda=\lb\dfrac{\log(p\vee q)}{n}\rb^{1/2}.
\end{equation}

 \section{Methodology}
 \label{sec: method}
 In this section we discuss the intuitions and details of our   main proposed methodology that we will analyze in later sections. The discussions are divided across three main subsections. The first Subsection \ref{sec: de-bias: general} presents the driving intuitions behind obtaining general de-biased estimators of generic parameters of interest that can be defined through generic optimization framework. Subsequently, Subsection \ref{subsec:debbiasing_method} translates this intuition to a working principle in the context of SCCA. In particular, we design a suitable optimization criterion which allows a principled application of the general de-biasing method and additionally lends itself to rigorous theoretical analyses. Finally, our last Subsection \ref{seubsec:why_others_methods_do_not_work} elaborates on the benefit of designing this specific optimization objective function over other possible choices of optimization problems for defining the leading canonical directions.

 \subsection{The Debiasing Method in General}
 \label{sec: de-bias: general}
 
 We first discuss the simple intuition behind reducing the bias of estimators defined through estimating equations. To that end, suppose we are interested in estimating  $\thk\in \RR^{p}$, which minimizes the function $f:\RR^p\mapsto\RR$. If $f$ is smooth, then $\thk$  solves the equation
$\dot f(\theta)=0$.
Suppose $\theta$ is in a small neighborhood of $\thk$. the Taylor series expansion of $f(\theta)$ around $\thk$ yields
$
\dot f(\theta)-\dot f(\thk)=\grad^2 f(\bth) (\theta-\thk)$, 
where $\bth\in\RR^p$ lies on the line segment joining $\thk$ and $\theta$. 
If $f$ has finitely many global minimums, then $f$ can not be flat at  $\thk$. In that case,   $f$  is strongly convex at some neighborhood of $\thk$.
Therefore  $\grad^2 f(\bth)$  is  positive definite, leading to 
$\thk=\theta-(\grad^2 f(\bth))^{-1}\dot{f}(\theta)$. Suppose   $\hth$ and $\hf$ are  reliable estimators of $\thk$ and $(\grad^2 f(\thk))^{-1}$, respectively. Correcting the first order bias of $\hth$ then yields the de-biased estimator $ \hdth=\hth-\hf\dot{f}(\hth)$. Thus, to find a bias-corrected estimator of $\thk$, it suffices to find a smooth function which is minimized at $\thk$ and has at most finitely many global minima. This simple intuition is the backbone of our strategy.
 
  
  \begin{remark}[Positive definiteness of $\grad^2 f(\thk)$]
  \label{remark: pd requirement}
  The positive definiteness of $\grad^2 f(\thk)$ is important because most existing  methods for estimating the inverse of a high dimensional matrix requires the matrix to be positive definite. These methods proceed via estimating the columns of $\Sigma^{-1}$ separately through a quadratic optimization step. Unless the original matrix is positive definite, these intermediate optimization problems are unbounded. Therefore,  the algorithms are likely to diverge  even with enough observations. For more details, see Section 1  of \cite{jankova2018} \citep[see also Section 2.1 of][]{yuan2010}. 
  \end{remark}


\subsection{The Debiasing Method for SCCA}\label{subsec:debbiasing_method}
To operationalize the intuition described above in Section \ref{sec: de-bias: general},  we begin with a lemma which represents  $\rhk^{1/2}\alk$ and $\rhk^{1/2}\bk$  as the unique minimizers (upto a sign flip) of a smooth objective function. We defer the proof of Lemma~\ref{lemma: objective function}  to Supplement~\ref{addlemma: methods}. 
\begin{lemma}\label{lemma: objective function}
For any $C>0$, we have
\[\pm(\rhk^{1/2}\alk,\rhk^{1/2}\bk)=\argmin_{x\in\RR^p, y\in\RR^q}h(x,y).\]
where 
$h(x,y)=(1-C/2)(x^T\Sx x)(y^T\Sy y)+C(x^T\Sx x)^2/4+C(y^T\Sy y)^2/4-2x^T\Sxy y$.
\end{lemma}

\textcolor{black}{
The proof of Lemma~\ref{lemma: objective function} hinges on  a 
seminal result on low rank matrix approximation dating back to \cite{eckart1936}, which implies that for any matrix $A$ with singular value decomposition $\sum_{i=1}^{r}\Lambda_i \tu_i \tv_i^T$,
\begin{equation}\label{def: Ak}
  \sum_{i=1}^{k}\Lambda_i \tu_i \tv_i^T=\argmin_{B\in \mathcal M_{k}}\|A-B\|_F^2\quad (k=1,\ldots, r),   \end{equation}
   where $\mathcal M_k$ is the set of all $p\times q$ matrices with rank $k$. Our main  inferential method for leading canonical directions  builds on Lemma  \ref{lemma: objective function}, and consequently, corrects for the bias  of estimating  $x^0=\rhk^{1/2}\alk$ and $y^0=\rhk^{1/2}\bk$ using preliminary plug-in estimators from literature. It is worth noting that we focus on the  the leading canonical directions up to a multiplicative factor since from our inferential point of view, this quantity is enough to explore the nature of projection operators onto these directions. in particular, for the sake of constructing tests for no-signal such as $H_0: (\alk)_i=0$ it is equivalent to the test $H_0: x^0_i=0$. }

\begin{remark}
Suppose $h$ is as in Lemma~\ref{lemma: objective function}.
It can be shown that the other stationary points  of $h(x,y)$, to be denoted by $(\tilde x_i, \tilde y_i)$, correspond to the canonical pairs with correlations $\Lambda_i$, $(i\geq 2)$. Moreover, the Hessian of $h(x,y)$ at $(\tilde x_i, \tilde y_i)$ has both positive and negative eigenvalues, indicating that the function is neither concave nor convex at these points. Therefore, all these stationary points are saddle points. Consequently, any minimum of $h(x,y)$ is a global minimum -- irrespective of the choice of $C>0$.
\end{remark}

Now note that
\begin{align}\label{def: delta h}
    {\pdv{h}{x}}(x,y)=&\ (2-C)(y^T\Sy y)\Sx x+C(x^T\Sx x)\Sx x-2\Sxy y,\nn\\
    {\pdv[2]{h}{x}}(x,y)=&\ (2-C)(y^T\Sy y)\Sx+C(x^T\Sx x)\Sx+2C\Sx xx^T\Sx,\nn\\
    {\pdv{h}{x}{y}}(x,y)=&\ 2(2-C)\Sx xy^T\Sy-2\Sxy,
\end{align}
 and hence by symmetry,  the Hessian  $H(x,y)$ of $h$ at $(x,y)$ is given by
  \[\begin{bmatrix}
 (2-C)(y^T\Sy y)\Sx+C(x^T\Sx x)\Sx & 2(2-C)\Sx xy^T\Sy-2\Sxy\\
 +2C\Sx xx^T\Sx & \\
 2(2-C)\Sy yx^T\Sx-2\Syx & (2-C)(x^T\Sx x)\Sy+C(y^T\Sy y)\Sy\\
 & +2C\Sy yy^T\Sy
 \end{bmatrix}.\]
At this point we note the flexibility of our approach in choosing $C$ so as to being able to work with a relatively amenable form of the Hessian  and its inverse that we need to estimate.
 We subsequently set $C=2$ so that the estimation of the cross term $\Sx xy^T\Sy$ can be avoided. 
 In particular, when $x^0={\rhk^{1/2}} \alk$ and $y^0={\rhk^{1/2}} \bk$, then $(x^0)^T\Sx x^0=(y^0)^T\Sy(y^0)=\rhk$. We denote the Hessian in this case as
 \begin{equation}\label{def: H knot}
    H^0=H(x,y):=2\rhk\begin{bmatrix}
\Sx+2\Sx \alk\alk^T\Sx & -\Sxy/\rhk\\
 -\Syx/\rhk & \Sy+2\Sy \bk\bk^T\Sy
 \end{bmatrix}.
 \end{equation}
A plug-in estimator $\hH(x,y)$ of $H^0$ is given by
 \begin{equation*}
\hH(x,y)= 2 \begin{bmatrix}
 (x^T\hSx x)\hSx+2\hSx xx^T\hSx & -\hSxy\\
 -\hSyx & (y^T\hSy y)\hSy+2\hSy yy^T\hSy
 \end{bmatrix}.
 \end{equation*} 

Because our $h$ is a sufficiently well-behaved function, it possesses a positive definite   Hessian  at the minima $\pm (x^0,y^0)$, thereby demonstrating the crucial strong convexity property mentioned in Remark~\ref{remark: pd requirement}. This property of $H^0$ is the content of our following lemma, the proof of which can be found in Supplement~\ref{addlemma: methods}.
\begin{lemma}\label{lemma: Hessian positive definite}
 Under Assumptions \ref{assump: eigengap assumptions} and \ref{assump: bounded eigenvalue}, 
the matrix  $H^0$ defined in \eqref{def: H knot} is positive definite with minimum eigenvalue $\Lambda_{min}(H^0)\geq 2(\rhk-\Lambda_2)/M$ where $M$ is as in Assumption \ref{assump: bounded eigenvalue}. 
 \end{lemma}

Lemma \ref{lemma: objective function} and Lemma \ref{lemma: Hessian positive definite} subsequently allows us to constructed de-biased estimators of the leading canonical directions as follows. 
\textcolor{black}{Suppose $\hx={|\hro|}^{1/2} \ha$ and $\hy={|\hro|}^{1/2}\hb$ are   estimators of $x^0$ and $y^0$,  where $\ha$ and $\hb$ are the preliminary estimators of $\alk$ and $\bk$, and $\hro$ is as defined in \eqref{def: hro}.
  Our construction of de-biased estimators in SCCA now relies on two objects: 
(a) estimators of 
 $\partial h(\hx,\hy)/\partial x$ and $\partial h(\hx,\hy)/\partial y$, which are simply given by 
 \begin{align}\label{def: hat dh general}
    {\pdv{\widehat h_n}{x}}(\hx,\hy)=&\ 2(\hx^T\hSx \hx)\hSx \hx-2\hSxy \hy,\nn\\
     {\pdv{\widehat h_n}{\hy}}(\hx,\hy)=&\ 2(\hy^T\hSy \hy)\hSy \hy-2\hSyx \hx,
\end{align}
 and (b) an estimator $\hf$ of $\Phi^0$ -- the inverse of $H^0$. Construction of such an estimator is can be involved and to tackle this we develop a version of the Node-wise Lasso algorithm (see Supplement \ref{sec: asymp: NL} for details) popularized in recent research \cite{vandegeer2014}.
 }\textcolor{black}{
Following the intuitions discussed in Section~\ref{sec: de-bias: general}, we can then complete the construction of the de-biased estimators, whose final form writes as}
 \begin{align}\label{def: de-biased estimators}
  \begin{bmatrix}
\hdai\\ \hdbi
\end{bmatrix}=
  \begin{bmatrix}
\hx\\ \hy
\end{bmatrix}-\hf^T \begin{bmatrix}
\pdv{\widehat h_n}{x}(\hx,\hy)\\ \pdv{\widehat h_n}{y}(\hx,\hy)
\end{bmatrix}.
\end{align}

  

  
 In Supplement~\ref{sec: connection to literature}, we will discuss how our proposed method connects to the broader scope of de-biased inference in high dimensional problems. In regard to the targets of our estimators, we note that if $\ha$ estimates $\alk$, then $\hdai$ also estimates $x^0$. However, if $\ha$ approximates $-\alk$ instead, then $\hdai$ instead approximates $-x^0$. The similar phenomenon can be observed for $\hb$ as well. Our theoretical analyses of these estimators will be designed accordingly.

    
 At this time, we are also ready to construct a de-biased estimator of $\rhk^2$. To that end, suppose $\hx$ and $\hy$ are such that $\hx^T\hSxy\hy\geq 0$. Note that if that is not the case, we can always switch $\hx$ to $-\hx$ so that $\hx^T\hSxy\hy\geq 0$. Our estimator of $\rhk^2$ can then be constructed as $\hro^{2,db}=\min(1,|\hro^{2,\text{raw}}|)$, where
 \[\hro^{2,\text{raw}}=\hx^T\hSxy\hdbi+(\hdai)^T\hSxy\hy-\hx^T\hSxy\hy.\]


Before moving onto the theoretical properties of our proposed methods, we make a slight relevant digression by noting that there are many ways to formulate the optimization program in \eqref{opt: sparse canonical correlation analysis} so that $\pm(\alk,\bk)$ can be characterized as the global optimizer. We therefore close this current section with a discussion on why the particular formulation in Lemma~\ref{lemma: objective function} particularly useful for our purpose.

\subsection{Subtleties with Other Representations of $\alk$ and $\bk$}\label{seubsec:why_others_methods_do_not_work}
Indeed, the most intuitive approach to characterize $\pm(\alk,\bk)$ is to see it as the maximizer of the constrained maximization problem \eqref{opt: sparse canonical correlation analysis}. This leads to 
 the Lagrangian 
  \begin{align}\label{Lagrange}
L(\alpha,\beta,l_1,l_2)=- \alpha^T\Sxy\beta+l_1(\alpha^T\Sx\alpha-1)+l_2(\beta^T\Sy\beta-1),
 \end{align}
 where $l_1$ and $l_2$ are the Lagrange multipliers.
Denoting $\theta=(\alpha,\beta,l_1,l_2)$, it can be verified that since $\theta_0=(\alk,\bk,\rhk/2,\rhk/2)$ is a stationary point of \eqref{opt: sparse canonical correlation analysis}, $\thk$ also solves $\dot{L}(\theta)=0$.  Using the first order Taylor series expansion of  $L$, one can subsequently show that any $\theta$  in a small neighborhood of $\theta_0$ has the approximate expansion
 \[\theta-\theta_0\approx \ddot{L}(\theta_0)^{-1}\dot{L}(\theta).\]
If we then replace $\theta$ by an estimator of $\thk$, one can use the above expansion to estimate the first order bias of this estimator provided $\ddot{L}(\theta_0)$ is suitably nice and estimable.  However, by \textit{strong max-min property} \citep[cf. Section 5.4.1][]{boyd2004}, $L$ satisfies
\begin{equation}\label{eq: lagrange}
   \sup_{l_1,l_2\in\RR}\inf_{\alpha\in\RR^p,\beta\in\RR^q}L(\alpha,\beta,l_1,l_2)=\inf_{\alpha\in\RR^p,\beta\in\RR^q}\sup_{l_1,l_2\in\RR}L(\alpha,\beta,l_1,l_2), 
\end{equation}
 which implies $(\alk,\bk,\rhk/2,\rhk/2)$ is a saddle point of $L$. Thus $\ddot{L}(\theta_0)$ fails to be positive definite. 
 In fact, any constrained optimization program fails to provide a Lagrangian with positive definite hessian, and thus violates the requirements outlined in Section~\ref{sec: de-bias: general}. 
 We have already pointed out in Remark~\ref{remark: pd requirement}  that  statistical tools for efficient estimation of the inverse of a high dimensional matrix  is scarce unless the matrix under consideration is positive definite. Therefore, we refrain from using the constrained optimization formulation in \eqref{opt: sparse canonical correlation analysis}  for the de-biasing procedure.

For any $C>0$, the function
 \[f:(\alpha,\beta)\mapsto -\frac{\alpha^T\Sxy\beta}{(\alpha^T\Sx\alpha)^{1/2}(\beta^T\Sy\beta)^{1/2}}+C( \alpha^T\Sx\alpha-1)^2+C( \beta^T\Sy\beta-1)^2,\]
 however, is a valid choice for the $f$ outlined in Subsection~\ref{sec: de-bias: general}  since its only global minimizers are $\pm(\alk,\bk)$, which also indicates strong convexity at $\pm(\alk,\bk)$. However, the gradient and the Hessian of this function takes a complicated form. Therefore, establishing asymptotic results for the de-biased  estimator based on this $f$ is significantly more cumbersome than its counterpart based on the $h$ in Lemma~\ref{lemma: objective function}. Hence, we refrain from using this objective function for our de-biasing procedure as well.

\section{Asymptotic Theory for the  De-biased Estimator}
\label{sec: asymptotic theory}

In this section we establish theoretical properties of our proposed estimators under a high dimensional sparse asymptotic framework. To set up our main theoretical results, we first present assumptions on sparsities of the true canonical directions and desired conditions on initial estimators of $\alk,\bk,\Phi^0$ in Subsection \ref{subsec: choices of ha, hb, hf}. The construction of estimators with these desired properties are discussed in Appendices \ref{algo:Modified COLAR} and \ref{algo: nodewise lasso}. Subsequently, we present the main asymptotic results and its implications for construction of confidence intervals of relevant quantities of interest in Subsection \ref{subsec:theoretical_results}.


\subsection{Assumptions on $\ha$, $\hb$, and $\hf$}
\label{subsec: choices of ha, hb, hf}

For the de-biasing procedure to be successful, it is  important that $\ha$ and $\hb$ are both $l_1$ and $l_2$ consistent for $\alk$ and $\bk$ with suitable rates of convergence. In particular, we will require them to satisfy the following condition.
 \begin{condition}
 [Preliminary estimator condition]
 \label{cond: preliminary estimator}
 The preliminary estimators $\ha$ and $\hb$ of $\alk$ and $\bk$ satisfy the followings for some 
  $\kappa\in[1/2,1]$, $s=s_U+s_V$, and $\lambda$ as defined in \eqref{def:lambda}:
 \begin{equation*}
\inf_{w\in\{\pm 1\}}\|w\ha-\alk\|_2+\inf_{w\in\{\pm 1\}}\|w\hb-\bk\|_2=O_p(s^{\kappa}\lambda),
\end{equation*}
and
\begin{equation*}
\inf_{w\in\{\pm 1\}}\|w\ha-\alk\|_1+\inf_{w\in\{\pm 1\}}\|w\hb-\bk\|_1=O_p(s^{\kappa+1/2}\lambda).
\end{equation*}
\end{condition}
\textcolor{black}{
 We present discussions regarding the necessity of the rates presented above as well as the motivation behind the exponent $\kappa \in [1/2,1]$ in Supplement \ref{sec:assumptions_necessities}.  Moreover, we also discuss the construction of estimators satisfying  Condition~\ref{cond: preliminary estimator} in Supplement \ref{app: COLAR}. Our method for developing these initial estimators  is motivated by the recent results in \cite{gao2017}, who jointly estimate $U$ and $V$ up to an orthogonal rotation with desired $l_2$ guarantees. However, our situation is somewhat different since we need to estimate $\alk$ and $\bk$  up to a sign flip, which might not be obtained from the joint estimation of all the directions up to orthogonal rotation. This is an important distinction  since the remaining directions act as nuisance parameters in our set up.
 It turns out that the asymptotics of the sign-flipped version requires crucial modification of the arguments of \cite{gao2017}. The analysis of this modified procedure presented in Supplement \ref{algo:Modified COLAR} in turn allows us to extract both the desired $l_1$ and $l_2$ guarantees in the process. }

  \textcolor{black}{
 We will also require an assumption on the sparsities $s_U$ and $s_V$, the number of nonzero rows of $U$ and $V$, respectively. We present this next while deferring the discussions on the necessity of such assumptions to Appendix~\ref{sec:assumptions_necessities}. 
 }
 \begin{assumption}[Sparsity Assumption]
\label{assumption: sparsity}
We assume   $s_U=o(p)$, $s_V=o(q)$, and $s^{2\kappa}\lambda^2=o(n^{-1/2})$ where $s=s_U+s_V$ and $\kappa$ is as in Condition~\ref{cond: preliminary estimator}.
\end{assumption}

 Finally, our last condition pertains to the  estimator  $\hf$ on $\Phi^0$. Most methods for estimating precision matrices can be adopted to estimate $\Phi^0$ using an estimator of $H^0$.  However, care is needed since $\hf$ needs to  satisfy some rates of convergence for the de-biased estimators in \eqref{def: de-biased estimators} to be  $\sqrt{n}$-consistent. We collect this condition below.
 \begin{condition}[Inverse hessian Conditions]
 \label{assumption: precision matrix}
  The estimator $\hf$  satisfies
\[\max_{1\leq j\leq p+q}\|(\hf)_j-\Phi^0_j\|_1=O_p(s^{\kappa+1/2}\lambda),\]
 and
 \[\max_{1\leq j\leq p+q}\|(\hf)_j-\Phi^0_j\|_2=O_p(s^{\kappa}\lambda),\]
 where $\kappa$ is as in Condition~\ref{cond: preliminary estimator}.
 \end{condition}
 \textcolor{black}{
We defer the discussion on   the construction of $\hf$  to Appendix \ref{App: nodewise lasso}, where, in particular, we will show that the a nodewise Lasso type estimator, which appeals to the ideas in \cite{vandegeer2014}, satisfies Condition~\ref{assumption: precision matrix}.
 }

 \subsection{Theoretical Analyses}\label{subsec:theoretical_results}
  In what follows, we only present  the results on inference for $\alk$. Parallel results for $\bk$ can be obtained similarly. 
Before stating the main theorem, we introduce a few additional notation. We partition the $i^{\mathrm{th}}$ column of $\Phi^0$ comfortably w.r.t. the dimensions of $X$ and $Y$ as $\Phi^0_i=(\pa, \pb)$ where $\pa\in\RR^p$ and $\pb\in\RR^q$.
We subsequently define the random variable
\begin{align}\label{def: mathcal Z}
 \mathcal Z(i)=  &\ [\rhk(\pa)^T+\{(\pa)^T\Sx x^0 \}(x^0)^T] XX^Tx^0+[\rhk(\pb)^T+\{(\pb)^T\Sy y^0\}(y^0)^T] YY^T y^0\nn\\
   &\ -(\pa)^TXY^T y^0-(x^0)^TXY^T\pb,
\end{align}
and its associated variance as
\begin{align}\label{def_sigma_i}
    \sigma_i^2=\text{var}(\mathcal Z(i)).
\end{align}
Since $X$ and $Y$ are sub-Gaussian, 
it can be shown that all moments of $\mathcal Z(i)$, and in particular, the $\sigma^2_i$'s are finite under Assumption~\ref{assump: bounded eigenvalue}. \textcolor{black}{Indeed, we show the same through the proof of Theorem~\ref{thm: for alpha} 
. }
Finally define
\begin{align}\label{def: L}
  \mathcal L=\begin{bmatrix}
     \mathcal L_{(1)}\\ \mathcal L_{(2)}
  \end{bmatrix}=2\Phi^0\begin{bmatrix}
\rhk(\hSx-\Sx)x^0-(\hSxy-\Sxy)y^0+((x^0)^T(\hSx-\Sx)x^0)\Sx x^0\\ \rhk(\hSy-\Sy)y^0-(\hSyx-\Syx)x^0+((y^0)^T(\hSy-\Sy)y^0)\Sy y^0
\end{bmatrix}  
\end{align}
 With this we are ready to state the main theorem of this paper.

\begin{theorem}[Asymptotic representation of \hdai]\label{thm: for alpha}
Suppose 
 $\mathcal L_1$ is as defined in \eqref{def: L}, $\ha$ and $\hb$ satisfy Condition~\ref{cond: preliminary estimator},  and $\hf$ satisfies Condition~\ref{assumption: precision matrix}.
Then under Assumption~\ref{assump: eigengap assumptions}, \ref{assump: bounded eigenvalue}, and  \ref{assumption: sparsity},
 the estimator $\hdai$ defined in \eqref{def: de-biased estimators} can be expanded as either
\[\quad  \hdai=x^0- \mathcal{L}_{(1)} +\text{rem},\quad\text{or}\quad \hdai=-x^0- \mathcal{L}_{(1)} +\text{rem},\]
where 
$\|\text{rem}\|_\infty=O_p(s^{2\kappa}\lambda^2)$ with $s$ and $\lambda$ as defined in Assumption~\ref{assumption: sparsity} and \eqref{def:lambda}, respectively.
  \end{theorem}

  A few remarks are in order about the statement and implications of Theorem \ref{thm: for alpha}. First, we not that under Assumption~\ref{assumption: sparsity}, $\|\text{rem}\|_\infty=o_p(n^{-1/2})$. The importance of Theorem~\ref{thm: for alpha} subsequently lies in the fact that it establishes the equivalence between  $\hx^{db}$ and the more tractable random vector $\mathcal L$ under Assumption~\ref{assumption: sparsity}.  In particular, one  immediately can derive a simple yet relevant corollary about the asymptotic normal nature of the distributions of our de-biased estimators. 

\begin{corollary}\label{corollary: main theorem single}
Under the set up of Theorem~\ref{thm: for alpha}, for any $i=1,\ldots,p$, the following assertions hold:
\begin{itemize}
   \item[1.]If $\alpha_{0,i}\neq 0$, then $n^{1/2}\slb (\hdae)^2-(x^0_i)^2\srb$ converges in distribution to a centered  Gaussian random variable with variance $16\sigma_i^2(x^0_i)^2$.
    \item[2.] If $\alpha_{0,i}= 0$, then $n(\hdae)^2$ converges in distribution to a central Chi-squared random variable with degrees of freedom one and scale parameter $4\sigma_i^2$.
\end{itemize}
 \end{corollary}
 
  The proof of Corollary~\ref{cor: projection matrix} is deferred to the appendix. Before proceeding, it is worth mentioning that
the decision to provide inference on $(\hdae)^2$ instead of $\hdae$ is driven by the fact that the former is unaffected by the sign flip of  $\hdae$, which, unbeknown to us, can be centered at either $x^0_i$ or $-x^0_i$. However, a result on $\hdae$ can also be derived under the set up of Theorem~\ref{thm: for alpha} and one has 
 \begin{equation}\label{statement: convergencee: hdai: single}
     \sqrt{n}\slb \hdae-x^0_i\srb\to_d N(0,4\sigma_i^2)\quad \text{or}\quad \sqrt{n}\slb \hdae+x^0_i\srb\to_d N(0,4\sigma_i^2)\quad (i=1,\ldots,p).
 \end{equation}
 Moreover, we note that, often the inference on $(x^0_i)^2$ suffices since
in practice the sign of $x^0_i$ is typically of little interest. As a specific example, testing $H_0: x^0_i=0$, is equivalent to testing $H_0: (x^0_i)^2=0$. More importantly one of the central objects of interest in low dimensional representations obtained through SCCA is the projection operators onto the leading canonical directions. It is easy to see that for this operator it is sufficiently to understand the squared $(x^0_i)^2$ and the cross terms $x^0_ix^0_j$ respectively. We will also present asymptotic characterization of estimators for the cross-terms $x^0_ix^0_j$. However, we first present a somewhat uniform nature of the joint asymptotic normal behavior for the entire vector $\hat{x}_n^{db}$. 
To this end, we verify in our next proposition that if $\log p=o(n^{-1/7})$, then the convergence in \eqref{statement: convergencee: hdai: single} is uniform across $i=1,\ldots,p$ while restricted to sets of suitably nice nature.

\begin{proposition}\label{corollary: main theorem}
Let $\mathcal A_p$ be the set of all hyperrectangles in $\RR^p$ and let $\Sigma_p$  the covariance matrix of the $p$-variate random vector $(\mathcal Z(1),\ldots,\mathcal Z(p))$.  Assume the set up of Theorem~\ref{thm: for alpha},  $\inf_{1\leq i\leq p}\sigma^2_i>c$ for some $c>0$, and that $\log(p+q)=o(n^{-1/7})$. Then as $n\to\infty$, either
\[\sup_{A\in\mathcal A_p}\bl P\slb n^{1/2}(\hdai-x^0)\in A\srb-P\slb 2\mathbb X\in A\srb\bl\to 0,\]
or
\[\sup_{A\in\mathcal A_p}\bl P\slb n^{1/2}(\hdai+x^0)\in A\srb-P\slb 2\mathbb X\in A\srb\bl\to 0,\]
where $\mathbb X$ is a random vector distributed  as $ N_{p}(0,\Sigma_p)$.
\end{proposition}

 Proposition~\ref{corollary: main theorem} can in turn be used, as promised earlier, to infer on the non-diagonal elements of the  matrix $x^0(x^0)^T$.  This is the content of our next corollary -- the proof of which can be found in Supplement~\ref{sec: proof of corollary main theorem}.
 \begin{corollary}\label{cor: projection matrix}
Consider the set up of Proposition~\ref{corollary: main theorem}. Suppose $\Sigma_p$ is positive definite. Let $i,j\in[p]$, and $i\neq j$. Denote by $\sigma_{ij}$  the covariance between $\mathcal Z(i)$ and $\mathcal Z(j)$, where $\mathcal Z(i)$'s are as defined in \eqref{def: mathcal Z}.
 Then the following assertions hold:
 \begin{enumerate}
     \item Suppose $x^0_ix^0_j\neq 0$. Then  
     \[n^{1/2}\slb (\hx)_i(\hx)_j -x^0_ix^0_j\srb\to_d N\slb 0, 4\{(x^0_i)^2\sigma^2_j+(x^0_j)^2\sigma^2_i+2x^0_ix^0_j\sigma_{ij}\}\srb.\]
   \item   Suppose $x^0_ix^0_j= 0$. Then   \[
   n (\hx)_i(\hx)_j\to_d \mathbb Z_i\mathbb Z_j,\]
   where $\mathbb Z_i\sim N(0,\sigma^2_i)$,  $\mathbb Z_j\sim N(0,\sigma^2_j)$, and $\text{cov}(\mathbb Z_{i},\mathbb{Z}_j)=\sigma_{ij}$. 
 \end{enumerate}
\end{corollary}
Here once again we  observe that
 the de-biased estimators of $x^0_ix^0_j$ have  different asymptotics depending on whether $x^0_ix^0_j=0$ or not -- which parallels the behavior of the de-biased estimators of the diagonal elements we demonstrated earlier through Corollary~\ref{corollary: main theorem single}.

 \begin{remark}
 Proposition~\ref{corollary: main theorem} can also used  to simultaneously test the null hypotheses $H_0: (\alk)_i=0$ $(i=1,\ldots,p)$. The uniform convergence in Proposition~\ref{corollary: main theorem} can be used o justify multiple hypothesis testing for the coordinates of $x^{0}$  -- whenever the coreesponding p-values are defined through rectangular rejection regions based on $\hat{x}_n^{db}$. To this end, one can use standard  methods like Benjamini and Hochberg (BH) and Benjamini and Yekutieli (BY) procedures for FDR control. 
 The simultaneous testing procedure can thereby also be connected to variable selection procedures. However, we do not pursue it here since specialized methods are available for the latter in  SCCA context \citep{laha2021}.
 \end{remark}
 
The proof of Proposition~\ref{corollary: main theorem}, which can be found in Supplement~\ref{sec: proof of corollary main theorem}, relies on a Berry-Esseen type result.  The lower bound requirement on the $\sigma^2_i$'s  is typical for such Berry Esseen type theorems -- see e.g.\cite{kato2016}. To check whether this assumptions actually can hold in specific examples, we provide Corollary~\ref{corollary to proposition} below to establish the validity of $\inf_{1\leq i\leq p}\sigma^2_i>c$ for some $c>0$ when $(X,Y)$ is jointly Gaussian. The proof of Corollary~\ref{corollary to proposition} can be found in Supplement~\ref{sec: proof of corollary main theorem}.
\begin{corollary}\label{corollary to proposition}
Suppose $X,Y$ are jointly Gaussian and $\rhk$ is bounded away from zero and one. Further suppose $\log(p+q)=o(n^{-1/7})$.
Then under the set up of Theorem~\ref{corollary: main theorem}, the assertion of $\inf_{1\leq i\leq p}\sigma^2_i>c$ for some $c>0$ used in Proposition~\ref{corollary: main theorem} holds.
\end{corollary}

  
 We end our discussions regarding the inference of $x^{0}$ with a method for consistent estimation of the $\sigma_i^2$'s. indeed, this will allow us to develop tests for the hypotheses $H_0: x_i^0=0$ or build confidence interval for $(x_i^0)^2$. To this end we partition $(\hf)_i=(\pha,\phb)$ where  $\pha\in\RR^p$, and $\phb\in\RR^q$. Because $\sigma_i^2=\text{var}(\mathcal Z)$, for $i=1,\ldots,p$, it can be shown that a consistent estimator is given by the  variance of pseudo-observations $\{\widehat Z_j(i)\}_{j=1}^n$ $(i=1,\ldots,p+q)$, which are defined by 
 \begin{align*}
   \widehat Z_j(i)= &\    [\hro\pha^T+\{\pha^T\hSx \hx \}(\hx)^T] X_jX_j^T\hx-\pha^TX_jY_j^T \hy\\
   &\ +[\hro\phb^T+\{\phb^T\hSy \hy\}(\hy)^T] Y_jY_j^T \hy\nn-(\hx)^TX_jY_j^T\phb.
 \end{align*}





Our final result pertains to the asymptotic distribution of $\hro^{2,db}$. 
\begin{theorem}\label{corrolary: main: rho}
Suppose $s^{2\kappa+1/2}\lambda^2=o(n^{-1/2})$ and $\rhk<1$. Then under the set-up of Theorem~\ref{thm: for alpha},
\[n^{1/2}(\hro^{2,db}-\rhk^2)\to_d N(0,\sigma^2_\rho),\]
where
$\sigma^2_\rho=\text{var}(\rhk(X^Tx^0)^2+\rhk(Y^Ty^0)^2-2(X^Tx^0)(Y^Ty^0))$.
In particular, when the observations are Gaussian, $\sigma^2_\rho=\rhk^2(1-\rhk^2)^2$.
\end{theorem}
A few remarks are in order regarding content of Theorem \ref{corrolary: main: rho}. First, one can  $\sigma_\rho^2$ is consistently  
\[\widehat\sigma^2_\rho=\sum_{j=1}^n \frac{(\hx^T X_j)^2(\hy^T Y_j)^2}{n}-\hro^4,\]
and thereby use Theorem \ref{corrolary: main: rho} to create asymptotically valid confidence intervals for leading canonical signal strength. Further note that Theorem~\ref{corrolary: main: rho} requires stricter condition on $s$ compared to Theorem \ref{thm: for alpha}.  Although we have not explored the sharpness of this assumption, one can find similar stricter sparsity requirement in \cite{jankova2018} while demonstrating $n^{1/2}$-consistency of a  de-biased estimator for the largest eigenvalue in the sparse PCA problem.
Finally, the value of $\sigma^2_\rho$ in the Gaussian case matches that of the parametric MLE of $\rhk^2$ under the Gaussian model \citep[][p.505]{anderson2003}. Such agreement is generally observed in case of the  de-biased estimators, e.g. the de-biased estimator of the  principal eigenvalue \citep{jankova2018}.
 

\section{Numerical Experiments}
\label{Sec: simulations}
\subsection{Preliminaries}
In this section we explore aspects of finite sample behavior of the methods discusses in earlier sections. Further numerical experiments are collected in  Supplement~\ref{sec: extra: bias}  where we compare the bias of our method  to  popular SCCA alternatives. We start with some preliminary discussions on the choice for the set-up, initial estimators, and tuning parameters.

 \textbf{Set Up:} The set-ups under which we will conduct our comparisons can be described through specifying the nuisance parameters (marginal covariance matrices of $X$ and $Y$) along with the strength ($\rho$), sparsity, rank of $\Sigma_{xy}$, and the joint distribution of $X,Y$. For the marginal marginal covariance matrices of $X$ and $Y$, motivated by previously studied cases in the literature \citep{mai,gao2017}  we shall consider two cases as follows:
\begin{description}
\item[\textit{Identity}.] This will correspond to the case where $\Sx=\Sy=I_p$
\item[\textit{Sparse-inverse}.] This will correspond to the case where $\Sx=\Sy$ is the correlation matrix obtained from  $\Sigma_0$, where $\Sigma_0=\Omega^{-1}$, and $\Omega$ is a sparse matrix with the form
\[\Omega_{ij}=1_{\{i=j\}}+0.5\times 1_{\{|i-j|=1\}}+0.4\times 1_{\{|i-j|=2\}},\quad i,j\in[p].\]
\end{description}
Analogous to \cite{mai} and \cite{gao2017}, we shall also take $\Sxy=\rho_0 \Sx\alk\bk^T\Sy$ to be  a rank one matrix, where  we consider the canonical vectors $\alpha_0$ and $\beta_0$ with sparsity $2$ as follows:
\begin{align*}
    \alpha_*=(1,1,0,\ldots,0)^T, &\quad \beta_*=(1,1,0,\ldots,0)^T,\quad 
    \alk = \frac{\alpha_*}{\sqrt{\alpha_*^T\Sigma_x\alpha_*}}, &\quad \bk = \frac{\beta_*}{\sqrt{\beta_*^T\Sigma_y\beta_*}}.
\end{align*}
The canonical correlation $\rhk$ depicts the signal strength in our set up. We will explore three different values for the  $\rho_0$: 0.2, 0.5, and 0.9, which will be referred as the small, medium, and the high signal strength settings, respectively. The joint distribution of $X,Y$ is finally taken to be Gaussian with mean $0$. Also, throughout we set the $(p,q,n)$ combination to be  $(80,80,500)$, $(300, 200, 500)$, and $(600, 200, 500)$, which correspond to  $p+q$ being small, moderate, and moderately high, respectively. Finally, we will always 
 consider $N=1000$ Monte Carlo samples.
 
\textbf{Initial Estimators and Tuning Parameters:}  We construct the preliminary estimators using the modified COLAR algorithm (see Algorithm~\ref{algo:Modified COLAR}). For the rank one case, the latter coincides with \cite{gao2017}'s COLAR estimator.  Recall that throughout we set the $(p,q,n)$ combination to be  $(80,80,500)$, $(300, 200, 500)$, and $(600, 200, 500)$. One of the reasons  we do not accommodate higher $p$ and $q$  because the COLAR algorithm, as it is,  does not scale well with $p$ and $q$ \footnote{This was also noted by \cite{mai}. } Also, we do not consider smaller values of $n$ since it is expected that de-biasing procedures generally require $n$ to be at least moderately large  (see e.g. \cite{jankova2018}).


In our proposed methods, tuning parameters arise from two sources: (a) estimation of the preliminary estimators and (b) precision matrix estimation.
  To implement the modified COLAR algorithm,  we mostly follow the code for COLAR  provided by the authors \cite{gao2017}. The COLAR penalty parameters, $\lambda_1$ and $\lambda_2$, were left as specified in the COLAR code, namely $\lambda_1 = 0.55\{\log(p)/n\}^{1/2}$ and $\lambda_2 = [\{1 + \log(p)\}/n]^{1/2}$. The tolerance level was fixed at $10^{-4}$ with a fixed maximum of 200 iterations for the first step of the COLAR algorithm.
Next consider the tuning strategy for the nodewise lasso algorithm (Algorithm~\ref{algo: nodewise lasso}), which involves the lasso penalty parameter $\lambda_j^{nl}$ and the  parameter $B_j$ ($j=1,\ldots,p+q$). 
 Theorem~\ref{thm: nodewise Lasso theorem} proposes the choice  $\lambda_j^{nl} = C \cdot \sqrt{{\log(p+q)}/{n}}$ for all $j\in[p+q]$. In our simulations, the parameter $C$ is empirically determined to minimize  $|\hf \widehat H(\hx,\hy)-I_{p+q}|_{\infty}$. For the settings $(80,80,500)$ and $(300, 200, 500)$, this parameter is set at $40$ and $50$ for the identity and sparse inverse cases, respectively. For the moderately high $p+q$ setting, this parameter is set at $20$.  The nodewise lasso parameter $B_j$ is taken to be $10/\lambda_j$, which is in line with  \cite{jankova2018}, who recommends taking $B_j\approx 1/\lambda_j$.


\textbf{Targets of Inference:} We present our results for the $1^{\rm st}$ and the $20^{\text{nd}}$ element of $x^0$. The former stands for a typical non-zero element, where the latter represents a typical zero element.  For each element, we compute
 confidence intervals for $(x^0_i)^2$, and test the null $H_0:$ $|x^0_i|=0$ $(i=1,20)$. For the latter, we use a  $\chi^2$-squared test based on the asymptotic null distribution of $(\hdai)^2$ given in  part two of Corollary~\ref{corollary: main theorem single}. As mentioned earlier, this test is equivalent to testing $H_0: (\alk)_i=0$. The construction of the confidence intervals, which we discuss next, is a little more subtle.
 
 We construct two types of confidence interval. For any $i\in[p]$, the first confidence interval, which will be referred as the ordinary interval from now on, is given by
 \begin{equation}\label{def: CI}
     \slb \max\{0, (\hdae)^2-l_{\text{CI},i}\}, (\hdae)^2+l_{\text{CI},i}\srb,\quad \text{where}\quad l_{\text{CI},i}=4z_{0.975}|\hdae|\widehat\sigma_i/\sqn.
 \end{equation}
Here $z_{0.975}$ is the $0.975^{\text{th}}$  quantile of the standard Gaussian distribution.
Corollary~\ref{corollary: main theorem single} shows that the asymptotic coverage of the above confidence interval is  $95\%$  when $x^0_i\neq 0$. For $x_i^0=0$, however,   the above confidence interval can have asymptotic coverage higher than $0.95\%$. To see why, note that  $(\hdae)^2=O_p(1/n)$  by 
Corollary~\ref{corollary: main theorem single}  in this case.  Since both the length and the center of the ordinary interval depends on $(\hdae)^2$, the coverage can suffer greatly if $(\hdae)^2$ underestimates $(x^0_i)^2$. Therefore, we 
construct another confidence interval by relaxing the length of the ordinary intervals. This second interval, to be referred as the conservative interval from now on, is obtained by simply substituting the $\hdae$ in the standard deviation term  $l_{\text{CI}}$ in \eqref{def: CI}  by $\max(|\hdae|,1)$.
Clearly, the conservative interval can   have potentially higher coverage than $95\%$, which motivates our  nomenclature. 

\subsection{Results} We divide the presentation of our results on coordinates with and without signal, followed by discussions about issues regarding distinctions between asymptotic and finite sample considerations of our method. \\

 \textbf{Inference when there is no signal:}
 If $x^0_i=0$, both confidence intervals (CI) exhibit high coverage,  often exceeding  $95\%$, across all settings; see 
 Figures \textcolor{blue}{[$x^0_{20}$ plots]}
 in Supplement~\ref{Sec: extra simulations}. This is unsurprising in view of the discussion in the previous paragraph.
 The conservative confidence intervals have substantially larger length, which is understandable because the ratio between the ordinary and the conservative CI length  is $O_p(1/n)$ in this case. Also, the length of the confidence intervals generally decrease as the signal strength increases, as expected. The rejection frequency of the tests (the type I error in this scenario),   generally stays below $0.05$, especially at medium to high signal strength. \\
 
 \textbf{Inference when there is signal:} 
  When $x^0_i\neq 0$, the  ordinary intervals exhibit poor coverage at the low and medium signal strength  regardless of the underlying covariance matrix structure, although the performance seems to be worse for sparse inverse matrices.  Figure \ref{fig: CI  for x1: ordinary} entails that this  underperformance  is  due to the underestimation of small signals  $(\widehat x^{db}_1)^2$, which is tied to the high negative bias of the preliminary estimator  in these cases; see the histograms in Figure~\ref{fig: hist for x1}. This issue will be discussed in more detail in Supplement~\ref{sec: extra: bias}.
  Figure \ref{fig: CI  for x1: ordinary} also implies that if $(x^0_i)^2$ is small, the confidence intervals crowd near the origin.
  Also at the high signal strength, the coverage of the ordinary intervals fail to reach the desired $95\%$ level.

The relaxation of the ordinary confidence interval length, which leads to  the conservative intervals, substantially improve the coverage, with the improvement being dramatic at low signal. In the latter case, the conservative intervals enjoy high coverage, which is well over $95\%$ for moderate or higher $p,q$. In this case, in general, the relaxation results in a  four-fold or higher increase in the confidence interval length. As signal strength increases, the increase in  the confidence interval length gets smaller, and consequently, the increase in the coverage 
slows down. This is unsurprising noting  the ratio between the length of the conservative and the ordinary interval is proportional to $\hro^{-1}$. 
One should be cautious with the relaxation, however, because it may lead to inclusion of not only the   true signal, as desired, but also  zero. This can be clearly seen in   the  medium signal strength case of the sparse inverse matrix; compare the middle column of Figure~\ref{fig: CI  for x1: ordinary} (b) with that of Figure~\ref{fig: CI  for x1: relaxed} (b). The inclusion of origin does not bring any advantage for the relaxed intervals in the no-signal case either, because as discussed earlier, in the latter case the ordinary intervals are themselves efficient, with the relaxed versions hardly making any  improvement.

\textbf{Discussion on Asymptotics:} The performance of the confidence intervals improve  if $(n,p,q)$ increase. See for example the illustration in Figure~\ref{fig: double} in Supplement~\ref{sec: extra plots} where the triplet has been doubled. Interestingly, the asymptotics successfully kicks in for the corresponding tests   as soon as the signal strength reaches the medium level. The test attains  power  higher than $0.673$ at the medium signal strength, and the perfect power of one at high signal strength.  This phenomenon is the result of the  super-efficiency  of the de-biased estimator at $x^0_i=0$, as elicited by Corollary~\ref{corollary: main theorem single}. Since the test exploits the knowledge of this faster convergence  under the null, it has better precision than the confidence interval, which is oblivious to this fact. In many situations,  the test may get rejected but the confidence intervals, even the ordinary one, may include zero. During implementation, if one faces such a situation, they should conclude that either the signal strength is too small or the  sample size  is not sufficient for the confidence intervals to be too precise. 


\textbf{Discussions on Performance of De-biased SCCA:} We conclude that since the de-biased estimators work on sparse estimators which are \textcolor{black}{super efficient} at zero, the inference does not face any obstacle if  the true signal $x^0_i=0$. In presence of signal, 
the tests are generally reliable if the signal strength is at least moderate. In contrast, the  ordinary confidence intervals, which are blindly based on Corollary~\ref{corollary: main theorem single}, struggle whenever the initial COLAR estimators incur a bias too large for the de-biasing step to overcome.
 This is generally observed at low to medium signal strength. 
 The conservative intervals can solve this problem partially at the cost of increased length. 
At present, the $l_1$ and $l_2$ guarantees as required by Condition~\ref{cond: preliminary estimator} are only available for  COLAR type estimators. The performance of the ordinary confidence intervals may improve if one can construct a SCCA preliminary estimator with similar strong theoretical garuantees, but better empirical performance in picking up small signal. Searching for  a different SCCA preliminary estimator is important for another reason --  COLAR is not scalable to ultra high dimension. This problem occurs because COLAR relies on semidefinite programming, whose scalability issues are well noted \citep{dey2018}.


 \begin{figure}
    \centering
    \begin{subfigure}{\textwidth}
     \centering
    \includegraphics[height=3.5 in]{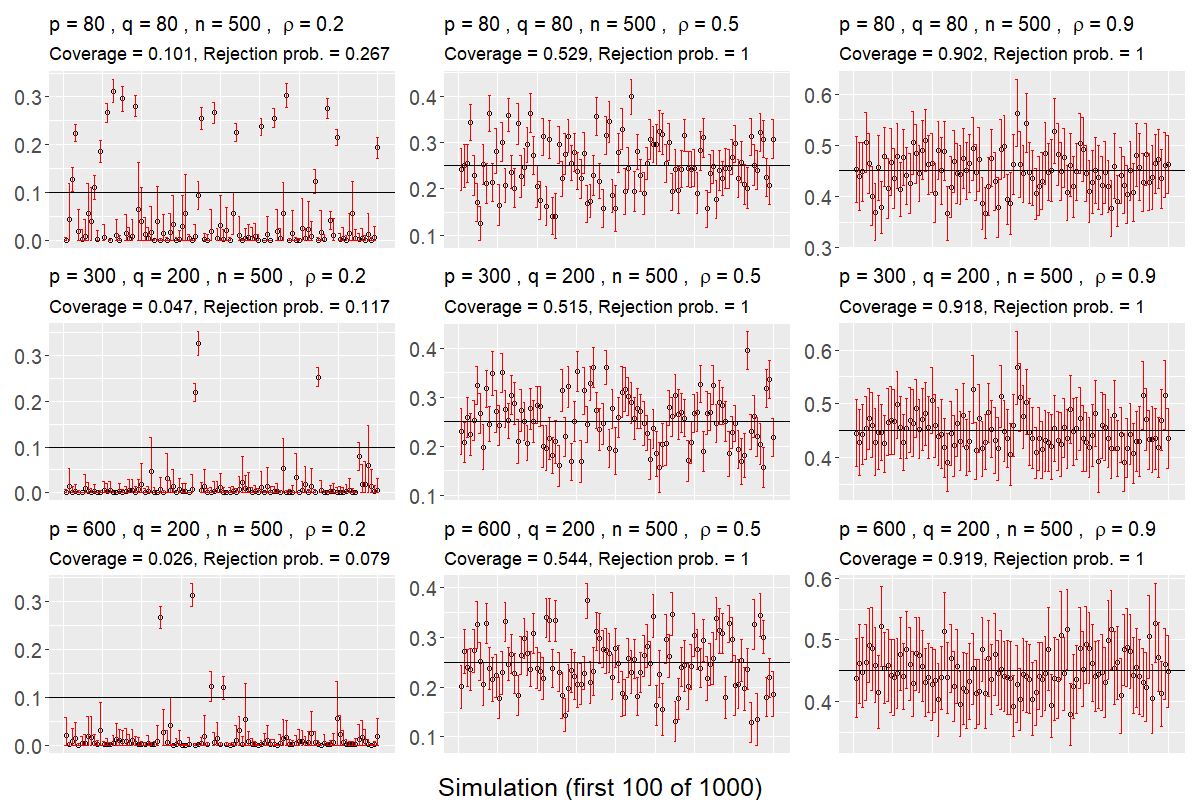}
    \caption{Ordinary confidence intervals for identity matrix}
    \end{subfigure}
    \begin{subfigure}{\textwidth}
     \centering
    \includegraphics[height= 3.5in]{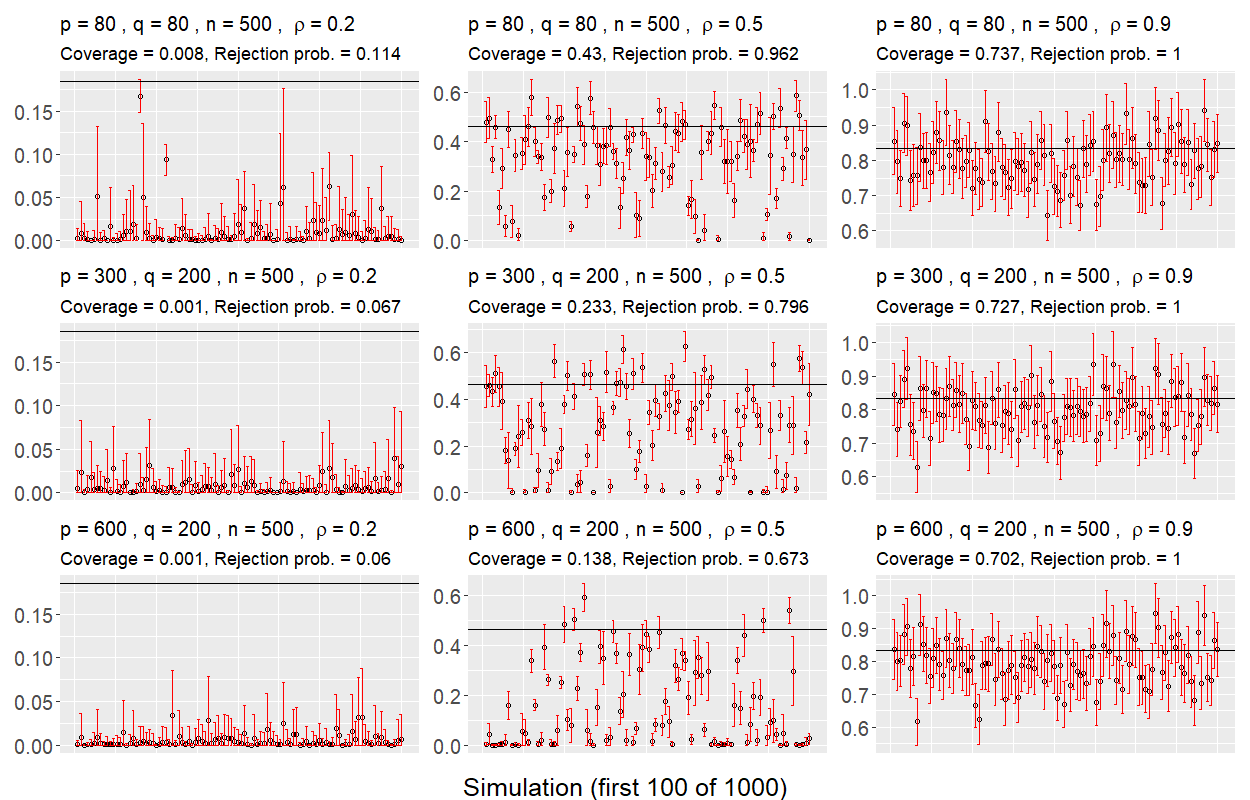}
    \caption{Ordinary confidence intervals for sparse inverse matrix}
    \end{subfigure}
    \caption{Ordinary confidence intervals for $(x^0_1)^2$}
    \label{fig: CI  for x1: ordinary}
\end{figure}
 
\begin{figure}
    \centering
    \begin{subfigure}{\textwidth}
     \centering
    \includegraphics[height=3.5 in]{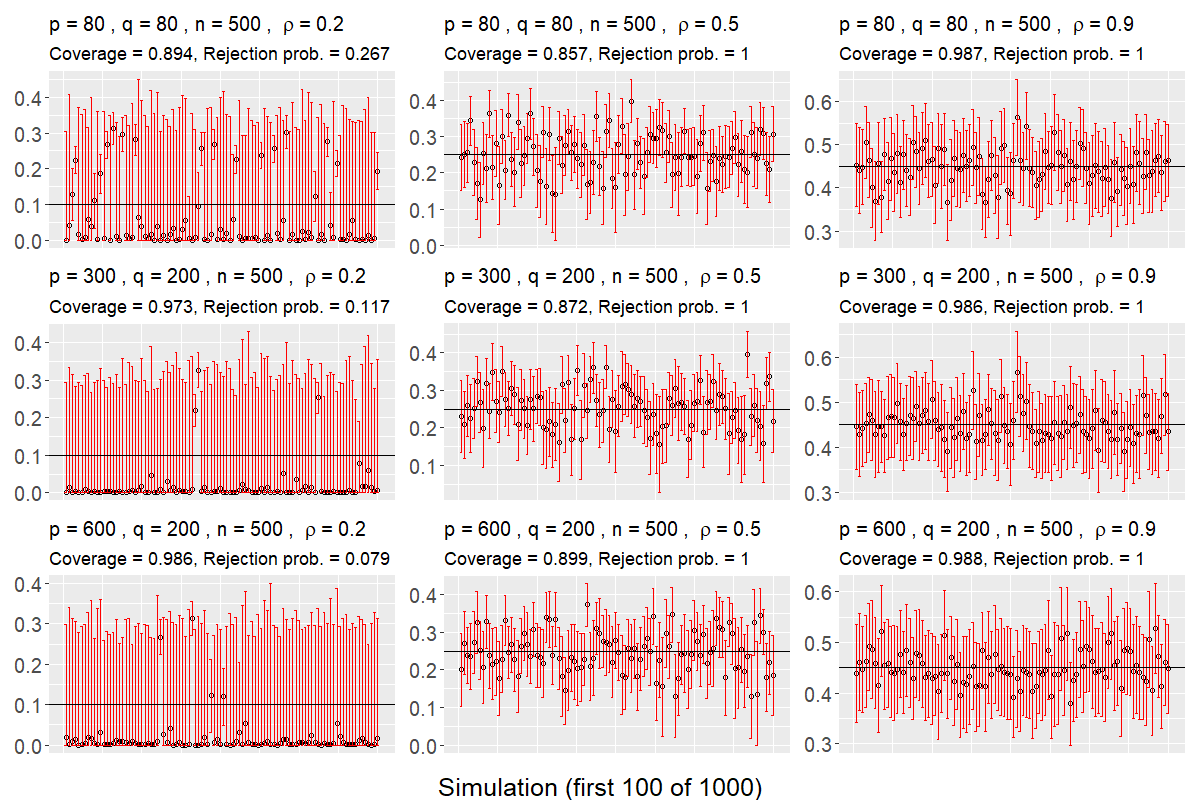}
    \caption{Conservative confidence intervals for identity matrix}
    \end{subfigure}
    \begin{subfigure}{\textwidth}
     \centering
    \includegraphics[height= 3.5in]{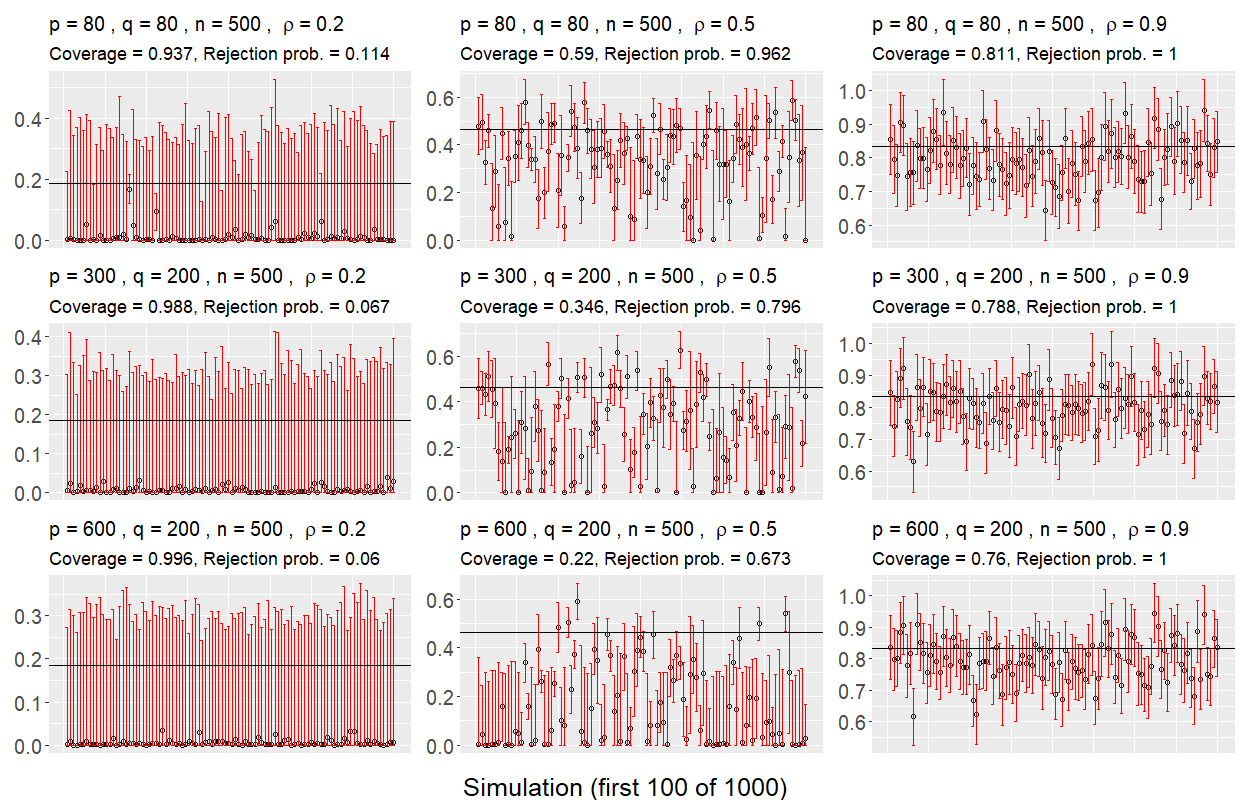}
    \caption{Conservative confidence intervals for sparse inverse matrix}
    \end{subfigure}
    \caption{Conservative confidence intervals for $(x^0_1)^2$}
    \label{fig: CI  for x1: relaxed}
\end{figure}

\section{Real Data Application}\label{sec:real_data_application}
The physiological functions in human bodies are  controlled by complex pathways, whose deregulation lead to myriad diseases. Therefore it is important to  understand  the interaction between different factors participating in these biological pathways, such as proteins, genes etc. We consider two important pathways:
(a) Cytokine-cytokine receptor interaction pathway and (b) Adipocytokine signalling pathway. Cytokines are released in response to inflammation in the body, and   pathway (a) is thus related to viral infection, cell-growth, differentiation, and cancer progression \citep{lee2017}. Pathway (b) is involved in fat metabolism and insulin resistance, thus playing a vital role in  diabetes \citep{pittas2004}. We wish to study the linear interaction between the group of genes and proteins that are involved in these pathways.  To that end, we use the 
Microarray and proteomic datasets analysed by \cite{lee2011}, which are originally  from the National Cancer Institute, and 
available at \textit{http://discover.nci.nih.gov/cellminer/}.

The dataset contains sixty human
cancer cell lines. We use $59$ of the sixty observations because one has missing microarray  information. Although the microarray data has information on many genes, we considered only those involved in pathways (a) and (b), giving $p=230$ and $62$ miRNAs, respectively. To this end, we use  \textit{https://www.genscript.com/} to get the list of genes participating in these pathways.
The dataset contains  $q=94$ proteins. We center and scale all variables prior to our analysis.

  
   Figure~\ref{fig: data: cor} indicates that  most genes and proteins have negligible correlation, which hints that only a handful of genes and proteins share linear interactions in the pathways under concern -- thus supporting the possibility of $\alk$ and $\bk$ being low dimensional. On the other hand, Figure~\ref{fig: data: var} hints at the existence of low dimensional structures in the variance matrices of both the genes and the proteins. However, it seems unlikely that they are totally uncorrelated among themselves, which questions the applicability of popular methods only suited for diagonal variance matrices, e.g. PMA \citep{witten2009}.
  
Apart from the de-biased estimators, we also look into the SCCA estimates of the leading canonical covariates using \cite{mai}, \cite{gao2017}, \cite{witten2009}, and \cite{wilms2015}'s  methods. The first three methods were implemented as discussed in Supplement~\ref{sec: extra: bias}. To apply \cite{wilms2015}'s methods, we used the code provided by the authors with the default choice of tuning parameters. Among these methods, only \cite{witten2009}'s method requires $\Sx$ and $\Sy$ to be diagonal. For these methods, we say a gene or protein is ``detected" if the corresponding loading, i.e. the estimated $(\ha)_i$ or $(\hb)_i$, is nonzero. 

We construct confidence intervals, both ordinary and conservative, and test the null that $x^0_i=0$ or $y^0_j=0$ for each $i\in[p]$ and $j\in[q]$, as  discussed in Section~\ref{Sec: simulations}.
We apply the false discovery rate corrections of Benjamini and Hochberg (BH) as well as Benjamini and Yekutieli (BY), the latter of which does not assume independent P-values. Table~\ref{tab: loadings} tabulate the number of detections by  the above-mentioned methods. Even after false discovery rate adjustment, most discoveries seem to include zero in the confidence intervals. We discussed this situation in Section~\ref{Sec: simulations}, where it was indicated that the former can occur if the signal strength is small or the sample size is insufficient. To be conservative, we consider only those genes and proteins whose ordinary interval excludes zero. These discoveries are reported in Tables~\ref{tab: genes and proteins: pathway a} and 
\ref{tab: discovery: adipo} along with the confidence intervals. The pictorial representation of the confidence intervals can be found in Figure~\ref{fig: data: cytokine} and Figures~\ref{fig: data: adipo} in Supplement~\ref{sec: extra plots}.

   Using Gene Ontology toolkit available at \emph{http://geneontology.org/}, we observe that  our discovered  from pathway (a) are mainly involved in biological processes like positive regulation of gliogenesis  and molecular function like  growth factor activity, where the selected proteins play a role in regulating membrane assembly, enzyme function, and other cellular functions.   Gene Ontology toolkit also entails that the discovered  genes from pathway (b) are involved in positive regulation of cellular processes, and molecular function like growth factor activity. The  only discovered gene in pathway (b) is ANXA2, which, according to UNIPORT at \emph{https://www.uniport.org}, is a membrane-binding protein involved in RNA binding and host-virus infection.

\begin{table}[H]
\centering
\begin{tabular}{lrrrrrr}

 Variable  & \citeauthor{mai} & \citeauthor{wilms2015} & \citeauthor{gao2017} & \citeauthor{witten2009} & DB+BH & DB+BY\\
 
  & \multicolumn{6}{c}{  Pathway (a) }
  \\

Genes  &  2 (2) & 1 (1) & 3 (3) & 41 (5) & 13 & 6\\ 
  Proteins & 4 (3) & 1 (1) & 7 (5) & 13 (5) & 36 & 22\\
 
  & \multicolumn{6}{c}{  Pathway (b) } \\
  
Genes  &  2 (1) & 1 (1) & 4 (3) & 11 (2) & 8 & 5\\ 
  Proteins & 7 (1) & 1 (1) & 9 (1) & 12 (1) & 22 & 2\\
   
\end{tabular}
\caption{Number of detections: number of non-zero loadings in different SCCA estimators and number of detections by our tests (DB) after Benjamini and Hochberg (BH) and Benjamini and Yekutieli (BY) false discovery rate correction. For the SCCA estimators, size of their intersection with DB+BY are given in parentheses.}
\label{tab: loadings}
\end{table}


\begin{table}[H]
\centering
\begin{tabular}{lrrrr}
  
  Gene & $p$-value* & 95\% CI & Relaxed CI &Discovered by\\
  
CLCF1 & 2.0E-07 & (0.055, 0.39) & (0, 0.58)&\citeauthor{witten2009}\\ 
  EGFR & 8.8E-09 & (0.11, 0.58) & (0,0.74)&\citeauthor{mai}, \citeauthor{witten2009},\\
  & & & & \citeauthor{gao2017}\\ 
  LIF & 1.6E-05 & (0.022, 0.45) & (0, 0.68)&\citeauthor{witten2009}, \citeauthor{gao2017}\\ 
  PDGFC & 1.4E-07 & (0.094, 0.64)&(0, 0.82) &\citeauthor{witten2009}\\ 
  TNFRSF12A & 7.8E-11 & (0.15, 0.60)& (0.01, 0.75)&\citeauthor{mai}, \citeauthor{witten2009}, \\
  & & & & \citeauthor{gao2017}, \citeauthor{wilms2015}\\ 
  & & & & \\
 Protein & $p$-value* & 95\% CI& Relaxed CI&Discovered by\\  
  
ANXA2 & 1.3E-15 & (0.13, 0.38)&(0.01, 0.51) &\citeauthor{mai}, \citeauthor{witten2009}, \\
& & & &\citeauthor{gao2017}, \citeauthor{wilms2015}\\ 
  CDH2 & 5.1E-09 & (0.22, 1.1)&(0.12, 1.23)&\citeauthor{mai}, \citeauthor{witten2009}, \\
  & & & & \citeauthor{gao2017}\\ 
  FN1 & 4.2E-07 & (0.96, 7.6)&(0.96, 7.6)&none\\ 
  GTF2B & 6.7E-05 & (0.034, 4.0)&(0.034, 4.0)&none\\ 
  KRT20 & 1.2E-05 & (0.015, 0.27) &(0, 0.48)&none\\ 
  MVP & 2.6E-05 & (0.021, 0.59)&(0, 0.82)& \citeauthor{witten2009} \\ 
  
\end{tabular}
\caption{Discovered genes and protein from  pathway (a). The confidence intervals are obtained using the methods described in Section~\ref{Sec: simulations}. The P-values are the original P-values before false discovery rate correction. \\ *All genes and proteins were also detected by  Benjamini and Yekutieli method.}
\label{tab: genes and proteins: pathway a}
\end{table}

\begin{table}[H]
\centering
\begin{tabular}{lrrrr}
  
  Gene & $p$-value* & 95\% CI & Relaxed CI &  Discovered by\\
  
ACSL5 & 2.9E-05 & (0.014, 0.45) &(0, 0.68) &none\\ 
RXRG & 4.1E-10 & (0.073, 0.32) & (0, 0.47)&\citeauthor{wilms2015}, \citeauthor{gao2017}, \\ 
& & & & \citeauthor{mai}\\
TNFRSF1B & 1.1E-09 & (0.49, 2.2)& (0.49, 2.2)&none\\ 
   & & & & \\
 Protein & $p$-value* & 95\% CI  & Relaxed CI &Discovered by\\

   ANXA2 & 2.7E-74 & (1.1, 1.7)& (1.1, 1.7)&none\\ 
 
\end{tabular}
\caption{Discovered genes and protein from  pathway (b).
 The confidence intervals are obtained using the methods described in Section~\ref{Sec: simulations}. The P-values are the original P-values before false discovery rate correction.
\\ *All genes and proteins were also detected by Benjamini and Yekutieli method.}
\label{tab: discovery: adipo}
\end{table}

\section{Acknowledgements}
Rajarshi Mukherjee’s research was partially
supported by NSF Grant EAGER-1941419 and
NIH Grant NIH/NIEHS P42ES030990.

\bibliographystyle{biometrika}
\bibliography{sparseCCA.bib}
\newpage
\begin{center}
 \huge{ Supplement to ``On Statistical Inference with High Dimensional Sparse CCA"}
\end{center}
  \section{ Extra Simulations}
\label{Sec: extra simulations}

\subsection{Bias estimation}
\label{sec: extra: bias}
This section compares the elementwise bias of our de-biased CCA estimator with other commonly used sparse CCA estimators.
We use the same simulation settings as in Section~\ref{Sec: simulations}. Also, the  tuning parameters for the de-biased estimators are kept exactly as in  Section~\ref{Sec: simulations}.
As competitors, we choose the 
 COLAR estimator of \cite{gao2017}, and  the SCCA of \cite{mai} and  \cite{witten2009}. Since we are in the rank one setting, the COLAR estimator coincides with the modified COLAR estimator, which is our preliminary estimator, and has already been discussed  in Section~\ref{Sec: simulations}. The SCCA of \cite{mai} is computed using the code provided by the authors, where we set  the penalty parameters lambda.alpha and lambda.beta to be $\log(p)/n$ and $\log(q)/n$, respectively.  \cite{witten2009}'s method is implemented using the  R package \texttt{PMA} with $l_1$ penalty, using the default  tuning parameters.
Finally, we consider $N=1000$ Monte Carlo replications as before. 

   Table~\ref{tab:firstabsolutebias} and Table~\ref{tab: 20} tabulate the  absolute bias and the standard deviation of $|\widehat x_i|$ for $i=1$ and $20$, respectively, estimated using the 1000 Monte Carlo samples. Recall from Section~\ref{Sec: simulations} that  $x^0_1$ is nonzero but $x^0_{20}$ is zero. 
   

\textbf{Bias in the estimation of $|x^0_1|$:}
 Table~\ref{tab:firstabsolutebias}  entails
  that the de-biased estimators of $x^0_1$ almost always outperform the remaining estimators in terms of the absolute bias, and the difference is more prominent when the signal strength is small. The only exception is the high signal strength setting, where sometimes the bias of the initial COLAR estimator is so small such that the de-biasing step does not lead to further improvement. The bias of the de-biased estimator and COLAR, in general, is close, and they exhibit the same pattern. A sharp decrease in the bias of the COLAR and the de-biased estimator can be observed at signal strength  $0.5$  and $0.9$, respectively, for identity and sparse inverse matrix. The QQ plots in Figure~\ref{fig: QQ plots for x1} and the histograms in Figure~\ref{fig: hist for x1} also reveal that the de-biased estimators  attain asymptotic normality at these signal strength.  These observations explain why the ordinary confidence intervals in Section~\ref{Sec: simulations}, which rely on Corollary~\ref{corollary: main theorem single}, have poor coverage at lower signal strength in the above cases. 
 In the sparse inverse case, the bias of \cite{witten2009}'s estimator stays substantially high, and  increases with  the  signal strengths for high $p$, $q$. This is unsurprising because \cite{witten2009}'s method is best suited for diagonal covariance matrices.
 
 \textbf{Bias in the estimation of $|x^0_{20}|$:}
 In this case, the SCCA estimators have much smaller bias than our de-biased estimator, which is expected because sparse estimators would generally set this co-ordinate to zero. However, as the QQ plots in Figure~\ref{fig: QQ plots for x20} indicate, the de-biased estimator attains asymptotic normality pretty quickly, even at low signal strength, while the initial COLAR estimator stays quite non-normal unless the signal strength is high. This observation explains the satisfactory performance of the confidence intervals for $x^0_{20}$.

\begin{table}[H]
\def~{\hphantom{0}}
{%
\begin{tabular}{lcccccc}

 Method & \multicolumn{3}{c}{Identity} & \multicolumn{3}{c}{Sparse Inverse}\\

  & \multicolumn{6}{c}{$ n=500,\ p=80,\  q=80$} \\
  &$\rho_0 = 0\cdot 2$&$\rho_0 = 0\cdot 5$ &$\rho_0 = 0\cdot 9$ &$\rho_0 = 0\cdot 2 $  &$\rho_0 = 0\cdot 5$  &$\rho_0 = 0\cdot 9$\\
 
 PMA & 28 (18) & 9.1 (15) & 2.9 (3.2) & 29 (10) & 43 (22) & 31 (23)\\
 \citeauthor{mai} & 28 (15) & 5.0 (6.3) & 2.5 (3.1) & 30 (10) & 37 (37) & 54 (6.4)\\
 COLAR & 28 (19) & 5.0 (6.4) & 2.2 (2.8) & 29 (25) & 7.3 (9.6) & 3.0 (2.6)\\
 Db & 21 (23) & 4.5 (5.6) & 2.1 (2.6) & 25 (24) & 6.0 (8.7) & 2.3 (2.5)
  \\
 \\
 & \multicolumn{6}{c}{$ n=500,\ p=300,\  q=200$} \\
 PMA & 29 (16) & 33 (29) & 8.5 (16) & 29 (13) & 47 (21) & 57 (36) \\
 \citeauthor{mai} & 31 (9.3) & 5.2 (6.4) & 2.5 (3.1) & 30 (5.6) & 39 (40) & 54 (7.4)\\
 COLAR & 30 (15) & 5.3 (6.6) & 2.2 (2.7) & 29 (27) & 42 (38) & 3.2 (2.5) \\
 Db & 24 (25) & 4.7 (5.8) & 2.1 (2.6) & 26 (26) & 39 (40) & 2.3 (2.4)
  \\
  \\
 & \multicolumn{6}{c}{$ n=500,\ p=600,\  q=200$} \\
  PMA &30 (15) & 40 (30) & 17 (27) & 42 (20) & 67 (34) & 87 (49)\\
 \citeauthor{mai} & 31 (8.3) & 4.7 (6.1) & 2.4 (3.0) & 43 (9.7) & 44 (34) & 3.8 (4.6)\\
 COLAR & 31 (11) & 5.0 (6.3)& 2.1 (2.7) & 43 (20) & 42 (49) & 3.9 (4.2)\\
 Db & 25 (26) & 4.3 (5.4) & 2.0 (2.5) & 37 (37) & 35 (40) & 3.4 (3.7)\\
 
\end{tabular}}
\caption{Table of the estimated bias of $|\widehat x_1|$. The standard deviation estimate is given in the parentheses. The bias and the standard error is estimated from 1000 Monte Carlo samples. All entries are  scaled by $10^{-2}$. Here PMA: Penalized Multivariate Analysis \cite{witten2009}; Db: The de-biased estimator.}
\label{tab:firstabsolutebias}
\end{table}

\begin{table}[H]
\centering
\def~{\hphantom{0}}
{%
\begin{tabular}{lcccccc}

 Method & \multicolumn{3}{c}{Identity} & \multicolumn{3}{c}{Sparse Inverse}\\

  & \multicolumn{6}{c}{$ n=500,\ p=80,\  q=80$} \\
 
  &$\rho_0 = 0\cdot 2$&$\rho_0 = 0\cdot 5$ &$\rho_0 = 0\cdot 9$ &$\rho_0 = 0\cdot 2 $  &$\rho_0 = 0\cdot 5$  &$\rho_0 = 0\cdot 9$\\
  
 PMA &  2.5 (8.3)& 1.6 (4.0)&1.8 (3.2)& 2.9 (9.3) & 2.5 (8.5) & 1.1 (5.0)\\
 \citeauthor{mai} &  1.3 (6.7)& 0.05 (0.5)&0 (0)& 0.98 (5.2) & 0.32 (2.7) & 0 (0)\\
 COLAR &   1.0 (6.2)& 0.02 (0.34)& 0 (0) & 0.59 (4.2)& 0.06 (1.4) & 0 (0)\\
 Db & 7.6(10)  &  4.4 (5.5)&1.7 (2.1) & 6.9 (9.2) & 4.6 (5.9) & 1.7 (2.1)
 \\
\\
 & \multicolumn{6}{c}{$ n=500,\ p=300,\  q=200$} \\

PMA & 1.7 (6.1)&1.6 (5.2) & 1.5 (3.3)& 1.9 (6.5) & 1.9 (6.5) &1.9 (6.2) \\
 \citeauthor{mai} & 0.40 (3.6)&0.01 (0.24) &0 (0) &0.31 (2.6) & 0.07 (0.91) & 0 (0)\\
COLAR & 0.31 (3.5)& 0 (0.07)& 0 (0) & 0.16 (2.2)& 0.07 (1.5) & 0 (0) \\
 Db &  6.8 (8.9) &4.2 (5.3) & 1.6 (2.0) & 6.3 (8.0)& 5.0 (6.5)& 1.6 (2.1)\\
 \\
 & \multicolumn{6}{c}{$ n=500,\ p=600,\  q=200$} \\
 
 PMA & 1.3 (4.1) & 1.3 (3.9) &1.2 (3.2)& 1.8 (5.9)& 1.8 (6.0)& 1.8 (5.9)\\
 \citeauthor{mai} & 0.28 (2.8) & 0.02 (0.3)& 0 (0)&0.15 (1.2)& 0.12 (1.1)& 0 (0)\\
 COLAR & 0.17 (2.3) &0 (0.05) & 0 (0) & 0.17 (1.9)& 0.09 (1.4) & 0 (0)\\
 Db &  6.3 (8.0) & 4.2 (5.3) & 1.6 (2.0) & 6.1 (7.7) & 5.1 (6.6) & 1.7 (2.1)\\
 
\end{tabular}}
\caption{Table of the estimated bias of $|\widehat x_{20}|$. The standard deviation estimate is given in the parentheses. The bias and the standard error is estimated from 1000 Monte Carlo samples. All entries are  scaled by $10^{-2}$. PMA: Penalized Multivariate Analysis \cite{witten2009}; Db: The de-biased estimator.}
\label{tab: 20}
\end{table}
\FloatBarrier
\subsection{Extra plots: simulation}
\label{sec: extra plots}

\begin{figure}[H]
    \centering
    \begin{subfigure}{\textwidth}
     \centering
    \includegraphics[height=3.6 in]{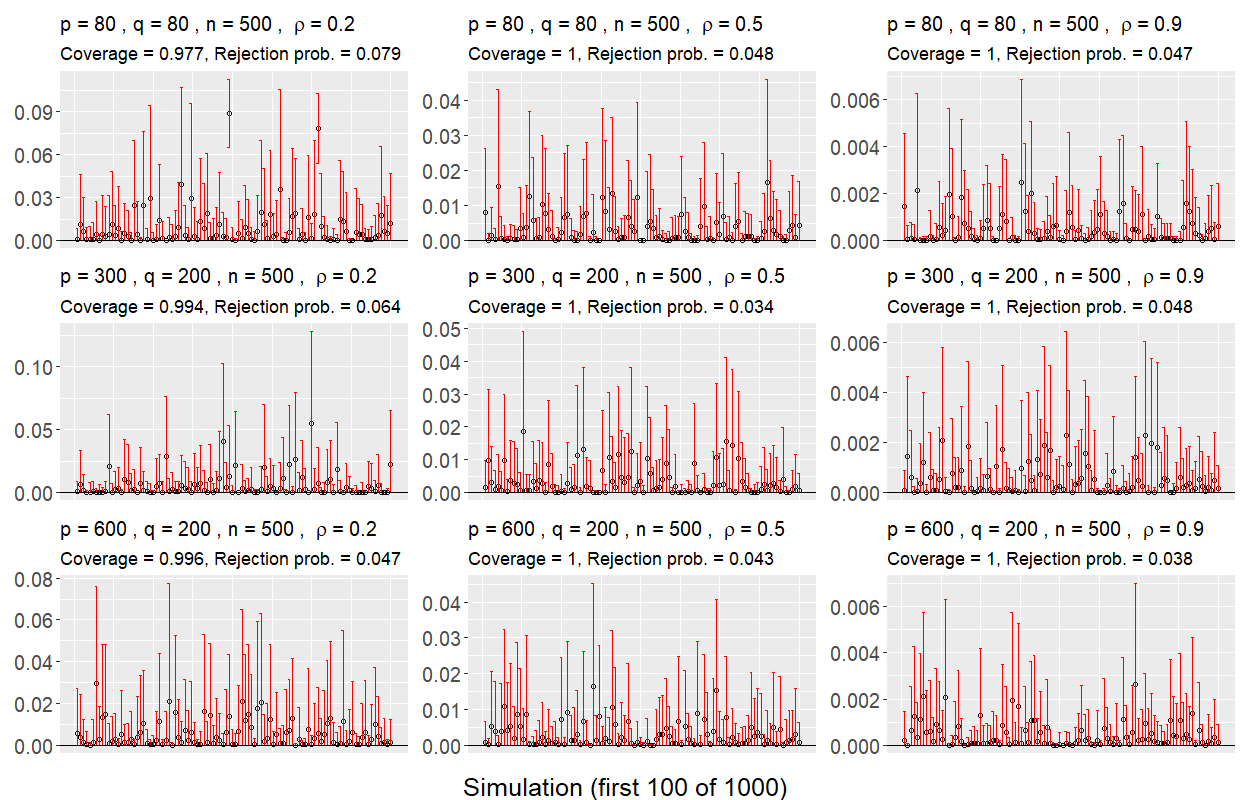}
    \caption{Ordinary confidence intervals for identity matrix}
    \end{subfigure}
    \begin{subfigure}{\textwidth}
     \centering
    \includegraphics[height= 3.6in]{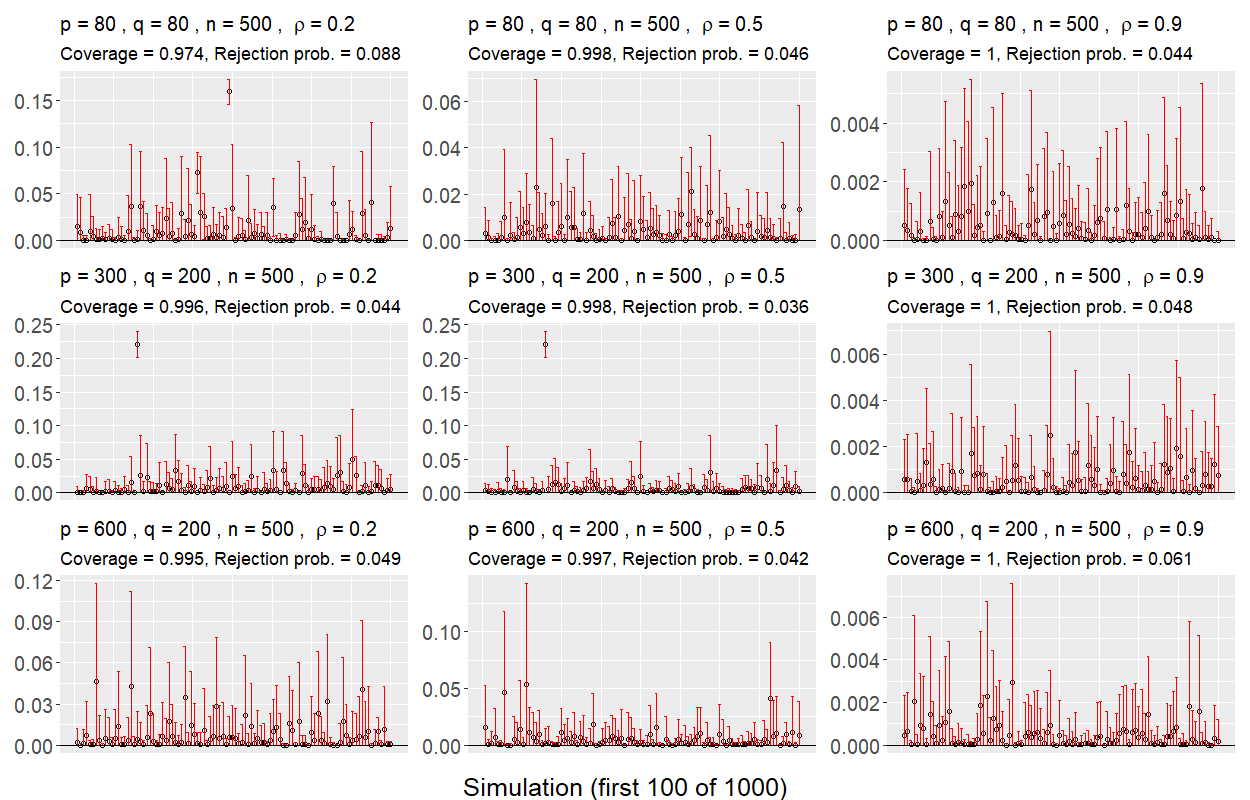}
    \caption{Ordinary confidence intervals for sparse inverse matrix}
    \end{subfigure}
    \caption{Ordinary confidence intervals for $(x^0_{20})^2$}
    \label{fig: CI  for x20: ordinary}
\end{figure}
 
\begin{figure}[H]
    \centering
    \begin{subfigure}{\textwidth}
     \centering
    \includegraphics[height=4 in]{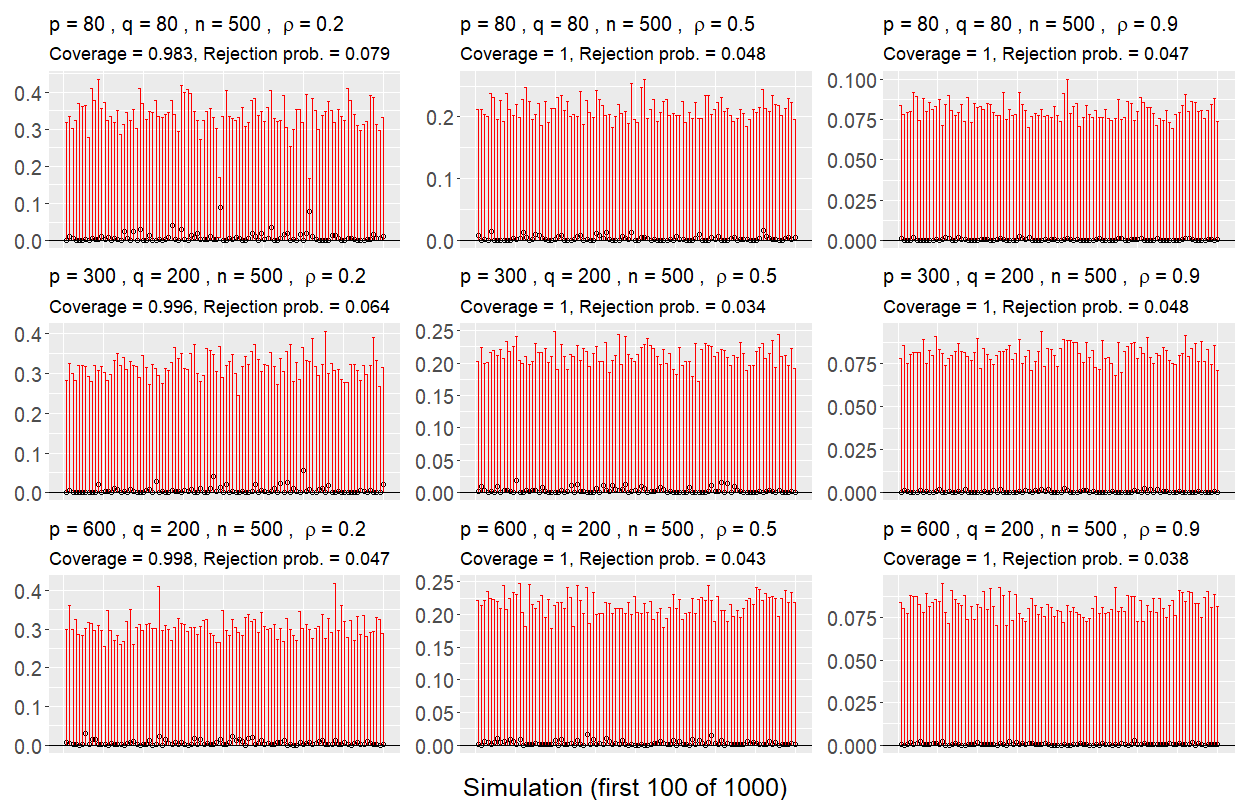}
    \caption{Conservative confidence intervals for identity matrix}
    \end{subfigure}
    \begin{subfigure}{\textwidth}
     \centering
    \includegraphics[height= 4in]{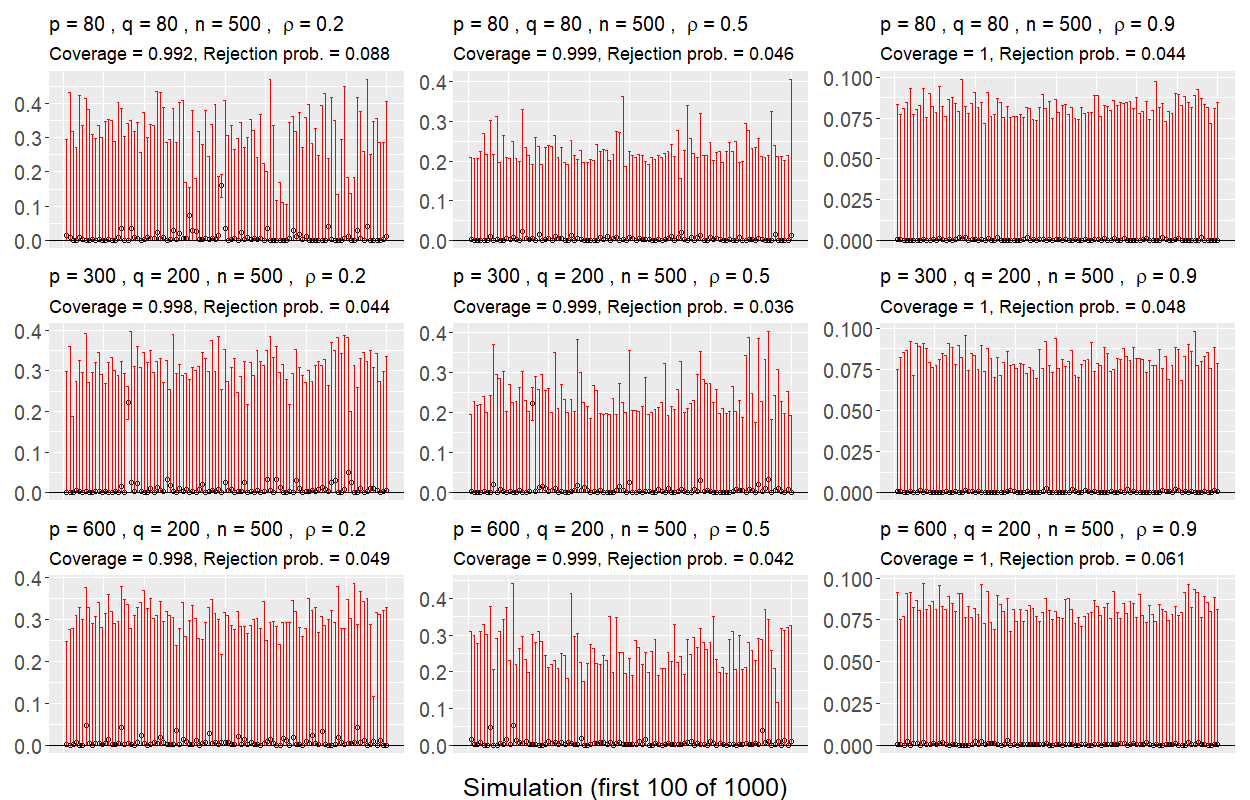}
    \caption{Conservative confidence intervals for sparse inverse matrix}
    \end{subfigure}
    \caption{Conservative confidence intervals for $(x^0_{20})^2$}
    \label{fig: CI  for x20: relaxed}
\end{figure}

\begin{figure}[H]
    \centering
    \begin{subfigure}{\textwidth}
     \centering
    \includegraphics[height=2.5 in]{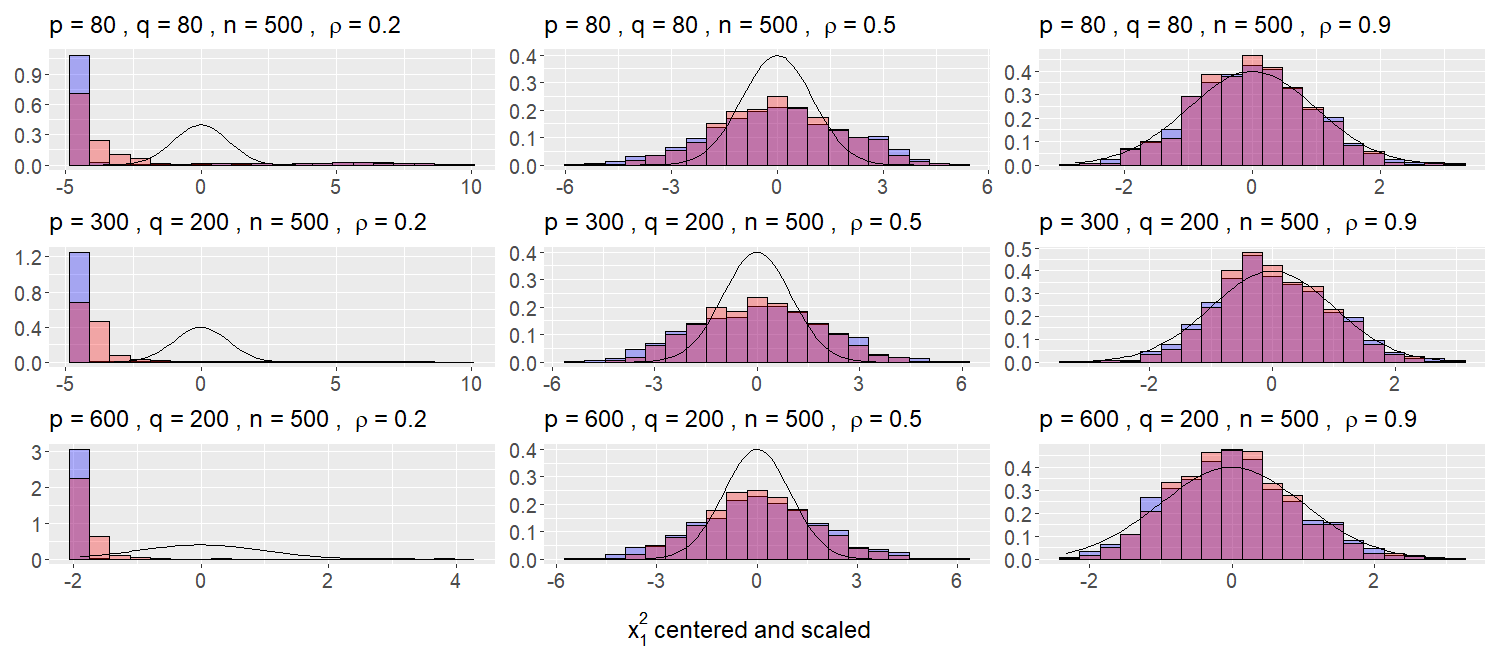}
    \caption{Identity matrix}
    \end{subfigure}
    \begin{subfigure}{\textwidth}
     \centering
    \includegraphics[height= 2.5in]{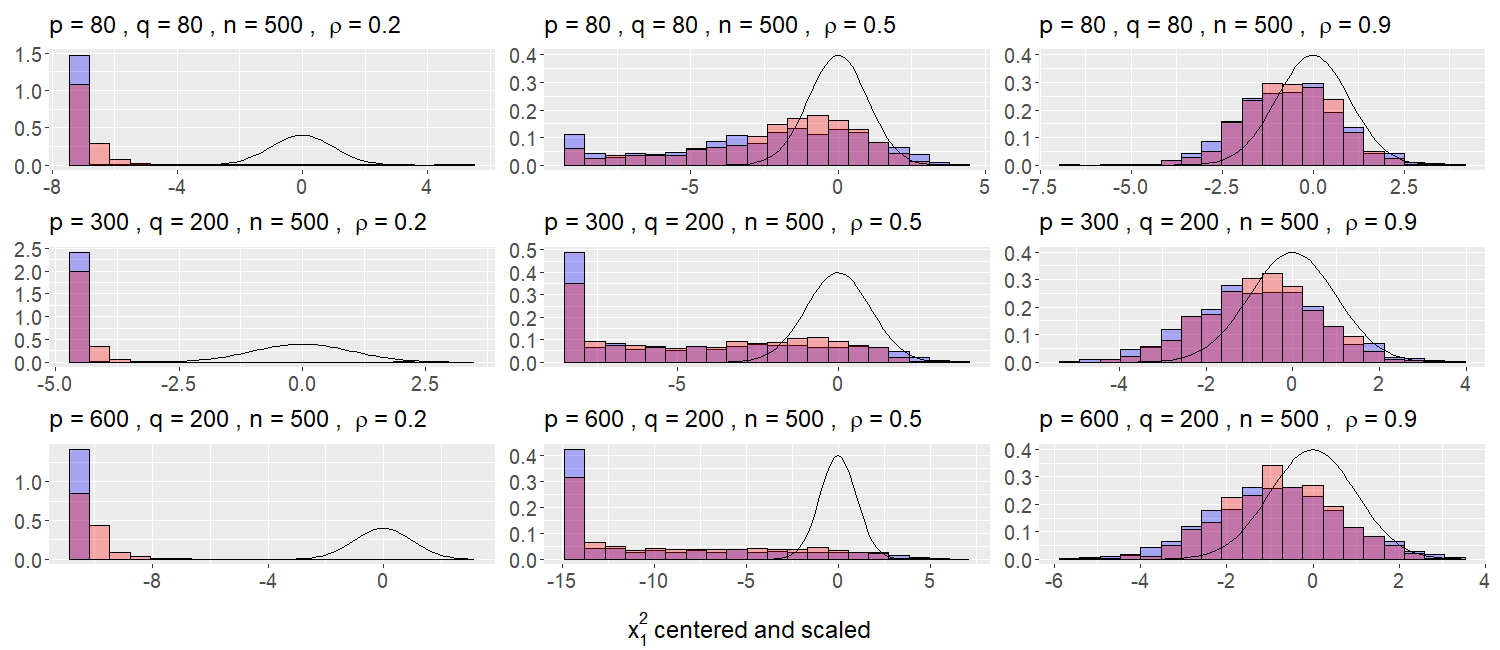}
    \caption{Sparse inverse matrix}
    \end{subfigure}
    \caption{Histograms of $(\widehat x^0_1)^2$: the estimates were centered by $(x^0_1)^2$ and scaled by\\ $4|x^0_1|\sigma_i n^{-1/2} $, where $\sigma_i$ is as in Theorem~\ref{thm: for alpha}. Preliminary estimates in blue and de-biased versions in red. A standard normal curve is imposed.}
    \label{fig: hist for x1}
\end{figure}
 \begin{figure}[H]
    \centering
    \begin{subfigure}{.49\textwidth}
     \centering
    \includegraphics[width=\textwidth]{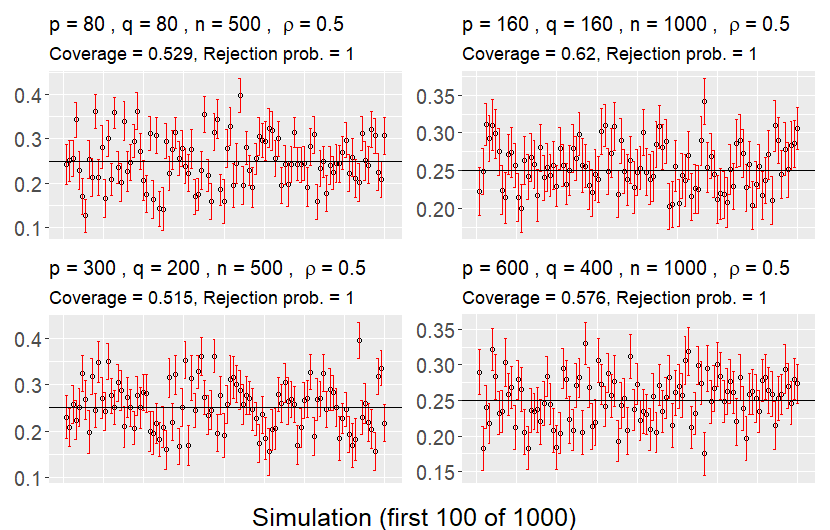}
    \caption{Ordinary intervals}
    \end{subfigure}
    \begin{subfigure}{.49\textwidth}
     \centering
    \includegraphics[width=\textwidth]{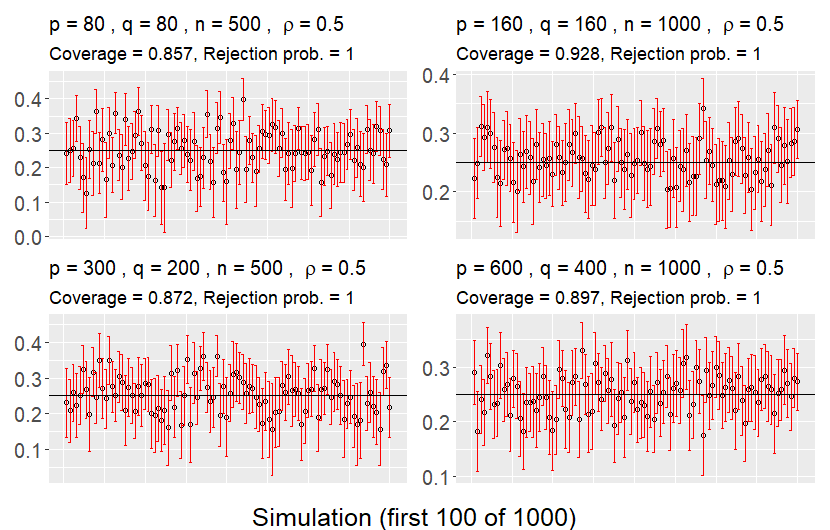}
    \caption{Conservative intervals}
    \end{subfigure}
    \caption{Effect of doubling $(p,q,n)$: note that the coverage of both ordinary and conservative confidence intervals increase. Here the underlying covariance matrices are taken to be identity. The coverage and the rejection probability of the tests are calculated using 1000 Monte Carlo samples.}
    \label{fig: double}
\end{figure}

\begin{figure}[H]
    \centering
    \begin{subfigure}{.49\textwidth}
     \centering
    \includegraphics[height=4.1in]{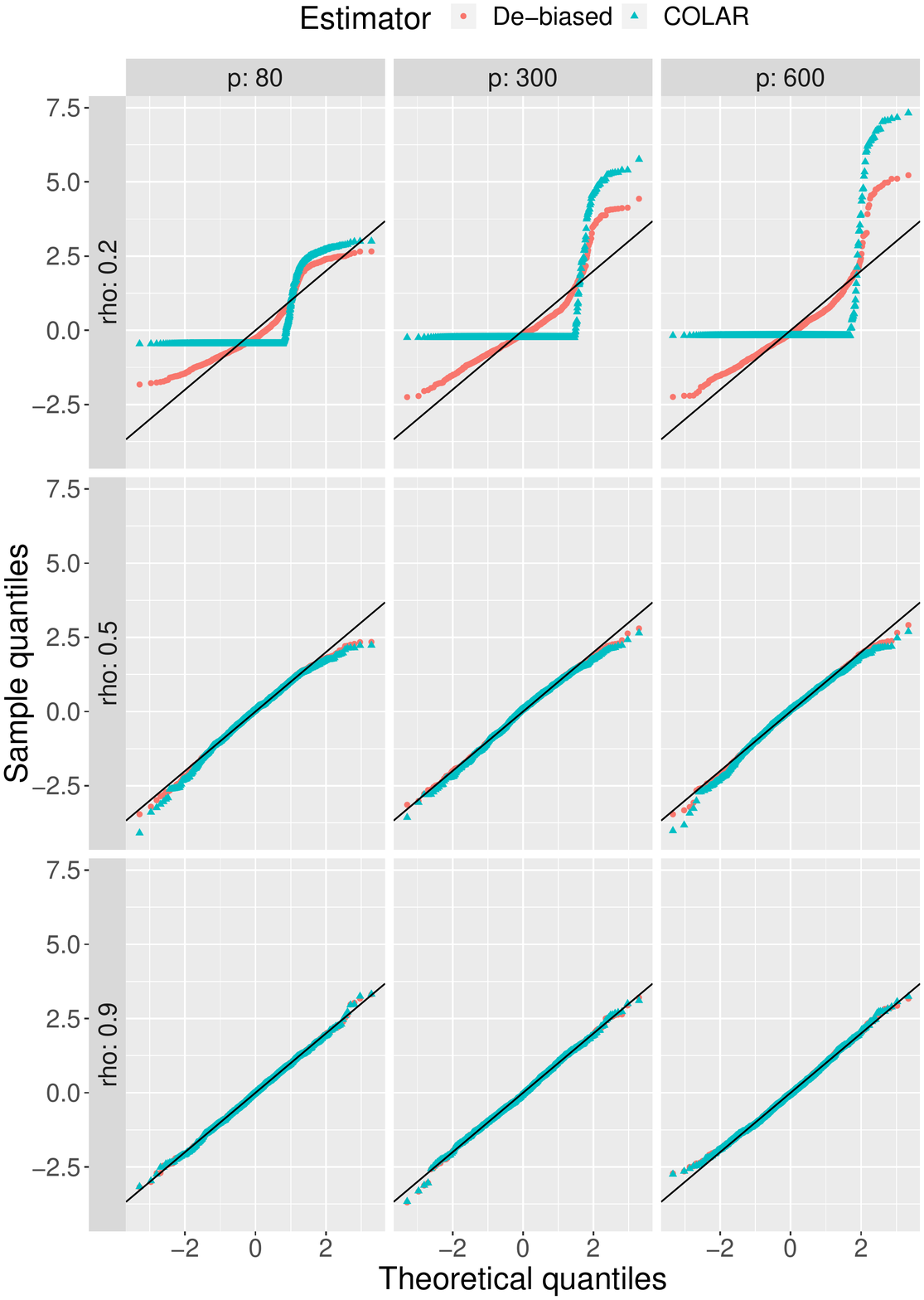}
    \caption{QQ-plots for identity matrix}
    \end{subfigure}
    \begin{subfigure}{.49\textwidth}
     \centering
    \includegraphics[height=4.1in]{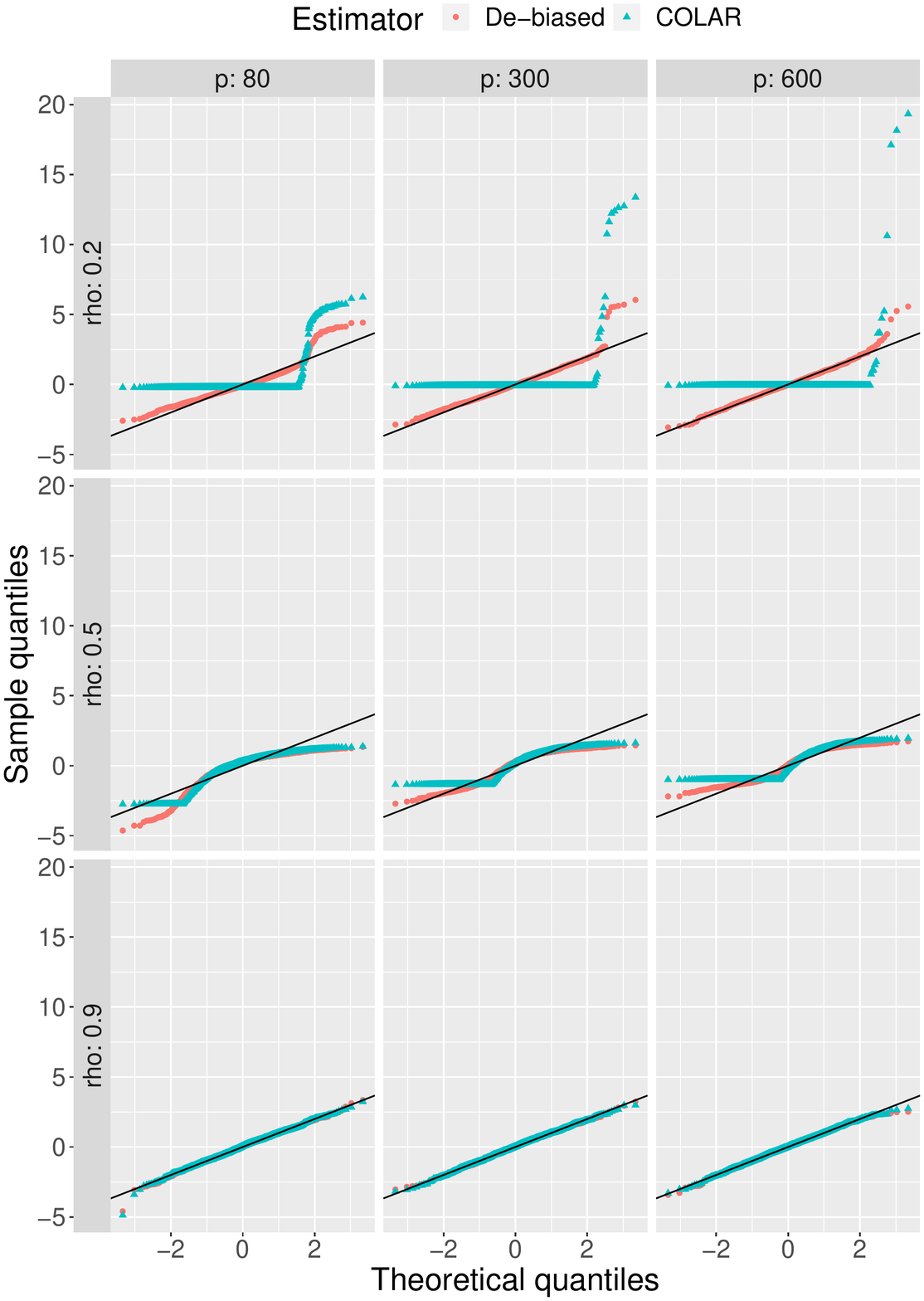}
    \caption{QQ-plots for sparse inverse matrix}
    \end{subfigure}
    \caption{QQ plots for  $\widehat {x}_{1}^2$.}
    \label{fig: QQ plots for x1}
\end{figure}

\begin{figure}[H]
    \centering
    \begin{subfigure}{.49\textwidth}
     \centering
    \includegraphics[height=4.1in]{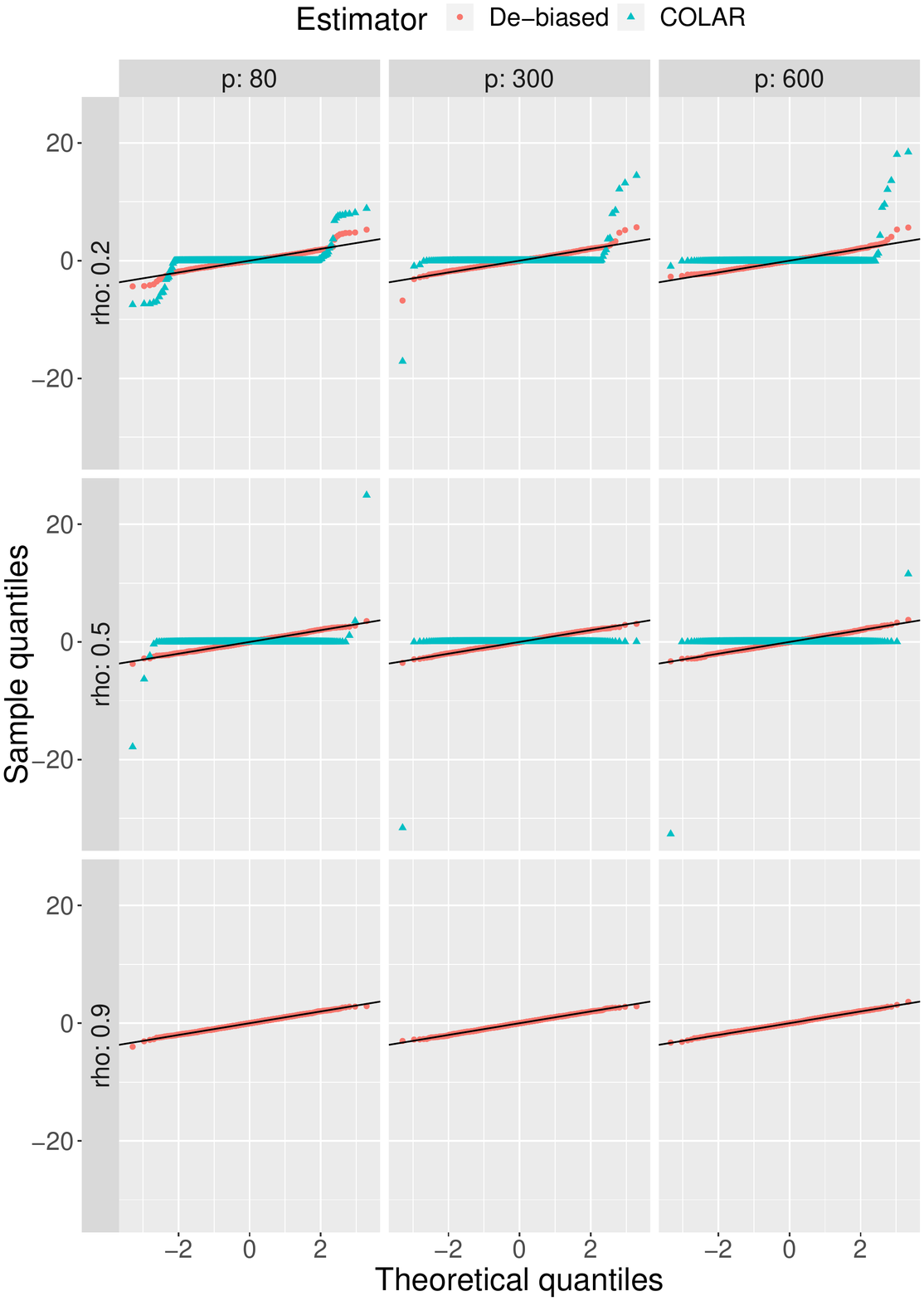}
    \caption{QQ-plots for identity matrix}
    \end{subfigure}
    \begin{subfigure}{.49\textwidth}
     \centering
    \includegraphics[height=4.1in]{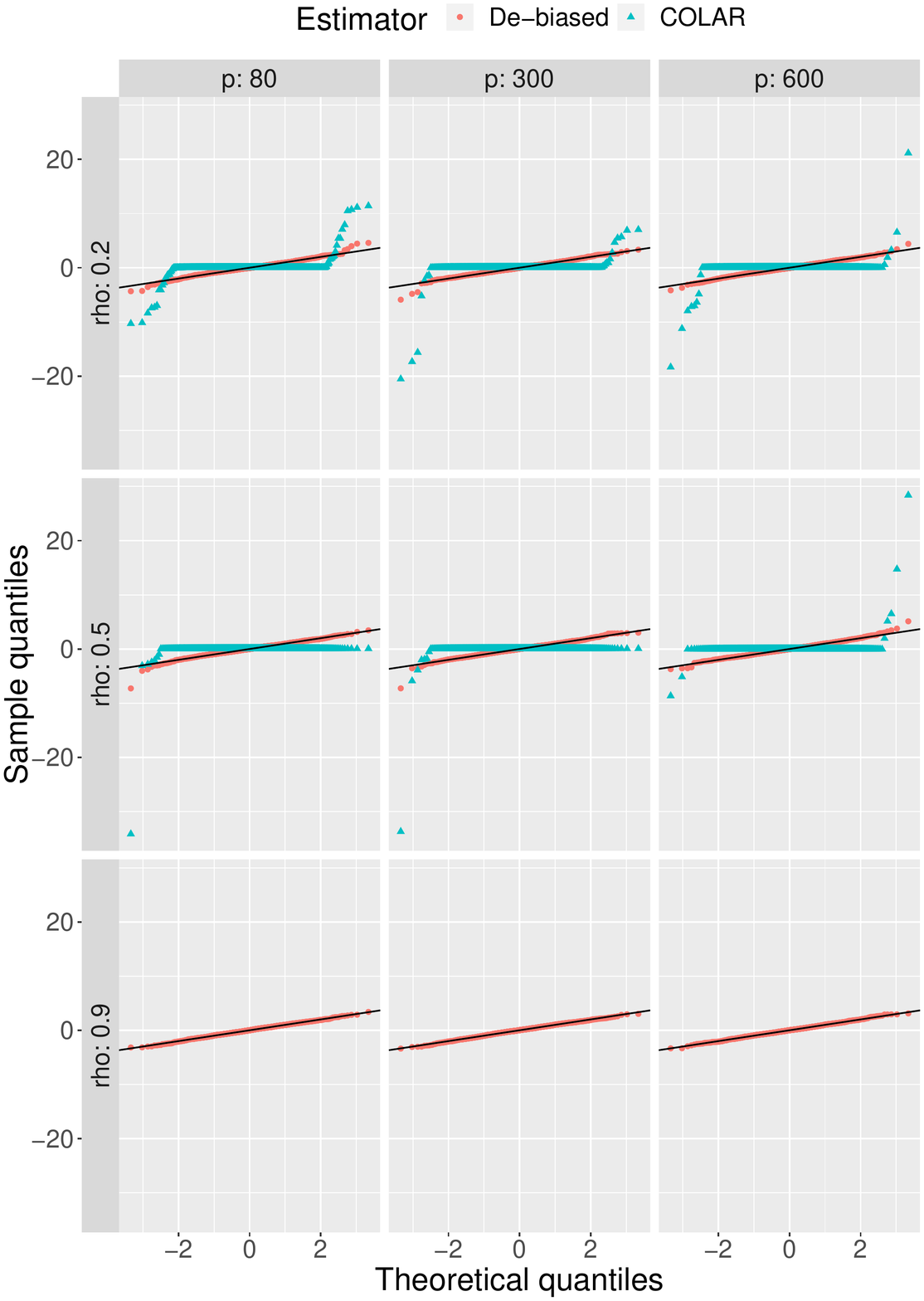}
    \caption{QQ-plots for sparse inverse matrix}
    \end{subfigure}
    \caption{QQ plots for  $\widehat {x}_{20}^2$.}
    \label{fig: QQ plots for x20}
\end{figure}

\subsection{Extra plot: data application}
\begin{figure}[H]
    \centering
     \begin{subfigure}{\textwidth}
     \centering
    \includegraphics[height=2.0 in]{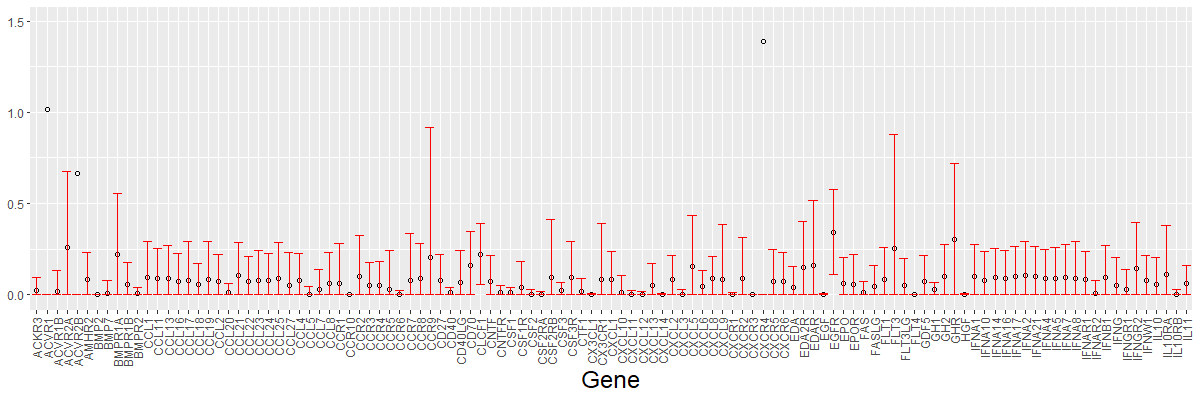}
    \caption{Confidence intervals for gene measurements: first half of the genes}
    \end{subfigure}
    \begin{subfigure}{\textwidth}
     \centering
    \includegraphics[height=2.0 in]{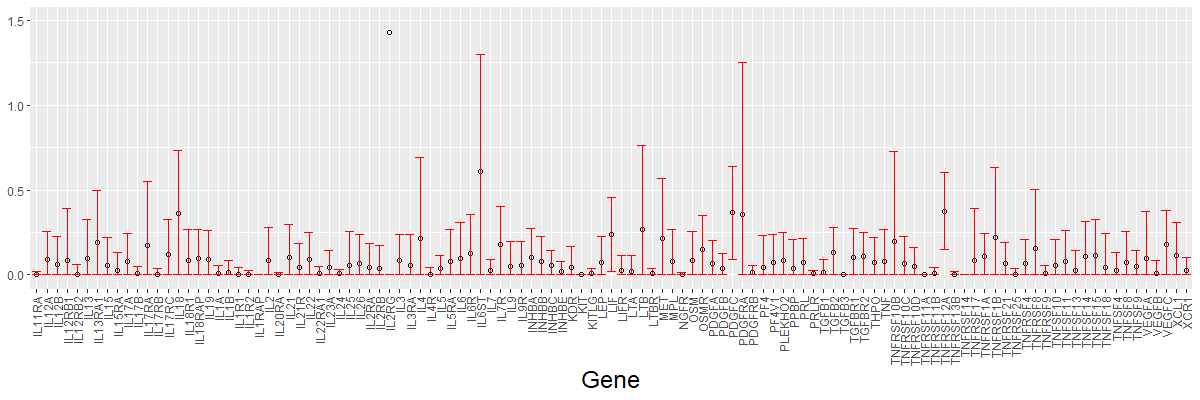}
    \caption{Confidence intervals for gene measurements: second half of the genes}
    \end{subfigure}
    \begin{subfigure}{\textwidth}
     \centering
    \includegraphics[height= 2.0in]{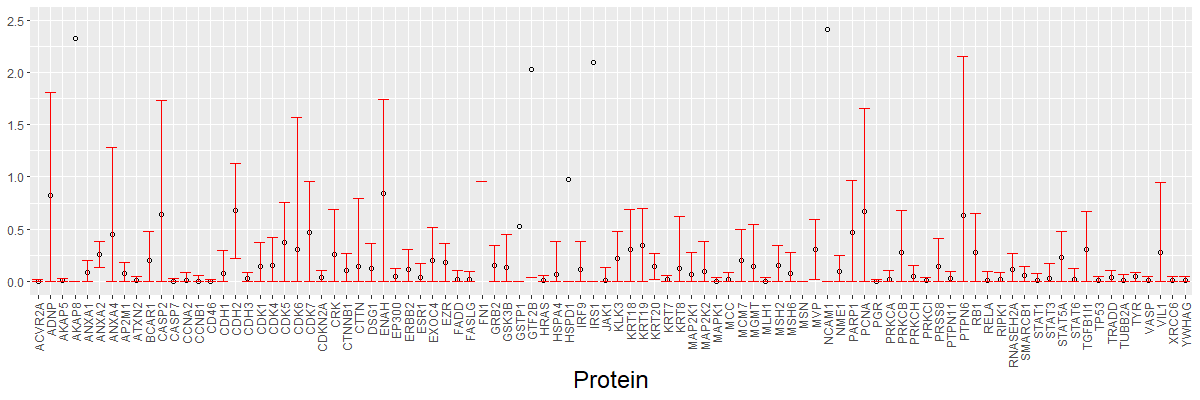}
    \caption{Confidence intervals for protein measurements}
    \end{subfigure}
    \caption{Confidence intervals for pathway (a) Cytokine-Cytokine receptor interaction pathway. }
    \label{fig: data: cytokine}
\end{figure}

\begin{figure}[H]
    \centering
    \begin{subfigure}{\textwidth}
     \centering
    \includegraphics[height=2.2 in]{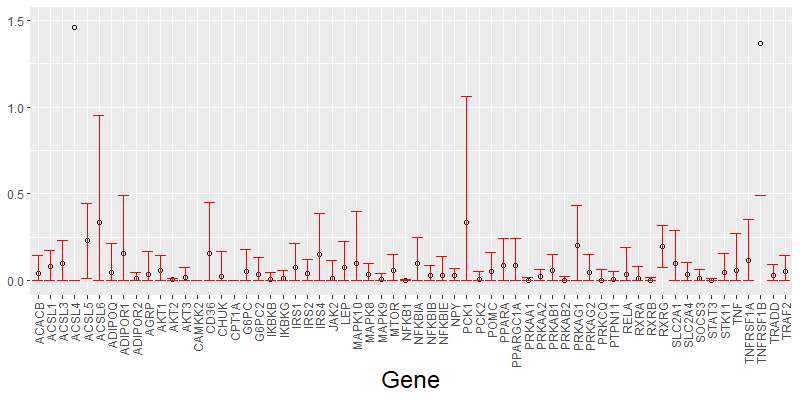}
    \caption{Confidence intervals for gene measurements}
    \end{subfigure}
    \begin{subfigure}{\textwidth}
     \centering
    \includegraphics[height= 2.2in]{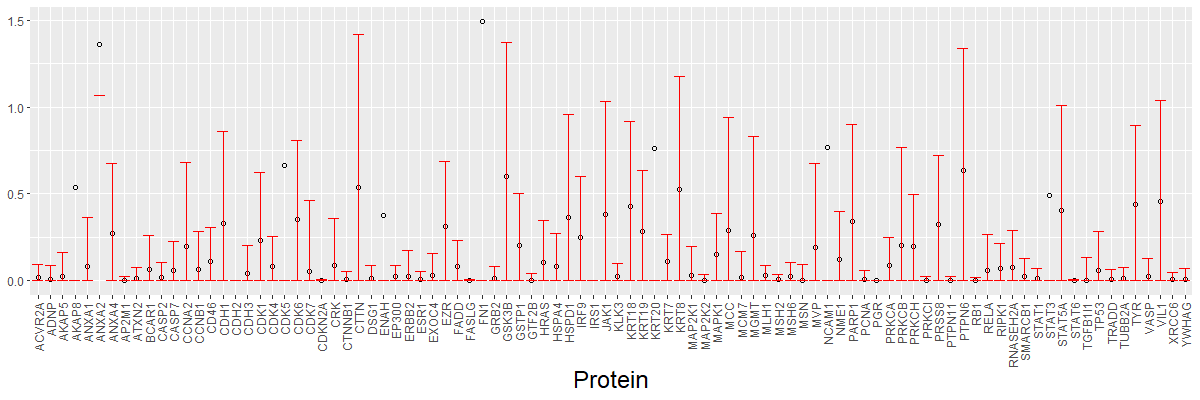}
    \caption{Confidence intervals for protein measurements}
    \end{subfigure}
    \caption{Confidence intervals for pathway (b), i.e.  Adipocytokine  signal pathway. }
    \label{fig: data: adipo}
\end{figure} 

\begin{figure}[H]
    \centering
    \begin{subfigure}{.88\textwidth}
     \centering
    \includegraphics[width=.55\textwidth]{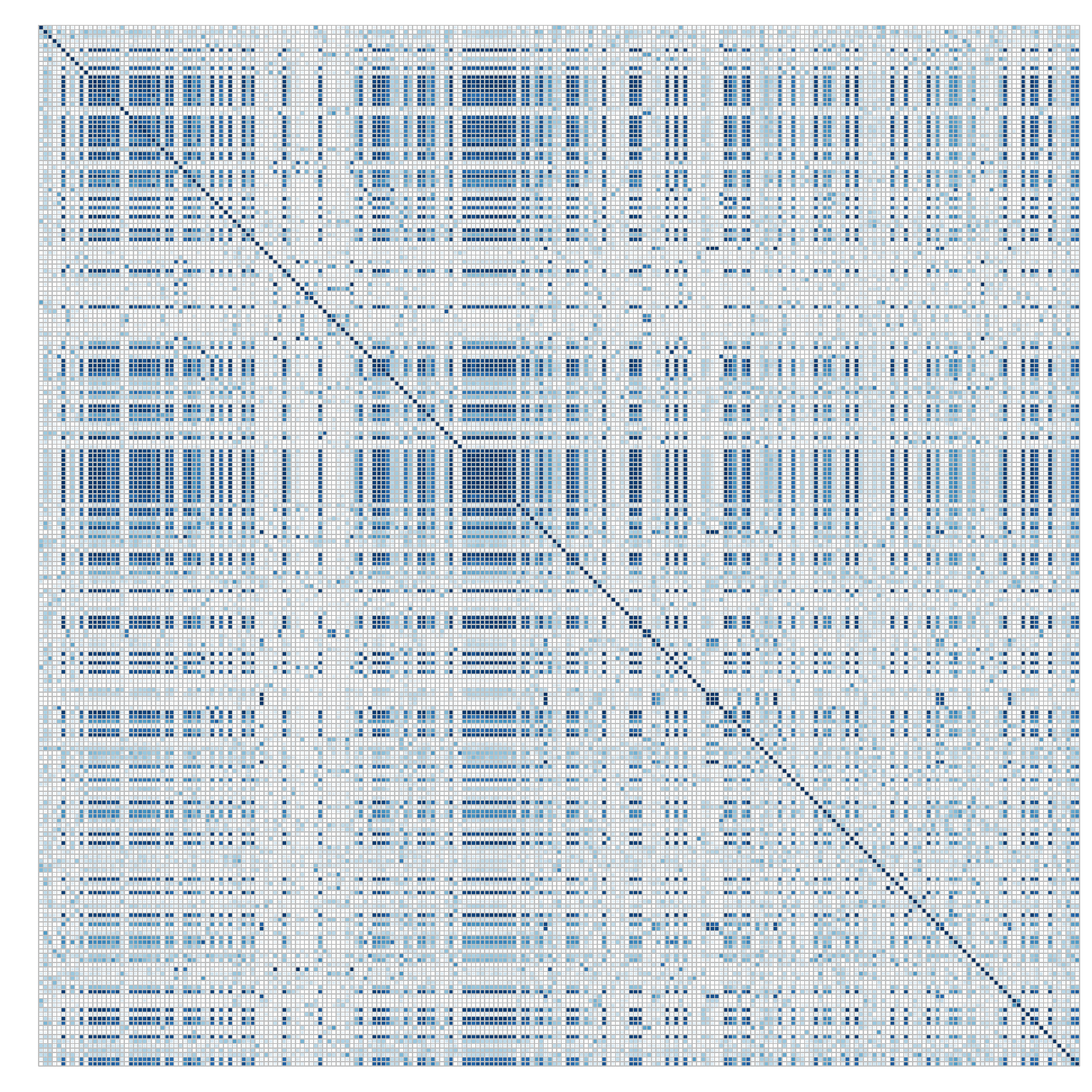}
    \caption{Genes in pathway (a)}
    \end{subfigure}
     \begin{subfigure}{.1\textwidth}
    \includegraphics[height=1.5in]{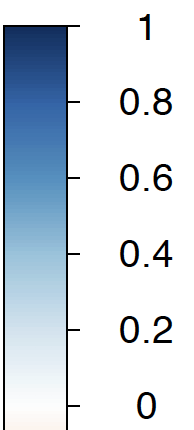}
    \end{subfigure}
    \begin{subfigure}{.49\textwidth}
     \centering
    \includegraphics[width=\textwidth]{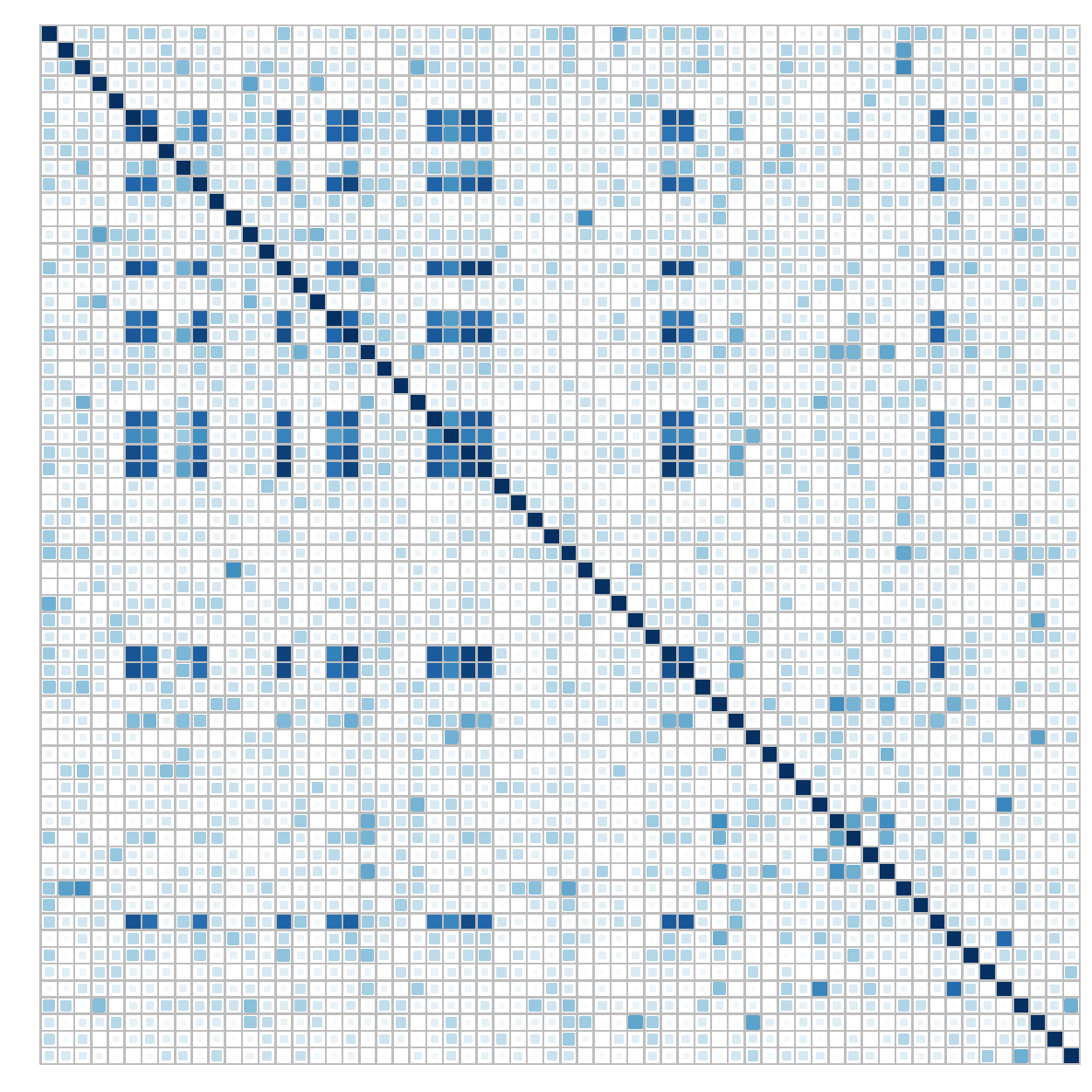}
    \caption{Genes in  pathway (b)}
    \end{subfigure}
    \begin{subfigure}{.49\textwidth}
     \centering
    \includegraphics[width=\textwidth]{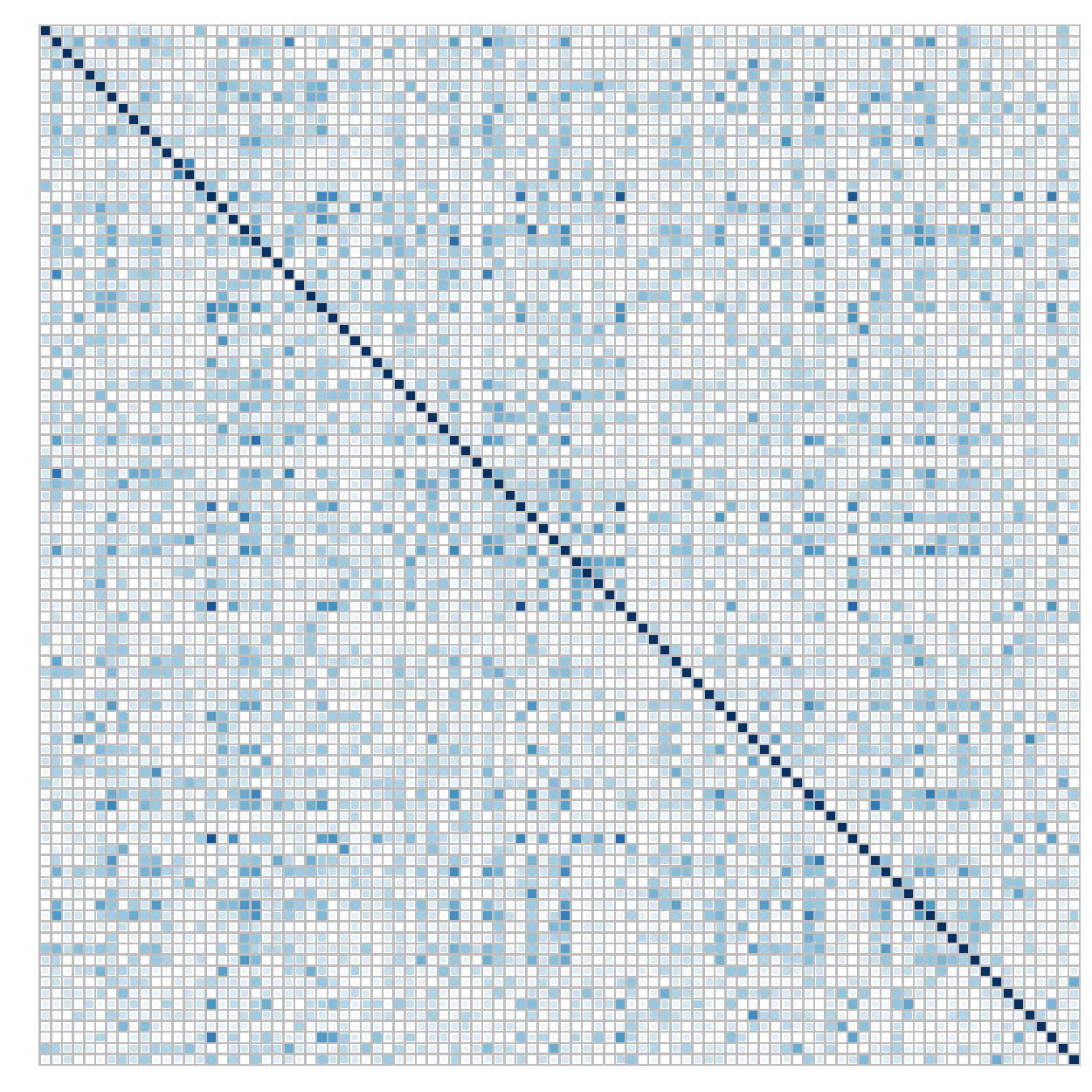}
    \caption{Proteins}
    \end{subfigure}
    
    \caption{Variance plot (in absolute values) of genes and proteins: here darker color means higher correlation. The color scale is given to the right.}
    \label{fig: data: var}
\end{figure}

\begin{figure}[H]
    \centering
    \begin{subfigure}{.49\textwidth}
     \centering
    \includegraphics[width=\textwidth]{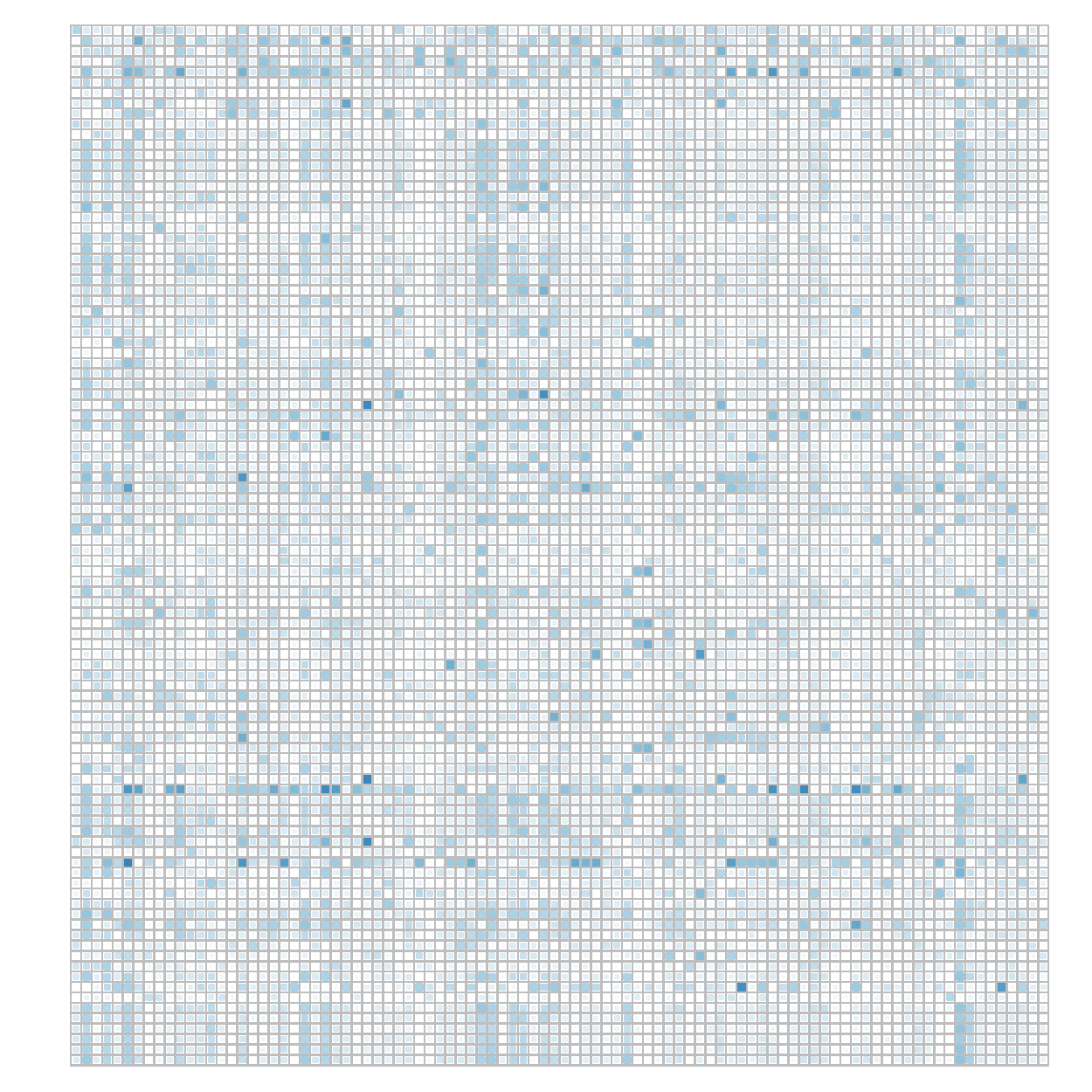}
    \caption{Pathway (a): first 100 genes}
    \end{subfigure}
    \begin{subfigure}{.49\textwidth}
     \centering
    \includegraphics[width=\textwidth]{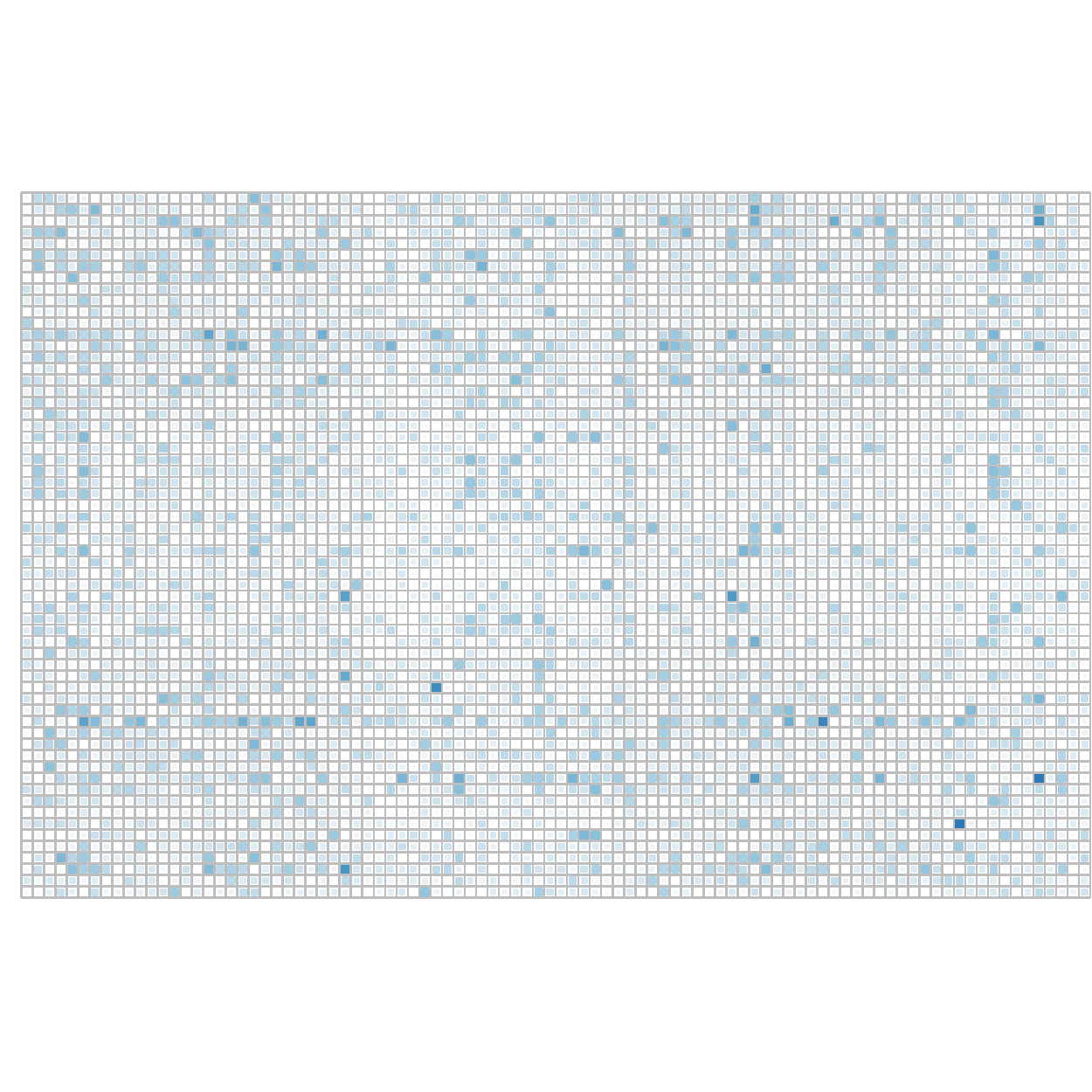}
    \caption{Pathway (b)}
    \end{subfigure}\\
     \begin{subfigure}{.8\textwidth}
     \centering
    \includegraphics[width=.7\textwidth]{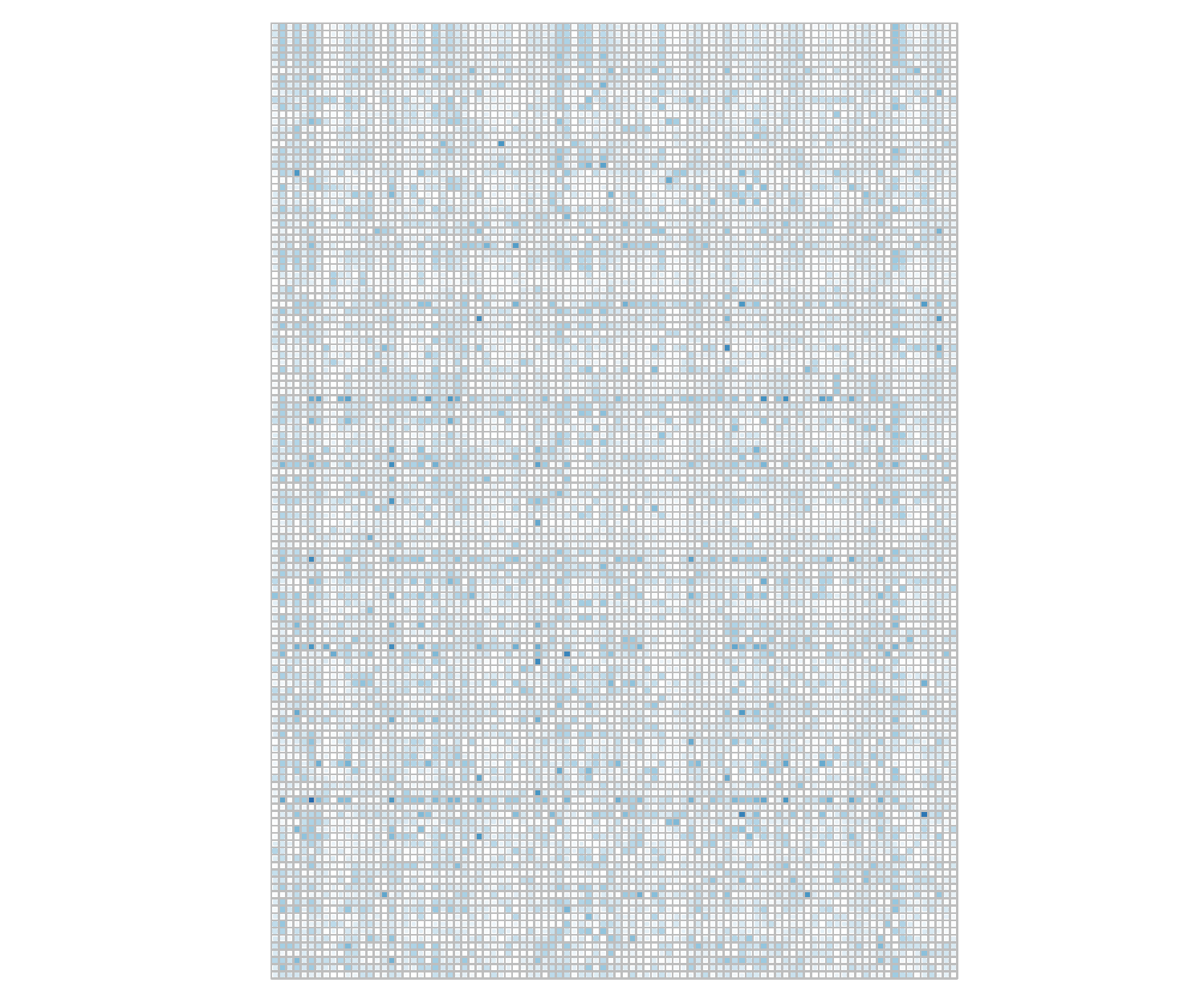}
    \caption{Pathway (a): last 131 genes}
    \end{subfigure}
    \begin{subfigure}{.17\textwidth}
     \centering
    \includegraphics[width=.7\textwidth]{Plots/application/Assumptions/scale.png}
    \end{subfigure}
     \caption{Covariance  plot: plotting the absolute value of covariance between genes and proteins. Here proteins are in the X axis and genes are in the Y axis. Darker color is associated with higher correlation. The color scale is given to the right. Because pathway (a) has 231 genes in comparison to only 94 proteins, it is split into two parts for easier representation.}
    \label{fig: data: cor}
\end{figure}
\FloatBarrier

\section{ Modified COLAR  Estimators}
\label{app: COLAR}
\subsection{Additional notation}
We will need a few additional notation for presenting our method for estimating $\alk$ and $\bk$. The trace inner product between two matrices $A,B\in\RR^{p\times q}$ is defined by $\langle A,B\rangle=\tr(A^TB)$. For any matrix $A$, $A_{-j}$ will denote the matrix obtained by deleting the $j$th column of $A$. Further,
we let $A_{-j,j}$ denote the vector derived from $A_j$ by deleting its $j$th element. We denote by $A_{-j,-j}$  the matrix obtained by deleting the $j$th column and $j$th row of $A$. We let $A_{j,-j}$  be the the vector obtained by deleting the $j$th element of the $j$th row of $A$. Also, for any symmetric matrix $A$, we let $A^{1/2}$ denote the matrix $P D^{1/2}P^T$ where $PDP^T$ is a spectral decomposition of $A$. 
 Moreover, for $j\in\NN$, we will use $e_j \in \RR^n$ to represent a unit vector with a one in $j$th position and $0$'s elsewhere with $n$ determined by context. Finally, we define the class of orthogonal matrices
\begin{equation}\label{notation: O(p,r)}
   \mathcal O(p,r)=\lbs U\in\RR^{p\times r}: U^T U=I_r\rbs. \end{equation}

\subsection{Modified COLAR \citep{gao2017} algorithm}
 \label{sec: chao algo}

 Our algorithm  runs in three stages and is motivated by \cite{gao2017}. The main difference between \cite{gao2017}'s algorithm and ours is that we set the parameter $r$ in the  COLAR algorithm from \cite{gao2017} to a working and possibly  misspecified value of $1$. This step results in some differences in the analytic details as well as  some simplifications of the original algorithm from \cite{gao2017}.
To set up the algorithm, we split the data into two equal parts indexed by $\{0,1\}$, and   calculate the empirical estimators ($\hSxy^{(0)}, \hSx^{(0)}$) and ($\hSxy^{(1)}, \hSx^{(1)}$) based on the  sub-samples $(X^{(0)},Y^{(0)})$ and $(X^{(1)},Y^{(1)})$, respectively. Here the superscripts refers to the specific sub-sample used to calculate the empirical estimators.  
 
 The first stage of our method produces a good preliminary estimator of $\alk$ and $\bk$, which is used by the second stage to produce  improved estimators.  This step is similar  to \cite{gao2017} but we present the details for the sake of completeness as well as the ease of using the notation from this stage of the algorithm while developing the analytic justification of the method. The first stage solves a convex relaxation of  \eqref{opt: sparse canonical correlation analysis} to obtain a set of preliminary estimators $\haz$ and $\hbz$. 
 The convex relaxation hinges on the idea that \eqref{opt: sparse canonical correlation analysis} can be written as a convex program after the change of variable  $F=\alpha\beta^T$.  To see this, first note that since
   $\alpha^T\hSxy^{(0)}\beta=tr(\hSyx^{(0)} F)$, the objective function in \eqref{opt: sparse canonical correlation analysis}  can be written as a linear functional of   $F$.
    The feasible set of \eqref{opt: sparse canonical correlation analysis} can also be written in terms of $F$ but it is not a convex set in general. However, since the objective is linear in $F$, if we replace the feasible set with its convex hull $\mathcal G$, 
    the solutions remain unchanged. The latter follows since a linear  function is always maximized at the boundary of any convex set. From \cite{gao2017} it then follows that $\mathcal G$ takes the form
  \begin{align}\label{def: G}
  \mathcal G= \lbs F\in\RR^{p\times q}\ :\ \|\hSx^{0,1/2}F\hSy^{0,1/2}\|_*\leq 1,\quad  \|\hSx^{0,1/2}F\hSy^{0,1/2}\|_{op}\leq 1\rbs,  \end{align}
  and our optimization problem  reduces to maximizing $tr(\Syx F)$ with respect to $F\in\mathcal G$. To obtain sparse solutions, we also add an $l_1$ penalty to the objective function.
  Therefore, at the end, the first stage solves
  \begin{maxi}
      {F\in\mathcal G}{tr(\hSyx^{(0)}F)-\lambda_1\|F\|_1,}{\label{COLAR: First stage}}{}
  \end{maxi}
  where $\lambda_1$ is a tuning parameter and $\|F\|_1=\sum_{i=1}^q\sum_{j=1}^p|F_{ij}|$ is the vector $l_1$ norm of the matrix $F$. We  take $\lambda_1=C\lambda$, where $\lambda$ is as defined in \eqref{def:lambda}, and  $C>0$ is some constant, whose value will be chosen later. The above optimization program gives an estimate $\widehat{F}_n$ of $F_0=\alk\bk^T$.  The first pair of left and right singular vectors of $\widehat{F}_n$ give the preliminary estimators of $\alpha$ and $\beta$, which we denote by $\haz$ and $\hbz$, respectively.
  
Our second stage is where we differ from \cite{gao2017}.  This modified second stage improves upon the preliminary estimators  and estimates $\alk$ and $\bk$ up to a sign flip. This is crucial, since estimating all the canonical directions simultaneously, as developed in \cite{gao2017}, does not lend itself to identifying the first directions only up to a sign flip. Henceforth, we will only consider the estimation of $\alk$  because the estimation of $\bk$ will be similar. To obtain an improved estimator of $\alk$, the second stage solves
   \begin{mini}[4]
    {x\in \RR^p} {x^T\hSx^{(1)}x-2x^T\hSxy^{(1)} \hb^0+\lambda_2 \|x\|_1}{\label{COLAR: stage 2}}{}
   \end{mini}
  where $\lambda_2$ is a penalizing parameter. We will take $\lambda_2=C\lambda$ for $\lambda$ defined in \eqref{def:lambda} and some constant $C>0$ whose value will be chosen later.
  When the observations are centered, i.e. $\hSx^{(1)}=(X^{(1)})^TX^{(1)}$ and $\hSxy^{(1)}=(X^{(1)})^TY^{(1)}$, \eqref{COLAR: stage 2} can be re-written as
    \begin{mini*}[4]
    {x\in \RR^p} {\|X^{(1)}x-Y^{(1)} \hb^0\|_2^2+\lambda_2 \|x\|_1.}{}{}
   \end{mini*}
   We will denote the solution to \eqref{COLAR: stage 2} by $\tx$.  \cite{gao2017} uses a group lasso penalty instead of the $l_1$ penalty in \eqref{COLAR: stage 2}. These panalties are, however, equivalent when $x$ is a vector, as in our case. Had we been estimating more than one leading canonical vector, as in \cite{gao2017}'s case, $x$ would be a matrix, and the group lasso penalty is no longer equivalent to the $l_1$ penalty.

 The third stage is the normalization step, which simply sets
 \begin{equation}\label{def: COLAR: ha}
   \ha=\begin{cases}\tx\slb(\tx)^T\hSx\tx\srb^{-1/2} & (\tx)^T\hSx\tx>0\\ 0 & o.w.\end{cases} \end{equation}
   It will be later shown in Lemma~\ref{lemma: colar: thm 4.2: T1} that  the  quadratic form $(\tx)^T\hSx\tx$ is non-zero  and $\ha=\tx\slb(\tx)^T\hSx\tx\srb^{-1/2}$ with probability tending to one. Our third stage is slightly different from \cite{gao2017}, who used the sample covariance matrix from a third part of the data to normalize $\tx$, where we use the full covariance matrix $\hSx$. Since we want to estimate only $\alk$ instead of the whole matrix $U$ as in \cite{gao2017}, normalization is simpler in our case, which circumvents the need of the stage final data splitting.
   We remark on passing that we could use \cite{gao2017}'s third step as well, and   the asymptotics would remain the same. However, we avoid three way data splitting because unnecessary data splitting may not be beneficial in finite sample.
   For convenience, we list the modified COLAR algorithm in Algorithm~\ref{algo:Modified COLAR}.
   \textcolor{red}{Somewhere write the full form of COLAR in main text; as well as in supplement.}
\begin{algorithm}[h]

\caption{{\bf Modified COLAR \citep{gao2017} algorithm}}
\label{algo:Modified COLAR}
\KwIn{} \qquad \qquad \qquad\enspace $\hSx^{(i)}$, $\hSxy^{(i)}$, $\hSy^{(i)}$  $(i=0,1)$,   $\lambda_1$ and $\lambda_2$.\\ 
{\bf \underline{Stage 1:}} \\
1. Solve  the convex program 
\[\widehat{F}_n=\argmax_{F\in\mathcal G}\lbs tr(\hSyx^{(0)}F)-\lambda_1\|F\|_1\rbs\]
 where $\mathcal G$ is as in \eqref{def: G}.\\
  2. Obtain the first pair of singular vectors $\haz\in\RR^p$ and $\hbz\in\RR^q$ of $\widehat{F}_n$.\\
 {\bf \underline{Stage 2:}} \\
  Solve the convex program
    \[\tx=\argmin_{x\in \RR^p} \{x^T\hSx^{(1)}x-2x^T\hSxy^{(1)} \hb^0+\lambda_2 \|x\|_1\}\] \\
{\bf \underline{Stage 3:}} Set \\
  \qquad \qquad \qquad\enspace  $\ha=\tx\slb(\tx)^T\hSx\tx\srb^{-1/2}$ \\
\noindent \KwOut{$\ha$}

\end{algorithm}



 \subsection{Asymptotic properties of the COLAR estimator}
 In Section~\ref{sec: asymptotic theory}, we noted that
 $n^{1/2}$-consistency of the de-biased estimators requires some restrictions on the $l_1$ and $l_2$ errors of the preliminary estimators of $\alk$ and $\bk$, which are satisfied by our $\ha$ and $\hb$. \begin{theorem}\label{thm: Chao Thm 4.2}
  Suppose Assumption~\ref{assump: eigengap assumptions} and Assumption~\ref{assump: bounded eigenvalue} hold. Further suppose 
 $s=s_U+s_V$ satisfies $s\lambda\to 0$, where $\lambda$ is as in \eqref{def:lambda}, and  $s=o(p)$. Then there exist $C_1$, $C_2>0$ such that for $\lambda_1=C\lambda$ with $C>C_1$, and $\lambda_2=C'\lambda$ with $C'>C_2$,
 the estimators $\ha$ and $\hb$ defined in \eqref{def: COLAR: ha} satisfy
Condition \ref{cond: preliminary estimator}
 with 
\begin{align}\label{def: kappa}
    \kappa=\begin{cases}
    1/2 & \text{if }r=1\\
    1 & \text{o.w.}
    \end{cases}
\end{align}
\end{theorem}

Note that the  sparsity requirement on $s$ is $s\lambda=o(1)$, which is a weaker condition than our Assumption~\ref{assumption: sparsity} that requires $s^{2\kappa}\lambda^2=o(n^{-1/2})$. Fact~\ref{fact: slambda goes to zero} indicates that Assumption~\ref{assumption: sparsity} implies $s\lambda=o(1)$.

 \begin{remark}
 When $r=1$, the proof of Theorem~\ref{thm: Chao Thm 4.2} implies  that a slightly stronger result  holds than that implied by Condition \ref{cond: preliminary estimator}. More explicitly, the $l_1$ and $l_2$ errors of $\ha$ depends only on $s_U$, and not on $s_V$. Similarly, the  asymptotics of $\hb$ depend only on $s_V$. To be more precise,
 \[\inf_{w\in\{\pm 1\}}\|w\ha-\alk\|_1=O_p(s_U\lambda),\quad \inf_{w\in\{\pm 1\}}\|w\hb-\bk\|_1=O_p(s_V\lambda)\]
 and
  \[\inf_{w\in\{\pm 1\}}\|w\ha-\alk\|_2=O_p(s_U^{1/2}\lambda),\quad \inf_{w\in\{\pm 1\}}\|w\hb-\bk\|_2=O_p(s_V^{1/2}\lambda).\]
  The above result is substantially sharper than that implied by the statement of Condition \ref{cond: preliminary estimator} if $s_U\ll s_V$ or vice versa.
 \end{remark}

  The optimal value of $\lambda_1$ and $\lambda_2$ rely on $C_1$ and $C_2$, which depend unknown quantities like $M$ in   Assumption~\ref{assump: bounded eigenvalue}. Therefore, cross-validation may be required to choose these tuning parameters efficiently. According to \cite{gao2017}, there is scope of improving the algorithm so that it adapts to the unknown $M$.
  However, it is beyond the scope of the current paper.

\begin{remark}[\cite{gao2013}'s estimators]
\label{remark: CAPIT}
 \textcolor{black}{ Although \cite{gao2013} uses an iterative thresholding type method to estimate $\alk$ and $\bk$ upto a sign flip and the resulting estimators attain the minimax rate in $l_2$ norm, they consider the rank one model. It is remains unknown whether their method continues to work similarly for $r>1$ case while estimating purely the leading canonical directions up to a sign flip.  Here we discuss one potential roadblock  on its straightforward extension to the general $r>1$ case. The theoretical guarantees of \cite{gao2013}'s method rely heavily on the initial estimators. To obtain these initial estimators, they apply singular value decomposition on a suitably chosen estimator of $\Sx^{-1}\Sxy\Sy^{-1}$. When $r=1$, the matrix $\Sx^{-1}\Sxy\Sy^{-1}$ has rank one, and its leading singular vectors are proportional to $\alk$ and $\bk$. Therefore the above-mentioned initialization method works. However, when $r>1$, unless $\Sx$ and $\Sy$ are identity, the leading singular vectors of $\Sx^{-1}\Sxy\Sy^{-1}$ are no longer proportional to $\alk$ and $\bk$. Therefore, the idea behind the initialization method of \cite{gao2013} ceases to work for $r>1$. }
\end{remark}

\section{ Nodewise Lasso Estimator}
\label{App: nodewise lasso}

\subsection{The main algorithm}
 The nodewise lasso algorithm was first implemented by \cite{meinshausen2006}, who used the name  graphical lasso.  \cite{meinshausen2006} showed that the $p\times p$ dimensional precision matrix can be estimated by regressing each of the $p$ variables against the other; see also the nodewise regression of \cite{vandegeer2014}. Although originally invented for precision matrix estimation, the
 basic idea of nodewise lasso applies to the inversion of any real symmetric matrix; cf. 
   \cite{jankova2018}. 
Asymptotic garuantees, however, can  be established only if the input matrix  consistently estimates a positive definite matrix.  Also, for the most part, the asymptotics of the  nodewise lasso algorithm is case-specific, which is to say that the convergence results solely depend on the matrix to be inverted, which is $\widehat H_n(\hx,\hy)$ in our case.
For the sake of completeness, we include this algorithm in our paper; see Algorithm~\ref{algo: nodewise lasso}. 
\begin{algorithm}[H]\label{algo: nodewise lasso}
\caption{Non-convex Nodewise Lasso}
\KwIn{} \qquad \enspace $A \in \mathbb{R}^{m\times m}$ where $m\in\NN$, positive penalty parameters $(\lambda^{nl}_{j},B_j)$, $j=1,\ldots,m$.

{{\bf for} $j=1,\dots m:$} 

\begin{itemize}
\item[\underline{NL1.}] Compute any stationary point $\widehat{\eta}_j$ of the minimization program 
\begin{mini}
    {\eta_j\in\RR^{p+q-1},\|\eta_j\|_1\leq B_j}{\eta_j^TA_{-j,-j}\eta_j-2A_{-j,j}^T\eta_j+\lambda^{nl}_{j}\|\eta_j\|_1,}{\label{opt: nodewise Lasso}}{}
\end{mini}
      where we remind the readers that $A_{-j,-j}$ is the matrix obtained by deleting the $j$th row and the $j$th column of the matrix $A$, and $A_{-j,j}$ is the vector obtained by deleting the $j$th element of $A_{j}$.
      \item[\underline{NL2.}] Compute the estimator of the noise-level 
      \begin{equation}\label{def: widehat tau j}
          \widehat{\tau}_j^2=\widehat{\Gamma}_j^TA\widehat{\Gamma}_j+\dfrac{1}{2}\lambda^{nl}_{j}\|\widehat{\eta}_j\|_1,    \end{equation}
        where
        \begin{equation}\label{def: widehat gamma j}
          \widehat{\Gamma}_j=(-(\heta)_1,\ldots,-(\heta)_{j-1},1,-(\heta)_{j+1},\ldots,-(\heta)_{m})  
        \end{equation}
  \end{itemize}
  Set \[\overline{A}=[\widehat{\Gamma}_1/\widehat{\tau}_1^2,\ldots,\widehat{\Gamma}_m/\widehat{\tau}_{m}^2]\] \\
\noindent \KwOut{$\overline{A}$}
\end{algorithm}
A couple of remarks are in order. First,
the $l_1$ penalty in \eqref{opt: nodewise Lasso} is introduced to enforce a sparse solution. Second, the constraint $\|\eta_j\|_1\leq B_j$ ensures a bounded solution to the problem.  Without this boundedness condition, the optimization problem \eqref{opt: nodewise Lasso} can become unbounded since $A_{-j,-j}$ is  potentially singular. Third, there is no guarantee that $\overline{A}$ will be symmetric when $A$ is symmetric. Therefore, we have to compute  the full matrix $\overline{A}$ even if  $A$ is known to be symmetric.  
Finally, notice  that Algorithm~\ref{algo: nodewise lasso} does not require us to solve   \eqref{opt: nodewise Lasso}, which is  possibly non-convex, since a stationary point of \eqref{opt: nodewise Lasso} suffices. In Section~\ref{sec: asymp: NL}, we will demonstrate how to choose the tuning parameters $\lnlm_j$ and $B_j$.
\begin{remark}[Possible other choices of $\hf$]
  The de-biasing literature borrows nodewise lasso from precision matrix estimation literature. Other methods for precision matrix estimation, e.g. Constrained
$l_1$-minimization for Inverse Matrix Estimation aka CLIME \citep{cai2011}, the graphical lasso aka GLASSO \citep{friedman2008} etc. may also be used in place of nodewise lasso to construct $\hf$ provided  Condition~\ref{assumption: precision matrix} is satisfied under realistic structural assumptions. In this regard, 
CLIME solves convex optimization problems and has fast implementation. It has also seen application in context of de-biasing \citep{neykov2018unified}.
We conjecture that if the columns of $\Phi^0$ are bounded in $l_1$ norm, then the CLIME estimator  satisfies the desired Condition~\ref{assumption: precision matrix} as well.  This requirement, however, is stricter than that of the nodewise lasso; see Assumption~\ref{assump:Phi 0} in Section~\ref{sec: asymp: NL}. To keep our discussions focused, we refrain from further discussion on the asymptotics of CLIME here.
Similar to CLIME, GLASSO also has fast implementation and is widely used in precision matrix estimation.   However, the current literature lacks results supporting its consistency. There is, instead, some evidence against its asymptotic convergence to the precision matrix, at least in $l_{\infty}$ norm  \citep{mazumder2012}. 
  \end{remark}
\subsection{Intuition behind the nodewise lasso algorithm}
\label{subsec: intuition of NL}
 To provide an intuition behind why Algorithm~\ref{algo: nodewise lasso} works, we argue that if the input matrix $A$ in Algorithm~\ref{algo: nodewise lasso} is positive definite, the algorithm  outputs $A^{-1}$ when the penalty parameters $\lnlm_j$'s are set to zero.  To that end, we first invoke a standard linear algebra result \citep[cf.]{rao2000}.
 \begin{lemma}\label{Lemma: linear algebra}
Suppose $m\in\NN$ and $A$ is an $m\times m$ positive definite matrix. Then $A_{-j,-j}$ is invertible for $j=1,\ldots,m$. Moreover,
\[(A^{-1})_{j,j}=\dfrac{1}{A_{j,j}-A_{j,-j}^TA_{-j,-j}A_{-j,j}}\]
\[(A^{-1})_{-j,j}=-(A^{-1})_{j,j}(A_{-j,-j})^{-1}A_{-j,j}.\]
 \end{lemma}

  Defining $\eta_j=(A_{-j,-j})^{-1}A_{-j,j}$, we note that
\[\argmin_{\eta\in\RR^{p-1}}(\eta^TA_{-j,-j}\eta-2A_{-j,j}^T\eta)=\eta_j,\quad j=1,\ldots,m.\]  
 Also, in parallel with \eqref{algo: nodewise lasso},
 we define
 \[\Gamma_j=(-(\eta_j)_1,\ldots,-(\eta_j)_{j-1},1,-(\eta_j)_{j+1},\ldots,-(\eta_j)_{r-1}),\].
 Then it follows that
 \begin{equation}\label{eq:inverse elements}
 \tau_j^2=\Gamma_j^TA\Gamma_j=A_{j,j}-A_{j,-j}^TA_{-j,-j}A_{-j,j}\stackrel{(a)}{=}1/(A^{-1})_{j,j},
 \end{equation}
 where (a) follows from Lemma~\ref{Lemma: linear algebra}. 
Applying  Lemma~\ref{Lemma: linear algebra} again, we can show that the $j$th element of the output matrix $\overline{A}$ equals 
\begin{equation}\label{eq: inverse columns}
\Gamma_j/\tau_j^2=(A^{-1})_j.
\end{equation}

\subsection{Nodewise lasso for our case }
In this section, we discuss the finite sample properties of our nodewise lasso estimator $\hf$. We begin with some implications of
the  discussion in Section~\ref{subsec: intuition of NL} for the special case when the input matrix  $A=H^0$. First, note that, in this case, for $j=1,\ldots,p+q$,
\[\eta_j^0=\argmin _{\eta\in\RR^{p-1}}\slb\eta^TH^0_{-j,-j}\eta-2(H^0_{-j,j})^T\eta\srb\] 
satisfies
\begin{equation}\label{def: eta j knot}
 \eta_j^0= (H^0_{-j,-j})^{-1}H^0_{-j,j}.
\end{equation}
Moreover, \eqref{eq:inverse elements} implies for $j=1,\ldots,p+q$,
\begin{equation}\label{def: tau j knot}
    (\tau^0_j)^2=H^0_{j,j}-(H^0_{j,-j})^TH^0_{-j,-j}H^0_{-j,j}
\end{equation}
satisfies $(\tau^0_j)^2=(\Phi^0_{j,j})^{-1}$. From \eqref{eq: inverse columns}, it then follows that 
\begin{equation}\label{rel: phi and eta}
    (\Phi^0)_{-j,j}=-\eta_j^0/(\tau^0_j)^2.
\end{equation}
Because $\heta$ is a stationary point of \eqref{opt: nodewise Lasso}, it satisfies
the KKT condition, which takes the form 
\begin{equation*}
    -2A_{j,-j}+2A_{-j,-j}\widehat {\eta}_j+\lambda^{nl}_j\partial\|\widehat {\eta}_j\|_1=0,
\end{equation*}
where $\lnlm_j$ is as in \eqref{opt: nodewise Lasso} and $\partial\|\widehat {\eta}_j\|_1$ is the partial derivative of the $l_1$ norm evaluated at $\heta$.
It then follows that \citep[cf. Section 3.1 of ][]{jankova2018}
\begin{equation}\label{KKT: NL: consequences}
    A_j^T\widehat\Gamma_j=\widehat\tau_j^2\quad\text{and}\quad \|A_{-j}^T\widehat{\Gamma}_j\|_{\infty}\leq\lambda_j^{nl}/2
\end{equation}
 \[\A^T\overline{A}-I_{p+q}|_{\infty}=O\slb\max_{1\leq j\leq p+q}\lambda^{nl}_{j}/{\widehat\tau}_j^2\srb.\]



\subsection{Asymptotic properties of the nodewise lasso estimator}
\label{sec: asymp: NL}
In this Section, we will show that the nodewise lasso estimator satisfies Condition~\ref{assumption: precision matrix} under some regulatory conditions. We will go through these regulatory conditions first.

Recall from \eqref{def: eta j knot} the definition of $\eta_j^0$. 
  We will require the number of non-zero elements in $\eta_j^0$, i.e. $\|\eta_j^0\|_0$, to be small, which is in parallel with  \cite{jankova2018}. 
  \begin{assumption}[ Assumption on the column sparsity of $\Phi^0$]
\label{assump:Phi 0}
 $\max_{1\leq j\leq p+q} \|{\eta^0}_j\|_0=O(s)$, where $s=s_U+s_V$.
 \end{assumption}
  Since $\abs{\|\eta^0_j\|_0- \|\Phi^0_j\|_0}\leq 1$ by \eqref{rel: phi and eta}, a restriction on $\|\eta^0_j\|_0$ actually induces a restriction on $\|\Phi^0_j\|_0$, which explains the nomenclature of Assumption~\ref{assump:Phi 0}. 
  
  Assumption~\ref{assump:Phi 0} can be hard to decipher, and it may  be hard to verify. Therefore we will now give a sufficient condition for Assumption~\ref{assump:Phi 0}.
  Lemma~\ref{lemma: NL: form of Phi knot} in Supplement~\ref{sec: proof of rho lemmas}
 gives the explicit form of $\Phi^0$, which indicates that
 \[ \|{\eta^0}_j\|_0\leq s_U+s_V+\|(\Sx)_j^{-1}\|_0+\|(\Sy)_j^{-1}\|_0\quad(j=1,\ldots,p).\]
  Therefore, 
we only require the column sparsities of $\Sx^{-1}$ and $\Sy^{-1}$ to be $O(s)$ for  Assumption~\ref{assump:Phi 0} is satisfied. This sparsity requirement is formulated as Condition  \ref{assumption: precision matrix (Sx)}.

    
\begin{condition}[ A  sufficient condition for Assumption~\ref{assump:Phi 0}]
\label{assumption: precision matrix (Sx)}
  The maximum number of non-zero elements per column of $\Sx^{-1}$ or $\Sy^{-1}$ is $O(s)$.
\end{condition}
Sparsity restriction on the columns of $\Sx^{-1}$ and $\Sy^{-1}$ is more intuitive than sparsity restriction on $\Phi^0$. It is a well known fact that $\Sx^{-1}$ or $\Sy^{-1}$ is  sparse if the partial correlation between the $X_i$'s or the $Y_i$'s are mostly zero, which may be satisfied when only a few of these variables interact among themselves. The latter is not unusual in high dimensional genomic data because genes, proteins etc. generally form clusters.
Such  sparsity restrictions are  also common in the literature; cf. \cite{sara2011, jankova2018}.

Now we are ready to state the main theorem of this section.
\begin{theorem}[Nodewise lasso Theorem]
\label{thm: nodewise Lasso theorem}
Suppose  Assumptions~\ref{assump: eigengap assumptions},   \ref{assump: bounded eigenvalue},  \ref{assumption: sparsity}, \ref{assump:Phi 0} hold, and the preliminary estimators $\hx$ and $\hy$ satisfy Condition~\ref{cond: preliminary estimator}. Further suppose \[\|\eta_j^0\|_1\leq B_j\leq C_Ts^{1/2}\quad (j=1,\ldots, p+q)\]
 for some $C_T>0$. Then there exists an absolute constant $C>0$ depending on $C_T$ such that  for
\[\lambda^{nl}_{j}=C\lambda,\]
 the estimator $\hf$ obtained by feeding $\widehat H_n(\hx,\hy)$ to Algorithm~\ref{algo: nodewise lasso} satisfies Condition~\ref{assumption: precision matrix}. here $\lambda$ is as in \eqref{def:lambda}.
\end{theorem}
The nodewise lasso estimator $\hf$ depends on $\hx$ and $\hy$ via $\widehat H_n(\hx,\hy)$, which does not rely on the sign of $\hx$ and $\hy$. Hence, the asymptotics of $\hf$, unlike the de-biased estimators, is unaffected by the sign flip of $\ha$ and $\hb$.


\section{ Connection to Related Literature}
\label{sec: connection to literature}
The study of asymptotic inference in the context of SCCA naturally connects to the popular research direction of de-biased/de-sparsified inference in high dimensional models \citep{zhang2014confidence,javanmard2014,vandegeer2014,jankova2018,zhang2014confidence,ning2017general,neykov2018unified,jankova2017honest,jankova2016confidence,cai2017,mitra2016benefit,bellec2019biasing}. \textcolor{black}{This line of research, starting essentially from the seminal work of \cite{zhang2014confidence}, more or less follows the general prescription laid out in Section~\ref{sec: de-bias: general}.
Similar to our case, these methods also often depend on potentially high dimensional  parameters -- and thereby require  initial good estimators of them.} For example,  asymptotically valid confidence interval for the coordinates of a sparse linear regression vector relies critically on good initial estimators of the regression vector  and nuisance parameter in form of the precision matrix of the covariates \citep{zhang2014confidence,javanmard2014,vandegeer2014}.
\textcolor{black}{The construction of a suitable estimating equation is however somewhat case specific, and can be involved based on the nature of the high dimensional nuisance parameters. Since SCCA  involves a list of high dimensional nuisance parameters including the covariance matrices $\Sx$ and $\Sy$, special attention is required in deriving our inferential procedures.}


Among the  above-mentioned methods, our approach bears the greatest resemblance  to the method recently espoused by \cite{jankova2018} in the context of Sparse Principal Component Analysis (SPCA). 
However, there are substantial differences between \cite{jankova2018}'s approach and ours.
First, due to the presence of high dimensional nuisance parameters $\Sx$ and $\Sy$,  the canonical correlation analysis problem in general is more complicated than the principal component analysis problem \citep{gao2015,gao2017}. Thus, blindly following \cite{jankova2018} works neither for the SCCA part nor for the actual de-biasing step. Second, Section~\ref{sec: de-bias: general} indicates that the correct choice  of the objective function $f$ is  crucial to any de-biasing method. \cite{jankova2018}'s objective function bases on the well-known fact that the first  principal component extraction problem can be written as an unconstrained Frobenius norm minimization problem. No such analogue, to the best of our knowledge, was previously available in the CCA literature. We had to construct a novel objective function whose unconstrained optimization yields the first canonical directions; see Lemma~\ref{lemma: objective function}. 
 Third,  \cite{jankova2018} applies the de-biasing procedure on some preliminary estimator, similar to us. However, their preliminary estimators are based on solving a penalized version of the non-convex principal component analysis optimization problem.  To aid the computation, the authors restrict the search space to a small neighborhood of a consistent estimator of the first principal component. The said consistent estimator is found by semi-definite programming. They also show that, any stationary point of the resulting optimization program consistently estimates the first principal component. This  removes the burden of finding the global minima, but the program still remains non-convex.   Our SCCA method, on the other hand,  is inspired by \cite{gao2017}'s approach, where the non-convex optimization part is replaced by a lasso.  

\section{On the Conditions and Assumptions of Section~\ref{sec: asymptotic theory}}\label{sec:assumptions_necessities}
 In this section we provide a detailed discussions on assumptions made for the sake of theoretical developments in Section \ref{subsec: choices of ha, hb, hf}. 
 
 \textbf{Discussion on Condition~\ref{cond: preliminary estimator}}
First, some remarks are in order regarding the range of $\kappa\in[ 1/2,1]$ in Condition \ref{cond: preliminary estimator}. Theorem 3.2 of \cite{gao2017} implies that it is impossible for $\kappa$ to be strictly less than $1/2$ since    the minimax rate of the $l_2$ error is roughly $s^{1/2}\lambda$ under Assumption~\ref{assump: eigengap assumptions} and Assumption~\ref{assump: bounded eigenvalue}. If $\kappa$ is larger, i.e. $\ha$ and $\hb$ have slower rates of convergence,  and we pay a price in terms of the sparsity restriction $s=o(n^{1/(4\kappa)}(\log(p+q))^{-1/(2\kappa)})$ in Assumption~\ref{assumption: sparsity}. 
Supplement~\ref{app: COLAR} shows that estimators satisfying Condition~\ref{cond: preliminary estimator} with $\kappa=1$ exist. In fact, most SCCA estimators with theoretical guarantees have $l_2$ error guarantee of $s^{\kappa}\lambda$ with $\kappa\in[1/2,1]$. The interested reader can refer to  \cite{gao2017, gao2015, gao2013} and references therein.  Subsequently, in view of the above, we let $\kappa\in[1/2,1]$. 


  
  In light of Condition~\ref{cond: preliminary estimator}, indeed $\ha$ and $\hb$ with faster rate of convergence, i.e. $\kappa=1/2$,  is  preferable. \textcolor{black}{ COLAR and  \cite{gao2013}'s   estimator attain this minimax rate when $r=1$. We do not yet know if there are SCCA estimators which attain the minimax rate for $r>1$ while only estimating the first canonical direction. For $r>1$, the estimation problem becomes substantially harder because the remaining $r-1$ canonical directions start acting as high dimensional nuisance parameters.  It is likely that  a trade-off between computational and estimation efficiency arises  in presence of these additional nuisance parameters. In particular, it is plausible that
    the minimax rate of $\kappa=1/2$ may not be achievable by polynomial time algorithms in this case.   To gather intuition about this, it is instructive to look at the literature on estimating the first principal component direction in high dimensions under sparsity. 
   In this case, to the best of our knowledge,  polynomial time algorithms   attain the minimax rate only in the single spike model, or a slightly relaxed version of the latter. We refer the interested reader to  \cite{wang2016}  for more details.   
 }   \textcolor{black}{
   The algorithms that do succeed to estimate the first  principal component under multiple spikes at the desired minimax rate  attempt to solve the underlying non-convex problem, and hence are not immediately clear to be polynomial time \citep{yuan2013, ma2013, jankova2018}.  In this case,  \cite{yuan2013} and \cite{ma2013}'s methods essentially reduce to power methods that induce sparsity by iterative thresholding. \cite{gao2013}'s method tries to borrow this idea in context of SCCA in the rank one case; see Remark~\ref{remark: CAPIT} for a discussion on the problems that  their method may face in presence of nuisance canonical directions.}

   Finally for the inferential question, it is natural to consider an extension of ideas from sparse PCA as developed in \citep{jankova2018}. When translated to SCCA, their approach will aim to solve
       \begin{mini}
          {x\in\RR^p,y\in\RR^q}{\widehat h_n(x,y)+C\lambda(\|x\|_1+\|y\|_1)}{\label{opt: sara}}{},
       \end{mini}
       where $C>0$ is a constant, and
       \[\widehat h_n(x,y)=(x^T\hSx x)^2/2+(y^T\hSy y)^2/2-2x^T\hSxy y.\]
We conjecture that for a suitably chosen $C$, the resulting estimators will satisfy Condition~\ref{cond: preliminary estimator} with $\kappa=1/2$. However, \eqref{opt: sara} is non-convex and solving \eqref{opt: sara} is computationally challenging for large $p$ and $q$. Analogous to \cite{jankova2018}, one can simplify the problem by searching for any stationary point of \eqref{opt: sara} over a smaller feasible set, namely a small neighborhood of a consistent preliminary estimator of $\alk$ and $\bk$. However, while this first stage does guarantee a good initialization, the underlying optimization problem still remains non-convex. Since the aim of the paper is  efficient inference of $\alk$ and $\bk$ whose computational efficiency is theoretically guaranteed, we stick with the modified COLAR estimators and refrain from exploring the above-mentioned route.



\textbf{Discussion on Assumption \ref{assumption: sparsity}:}
  It is  natural to wonder whether the condition $s^{2\kappa}\lambda^2=o(n^{-1/2})$ is at all necessary, especially since  it is much stricter than  $s\lambda=o(1)$, which is sufficient  for the $l_2$  consistency of $\ha$ and $\hb$ presented in Theorem \ref{thm: Chao Thm 4.2} of Supplement \ref{app: COLAR}. However, 
  current literature on inference in high dimensional sparse models bears evidence that 
  the restriction $s\lambda^2=o(n^{-1/2})$ might be unavoidable.
  In fact, this  sparsity requirement is a staple in most de-biasing approaches whose preliminary estimators are minimax optimal, including sparse principal component analysis \citep{jankova2018} and sparse generalized linear models \citep{vandegeer2014, javanmard2014}.  
   Indeed, in case of sparse linear regression, \cite{cai2017} shows that this  sparsity is necessary for adaptive inference. We believe similar results hold for our case as well. However, further enquiry in that direction is beyond the scope of the present paper.

Next,    it is natural to ask why Assumption~\ref{assumption: sparsity}  involves sparsity restriction not only on $\alk$ and $\bk$, but also on the other columns of $U$ and $V$. \textcolor{black}{This restriction stems from the initial estimation procedure of $\alk$ and $\bk$. Although we  estimate only  the first pair of canonical directions,
the remaining canonical directions act as nuisance parameters.}  Thus, to efficiently estimate $\alk$ and $\bk$ , we need to separate the other covariates  from $\alk$ and $\bk$. Therefore, we need to estimate the other covariates' effect efficiently enough. Consequently we require some regularity assumptions on these nuisance parameters as precisely quantified by Assumption~\ref{assumption: sparsity}.


\textbf{Discussion on Condition~\ref{assumption: precision matrix}:}
This is a standard assumption in de-biasing literature in that similar assumptions have  appeared in sparse PCA  \citep{jankova2018} and sparse generalized linear models literature \citep{vandegeer2014} -- both of whom use the nodewise lasso algorithm to construct $\hf$. 
We remark in passing that that  \cite{javanmard2014}'s  construction of de-biased lasso does not require the analogue of $\hf$, which is the precision matrix estimator in their case,  to satisfy any condition like Condition~\ref{assumption: precision matrix}. Instead, it requires $(\hf)_i^T \hSx(\hf)_i$'s to be small. It is unknown whether such constructions work in the more complicated scenario of CCA or PCA.

 \section{Proof Preliminaries}
 \label{sec: aditional lemmas}
This section states the facts and lemmas that are used repeatedly in the proofs. The proofs are deferred to Section  \ref{sec: proof of lemmas in fact} unless they are very trivial.

First, we derive some results for $\ha$ and $\hb$ satisfying Condition~\ref{cond: preliminary estimator}.

\begin{lemma}\label{corollary: rate of rho}
Suppose $\ha$ and $\hb$ satisfy  Condition~\ref{cond: preliminary estimator}. Further suppose Assumption~\ref{assump: bounded eigenvalue} and  Assumption~\ref{assumption: sparsity} hold. 
Then 
\[\|\hro|-\rhk|=O_p(s^{\kappa}\lambda).\]
Moreover
\[\ha^T\hSx\ha-1=O_p(s^{\kappa}\lambda)\quad \text{and}\quad \hb^T\hSy\hb-1=O_p(s^{\kappa}\lambda)\]
\end{lemma}

 Recall that we have defined
\begin{equation*}
    \hx=|\hro|^{1/2}\ha,\quad \hy=|\hro|^{1/2}\hb,\quad 
    x^0=(\rhk)^{1/2}\alk,\quad
    y^0=(\rhk)^{1/2}\bk.
\end{equation*}
The following lemma gives the rates of $\hx$ and $\hy$ when Condition~\ref{cond: preliminary estimator} holds.

\begin{lemma}\label{corollary: rate of x and y}
Under the set up of Lemma~\ref{corollary: rate of rho},
\[\inf_{w\in\{\pm 1\}}\|w\hx-x^0\|_1+\inf_{w\in\{\pm 1\}}\|w\hy-y^0\|_1=O_p(s^{\kappa+1/2}\lambda)\]
and
\[\inf_{w\in\{\pm 1\}}\|w\hx-x^0\|_2+\inf_{w\in\{\pm 1\}}\|w\hy-y^0\|_2=O_p(s^{\kappa}\lambda)\]
where $\kappa$ is as defined in \eqref{def: kappa}.
\end{lemma}

The following lemma entails that $\hx^T\hSx\hx$ and $\hy^T\hSy\hy$  consistently estimate  $\rhk^2$.
\begin{lemma}\label{corollary: colar: normalization of hx}
Under the set up of Lemma~\ref{corollary: rate of rho}, we have
\[\hx^T\hSx\hx-\rhk=O_p(s^{\kappa} \lambda)\quad \hy^T\hSy\hy-\rhk=O_p(s^{\kappa}\lambda)\]
where $\kappa$ is as defined in\eqref{def: kappa}.
\end{lemma}

\begin{proof}[of Lemma~\ref{corollary: colar: normalization of hx}]
Noting
\[|\hx^T\hSx\hx-\rhk|\leq |\hro(\ha^T\hSx\ha-1)|+\|\hro|-\rhk|,\]
the proof follows from Lemma~\ref{corollary: rate of rho} and the fact $|\hro|\leq 1$. The proof for  $\hy$ follows in a similar way.
\end{proof}

Now we state an implication of Assumption~\ref{assumption: sparsity}.
\begin{fact}\label{fact: slambda goes to zero}
Suppose $\lambda$ is as in \eqref{def:lambda}. 
Then Assumption~\ref{assumption: sparsity} implies $s^{\kappa+1/2}\lambda=o(1)$ and $s\lambda=o(1)$. 
\end{fact}

Now we state some linear algebra facts.
\begin{fact}\label{fact: frob norm of proj matrix}
 For any two matrices $A,B\in\RR^{p\times q}$, we have
 \[\|P_A-P_B\|_F^2=\text{rank}(A)+\text{rank}(B)-2 \text{tr}(P_AP_B),\]
 where $P_A$ and $P_B$ are the projection matrices onto the column spaces of $A$ and $B$, respectively.
 \end{fact}
 
 \begin{proof}
 Noting $P_A^2=P_A$ and $P_B^2=P_B$, we obtain
 \begin{align*}
     \|P_A-P_B\|_F^2= \text{tr}\slb (P_A-P_B)^T(P_A-P_B) \srb=\text{tr}(P_A)+\text{tr}(P_B)-2\text{tr}(P_AP_B),
 \end{align*}
 from which the result follows because for projection matrix $P_A$, $\text{tr}(P_A)=\text{rank}(A)$.
 \end{proof}
 
 \begin{fact}\label{fact: frob norm operator norm}
For any matrix $A\in\RR^{r\times r}$, $\|A\|_F\leq r^{1/2}\|A\|_{op}$ 
\end{fact}
\begin{proof}[of Fact~\ref{fact: frob norm operator norm}]
Suppose $\varsigma_i$'s are the singular values of $A$. Then
\[\|A\|_F^2=\sum_{i=1}^r\varsigma_i^2\leq r\max_{1\leq i\leq r}\varsigma_i^2.\]
Therefore $\|A\|_F\leq r^{1/2} \|A\|_{op}$.
\end{proof}

\begin{fact}[Lemma 2.1.3 of \cite{chen2020}]\label{fact: Chen 2020}
Suppose $A$ and $B$ are two matrices in $\RR^{p\times q}$. Then
\[2^{-1/2}\|P_A-P_B\|_F\leq \inf_{W\in\mathcal O(r,r)}\|AW-B\|_F\leq \|P_A-P_B\|_F\]
\end{fact}

\begin{fact}\label{fact: Projection matrix}
Suppose $x$ and $y\in\RR^p$. Then 
\[\|P_x-P_y\|_F\leq 4\inf_{w\in\{\pm 1\}}\frac{\|wx-y\|_2}{\max(\|x\|_2,\|y\|_2)}.\]
\end{fact}


The next lemma shows that the $l_1$ and $l_2$ norms of $\alk$ and $\bk$ are bounded.
\begin{lemma}\label{lemma:norm:  l1 and l2 norm}
Under Assumption~\ref{assump: bounded eigenvalue}, we have
\[\|\alk\|_1,\|\bk\|_1\leq (Ms)^{1/2},\]
where  $s=s_U+s_V$.
Also
\[\|\alk\|_2,\|\bk\|_2\leq M^{1/2},\]
where $M$ is as in Assumption~\ref{assump: bounded eigenvalue}.
\end{lemma}

\begin{proof}
Since
 $\|\alk\|^2_2\Lambda_{min}(\Sx)\leq |\alk^T\Sx\alk|$, we have $\|\alk\|_2\leq \sqrt M$ by the \ref{assump: bounded eigenvalue} Assumption. Similarly, $\|\bk\|_2\leq \sqrt M.$
 Now, Cauchy Schwartz inequality implies
  \[\|\alk\|_1\leq \sqrt s\|\alk\|_2=\sqrt{Ms}.\]
  The same can be proved for $\bk$, which completes the proof.
\end{proof}



Now we state some rate-results for  Sub-Gaussian covariance matrices. 
\begin{lemma}\label{result: inf norm: dif }
Suppose $X\in\RR^{n\times p}$ and $Y\in\RR^{n\times q}$ are sub-Gaussian matrices. 
 Then there exists constant $C$ depending only on the subgaussian parameter of $(X,Y)$ so that 
\[|\hSx-\Sx|_{\infty}, |\hSy-\Sy|_\infty, |\hSxy-\Sxy|_\infty\leq C\lambda\]
with high probability as $n,p,q\to\infty$.
Moreover,  for any $v\in\RR^p$, there exists constant $C$ depending only on the subgaussian parameter of $(X,Y)$ so that
\[\|(\hSx-\Sx)v\|_{\infty}\leq C\|v\|_2 \lambda\quad\text{and}\quad \|(\hSxy-\Sxy)v\|_{\infty}\leq C\|v\|_2 \lambda.\]
with high probability as $n,p,q\to\infty$.
\end{lemma}




\begin{lemma}\label{lem: quadratic form for sx}
Let $v\in\RR^p$ and the set $S\subset \{1,\ldots, p\}$ has cardinality $s$. Suppose $v$ satisfies the cone condition $\|v_{S^c}\|_1\leq C'\|v_S\|_1$ for some constant $C'>0$ and $S$ has cordinality $s$. Then under the set up of Lemma~\ref{result: inf norm: dif }, there exists $C>0$ depending only on the subgaussian parameter of $X$ and $C'$ so that 
\[|v^T(\hSx-\Sx)v|\leq C s\|v\|_2^2\lambda ,\]
with high probability as $n,p,q\to\infty$.
\end{lemma}

\begin{proof}
Noting
$\|v\|_1\leq (C'+1)\|v_S\|_1\leq (C'+1)s^{1/2}\|v_S\|_2$,  we obtain
\[\abs{v^T(\hSx-\Sx)v}\leq \|v\|_1^2|\hSx-\Sx|_\infty\leq (C'+1)^2s\|v\|_2^2|\hSx-\Sx|_\infty.\]
Then the proof follows by Lemma~\ref{result: inf norm: dif }. 
\end{proof}


\begin{lemma}\label{lemma: Quadratic:  bounded}
Supppose $X$ is sub-Gaussian and a random vector $\widehat z_n\in\RR^p$ satisfies $\|\widehat z_n\|_1=O_p(s^{1/2})$, where $s$ satisfies $\lambda^{1/2}s=o(1)$. Then there exists $C$ depending only on the sub-gaussian parameters of $X$  so that
\[|\widehat z_n^T(\hSx-\Sx)\widehat z_n|\leq C(s^{1/2}\lambda\|\widehat z_n\|_2^2+\lambda\|\widehat z_n\|_1) \]
with high probability for sufficiently large $n,p$, and $q$.
\end{lemma}

\begin{lemma}\label{Additional lemma: Quadratic:  bounded: corollary}
Suppose $X$ and $Y$ are sub-Gaussian and  random vectors $\widehat z_n\in\RR^p, \widehat w_n\in\RR^q$ satisfy 
\[\|\widehat z_n\|_1+\|\widehat w_n\|_1=O_p(s^{1/2}),\]
where $s$ satisfies Assumption~\ref{assumption: sparsity}. Then there exists $C$ depending only on the subgaussian parameters of $X$ and $Y$ so that
\[|\widehat z_n^T(\hSxy-\Sxy)\widehat w_n|\leq C\lambda\slb s^{1/2}(\|\widehat z_n\|_2^2+\|\widehat w_n\|_2^2)+(\|\widehat z_n\|_1+\|\widehat w_n\|_1)\srb\]
with high probability as $n,p,q\to\infty$.
\end{lemma}

\begin{lemma}[Lemma 8 of \citeauthor{jankova2018}]
\label{lemma: sara 8}
Suppose $X$ is sub-Gaussian and $z\in\RR^p$ is a vector with $\|z\|_0=s$. Then 
\[\sup_{z\in\RR^{p}}\frac{z^T(\hSx-\Sx)z}{\|z\|^2_2}=O_p((s\log p/n)^{1/2}).\]
\end{lemma}


\begin{lemma}\label{lemma: additional: l1 norm of x and l2 norm of z knot}
Suppose $\widehat z_n$ is a random vector, possibly depending on $X$, so that $\|\widehat z_n-z_0\|_1=o_p(1)$ where $z_0$ is a fixed vector with finite $l_2$ norm. Then depending only on the subgaussian parameter of $X$ so that \[|x^T(\hSx-\Sx)\widehat z_n|\leq C \|z_0\|_2\|x\|_1O_p(\lambda)\]
with high probability as $n,p,q\to\infty$.
\end{lemma}

Our next result is on multivariate normal distribution. This result quite well known and can be obtained via straightforward calculation.
\begin{fact}[Fourth moments of multivariate normal distribution]\label{fact: higher moments of bivariate normal}
 Suppose $X\sim N_p(0,\Sx)$. Then
 \[E[X_1^2X_2^2]=(\Sx)_{11}(\Sx)_{22}+2(\Sx)_{12}^2,\]
 \[E[X_1^3X_2]=3(\Sx)_{11}(\Sx)_{12},\]
 \[E[X_1X_2X_3X_4]=(\Sx)_{12}(\Sx)_{34}+(\Sx)_{13}(\Sx)_{24}+(\Sx)_{14}(\Sx)_{23}.\]
 \end{fact}
 
 
  
  The next result gives an expression for the variance of quadratic terms of multivariate Gaussian random vectors.
  \begin{fact}\label{fact: variance of quadratic terms}
  Suppose
  \[(X,Y)\sim N_{p+q}(0,\Sigma)\quad\text{where}\quad \Sigma=\begin{bmatrix}
     \Sx & \Sxy\\
     \Syx & \Sy
  \end{bmatrix}.\] 
  Further suppose $a,b,z\in\RR^p$ and $d\in\RR^q$. Then it follows that
  \[\text{var}(a^TXX^Tb)=(a^T\Sx a)( b^T\Sx b)+(a^T\Sx b)^2\ \text{and}\  \text{var}(z^TXY^Td)=(z^T\Sx z)(d^T\Sy d)+(z^T\Sxy d)^2.\]
  \end{fact}
  
  The next fact is regarding the sub-exponential norms of quadratic forms in $X$ and $Y$.
  \begin{fact}\label{fact: sub-exponential norms}
Suppose $a,c\in\RR^p$ and $b\in\RR^q$. Then sub-Gaussian random vectors $X$ and $Y$ satisfy
\[\|a^TXY^Tb\|_{\psi_1}\leq \|a\|_2\|b\|_2\|X\|_{\psi_2}\|Y\|_{\psi_2},\quad \|a^TXX^Tc\|_{\psi_1}\leq \|a\|_2\|c\|_2\|X\|_{\psi_2}^2.\]
\end{fact}
  
  Next, we present a result on Gaussian random vectors.
   \begin{lemma}\label{lemma: variance lemma facts}
  Suppose 
  \[(X,Y)\sim N_{p+q}(0,\Sigma)\quad\text{where}\quad \Sigma=\begin{bmatrix}
     \Sx & \Sxy\\
     \Syx & \Sy
  \end{bmatrix}.\] 
  Let  $a,b,z\in\RR^p$ and $c,d,\gamma\in\RR^q$ be some deterministic vectors. Then
  \begin{align*}
      T= a^TXX^Tb+c^TYY^Td-z^TXY^Td-b^TXY^T\gamma
  \end{align*}
  has variance
  \begin{align*}
    \MoveEqLeft   (a^T\Sx a)(b^T\Sx b)+(a^T\Sx b)^2+(c^T\Sy c)(d^T\Sy d)+(c^T\Sy d)^2\\
  &\ +(z^T\Sx z)(d^T\Sy d)+(z^T\Sxy d)^2+(b^T\Sx b)(\gamma^T\Sy\gamma)+(b^T\Sxy\gamma)^2\\
  &\ +2(a^T\Sxy c)(b^T\Sxy d)+2(a^T\Sxy d)(b^T\Sxy c)+2(z^T\Sx b)(d^T\Sy\gamma)+2(z^T\Sxy\gamma) (b^T\Sxy d)\\
  &\ -2 (a^T\Sx z)(b^T\Sxy d)-2(a^T\Sxy d)(b^T\Sx z)-2(a^T\Sx b)(b^T\Sxy \gamma)-2(a^T\Sxy \gamma)( b^T\Sx b)\\
  &\ -2(c^T\Syx z)(d^T\Sy d)-2(c^T\Sy d)(d^T\Syx z)-2(c^T\Syx b)( d^T\Sy \gamma)- 2(c^T\Sy\gamma) (d^T\Syx b).
  \end{align*}
 \end{lemma} 
 
 The next fact is a result obtained using the delta method.
 \begin{fact}\label{fact: delta method}
Suppose 
\[n^{1/2}\begin{bmatrix}
\widehat\theta_n-\theta\\ \widehat\vartheta_n-\vartheta
\end{bmatrix}\to_d N_2\left(\begin{bmatrix}
0 \\ 0
\end{bmatrix},\begin{bmatrix}
\sigma^2_{11} & \sigma_{12}\\
\sigma_{21} & \sigma^2_{22}
\end{bmatrix}\right) \]
where the covariance matrix is positive definite and $\theta\neq 0$.
Then 
\[n^{1/2}(\widehat\theta_n^{1/2}\widehat\vartheta_n-\theta^{1/2}\vartheta)\to_d N\lb 0,\frac{\vartheta^2\sigma_{11}^2}{4\theta}+\sigma_{22}^2\theta+\vartheta\sigma_{12}\rb.\]
\end{fact}

\begin{proof}[of Fact~\ref{fact: delta method}]
The proof follows by delta method. Let us denote $f(x,y)=x^{1/2}y$. Then the gradient of $f$ writes as $\grad f(x,y)=(x^{-1/2}y/2,x^{1/2})$. Observe that
\[\grad f(x,y)^T\begin{bmatrix}
\sigma^2_{11} & \sigma_{12}\\
\sigma_{21} & \sigma^2_{22}
\end{bmatrix}\grad f(x,y)=
\begin{bmatrix}
\frac{x^{-1/2}y}{2} & x^{1/2}
\end{bmatrix}\begin{bmatrix}
\sigma^2_{11} & \sigma_{12}\\
\sigma_{21} & \sigma^2_{22}
\end{bmatrix}\begin{bmatrix}
x^{-1/2}y/2\\x^{1/2}
\end{bmatrix}=\frac{y^2\sigma_{11}^2}{4x}+\sigma_{22}^2 x+y\sigma_{12}\]
is positive if $y>0$.

Note that since $\theta\neq 0$, $\grad  f(\theta,\vartheta)$ is non-zero. Therefore, an application of delta method establishes that 
 $n^{1/2}(\widehat\theta_n^{1/2}\widehat\vartheta_n-\theta^{1/2}\vartheta)$ is asymptotically centered normal with variance
\[\frac{\vartheta^2\sigma_{11}^2}{4\theta}+\sigma_{22}^2\theta+\vartheta\sigma_{12}.\]

\end{proof}

\section{Proof of Lemmas in Section ~\ref{sec: method}}
\label{addlemma: methods}

In this section, we prove the lemmas from Section ~\ref{sec: method}.
\begin{proof}[of Lemma~\ref{lemma: objective function}]
Suppose $A=\Sx^{-1/2}\Sxy\Sy^{-1/2}$.
 Denoting $\tU=\Sx^{1/2} U$ and $\tV=\Sy^{1/2} V$, we observe   that $\tU\in \mathcal O(p,r)$ and $\tV\in\mathcal O(q,r)$, where the latter sets are  defined in \eqref{notation: O(p,r)}. Hence,  $\sum_{i=1}^{r}\Lambda_i \tu_i \tv_i^T=\tU\Lambda\tV^T$ is a singular value decomposition of $A$. Let us also define $A_k=\sum_{i=1}^{k}\Lambda_i \tu_i \tv_i^T$.
When $k=1$. Then from \eqref{def: Ak} it can be shown that
 \begin{equation}\label{min: Ak part 2}
     \argmin_{(x,y)\in\RR^p\times\RR^q}\| A-xy^T\|_F^2=\{(c_1\tu_1,c_2\tv_1): c_1,c_2\in\RR,\ c_1c_2=\Lambda_1\}.
 \end{equation}
  From \eqref{min: Ak part 2} we  deduce that for any $c_1,c_2\in\RR$, $(c_1\tu_1,c_2\tv_1)$ is a solution to 
  \begin{mini}
      {(x,y)\in\RR^p\times\RR^q}{\|A-xy^T\|_F^2}{\label{min: Ak when k=1}}{}
  \end{mini}
 as long as  $c_1c_2=\Lambda_1=\rhk$. Thus, there is an infinite set of minimizers of \eqref{min: Ak when k=1}. 
Since $\|\tu_1\|_2=\|\tv_1\|_2=1$, if we add the additional constraint $x^T x=y^T y$ to \eqref{min: Ak when k=1}, its only minimizers  are $\pm(\rhk^{1/2}\tu_1,\rhk^{1/2}\tv_1)$.
More precisely, for any $C>0$ we have
\[\pm(\rhk^{1/2}\tu_1,\rhk^{1/2}\tv_1)=\argmin_{x\in\RR^p, y\in\RR^q} \lbs \|A-xy^T\|_F^2 +\frac{C}{4}(x^Tx-y^Ty)^2\rbs.\]
 Because $\Sx$ and $\Sy$ are positive definite,
the reparametrization $x\mapsto\Sx^{1/2} x$ and $y\mapsto \Sy^{1/2} y$ yields
\begin{align}\label{min: tu tv}
  \pm (\rhk^{1/2}\Sx^{-1/2}\tu_1,\rhk^{1/2}\Sy^{-1/2}\tv_1)=\argmin_{x\in\RR^p,y\in\RR^q}\lbs\|A-\Sx^{1/2} x y^T\Sy^{1/2}\|_F^2+\frac{C}{4}(x^T\Sx x-y^T\Sy y)^2\rbs. 
\end{align}
Finally, noting
$\tu_1=\Sx^{1/2}\alk$ and $\tv_1=\Sy^{1/2}\bk$,  we see that the left hand side of \eqref{min: tu tv} equals $\pm(\rhk^{1/2}\alk,\rhk^{1/2}\bk)$.
Hence, for any $C>0$
\[(\rhk^{1/2}\alk,\rhk^{1/2}\bk)=\argmin_{x\in\RR^p,y\in\RR^q}\lbs\|A-\Sx^{1/2} x y^T\Sy^{1/2}\|_F^2+\frac{C}{4}(x^T\Sx x-y^T\Sy y)^2\rbs. \]
A little algebra shows
\begin{align*}
   \MoveEqLeft \|A-\Sx^{1/2} x y^T\Sy^{1/2}\|_F^2\\
   =&\ \langle A-\Sx^{1/2} x y^T\Sy^{1/2}, A-\Sx^{1/2} x y^T\Sy^{1/2}\rangle\\
   =&\ \langle A,A\rangle -2\langle \Sx^{1/2} xy^T\Sy^{1/2}, A\rangle+(x^T\Sx x)(y^T\Sy y)\\
   =& \langle A,A\rangle -2Tr(\Sy^{-1/2}\Syx xy^T\Sy^{1/2})+(x^T\Sx x)(y^T\Sy y)\\
   =&\  \langle A,A\rangle- 2x^T\Sxy y+(x^T\Sx x)(y^T\Sy y)
\end{align*}
Since the minimizers do not depend on $\langle A, A\rangle$, the proof follows by elementary algebra.
\end{proof}

\begin{proof}[of Lemma~\ref{lemma: Hessian positive definite}]
Let us denote $\tu_i=\Sx^{1/2}u_i$
and $\tv_i=\Sy^{1/2}v_i$ for $i=1,\ldots, r$.
 Letting $D=\text{Diag}(\Sx^{1/2},\Sy^{1/2})$, and recalling $x^0=\rhk^{1/2}\alk$ and $y^0=\rhk^{1/2}\bk$, we rewrite $H^0$ in \eqref{def: H knot} as
  \begin{align}\label{def: Hessian general}
      H^0= 2\rhk D\begin{bmatrix}
      I_p+2\tu_1\tu_1^T & -\Sx^{-1/2}\Sxy\Sy^{-1/2}/\rhk\\
      -\Sy^{-1/2}\Syx\Sx^{-1/2}/\rhk & I_q+2\tv_1\tv_1^T
      \end{bmatrix}D.
  \end{align}
  
 Let us consider
 \[A=\begin{bmatrix}
      I_p+2\tu_1\tu_1^T & -\Sx^{1/2}U\Lambda V^T\Sy^{1/2}/\rhk\\
      -\Sy^{1/2}V\Lambda U^T\Sx^{1/2}/\rhk & I_q+2\tv_1\tv_1^T
      \end{bmatrix}.\]
      If we can show that $\Lambda_{min}(A)>0$ then it will follow that $A$ is invertible, implying
  \[H^0(x^0,y^0)^{-1}=(2\rhk)^{-1}D^{-1}A^{-1}D^{-1},\]
  leading to
  \[\|H^0(x^0,y^0)^{-1}\|_{op}\leq (2\rhk)^{-1}\|D^{-1}\|_{op}\|A^{-1}\|_{op}\|D^{-1}\|_{op}\]
  which, combined with  Assumption \ref{assump: bounded eigenvalue}, yields 
  \[\Lambda_{min}(H(x^0,y^0))^{-1}\leq \frac{M\Lambda_{min}(A)^{-1}}{2\rhk }.\]
  Therefore,
  \begin{equation}\label{inlemma: pd Hessian}
      \Lambda_{min}(H(x^0,y^0))\geq 2\rhk \Lambda_{min}(A)/M.
  \end{equation}

  Therefore, it suffices to find a lower bound on $\Lambda_{min}(A)$. To that end, first note that since $\{\tu_1,\ldots,\tu_r\}$  is a set of orthogonal vectors, they can be extended to an orthogonal basis $\{\tu_1,\ldots,\tu_r,\tu_{r+1},\ldots,\tu_p\}$ of $\RR^p$. Similarly, we can extend 
  $\{\tv_1,\ldots,\tv_r\}$  to an orthogonal basis $\{\tv_1,\ldots,\tv_r,\tv_{r+1},\ldots,\tv_q\}$ of $\RR^q$. 
  
  Let us consider
  $z=(\tu_i,\tv_i)$ for $2\leq i\leq r$. Since $\tu_i^T\tu_1=\tv_i^T\tv_1=0$, and
  \[Az=\begin{bmatrix}
  (1-\Lambda_i/\rhk) \tu_i \\ (1-\Lambda_i/\rhk) \tv_i
  \end{bmatrix}=(1-\Lambda_i/\rhk)z.\]
  
   Thus $z$ is an eigenvector with eigenvalue $1-\Lambda_i/\rhk$. A similar case is when $z=(\tu_i,-\tv_i)$. Then
   \[Az=\begin{bmatrix}
   \tu_i+\Lambda_i/\rhk \tu_i\\ -\tv_i-\Lambda_i/\rhk \tv_i
   \end{bmatrix}=(1+\Lambda_i/\rhk)z.\]
   
   In this case also $Az=z$ with eigenvalue $1-\Lambda_i/\rhk$.
   Now suppose $z=(\tu_1,\tv_1)$. Then
    Then $Az=2z$, implying it is an eigenvector with eigenvalue 2. Now consider $z=(\tu_1,-\tv_1)$. Then $Az=4z$
  which implies $z$ is an eigenvector with eigenvalue 4. Therefore, we have obtained $2r$ orthogonal eigenvectors of $A$.
  Next, consider $z=(\tu_i,0)$ where $r+1\leq i\leq p $. Then $U^T\Sx^{1/2}\tu_i=0$ as well as $\tu^T\tu_i=0$. Hence $Az=z$, i.e. $z$ is an eigenvector with eigenvalue 1.
  Similarly,  $z=(0,\tu_j)$ for $r+1\leq j\leq q$ is also an eigenvector of $A$ with eigenvalue one. Therefore, we have obtained total $2r+(p-r)+(q-r)=p+q$ many orthogonal eigenvectors of $A$ with non-zero eigenvalues. and  $\Lambda_{min}(A)=1-\Lambda_2/\rhk$. Therefore the current lemma follows from \eqref{inlemma: pd Hessian}.
\end{proof}


\section{ Proof of Theorem~\ref{thm: for alpha} }
\label{app: prf of main theorem}

 \subsection{Preliminaries for the proof of Theorem~\ref{thm: for alpha}}
We keep using the notations defined in the earlier sections. Especially, recall the $\lambda$ defined in \eqref{def:lambda}. Several times we will use the followings  without stating which holds by Condition~\ref{assumption: precision matrix}:
\[\max_{1\leq j\leq p+q}\|(\hf)_j-\Phi^0_j\|_1=O_p(s^{\kappa+1/2}\lambda),\]
 and
 \[\max_{1\leq j\leq p+q}\|(\hf)_j-\Phi^0_j\|_2=O_p(s^{\kappa}\lambda).\]
 Note that Lemma~\ref{corollary: rate of x and y} implies
\[\inf_{w\in\{\pm 1\}}\|w\hx-x^0\|_1+\inf_{w\in\{\pm 1\}}\|w\hy-y^0\|_1=O_p(s^{\kappa+1/2}\lambda),\]
\[\inf_{w\in\{\pm 1\}}\|w\hx-x^0\|_2+\inf_{w\in\{\pm 1\}}\|w\hy-y^0\|_2=O_p(s^{\kappa}\lambda),\]
and Lemma~\ref{corollary: rate of rho} implies  $|\hro-\rhk|=O_p(s^{1/2}\lambda)$. It turns out that if  $\|\hx-x^0\|_1$ and $\|\hx-x^0\|_2$ are small, then $(\hdai)_i-\rhk^{1/2}(\alk)_i$ is asymptotically normal for $1\leq i\leq p$, but if $\|\hx+x^0\|_1$ and $\|\hx+x^0\|_2$ are small, then $(\hdai)_i+\rhk^{1/2}(\alk)_i$ will be asymptotically normal. An analogous result holds for $\hdbi$ and $\bk$. Therefore, there can be four different scenarios depending on whether $\hx$ or $\hy$ has a sign flip. Since the proofs for all these cases are identical, we will only consider the case when $\hx$ and $\hy$ are aligned with $\alk$ and $\bk$, i.e.
\[\|\hx-x^0\|_1=\inf_{w\in\{\pm 1\}}\|w\hx-x^0\|_1,\quad \|\hy-y^0\|_1=\inf_{w\in\{\pm 1\}}\|w\hy-y^0\|_1,\]
and
\[\|\hx-x^0\|_2=\inf_{w\in\{\pm 1\}}\|w\hx-x^0\|_2,\quad \|\hy-y^0\|_2=\inf_{w\in\{\pm 1\}}\|w\hy-y^0\|_2.\]
Therefore, we will have
\[\|\hx-x^0\|_1+\|\hy-y^0\|_1=O_p(s^{\kappa+1/2}\lambda),\quad \|\hx-x^0\|_2+\|\hy-y^0\|_2=O_p(s^{\kappa}\lambda). \]
The following fact follows immediately from the above:
\begin{align}\label{intheorem: main: fact: hx and hy}
    \|\hx\|_1=O_p(s^{1/2}),\ \|\hx\|_2=O_p(1),\  \|\hy\|_1=O_p(s^{1/2}),\ \|\hy\|_2=O_p(1).
\end{align}

Suppose $x\in\RR^p$ and $y\in\RR^q$. 
Recall the definitions of $\partial \widehat h_n/\partial x$ and $\partial \widehat h_n/\partial y$ from \eqref{def: hat dh general}. Also, recall from \eqref{def: delta h} that when $C=2$,
\begin{gather*}
    {\pdv{h}{x}}(x,y)= 2(x^T\Sx y)\Sx x-2\Sxy y\nn\\
     {\pdv{h}{y}}(x,y)= 2(y^T\Sy y)\Sy y-2\Syx x.
\end{gather*}
 For the sake of brevity, we will use the notations
\[\grad h(x,y)=\begin{bmatrix}
{\pdv{ h}{x}}(x,y)\\{\pdv{ h}{y}}(x,y)
\end{bmatrix},\quad \grad \widehat h_n(x,y)=\begin{bmatrix}
{\pdv{\widehat h_n}{x}}(x,y)\\{\pdv{\widehat h_n}{y}}(x,y)
\end{bmatrix}.\]
Notice also that
$\grad h(x^0,y^0)=0$ when $x^0=\rhk^{1/2}\alk$ and $y^0=\rhk^{1/2}\bk$.

\subsection{Proof architecture}
Now we will start the proof of Theorem~\ref{thm: for alpha}.
From  Definiton~\ref{def: de-biased estimators}, we find the decomposition
\begin{align*}
  -\begin{bmatrix}
\hdai\\ \hdbi
\end{bmatrix}+
  \begin{bmatrix}
x^0\\ y^0
\end{bmatrix}=&\ \Phi^0\slb \grad \widehat h_n(x^0,y^0)-\grad h(x^0,y^0)\srb\\
&\ +(\hf^T-\Phi^0)\slb \grad \widehat h_n(x^0,y^0)-\grad h(x^0,y^0)\srb\\
&\ + \hf^T\slb \grad \widehat h_n(\hx,\hy)-\grad \widehat h_n(x^0,y^0)\srb +\begin{bmatrix}
x^0-\hx\\ y^0-\hy
\end{bmatrix}.
\end{align*}
Here we used the fact that $\grad h(x^0,y^0)=0$. The above decomposition 
 indicates for $1\leq i\leq p$,
\begin{align*}
\MoveEqLeft   x^0_i-(\hdai)_i\\
= &\ 2\underbrace{(\Phi^0_i)^T\begin{bmatrix}
((x^0)^T\hSx x^0)(\hSx-\Sx)x^0-(\hSxy-\Sxy)y^0+((x^0)^T(\hSx-\Sx)x^0)\Sx x^0\\ ((y^0)^T\hSy y^0)(\hSy-\Sy)y^0-(\hSyx-\Syx)x^0+((y^0)^T(\hSy-\Sy)y^0)\Sy y^0
\end{bmatrix}}_{T_1(i)}\\
&\ + \underbrace{((\hf)_i-\Phi^0_i)^T (\grad \widehat  h_n(x^0,y^0)-\grad h(x^0,y^0))}_{T_2(i)}\\
&\ + \underbrace{(\hf)_{i}^T\lb\widehat \grad h_n(\hx,\hy)-\widehat\grad h_n(x^0,y^0)-\grad h(\hx,\hy)+\grad h(x^0,y^0)\rb}_{T_3(i)}\\
&\ +\underbrace{(\hf)_{i}^T\slb\grad  h(\hx,\hy)-\grad  h(x^0,y^0) -H^0\begin{bmatrix}
\hx-x^0\\ \hy-y^0
\end{bmatrix}\srb}_{T_4(i)}\\
&\ +\underbrace{\slb e_i^T-(\hf)_i^T H^0\srb\begin{bmatrix}
x^0-\hx\\ y^0-\hy
\end{bmatrix}}_{T_5(i)}
\end{align*}
The term  $T_1(i)$ is the main contributing term in that it is asymptotically equivalent to $\mathcal L_i$. We will prove the theorem in two steps. The first step shows  that 
\[\max_{1\leq i\leq p}|T_1(i)-\mathcal{L}_i|=O_p(s\lambda^2) \]
and $n^{1/2}\mathcal{L}_i$ converges weakly to a centered Gaussian random variable with variance $4\sigma^2_i$.
The last four  steps show that the remaining terms  are asymptotically negligible, i.e.
\[\max_{1\leq i\leq p}\sum_{k=2}^5|T_{k}(i)|=O_p(s^{2\kappa}\lambda^2).\]
Because $s^{2\kappa}\lambda^2=o(n^{-1/2})$, the proof follows.

As in Section~\ref{sec: asymptotic theory}, we denote  $\Phi^0_i=((\Phi^0_i)_1, (\Phi^0_i)_2)$ where $(\Phi^0_i)_1\in\RR^p$ and $(\Phi^0_i)_2\in\RR^q$.
For notational convenience, we will denote $\pa=(\Phi^0_i)_1$ and $\pb=(\Phi^0_i)_2$. Similarly we define $\pha$ and $\phb$ so that $(\hf)_i=(\pha,\phb) $. We drop the $n$ from the subscript of $\hf$ for the sake of simplicity. 
The following fact, which follows from Lemma~\ref{lemma: Phi and H knot } and Assumption~\ref{assump: eigengap assumptions}, will be used repeatedly:
\begin{equation}\label{intheorem: main: bound: l2 norm of pa}
    \max_{1\leq i\leq p+q}\max\{\|\pa\|_2,\|\pb\|_2\}\leq \max_{1\leq i\leq p+q}\|\Phi^0_i\|_2\leq  \|\Phi^0\|_{op}\leq \frac{M}{2(\rhk-\Lambda_2)}\leq \frac{M}{2\epsilon_0}.
\end{equation}

\subsection{Step 1: showing the asymptotic normality of $T_1(i)$}
We can split $T_1(i)$ into two terms:
\begin{align*}
    T_1(i)=&\ \underbrace{2(\Phi^0_i)^T\begin{bmatrix}
\rhk(\hSx-\Sx)x^0-(\hSxy-\Sxy)y^0+((x^0)^T(\hSx-\Sx)x^0)\Sx x^0\\ \rhk(\hSy-\Sy)y^0-(\hSyx-\Syx)x^0+((y^0)^T(\hSy-\Sy)y^0)\Sy y^0
\end{bmatrix}}_{\mathcal{L}_i}\\
&\ + \underbrace{2(\Phi^0_i)^T\begin{bmatrix}
((x^0)^T\hSx x^0-\rhk)(\hSx-\Sx)x^0\\
((y^0)^T\hSy y^0)-\rhk)(\hSy-\Sy)y^0
\end{bmatrix}}_{T_{12}(i)}
\end{align*}
Note that the second term $T_{12}$ is bounded by 
\begin{align*}
 |(x^0)^T(\hSx-\Sx) x^0\|(\pa)^T(\hSx-\Sx)x^0|+|(y^0)^T(\hSy-\Sy) y^0\|(\pb)^T(\hSy-\Sy)y^0|.
\end{align*}
By Lemma~\ref{lemma: sara 8} it follows that
\[|(x^0)^T(\hSx-\Sx) x^0|=\|x^0\|^2_2O_p(s^{1/2}\lambda).\]
From Lemma~\ref{lemma: additional: l1 norm of x and l2 norm of z knot} it follows that
\[\max_{1\leq i\leq p}|(\pa)^T(\hSx-\Sx)x^0|=\max_{1\leq i\leq p}\|\pa\|_2\|x^0\|_1O_p(\lambda).\]
From Lemma~\ref{lemma: Phi and H knot } it follows that there exists $C>0$ so that $\|\Phi^0\|_{op}\leq C$. Therefore,
\[\max_{1\leq i\leq p}\|\pa\|_2\leq \max_{1\leq i\leq p}\|(\Phi^0)_i\|_2\leq \|\Phi^0\|_{op}\leq C.\]
  From Lemma~\ref{lemma:norm:  l1 and l2 norm} it also follows that $\|x^0\|_2=O_p(1)$ and $\|x^0\|_1=O(s^{1/2})$. Thus
\[ |(x^0)^T(\hSx-\Sx) x^0\|(\pa)^T(\hSx-\Sx)x^0|=O_p(s\lambda^2).\]
Similarly we can show that
\[|(y^0)^T(\hSy-\Sy) y^0\|(\pb)^T(\hSy-\Sy)y^0|=O_p(s\lambda^2).\]
Thus, we conclude
We have  established in \eqref{rate: T12} that
\begin{align}\label{rate: T12}
   \max_{1\leq i\leq p}|T_1(i)-\mathcal{L}_i|= \max_{1\leq i\leq p} |T_{12}(i)|=O_p(s\lambda^2).
\end{align}


\subsection{Step 2: Showing $T_2(i)$ is small}
We have
\begin{align*}
\grad \widehat h_n(x^0,y^0)-\grad h(x^0,y^0)
=&\ 2\begin{bmatrix}
\rhk(\hSx-\Sx)x^0-(\hSxy-\Sxy)y^0+((x^0)^T(\hSx-\Sx)x^0)\Sx x^0\\ \rhk(\hSy-\Sy)y^0-(\hSyx-\Syx)x^0+((y^0)^T(\hSy-\Sy)y^0)\Sy y^0
\end{bmatrix}\\
&\ + \begin{bmatrix}
((x^0)^T\hSx x^0-\rhk)(\hSx-\Sx)x^0\\
((y^0)^T\hSy y^0)-\rhk)(\hSy-\Sy)y^0
\end{bmatrix}
.
\end{align*}
 Lemma~\ref{lemma:norm:  l1 and l2 norm} and \ref{lemma: sara 8} imply that
\[(x^0)^T(\hSx-\Sx) x^0=O_p(s^{1/2}\lambda) \quad\text{and}\quad (y^0)^T(\hSy-\Sy) y^0=O_p(s^{1/2}\lambda).\]
 Therefore using Assumption~\ref{assump: bounded eigenvalue}, we  obtain that 
\begin{align*}
|T_2(i)|\leq &\ \|\pha-\pa\|_1\slb \|(\hSx-\Sx) x^0\|_{\infty}+\|(\hSxy-\Sxy)y^0\|_{\infty} +\|(\hSx-\Sx)x^0\|_{\infty} O_p(s^{1/2}\lambda)\srb\\
&\ +M\|\pha-\pa\|_2 \|x^0\|_2 O_p(s^{1/2}\lambda)\\
&\ +\|\phb-\pb\|_1\slb \|(\hSy-\Sy)y^0\|_{\infty}+\|(\hSyx-\Syx)x^0\|_{\infty} + \|(\hSy-\Sy)y^0\|_{\infty} O_p(s^{1/2}\lambda)\srb\\
&\  +M\|\phb-\pb\|_2 \|y^0\|_2 O_p(s^{1/2}\lambda).
\end{align*}
Now note that
\begin{align}\label{intheorem: main: rate of l1 distance pha and pa}
    \max\{\|\pha-\pa\|_1,\|\phb-\pb\|_1\}\leq \|(\hf)_i-\Phi^0_i\|_1=O_p(s^{\kappa+1/2}\lambda)
\end{align}
and
\begin{align}\label{intheorem: main: rate of l2 distance pha and pa}
    \max\{\|\pha-\pa\|_2,\|\phb-\pb\|_2\}\leq \|(\hf)_i-\Phi^0_i\|_2=O_p(s^{\kappa}\lambda)
\end{align}
 by Condition~\ref{assumption: precision matrix}.
Also  by Lemma~\ref{lemma:norm:  l1 and l2 norm}, Lemma~\ref{result: inf norm: dif },  and Lemma~\ref{result: inf norm: dif },
\begin{equation*}
    \|(\hSx-\Sx) x^0\|_{\infty},\|(\hSxy-\Sxy)y^0\|_{\infty}, \|(\hSy-\Sy)y^0\|_{\infty},\|(\hSyx-\Syx)x^0\|_{\infty}=O_p(\lambda).
\end{equation*}
Therefore, 
\[\max_{1\leq i\leq p}|T_2(i)|=O_p(s^{\kappa+1/2}\lambda^2+s^{\kappa+1}\lambda^3).\]

\subsection{Step 3: Showing $T_3(i)$ is asymptotically negligible}
For any vector $z\in\RR^{p\times q}$, consider the partition $(z_1,z_2)=z$ where $z_1\in\RR^p$ and $z_2\in\RR^q$. When $z=\widehat\grad h_n(\hx,\hy)-\widehat\grad h_n(x^0,y^0)-\grad h(\hx,\hy)+\grad h(x^0,y^0)$, we derive the expression 
  \begin{align*}
      \MoveEqLeft  2^{-1}\lb\widehat\grad h_n(\hx,\hy)-\widehat\grad h_n(x^0,y^0)-\grad h(\hx,\hy)+\grad h(x^0,y^0)\rb_1\\
      =&\ (\hx^T\hSx\hx)\hSx(\hx-x^0)-\hSxy(\hy-y^0)+2\slb \hx^T\hSx\hx-(x^0)^T\hSx x^0\srb \hSx x^0\\
      &\ -\lbs (\hx^T\Sx\hx)\Sx(\hx-x^0)-\Sxy(\hy-y^0)+\slb\hx^T\Sx\hx-(x^0)^T\Sx x^0\srb\Sx x^0\rbs\\
      =&\ (\hx^T\hSx\hx)(\hSx-\Sx)(\hx-x^0)
      +\slb \hx^T(\hSx-\Sx)\hx\srb \Sx(\hx-x^0)\\
      &\ -(\hSxy-\Sxy)(\hy-y^0)+(\hx^T\hSx\hx-(x^0)^T\hSx x^0)(\hSx-\Sx)x^0\\
      &\ +\slb \hx^T\hSx\hx-(x^0)^T\hSx x^0-\hx^T\Sx\hx+(x^0)^T\Sx x^0\srb \Sx x^0.
  \end{align*}
 Using Assumption~\ref{assump: bounded eigenvalue} we obtain that
  \begin{align*}
    \MoveEqLeft  2^{-1}\bl(\pa)^T\lb\widehat\grad h_n(\hx,\hy)-\widehat\grad h_n(x^*,y^*)-\grad h(\hx,\hy)+\grad h(x^*,y^*)\rb_1\bl\\
    \leq &\ \underbrace{(\hx^T\hSx\hx)|(\pa)^T(\hSx-\Sx)(\hx-x^0)|}_{T_{31}(i)}+\underbrace{M\|\pa\|_2\|\hx-x^0\|_2\bl  \hx^T(\hSx-\Sx)\hx\bl}_{T_{32}(i)}\\
    &\ +\underbrace{|(\pa)^T(\hSxy-\Sxy)(\hy-y^0)|}_{T_{33}(i)}+
    \underbrace{\bl \hx^T\hSx\hx-(x^0)^T\hSx x^0\bl |(\pa)^T(\hSx-\Sx)x^0|}_{T_{34}(i)}\\
    &\ +\underbrace{M \|\pa\|_2\|x^0\|_2\abs{\hx^T\hSx\hx-(x^0)^T\hSx x^0-\hx^T\Sx\hx+(x^0)^T\Sx x^0}}_{T_{35}(i)}
  \end{align*}
 We will provide some bounds on the $T_{2k}(i)$'s $(k=1,\ldots,5)$. The $O_p$ terms appearing in the bounds do not depend on $i$, and depend only on $M$, and the Sub-gaussian norms of  $X$ and $Y$.
 
 From Lemma~\ref{corollary: colar: normalization of hx} it follows that $\hx^T\hSx\hx=\rhk+o_p(1)$. Using  Lemma~\ref{lemma: additional: l1 norm of x and l2 norm of z knot} we then obtain
 \begin{align*}
 \max_{1\leq i\leq p}|T_{31}(i)|\leq &\  \max_{1\leq i\leq p}\|\pa\|_2\|\hx-x^0\|_1 O_p(\lambda)\stackrel{(a)}{=}\max_{1\leq i\leq p}\|\Phi^0_i\|_2O_p(s^{\kappa+1/2}\lambda^2)\\
 \leq &\ \|\Phi^0\|_{op} O_p(s^{\kappa+1/2}\lambda^2)\stackrel{(b)}{=}O_p(s^{\kappa+1/2}\lambda^2)
 \end{align*}
where (a) follows from Lemma~\ref{corollary: rate of x and y} and (b)  follows from Lemma~\ref{lemma: Phi and H knot }. Next, noting $\|\hx\|_2=O_p(1)$ and $\|\hx\|_1=O_p(1)$ by Lemma~\ref{corollary: rate of x and y}, and using Lemma~\ref{lemma: Quadratic:  bounded},   we obtain
\[|\hx^T(\hSx-\Sx)\hx|= O_p(s^{1/2}\lambda).\]
Since Lemma~\ref{corollary: rate of x and y} implies $\|\hx-x^0\|_2=O_p(s^{\kappa}\lambda)$, and Lemma~\ref{lemma: Phi and H knot } implies $\max_{1\leq i\leq p}\|\pa\|_2\leq \|\Phi^0\|_{op}=O(1)$, we have
\[\max_{1\leq i\leq p}|T_{32}(i)|=O_p(s^{\kappa+1/2}\lambda^2).\]
In the same way as we did for $T_{31}$,
We can deduce
$\max_{1\leq i\leq p}|T_{33}(i)|=O_p(s^{\kappa+1/2}\lambda^2)$.

To control $T_{34}(i)$, first note that
\begin{align*}
 \MoveEqLeft
 \bl \hx^T\hSx\hx-(x^0)^T\hSx x^0\bl\\
\leq &\ |( \hx+x^0)^T(\hSx-\Sx)(\hx-x^0)|+  | (\hx+x^0)^T\Sx(\hx-x^0)|,
\end{align*}
whose first term can be bounded by Lemma~\ref{lemma: additional: l1 norm of x and l2 norm of z knot} and Lemma \ref{corollary: rate of x and y} to yield
\[|( \hx+x^0)^T(\hSx-\Sx)(\hx-x^0)|\leq 2\|x^0\|_2\|\hx-x^0\|_1O_p(\lambda)=O_p(s^{1/2+\kappa}\lambda^2),\]
and the second term can be bounded using Assumption~\ref{assump: bounded eigenvalue} and Corollary \ref{corollary: rate of x and y} to yield
\[|(\hx+x^0)^T\Sx(\hx-x^0)|\leq M\|2x^0\|_2\|\hx-x^0\|_2=O_p(s^{\kappa}\lambda).\]
Because $s^{1/2}\lambda=o(1)$ under  Fact~\ref{fact: slambda goes to zero},
\[\bl \hx^T\hSx\hx-(x^0)^T\hSx x^0\bl=O_p(s^{\kappa}\lambda).\]
On the other hand, another application of Lemma~\ref{lemma: additional: l1 norm of x and l2 norm of z knot} yields
\[\max_{1\leq i\leq p}|(\pa)^T(\hSx-\Sx)x^0|= O_p(\lambda)\|x^0\|_1\max_{1\leq i\leq p}\|\pa\|_2,\]
which is $O_p(s^{1/2}\lambda)$ by Lemma~\ref{lemma:norm:  l1 and l2 norm} and \eqref{intheorem: main: bound: l2 norm of pa}. Therefore 
\[\max_{1\leq i\leq p}|T_{34}(i)|=O_p(s^{1/2+\kappa}\lambda^2).\]
For $T_{35}(i)$, note that Lemma~\ref{lemma:norm:  l1 and l2 norm} and \eqref{intheorem: main: bound: l2 norm of pa} implies
\[\max_{1\leq i\leq p}\|\pa\|_2\|x^0\|_2=O_p(1).\]
On the other hand,
\begin{align*}
  \MoveEqLeft  \abs{\hx^T\hSx\hx-(x^0)^T\hSx x^0-\hx^T\Sx\hx+(x^0)^T\Sx x^0}\\
  =&\ \abs{\hx^T(\hSx-\Sx)\hx-(x^0)^T(\hSx-\Sx)x^0}\\
  =&\ \abs{ (\hx-x^0)^T(\hSx-\Sx)\hx+ (x^0)^T(\hSx-\Sx)(\hx-x^0)}\\
  \leq &\ \|x^0\|_2\|\hx-x^0\|_1O_p(\lambda)+\|x^0\|_2\|\hx-x^0\|_1 O_p(\lambda)
\end{align*}
where the last step follows from Lemma~\ref{lemma: additional: l1 norm of x and l2 norm of z knot}. Using Lemma~\ref{corollary: rate of x and y}, we conclude
\[\abs{\hx^T\hSx\hx-(x^0)^T\hSx x^0-\hx^T\Sx\hx+(x^0)^T\Sx x^0}=O_p(s^{\kappa+1/2}\lambda^2).\]
Thus $\max_{1\leq i\leq p}|T_{35}(i)|=O_p(s^{\kappa+1/2}\lambda^2)$ as well. 
Since we have shown that $\max_{1\leq i\leq p}\sum_{k=1}^5|T_{3k}(i)|=O_p(s^{\kappa+1/2}\lambda^2)$, it then follows that
\[\max_{1\leq i\leq p}\bl(\pa)^T\lb\widehat\grad h_n(\hx,\hy)-\widehat\grad h_n(x^*,y^*)-\grad h(\hx,\hy)+\grad h(x^*,y^*)\rb_1\bl=O_p(s^{\kappa+1/2}\lambda^2).\]
Similarly we can show that
\[\max_{1\leq i\leq p}\bl(\pb)^T\lb\widehat\grad h_n(\hx,\hy)-\widehat\grad h_n(x^*,y^*)-\grad h(\hx,\hy)+\grad h(x^*,y^*)\rb_2\bl=O_p(s^{\kappa+1/2}\lambda^2),\]
which completes the proof of $\max_{1\leq i\leq p}|T_{3}(i)|=O_p(s^{\kappa+1/2}\lambda^2)$

\subsection{Step 5: Showing  $T_4(i)$ is $O_p(s^{2\kappa}\lambda^2)$}
By Taylor series expansion, we obtain that
\begin{align*}
 \MoveEqLeft   \grad h(\hx,\hy)-\grad h(x^0,y^0)
 = H(\ho,\widehat \vartheta_n)\begin{bmatrix}
 \hx-x^0\\
 \hy-y^0
 \end{bmatrix}
\end{align*}
where $(\ho, \hv)$ is on the line joining $(x^0,y^0)$ and  
$(\hx,\hy)$. Therefore,
\begin{align}\label{intheorem: main: ho, xo}
    \|\ho-x^0\|_k\leq \|\hx-x^0\|_k,\quad \|\hv-y^0\|_k\leq \|\hy-y^0\|_k  \quad (k=1,2).
\end{align}
Therefore,
\begin{align*}
    T_4(i)=&\ (\hf)_i^T\slb H(\ho,\hv)-H(x^0,y^0)\srb \begin{bmatrix}
    \hx-x^0\\
    \hy-y^0
    \end{bmatrix}\\
    =&\ 2\begin{bmatrix}
    (\pa)^T & (\pb)^T
    \end{bmatrix}\begin{bmatrix}
    (\ho^T\Sx\ho- (x^0)^T\Sx x^0)\Sx & 0\\ +2\Sx(\ho\ho^T-x^0(x^0)^T)\Sx &  \\
    0 & (\hv^T\Sy\hv- (y^0)^T\Sy y^0)\Sy\\ & +2\Sy(\hv\hv^T-y^0(y^0)^T)\Sy
    \end{bmatrix}\begin{bmatrix}
    \hx-x^0\\
    \hy-y^0
    \end{bmatrix}\\
    =&\ \underbrace{(\ho^T\Sx\ho- (x^0)^T\Sx x^0)2(\pa)^T\Sx(\hx-x^0)}_{T_{41}(i)} +4\underbrace{(\pa)^T\Sx(\ho\ho^T-x^0(x^0)^T)\Sx(\hx-x^0)}_{T_{42}(i)}\\
    &\ +\underbrace{(\hv^T\Sy\hv- (y^0)^T\Sy y^0)2(\pb)^T\Sy(\hy-y^0)}_{T_{43}(i)} +\underbrace{4(\pb)^T\Sy(\hv\hv^T-y^0(y^0)^T)\Sy(\hy-y^0)}_{T_{44}(i)}
\end{align*}
It suffices to show that 
\[\max_{1\leq i\leq p}|T_{41}(i)|=O_p(s^{2\kappa}\lambda^2)\quad \text{and}\quad \max_{1\leq i\leq p}|T_{42}(i)|=O_p(s^{2\kappa}\lambda^2).\]
The proof of $T_{43}$ and $T_{44}$ will follow in a similar way, and hence will be skipped.

To control $T_{41}$, note that
\begin{align*}
    |\ho^T\Sx\ho- (x^0)^T\Sx x^0|\leq &\  |\ho^T\Sx(\ho-x^0)|+|(\ho-x^0)^T\Sx x^0|\\
    \leq &\ M\|\ho-x^0\|_2(\|\ho\|_2+\|x^0\|_2)
\end{align*}
where the last step follows from Assumption~\ref{assump: bounded eigenvalue}. From \eqref{intheorem: main: ho, xo} and Lemma~\ref{corollary: rate of x and y}, it follows that
\[ |\ho^T\Sx\ho- (x^0)^T\Sx x^0|=O_p(s^{\kappa}\lambda).\]
On the other hand, by Assumption~\ref{assump: bounded eigenvalue},
\begin{align*}
 \max_{1\leq i\leq p}  | (\pa)^T\Sx(\hx-x^0)|\leq   \max_{1\leq i\leq p}M\|\pa\|_2\|\hx-x^0\|_2
\end{align*}
which is $O_p(s^{\kappa}\lambda)$ by \eqref{intheorem: main: bound: l2 norm of pa}
and Lemma~\ref{corollary: rate of x and y}.
Therefore, 
\[\max_{1\leq i\leq p}T_{41}(i)=O_p(s^{2\kappa}\lambda^2).\]
For $T_{42}(i)$, note that by Assumption~\ref{assump: bounded eigenvalue},
\begin{align*}
    |T_{4}(i)|\leq &\ 4 M^2\|\pa\|_2 \|\hx-x^0\|_2\|\ho\ho^T-x^0(x^0)^T\|_F.
\end{align*}
From \eqref{intheorem: main: bound: l2 norm of pa} it follows that $\max_{1\leq i\leq p}\|\pa\|_2=O(1)$ and  Lemma~\ref{corollary: rate of x and y}  entails that $\|\hx-x^0\|_2=O_p(s^{\kappa}\lambda)$.  Fact~\ref{fact: Projection matrix}, on the other hand, implies that
\[\|\ho\ho^T-x^0(x^0)^T\|_F\leq \|x^0\|_2^{-1}\|\ho-x^0\|_2,\]
which is $O_p(s^{\kappa}\lambda)$ because $\|x^0\|_2\geq \rhk M^{-1/2}$ by Assumption~\ref{assump: bounded eigenvalue} and $\|\ho-x^0\|_2=O_p(s^{\kappa}\lambda)$ by \eqref{intheorem: main: ho, xo} and Lemma~\ref{corollary: rate of x and y}.
Therefore, it follows that
\[\max_{1\leq i\leq p}T_{42}(i)=O_p(s^{2\kappa}\lambda^2),\]
completing the proof of this step.

\subsection{Step 5: Showing  $T_5(i)$ is $O_p(s^{2\kappa}\lambda^2)$}
$T_5(i)$ is bounded by
\begin{align*}
\MoveEqLeft \|e_i-(\hf)_i^TH^0\|_2\slb \|\hx-x^0\|_2+\|\hy-y^0\|_2\srb\\
\stackrel{(a)}{=}&\ \|(\Phi^0_i-(\hf)_i)^TH^0\|_2 O_p(s^{\kappa}\lambda)\\
\leq &\ \|\Phi^0_i-(\hf)_i\|_2\|H^0\|_{op}O_p(s^{\kappa}\lambda)
\end{align*}
where (a) follows from Lemma~\ref{corollary: rate of x and y}.
 From Lemma~\ref{lemma: Phi and H knot } and   \eqref{intheorem: main: rate of l2 distance pha and pa}
 it thus follows that $\max_{1\leq i\leq p}|T_5(i)|=O_p(s^{2\kappa}\lambda^2)$.
 
\section{Proof of Proposition~\ref{corollary: main theorem}}
\label{sec: proof of prop}

 First we state and prove a lemma that is key to proving Proposition~\ref{corollary: main theorem}. This lemma establishes the joint asymptotic distribution of the random vector $\mathcal L$ defined in \eqref{def: L} in terms of the $p+q$-variate Gaussian vector $\mathbb Z$ appearing in the statement of Proposition~\ref{corollary: main theorem}.

\begin{lemma}\label{lemma: berry essen}
Let $\Sigma_z$ be the covariance matrix of $\mathcal Z=(\mathcal Z(1),\ldots,\mathcal Z(p+q))$.
Under the set-up of Proposition~\ref{corollary: main theorem}, there exists a constant $C>0$ depending only on the sub-Gaussian norms of $X$ and $Y$, and the constants $M$ and $\epsilon_0$, so that
\[\sup_{A\in\mathcal A}\bl P\slb n^{1/2}\mathcal L\in A\srb-P\slb 2\mathbb Z\in A\srb\bl\leq C\lb\frac{\log^7((p+q)n)}{n}\rb^{1/6}\]
where $\mathbb Z$ is a $p+q$-variate centered Gaussian vector with covariance matrix $\Sigma_z$.
Here  $\epsilon_0$ and $M$ are as in Assumption~\ref{assump: eigengap assumptions}, and  Assumption~\ref{assump: bounded eigenvalue}, respectively.
\end{lemma}

\begin{proof}[of Lemma~\ref{lemma: berry essen}]
 Note that for $i=1,\ldots, p+q$,
\begin{align*}
   2^{-1} \mathcal{L}_i= &\ \rhk(\pa)^T(\hSx-\Sx)x^0+ \rhk(\pb)^T(\hSy-\Sy)y^0  -(\pa)^T(\hSxy-\Sxy)y^0\\
   &\ -(\pb)^T(\hSyx-\Syx)x^0+\underbrace{((\pa)^T\Sx x^0)}_{\xi_1(i)}(x^0)^T(\hSx-\Sx)x^0\\
   &\ +(\underbrace{(\pb)^T\Sy y^0}_{\xi_2(i)})(y^0)^T(\hSy-\Sy)y^0
\end{align*}
Observe that $E \mathcal{L}_i=0$. moreover, $\mathcal{L}_i=2n^{-1}\sum_{j=1}^n(\mathcal Z_j(i)-E\mathcal Z_j(i))$, where  
\begin{align*}
 \mathcal Z_j(i)=  &\ (\rhk(\pa)^T+\xi_1(i) (x^0)^T) X_jX_j^Tx^0+(\rhk(\pb)^T+\xi_2(i)(y^0)^T) Y_jY_j^T y^0\\
   &\ -(\pa)^TX_jY_j^T y^0-(x^0)^TX_jY_j^T\pb.
\end{align*}
Let us consider the $p+q$ variate
iid random vectors $\mathcal Z_j=(\mathcal Z_j(i))_{1\leq i\leq p+q} $ $(j=1,\ldots,n)$, which are iid copies of $\mathcal Z$. Note that we can express $\mathcal L$ in terms of $\mathcal Z_j$'s since $\mathcal  L=2n^{-1}\sum_{j=1}^n(\mathcal Z_j-E[\mathcal Z_j])$. 
We intend to use a Berry-Esseen type theorem. In particular, we apply
 Theorem 2.1 of \cite{kato2016}.
 Note that we can express $\mathcal L$ in terms of $\mathcal Z_j$'s since $\mathcal  L/2=n^{-1}\sum_{j=1}^n(\mathcal Z_j-E[\mathcal Z_j])$.  Let $\mathcal A$ be the set of all hyperrectangles in $\RR^{p+q}$. 
 Theorem 2.1 of \cite{kato2016} states that
 \begin{equation}\label{inlemma: kato statement}
     \sup_{A\in\mathcal A}\bl P\slb n^{-1/2}\sum_{j=1}^n(\mathcal Z_j-E[\mathcal Z_j])\in A\srb-P\slb \mathbb Z\in A\srb\bl\leq C\lb\frac{C_z^2\log^7((p+q)n)}{n}\rb^{1/6}
 \end{equation}
provided
\begin{itemize}
    \item[A1.]There exists $c>0$ so that $\min\limits_{1\leq i\leq p+q}\sigma_i^2>c$ where $\sigma_i^2=\text{var}(\mathcal Z(i))$.
    \item[A2.] There exists $C_z>0$ so that
    \[E\slbt|\mathcal Z(i)-E[\mathcal Z(i)]|^3\srbt\leq C_z,\quad E\slbt\slb\mathcal Z(i)-E[\mathcal Z(i)]\srb^4\srbt\leq C_z^2\quad i=1,\ldots,(p+q).\]
   \item[A3.] The $C_z$ in A3 also satisfies
   \[\max_{1\leq i\leq p+q}E\slbt\exp\slb C_z^{-1}\left|\mathcal Z(i)-E[\mathcal Z(i)]\right|\srb\srbt\leq 2.\]
\end{itemize}
A1 follows from our assumption on the $\sigma_i^2$'s. To prove A2, 
first we will bound $\max_{1\leq i\leq p+q}E [|\mathcal Z_1(i)|^3]$ and $\max_{1\leq i\leq p+q}E[|\mathcal Z_1(i)|^4]$, which is not immediate since the moment expressions of the $\mathcal Z(i)$'s  involve $p$ and $q$ dimensional vectors.
Let us denote by $\mathcal S^{p+q-1}$ the unit $l_2$ ball in $\RR^{p+q}$.
Note that for $a,b\in\mathcal{S}^{p+q-1}$ and $k\in\NN$, by Cauchy-Schwarz inequality,
\begin{align*}
   E[|a^TXX^Tb|^k]\leq \lb E[(a^T X_1)^{2k}] E[(b^T X_1)^{2k}]\rb^{1/2}\quad k\in\NN.
\end{align*}
Because $X$ is a sub-Gaussian random vector, $a^TX$ is a sub-Gaussian variable, which implies \citep[cf. (5.11) of ][]{vershynin2010} 
\[E[|a^TX|^{2k}]\leq \|X\|_{\psi_2}^{2k}(2k)^{k},\]
where $\|\cdot\|_{\psi_2}$ is the sub-Gaussian norm  \citep[cf. Definition 5.7  of ][]{vershynin2010}. Note that $\|X\|_{\psi_2}<\infty$ because $X$ is sub-Gaussian.
Thus,
\[ E[|a^TXX^Tb|^k]\leq  \|X\|_{\psi_2}^{2k}(2k)^{k}.\]
Thus for $a\in\RR^p$ and $b\in\RR^p$, 
\begin{align}\label{intheorem: main: moment inequality}
    E|a^TXX^Tb|^k\leq  \|X\|^{2k}_{\psi_2}(2k)^k\|a\|_2^k\|b\|_2^k\quad k\in\NN.
\end{align}
Similarly, for $a\in\mathcal{S}^{p-1}$ and $b\in\mathcal{S}^{q-1}$, we can show that
\[E|a^TXY^Tb|^k \leq (2k)^{k}\lb \|X\|_{\psi_2}^{2k}\|Y\|_{\psi_2}^{2k}\rb^{1/2}.\]
Therefore, for $a\in\RR^p$ and $b\in\RR^q$, we can show that
\begin{equation}\label{intheorem: main: subgaussian}
    E|a^TXY^Tb|^k\leq  (2k)^{k}\lb \|X\|_{\psi_2}^{2k}\|Y\|_{\psi_2}^{2k}\rb^{1/2}\|a\|_2^k\|b\|_2^k \quad k\in\NN.
\end{equation}

Moreover, for any $k\in\NN$ and $\{a_i\}_{i=1}^k\in\RR$, there exists a universal constant $c>0$ so that
\[(\sum_{i=1}^k|a_i|)^3\leq c\sum_{i=1}^k|a_i|^3\quad\text{and}\quad(\sum_{i=1}^k|a_i|)^4\leq c\sum_{i=1}^k|a_i|^4.\]
Hence, there exists $C>0$ depending only on $\|X\|_{\psi_2}$ and $\|Y\|_{\psi_2}$ so that
\begin{align*}
    E|\mathcal Z_1(i)|^k\leq &\ C\lbs \|x^0\|_2^k\slb \|\pa\|_2^k+|\xi_1(i)|^k\|x^0\|_2^k\srb +\|y^0\|_2^k\slb \|\pb\|_2^k+|\xi_2(i)|^k\|y^0\|_2^k\srb\\
   &\ +\|\pa\|_2^k\|y^0\|_2^k+\|\pb\|_2^k\|x^0\|_2^k\rbs\quad k=3,4,
\end{align*}
where we used the fact that $\rhk\leq 1$.  Lemma~\ref{lemma:norm:  l1 and l2 norm} implies $\|x^0\|_2$ and $\|y^0\|_2$ bounded above by a constant. On the other hand, \eqref{intheorem: main: bound: l2 norm of pa} implies
\[\max_{1\leq i\leq p+q}\max\{\|\pa\|_2,\|\pb\|_2\}\leq \frac{M}{2\epsilon_0}.\]
  These facts also imply $|\xi_1(i)|$ and $|\xi_2(i)|$ are bounded. To see this, note that
  \begin{equation}\label{inlemma: main corr: bounds on xi one and xi two}
     \max_{1\leq i\leq p+q}|\xi_1(i)|\leq \|\Sx^{1/2}\|_{op}\|\Sx^{1/2}x^0\|_2\max_{1\leq i\leq p+q}\|\pa\|_2\leq  M^{1/2}\|\Phi^0\|_{op}, \end{equation}
which is bounded above.
Similarly, we can show that $|\xi_2(i)|$ is bounded uniformly over $i=1,\ldots,p+q$. Thus, we conclude that there exists $C>0$  depending only on $\|X\|_{\psi_2}$, $\|Y\|_{\psi_2}$,  $M$, and $\rhk-\Lambda_2$ so that 
\begin{align*}
 \max_{1\leq i\leq p+q}E[|\mathcal Z_1(i)|^3],  \max_{1\leq i\leq p+q}E[|\mathcal Z_1(i)|^4]\leq C.   \end{align*}
 Hence, $E[\mathcal Z_1(i)^2]$ is also bounded by $C$ uniformly across $i=1,\ldots,p+q$. Since
 \[E\slbt (\mathcal Z_1(i)-E[\mathcal Z_1(i)])^3\srbt = E[\mathcal Z_1(i)^3]-3E[\mathcal Z_1(i)^2]E[\mathcal Z_1(i)]+2E[\mathcal Z_1(i)]^3,\]
 \[E\slbt (\mathcal Z_1(i)-E[\mathcal Z_1(i)])^4\srbt = E[\mathcal Z_1(i)^4]-4E[\mathcal Z_1(i)^3]E[\mathcal Z_1(i)]+6E[\mathcal Z_1(i)^2]E[\mathcal Z_1(i)]^2-3E[\mathcal Z_1(i)]^4,\]
it  follows that there exists $C'>0$  depending only on $\|X\|_{\psi_2}$, $\|Y\|_{\psi_2}$,  $M$, and $\epsilon_0$ so that 
\begin{align}\label{intheorem: main: skewness}
 \max_{1\leq i\leq p+q}E\slbt|\mathcal Z_1(i)-E[\mathcal Z_1(i)]|^3\srbt, \max_{1\leq i\leq p+q}E\slbt|\mathcal Z_1(i)-E[\mathcal Z_1(i)]|^4\srbt\leq C'.   \end{align}
Let us denote $C_z'= \max(C',1)$. It is easy to see that $C_z$ satisfies A2 if  $C_z>C_z'$. 

Next, we will find the moment generating functions of $|\mathcal Z(i)-E[\mathcal Z(i)]|$, using which we will choose a $C_z$ that satisfies A2 and A3. First of all, note that since $X$ and $Y$ are sub-Gaussian, $a^TX$ and $b^TY$ are sub-Gaussian.
Since the product of sub-Gaussian random variables is sub-exponential \citep[cf. Lemma 2.7.5 of ][]{vershynin2018}, and sum of sub-exponential random variables is also sub-exponential\citep[cf. Bernstein inequality,  Theorem 2.8.2 of ][]{vershynin2018}, $\mathcal Z(i)$ is also sub-exponential. 

The sub-exponential norm  $\|Z\|_{\psi_1}$ of a sub-Gaussian random variable $Z$ is defined by \citep[cf. Definition 2.7.3 of][]{vershynin2018} 
\[\|Z\|_{\psi_1}=\inf\{ t\geq 0: E[\exp(|Z|/t)]\leq 2\}.\]
Therefore, for $t\geq \|\mathcal Z(i)\|_{\psi_1}$, we have
$E[\exp({|Z|/t})]\leq 2$. This implies if $C_z$ satisfies
\begin{equation}\label{inlemma: main: corr: Cz lower bound}
    C_z\geq \max_{1\leq i\leq p+q} \|\mathcal Z(i)-E[\mathcal Z(i)]\|_{\psi_1},
\end{equation}
then $C_z$ satisfies A3 as well. 
Now
\begin{equation}\label{inlemma: main: corr: subexponential norm central to raw}
   \|\mathcal Z(i)-E[\mathcal Z(i)]\|_{\psi_1}\leq \|\mathcal Z(i)\|_{\psi_1}+|E[\mathcal Z(i)]| 
\end{equation}
Fact~\ref{fact: sub-exponential norms} implies that there exists a constant $C$ depending on $\|X\|_{\psi_2}$ and $\|Y\|_{\psi_2}$ so that
\begin{align}\label{inlemma: main: corr: subexponential norm bound}
   \max_{1\leq i\leq p+q} \|\mathcal Z(i)\|_{\psi_1}\leq &\ C \max_{1\leq i\leq p+q}\lbs \|x^0\|_2\slb \|\pa\|_2+|\xi_1(i)|\|x^0\|_2\srb +\|y^0\|_2\slb \|\pb\|_2+|\xi_2(i)|\|y^0\|_2\srb\nn\\
   &\ +\|\pa\|_2\|y^0\|_2+\|\pb\|_2\|x^0\|_2\rbs.
\end{align}
The fact that the right hand side is bounded follows from Lemma~\ref{lemma:norm:  l1 and l2 norm},  \eqref{intheorem: main: bound: l2 norm of pa}, and \eqref{inlemma: main corr: bounds on xi one and xi two}, Assumption~\ref{assump: eigengap assumptions}. As in the proof of A1, it can also be shown that the bound depends only on $\|X\|_{\psi_2}$, $\|Y\|_{\psi_2}$, $M$, and the $\epsilon_0$ in Assumption~\ref{assump: eigengap assumptions}.
On the other hand, 
\begin{align*}
    \max_{1\leq i\leq p+q}E[\mathcal Z(i)]= &\ \rhk ((\pa)^T\Sx x^0+(\pb)^T\Sy y^0)+\rhk(\xi_1(i)+\xi_2(i))\\
    &\ -(\pa)^T\Sxy y^0-(\pb)^T\Syx x^0.
\end{align*}
Since $\Sxy y^0=\rhk\Sx x^0$ and $\Syx x^0=\rhk \Sy y^0$, we have
\[\max_{1\leq i\leq p+q}E[\mathcal Z(i)]=\rhk\max_{1\leq i\leq p+q}(\xi_1(i)+\xi_2(i))<M^{1/2}\|\Phi^0\|_{op},\]
where the last step follows by \eqref{inlemma: main corr: bounds on xi one and xi two}.  Combining the above with \eqref{inlemma: main: corr: subexponential norm central to raw} and \eqref{inlemma: main: corr: subexponential norm bound} implies that $\max_{1\leq i\leq p+q}\|\mathcal Z(i)-E[\mathcal Z(i)]\|_{\psi_1}$ can be bounded by some $D_z'>0$ depending only on $\|X\|_{\psi_2}$, $\|Y\|_{\psi_2}$, $M$, and $\epsilon_0$. Therefore, according to \eqref{inlemma: main: corr: Cz lower bound}, $C_z$ satisfies A3 if $C_z>D_z'$. Recall that we showed that $C_z$ satisfies A2 if $C_z>C_z'$ for some $C_z'=\max(C',1)$ where $C'$ is defined in \eqref{intheorem: main: skewness}. Therefore, if $C_z>\max(D_z',C_z')$, then A2 and A3 holds. Suppose $t=(t_i)_{1\leq i\leq p+q}\in\RR^{p+q}$. Since the constant $C_z$ does not depend on $p$, $q$, or $n$,
 \eqref{inlemma: kato statement} implies there exists a constant $C>0$ depending only on $\|X\|_{psi_2}$, $\|Y\|_{\psi_2}$, $M$ and $\epsilon_0$ so that
\[\sup_{A\in\mathcal A}\bl P\slb n^{-1/2}\sum_{j=1}^n(\mathcal Z_j-E[\mathcal Z_j])\in A\srb-P\slb \mathbb Z\in A\srb\bl\leq C\lb\frac{\log^7((p+q)n)}{n}\rb^{1/6}.\]
Since $A\in\mathcal A$ implies $2A\in\mathcal A$, and $\mathcal L_i=2\sum_{j=1}^n(\mathcal Z_j(i)-E[\mathcal Z(i)])/n$, we have
\[\sup_{A\in\mathcal A}\bl P\slb n^{1/2}\mathcal L\in A\srb-P\slb 2\mathbb Z\in A\srb\bl\leq C\lb\frac{\log^7((p+q)n)}{n}\rb^{1/6}.\]
In particular, for $t=(t_1,\ldots,t_{p+q})$, we have
\[ \sup_{t\in\RR^{p+q}}\bl P\slb n^{1/2}\mathcal L_i\leq t_i,\ 1\leq i\leq p+q\srb-P\slb 2\mathbb Z_i\leq t_i,\ 1\leq i\leq p+q\srb\bl\leq C\lb\frac{\log^7((p+q)n)}{n}\rb^{1/6}.\]
Hence the proof follows.
\end{proof}

\begin{proof}[of Proposition~\ref{corollary: main theorem}]
Since $\mathcal L_{(1)}$ consists of the first $p$ co-ordinates of $\mathcal L$, 
 from Lemma~\ref{lemma: berry essen}, we  obtain that
 \[\sup_{A\in\mathcal A_p}\bl P\slb -n^{1/2}\mathcal L_{(1)}\in A\srb-P\slb 2\mathbb X\in A\srb\bl\leq C\lb\frac{\log^7((p+q)n)}{n}\rb^{1/6},\]
 where $\mathbb X=(\mathbb{Z}_1,\ldots,\mathbb{Z}_p)$ consists of the first $p$ co-ordinates of $\mathbb Z$, a $p+q$-variate centred Gaussian random vector.  Since the covariance of $\mathbb Z$ is the covariance matrix of the $p+q$-variate random vector $\mathcal Z=(\mathcal Z(i))_{1\leq i\leq p+q}$, it follows that the covariance matrix of  $\mathbb X$ is the covariance matrix of  $(\mathcal Z(1), \ldots, \mathcal Z(p))$, which we denoted by 
 $\Sigma_p$.

Theorem~\ref{thm: for alpha} implies that either 
 \[\hdai-x^0=-\mathcal L_{(1)}+\text{rem}\quad\text{or}\quad \hdai+x^0=-\mathcal L_{(1)}+\text{rem}.\]
 Suppose the former holds. Then
 \begin{equation}\label{incorr: main: berry essen}
     \sup_{A\in\mathcal A_p}\bl P\slb n^{1/2}(\hdai-x^0-\text{rem})\in A\srb-P\slb 2\mathbb X\in A\srb\bl\leq C\lb\frac{\log^7((p+q)n)}{n}\rb^{1/6}.
 \end{equation}
 Hence
 \begin{align}\label{incorr: main: upper vound on prob: -2}
 \MoveEqLeft  \sup_{A\in\mathcal A_p}\bl P\slb n^{1/2}(\hdai-x^0)\in A\srb-P\slb 2\mathbb X\in A\srb\bl \nn\\
  \leq &\ \sup_{A\in\mathcal A_p}\bl P\slb n^{1/2}(\hdai-x^0-\text{rem})\in A\srb-P\slb 2\mathbb X\in A\srb\bl\nn\\
  &\ +\sup_{A\in\mathcal A_p}\bl P\slb n^{1/2}(\hdai-x^0)\in A\srb-P\slb n^{1/2}(\hdai-x^0-\text{rem})\in A\srb\bl\nn\\
  \leq &\ C\lb\frac{\log^7((p+q)n)}{n}\rb^{1/6}+  \sup_{A\in\mathcal A_p}P\lb n^{1/2}(\hdai-x^0-\text{rem})\in A-\text{rem},  n^{1/2}(\hdai-x^0-\text{rem})\in A^c\rb\nn\\
  =&\ C\lb\frac{\log^7((p+q)n)}{n}\rb^{1/6}+  \sup_{A\in\mathcal A_p}P\underbrace{\lb -n^{1/2}\mathcal L_{(1)}\in A-\text{rem},  -n^{1/2}\mathcal L_{(1)}\in A^c\rb}_{\mathcal E(A)}
 \end{align}
 For any set $A\in\RR^p$ and $\epsilon>0$, we denote
 \[D(A,\e)=\{x\in\RR^p: \text{there exists } y\in A \text{ so that }\|x-y\|_2\leq \e\}.\]
 Note that
 \begin{align*}
P(\mathcal E(A))\leq  P\lb -n^{1/2}\mathcal L_{(1)}\in D(A,\|rem\|_\infty)\setminus A\rb
 \end{align*}
 Since $A$ is a hyperrectangle in $\RR^p$, $D(A,\|rem\|_\infty)$ is also a hyperrectangle in $\RR^p$. Thus,
 \begin{align*}
\MoveEqLeft  \sup_{A\in\mathcal A_p}  \bl P\lb -n^{1/2}\mathcal L_{(1)}\in D(A,\|rem\|_\infty)\setminus A\rb-P\slb 2\mathbb X\in D(A,\|rem\|_\infty)\setminus A\srb\bl\\
  \leq &\ \sup_{A\in\mathcal A_p}  \bl P\lb -n^{1/2}\mathcal L_{(1)}\in D(A,\|rem\|_\infty)\rb-P\slb 2\mathbb X\in D(A,\|rem\|_\infty)\srb\bl\\
  &\ +\sup_{A\in\mathcal A_p}  \bl P\lb -n^{1/2}\mathcal L_{(1)}\in A\rb-P\slb 2\mathbb X\in A\srb\bl\\
  \stackrel{(a)}{\leq} &\ 2 \sup_{A\in\mathcal A_p}  \bl P\lb -n^{1/2}\mathcal L_{(1)}\in A\rb-P\slb 2\mathbb X\in A\srb\bl,
 \end{align*}
 which is bounded by $C(\log^7((p+q)n)/n)^{1/6}$ by \eqref{incorr: main: berry essen}. Here (a) follows because $D(A,\|\text{rem}\|_\infty)\in\mathcal A_p$.
 Therefore, 
 \begin{align}\label{incorr: main: upper bound on P -1}
   \sup_{A\in \mathcal A_p} P(\mathcal E(A))\leq   C\lb\frac{\log^7((p+q)n)}{n}\rb^{1/6}+ \sup_{A\in \mathcal A_p}P\slb 2\mathbb X\in D(A,\|\text{rem}\|_\infty)\setminus A\srb.
 \end{align}
 Let us consider a particular $A$. Any $A$ in $\mathcal A_p$  has the form $A=[x_1,y_1]\times\ldots[x_{p+q},y_{p+q}]$ where $-\infty \leq x_i<\infty$, and  $-\infty<y_i\leq\infty(i=1,\ldots,p+q)$. 
Some algebra leads to
\begin{align}\label{incorr: Prob upper bound}
P\slb 2\mathbb X\in D(A,\|rem\|_\infty)\setminus A\srb&\leq \max_{1\leq i\leq p}P\slb x_i-\|\text{rem}\|_\infty \leq  2\mathbb X_i\leq x_i\srb\\&+\max_{1\leq i\leq p}P\slb y_i \leq  2\mathbb X_i\leq y_i+\|\text{rem}\|_\infty\srb.
\end{align}
Note that if either $x_i=-\infty$ or $y_i=\infty$, then 
\begin{equation}\label{incorr: main:  xi infinity}
 P\slb x_i-\|\text{rem}\|_\infty \leq  2\mathbb X_i\leq x_i\srb=0\quad\text{or}\quad P\slb y_i \leq  2\mathbb X_i\leq y_i+\|\text{rem}\|_\infty\srb=0.   
\end{equation}
For any $\epsilon>0$, we calculate 
\begin{align*}
 \MoveEqLeft    P\slb x_i-\|\text{rem}\|_\infty \leq  2\mathbb X_i\leq x_i\srb\\
     \leq &\ P(\|\text{rem}\|_\infty\geq x_i- 2\mathbb X_i,\  x_i-2\mathbb X_i\geq \e)+P(0\leq x_i-2\mathbb X_i\leq \e)\\
     \leq &\ P(\|\text{rem}\|_{\infty}\geq \e)+\dint_{x_i-\e}^{x_i}\frac{e^{-\frac{t^2}{8\sigma_i^2}}}{(2\pi)^{1/2}2\sigma_i}dt\\
     \leq &\  P(\|\text{rem}\|_{\infty}\geq \e)+ \frac{C\e}{\min\limits_{1\leq i\leq p}\sigma_i},
 \end{align*}
where $C$ is an absolute constant not depending on $x_i$.  Our assumptions imply that there exists $c>0$ so that $c\leq \min_{1\leq i\leq p}\sigma_i$. This, combined with \eqref{incorr: main:  xi infinity}, implies that there exists $C>0$ so that
\[ \max_{1\leq i\leq p}P\slb x_i-\|\text{rem}\|_\infty \leq  2\mathbb X_i\leq x_i\srb\leq  P(\|\text{rem}\|_{\infty}\geq \e)+C\e.\]
Similarly, we can show that 
\[ \max_{1\leq i\leq p}P\slb y_i\leq  2\mathbb X_i\leq y_i+\|\text{rem}\|_\infty \srb\leq  P(\|\text{rem}\|_{\infty}\geq \e)+C\e.\]
Equation \ref{incorr: Prob upper bound} implies that
\[\sup_{A\in\mathcal A_p}P\slb 2\mathbb X\in D(A,\|rem\|_\infty)\setminus A\srb\leq  P(\|\text{rem}\|_{\infty}\geq \e)+C\e\]
Since $\e$ is arbitrary, and $\|\text{rem}\|_\infty=o_p(1)$, it follows that as $n\to\infty$,
\[\sup_{A\in\mathcal A_p}P\slb 2\mathbb X\in D(A,\|rem\|_\infty)\setminus A\srb\to 0,\]
which, combined with \eqref{incorr: main: upper vound on prob: -2}, \eqref{incorr: main: upper bound on P -1}, and the fact that $\log p=o(n^{-1/7})$ implies
\[\sup_{A\in\mathcal A_p}\bl P\slb n^{1/2}(\hdai-x^0)\in A\srb-P\slb 2\mathbb X\in A\srb\bl\to 0\]
as $n\to\infty$. If $\hdai+x^0=-\mathcal L_{(1)}+\text{rem}$, we can similarly show that
\[\sup_{A\in\mathcal A_p}\bl P\slb n^{1/2}(\hdai+x^0)\in A\srb-P\slb 2\mathbb X\in A\srb\bl\to 0.\]
\end{proof} 
 
 \section{Proof of  Corollaries in Section~\ref{sec: asymptotic theory}}
 \label{sec: proof of corollary main theorem}
 
 \begin{proof}[of Corollary~\ref{corollary: main theorem single}]
 Equation \ref{intheorem: main: skewness} implies that under Assumption~\ref{assump: eigengap assumptions} and \ref{assump: bounded eigenvalue}, there exists $C>0$  depending only on $\|X\|_{\psi_2}$, $\|Y\|_{\psi_2}$,  $M$, and $\rhk-\Lambda_2$ so that 
\[\max_{1\leq i\leq p}E\slbt|\mathcal Z_1(i)-E[\mathcal Z_1(i)]|^4\srbt\leq C.\]
Therefore $\sigma_i^2=\text{var}(\mathcal Z)$ is also finite. 
Letting
\[s_n^2=\sum_{j=1}^n \text{var}(\mathcal Z_j)=n \text{var}(\mathcal Z_1)=n\sigma^2_i,\]
we note
\[\sum_{j=1}^n\frac{E|\mathcal Z_j-E \mathcal Z_j|^4}{s_n^4}\leq \frac{n C}{n^2\sigma_i ^4}=O\lb \frac{1}{n}\rb.\]
Hence,  $\mathcal Z_i$'s satisfy the Lyaponov's condition \citep[cf. Theorem 27.3 of][]{billingsley2008}. Therefore 
\[ \frac{\sum_{i=1}^n(\mathcal Z_i-E\mathcal Z_i)}{s_n}=\frac{\sum_{i=1}^n(\mathcal Z_i-E\mathcal Z_i)}{(n \sigma_i^2)^{1/2}}\to_d N(0,1),\]
which implies $n^{1/2}\mathcal{L}_i$ converges in distribution to a centered Gaussian random variable with variance $4\sigma_i^2$.
  Therefore Theorem~\ref{thm: for alpha} implies that  $\hdae$ satisfies either
 \[\sqrt{n}\slb \hdae-x^0_i\srb\to_d N(0,4\sigma_i^2)\quad \text{or}\quad \sqrt{n}\slb \hdae+x^0_i\srb\to_d N(0,4\sigma_i^2)\quad (i=1,\ldots,p).\]
 Thus, when $x^0_i=0$, the result follows immediately. On the other hand, when $x^0_i\neq 0$, the result follows from an application of Delta method. 
   \end{proof}
   
   \begin{proof}[of Corollary~\ref{corollary to proposition}]
   The proof follows from Propossition~\ref{corollary: main theorem} and Lemma~\ref{lemma: main: sigmai square lower bound}.
   \end{proof}
   
   \begin{proof}[of Corollary~\ref{cor: projection matrix}]
   From Proposition~\ref{corollary: main theorem} it follows that
\begin{equation}\label{incorollary: bivariate convg}
    n^{1/2}\begin{bmatrix}
(\hdai)_i-x^0_i\\(\hdai)_j-x^0_j
\end{bmatrix}\to_d N_2\left(\begin{bmatrix}
0\\0
\end{bmatrix},\underbrace{\begin{bmatrix}
\sigma_i^2 & \sigma_{ij}\\
\sigma_{ij} & \sigma_j^2
\end{bmatrix}}_{\Sigma_{\text{pro}}}\right ).
\end{equation}
Consider the function $g(x,y)=xy$. If $x^0_ix^0_j\neq 0$, then either $x^0_i\neq 0$ or $x^0_j\neq 0$, which implies
\[\dot{g}(x^0_i,x^0_j)^T\Sigma_{\text{pro}}\dot{g}(x^0_i,x^0_j)=4(x^0_j)^2\sigma^2_i+4(x^0_i)^2\sigma^2_j+8x^0_ix^0_j\sigma_{ij}>0.\]
Then by delta method,
\[n^{1/2}\slb g((\hdai)_i,(\hdai)_j)-x^0_ix^0_j\srb \to_d N\slb 0, \dot{g}(x^0_i,x^0_j)^T\Sigma_{\text{pro}}\dot{g}(x^0_i,x^0_j)\srb,\]
which completes the proof of the first part.
Now suppose both $x^0_i=0$ and $x^0_j=0$. Then \eqref{incorollary: bivariate convg} reduces to
\[ \underbrace{n^{1/2}\begin{bmatrix}
(\hdai)_i\\(\hdai)_j
\end{bmatrix}}_{V_n}\to_d N_2(0,\Sigma_{\text{pro}} ).\]
Fix $t\in\RR$.
 Let $C_t=\{y\in\RR^2\ :\ y^TAy \leq t\}$, where
 \[A=2^{-1}\begin{bmatrix}
 0 & 1\\ 1& 0
 \end{bmatrix}.\]
 Note that
 \[P\slb n(\hdai)_i(\hdai)_j\leq t\srb=P(V_n^TAV_n\leq t)=P(V_n\in C_t).\]
 Since $V_n\to_d N_{2}(0,\Sigma_{\text{pro}})$, if we can show that $C_t$ is a continuity set of the latter distribution, it would follow that
 \[P(V_n\in C_t)\to P\slb N_{2}(0,\Sigma_{\text{pro}})\in C_t\srb=P(\mathbb Z_i\mathbb Z_j\leq t)\]
 where $\mathbb Z_i\sim N(0,\sigma_i^2)$ and $\mathbb Z_j\sim N(0,\sigma_j^2)$ so that their covariance is $\sigma_{ij}$. Hence it remains to prove that $C_t$ is a continuity point of $N_{2}(0,\Sigma_{\text{pro}})$ for all $t\in\RR$, which means $P(N_{2}(0,\Sigma_{\text{pro}})\in \partial C_t)=0$ for all $t\in\RR$, where $\partial C_t$ is the boundary of $C_t$, that is
 \[\partial C_t=\{y\in\RR^2\ :\ y^TAy=t\}.\]
 That  $P(N_{2}(0,\Sigma_{\text{pro}})\in \partial C_t)=0$ will trivially follow if we can show that $\Sigma_{\text{pro}}$ is a positive definite matrix. However, the latter follows from the fact that $\Sigma_p$ is positive definite noting $\Sigma_{\text{pro}}$ is a principal minor of $\Sigma_p$. Hence, the proof follows.
   \end{proof}
   
 \subsection{Additional lemmas for the proofs of Supplement~\ref{sec: proof of corollary main theorem}}
 \label{sec: addlemma: corollaries}
   The next lemma is essential in proving  Corollary~\ref{corollary to proposition}.
   \begin{lemma}\label{lemma: main: sigmai square lower bound}
  Suppose $X$ and $Y$ are Gaussian. Then
  \[4\sigma_i^2=\rhk(1-\rhk^2)\sum_{k=2}^r\frac{(\rhk^2+\Lambda_k^2)}{(\rhk^2-\Lambda_k^2)^2}(u_k)_i^2+ \frac{21\rhk^4-13\rhk^2+8}{32\rhk^2}(x_0)_i^2 +(1-\rhk^2)\frac{\sum_{i=r+1}^p(u_k)_i^2}{\rhk}.\]
  In particular,
 \[4\sigma_i^2\geq \min\lb 20.58945, \frac{1-\rhk^2}{\rhk}\rb.\]
  \end{lemma}
  
  \begin{proof}[of Lemma~\ref{lemma: main: sigmai square lower bound}]
  Suppose $1\leq i\leq p$ and as usual, we let $e_1$ be a unit vector whose first element is one, whose length depends on the context.
  
To find the variance of $\mathcal Z(i)$, we will use Lemma~\ref{lemma: variance lemma facts} with
 \begin{gather*}
     a= \rhk\pa+\xi_1(i)x^0,\quad 
     b= x^0, \quad 
     c= \rhk\pb +\xi_2(i)y^0,
     d= y^0, 
     z= \pa,
     \gamma= \pb.
 \end{gather*}
Noting
 \[\mathcal Z(i)=a^TXX^Tb+c^TYY^Td-z^TXY^Td-b^TXY^T\gamma,\]
  \[(x^0)^T\Sx x^0=\rhk,\quad (y^0)^T\Sy y^0=\rhk, \quad (x^0)^T\Sxy y^0=\rhk^2,\]
  and using Lemma~\ref{lemma: variance lemma facts}, we see that $\sigma^2_i=\text{var}(\mathcal Z(i))$ equals

  \begin{align*}
  =&\ (\pa)^T\Sx\pa (\rhk-\rhk^3)+(\pb)^T\Sy\pb (\rhk-\rhk^3)\\
      &\ +2\rhk^2(\xi_1(i)^2+\xi_2(i)^2)+
       (14\rhk^4+2-12\rhk^2)\xi_1(i)\xi_2(i),
  \end{align*}
 where $\xi_1(i)= (\pa)^T\Sx x^0$ and $\xi_2(i)=(\pb)^T\Sy y^0$. Note that Lemma~\ref{lemma: NL: form of Phi knot} implies
  \begin{align*}
 \xi_1(i)=  (\pa)^T\Sx x^0 = (2\rhk)^{-1}e_i^T(UO_4U^T+\Sx^{-1})\Sx x^0  =&\ (2\rhk^{1/2})^{-1}\slb U_{*i}^T O_4 e_1+(\alk)_i\srb\\
  =&\ (2\rhk^{1/2})^{-1}\slb -5(u_1)_i/8+(\alk)_i\srb.
  \end{align*}
  Since $u_1=\alk$, we obtain that
  \[ (\pa)^T\Sx x^0=3\rhk^{-1/2}(\alk)_i/16=\frac{3(x_0)_i}{16\rhk}.\]
 Lemma~\ref{lemma: NL: form of Phi knot} also implies
\begin{align*}
 \xi_2(i)= (\pb)^T\Sy y^0=(2\rhk)^{-1}e_i^T(U O_3 V^T)\Sy y^0=&\ (2\rhk^{1/2})^{-1}U_{i*}^TO_3e_1\\ 
  =&\ \rhk^{-1/2}U_{i*}^Te_1/16=\frac{(x_0)_i}{16\rhk}.
\end{align*}
  Therefore $\xi_1(i)=3\xi_2(i)$, which implies
  \begin{align*}
      \sigma_i^2=&\ 
       (\pa)^T\Sx\pa (\rhk-\rhk^3)+(\pb)^T\Sy\pb (\rhk-\rhk^3) +2(21\rhk^4-8\rhk^2+3)\xi_2(i)^2. \end{align*}
  Therefore it suffices to find the values of  $(\pa)^T\Sx\pa$ and $(\pb)^T\Sy\pb$. 
  Lemma~\ref{lemma: NL: form of Phi knot} yields that
  \begin{align*}
(\pa)^T\Sx\pa=&\  \frac{e_i^T(UO_4U^T+\Sx^{-1})\Sx (UO_4U^T+\Sx^{-1}) e_i}{4\rhk^2} \\
=&\ \frac{U^T_{i*}(O_4^2+2O_4)U_{i*}+(\Sx^{-1})_{ii}}{4\rhk^2}\\
=&\ \frac{(x_0)^2_i(25/64-10/8)}{4\rhk^3}+\sum_{k=2}^r\frac{\slb\Lambda^4_k+2\Lambda_k^2(\rhk^2-\Lambda_k^2)\srb (u_k)_i^2}{4\rhk^2(\rhk^2-\Lambda_k^2)^2}+\frac{(\Sx^{-1})_{ii}}{4\rhk^2}\\
=&\ -\frac{55(x_0)_i^2}{16^2\rhk^3}-\sum_{k=2}^r\frac{\Lambda_k^4(u_k)_i^2}{4\rhk^2(\rhk^2-\Lambda_k^2)^2}+\sum_{k=2}^r\frac{\Lambda_k^2(u_k)_i^2}{2(\rhk^2-\Lambda_k^2)^2}+\frac{(\Sx^{-1})_{ii}}{4\rhk^2},
  \end{align*}
 and
   \begin{align*}
  (\pb)^T\Sy\pb=&\    \frac{ e_i^T UO_3V^T \Sy VO_3U^T e_i}{4\rhk^2} =   \frac{ U_{i*}^TO_3^2U_{i*}}{4\rhk^2} =\frac{(x_0)_i^2}{16^2\rhk^3}+\sum_{k=2}^r\frac{\Lambda_k^2(u_k)_i^2}{4(\rhk^2-\Lambda_k^2)^2}.
  \end{align*} 
Therefore,
\begin{align*}
    \sigma^2_i=&\ (\rhk-\rhk^3)\sum_{k=2}^r\frac{(3-\Lambda^2/\rhk^2)\Lambda_k^2(u_k)_i^2}{4(\rhk^2-\Lambda_k^2)^2}+(\rhk-\rhk^3)\frac{(\Sx^{-1})_{ii}}{4\rhk^2}\\
    &\ +2(21\rhk^4-8\rhk^2+3)\frac{(x_0)^2_i}{16^2\rhk^2}
    -(1-\rhk^2)\frac{54(x_0)_i^2}{16^2 \rhk^2}\\
    =&\ (1-\rhk^2)\sum_{k=2}^r\frac{3\rhk^2\Lambda_k^2-\Lambda_k^4}{4\rhk(\rhk^2-\Lambda_k^2)^2}(u_k)_i^2+(1-\rhk^2)\frac{(\Sx^{-1})_{ii}}{4\rhk} +\frac{42\rhk^4+38\rhk^2-48}{16^2\rhk^2}\\
    =&\ (1-\rhk^2)\sum_{k=2}^r\frac{3\rhk^2\Lambda_k^2-\Lambda_k^4}{4\rhk(\rhk^2-\Lambda_k^2)^2}(u_k)_i^2+(1-\rhk^2)\frac{\sum_{i=1}^r(u_k)_i^2}{4\rhk} +\frac{42\rhk^4+38\rhk^2-48}{16^2\rhk^2}(x_0)_i^2\\
    &\ +(1-\rhk^2)\frac{\sum_{i=r+1}^p(u_k)_i^2}{4\rhk}\\
    =&\ (1-\rhk^2)\sum_{k=2}^r\frac{3\rhk^2\Lambda_k^2-\Lambda_k^4+(\rhk^2-\Lambda_k^2)^2}{4\rhk(\rhk^2-\Lambda_k^2)^2}(u_k)_i^2+ \frac{42\rhk^4+38\rhk^2-48+64-64\rhk^2}{16^2\rhk^2}(x_0)_i^2\\
    &\ +(1-\rhk^2)\frac{\sum_{i=r+1}^p(u_k)_i^2}{4\rhk}\\
     =&\ \rhk(1-\rhk^2)\sum_{k=2}^r\frac{(\rhk^2+\Lambda_k^2)}{4(\rhk^2-\Lambda_k^2)^2}(u_k)_i^2+ \frac{42\rhk^4-26\rhk^2+16}{16^2\rhk^2}(x_0)_i^2 +(1-\rhk^2)\frac{\sum_{i=r+1}^p(u_k)_i^2}{4\rhk}
\end{align*}
Thus $4\sigma_i^2$ equals
\[\rhk(1-\rhk^2)\sum_{k=2}^r\frac{(\rhk^2+\Lambda_k^2)}{(\rhk^2-\Lambda_k^2)^2}(u_k)_i^2+ \frac{21\rhk^4-13\rhk^2+8}{32\rhk^2}(x_0)_i^2 +(1-\rhk^2)\frac{\sum_{i=r+1}^p(u_k)_i^2}{\rhk}\]
which is bounded below by
\begin{align*}
   \MoveEqLeft \frac{(1-\rhk^2)}{\rhk}(1-(\alk)_i^2)+\frac{42\rhk^4-26\rhk^2+16}{64\rhk}(\alk)_i^2\\
    =&\ \frac{64-64\rhk^2+(42\rhk^4+38\rhk^2-48)(\alk)_i^2}{64\rhk}
\end{align*}
The quadratic $21x^2+19x-24$ has only positive root at
$x_r=\frac{-19\pm(19^2+4*21*24)}{2*21}$, which is approximately $0.708$.  Therefore $21\rhk^4+19\rhk^2-24$ has only positive root at $(x_r)^{1/2}$. which is approximately $0.841$. This polynomial is positive to the right of $(x_r)^{1/2}$, and negative to the left of it. Therefore, for $\rhk\geq (x_r)^{1/2}$,
\[
    4\sigma_i^2\geq \frac{1-\rhk^2}{\rhk}
\]
which is bounded away from $0$ because $\rhk$ is bounded away from one and zero. On the other hand, for $\rhk< (x_r)^{1/2}$, we have $42\rhk^4+38\rhk^2-48<0$, which, noting $(\alk)_i^2\leq 1$, leads to 
\begin{align*}
    4\sigma_i^2 =&\    \frac{64-64\rhk^2+(42\rhk^4+38\rhk^2-48)(\alk)_i^2}{64\rhk}\\
    \geq &\ \frac{64-64\rhk^2+42\rhk^4+38\rhk^2-48}{64\rhk}\\
    =&\ \underbrace{\frac{42\rhk^4-26\rhk^2+16}{64\rhk}}_{h_\rho(\rhk)}.
\end{align*}
The function $h_{\rho}:[0,1]\mapsto\RR$ is positive non-increasing in the interval $(0,1)$. Hence, for $\rhk\in(0,(x_r)^{1/2})$, we have  $h_{\rho}(\rhk)>h_{\rho}[(x_r)^{1/2}]\approx 20.58945$.
Therefore for all $\rhk \in(0,1)$, we have 
\[4\sigma_i^2\geq \min\lb 20.58945, \frac{1-\rhk^2}{\rhk}\rb.\]
\end{proof}




 \section{ Proof of Theorem~\ref{corrolary: main: rho}}
 \label{app: proof of rho theorem}
 Note that if we can show  $\hro^{2, \text{raw}}$ satisfies 
 \begin{equation}\label{intheorem: rho raw convg}
     n^{1/2}(\hro^{2, \text{raw}}-\rhk^2)\to_dN(0,\sigma_\rho^2),
 \end{equation}
  then $\hro^{2, \text{raw}}=\rhk^2+O_p(n^{-1/2})$ would follow. The latter implies  $P(\hro^{2, \text{raw}}\in(0,1))\to 1$, which leads to $P(\hro^{2, \text{raw}}=\hro^{2,db})\to1$ as $n\to\infty$. The latter, in conjunction with \eqref{intheorem: rho raw convg}, would complete the proof. Hence it suffices to show \eqref{intheorem: rho raw convg} holds.
  
\begin{proof}[of Theorem~\ref{corrolary: main: rho}]
Suppose $w_1=\argmin_{w\in\{\pm 1\}} \|w\hx-x^0\|_2$ and
$w_2=\argmin_{w\in\{\pm 1\}} \|w\hy-y^0\|_2$. The proof of Theorem~\ref{corrolary: main: rho} requires the following two lemmas to address the sign flip. Both these lemmas are proved in Subsection~\ref{sec: proof of rho lemmas}.

\begin{lemma}\label{lemma: main: sign lemma}
Suppose $\hx$ and $\hy$ satisfy $(\hx)^T\hSxy \hy>0$ for all $n$. Then for sufficiently large $n$,   $P(w_1w_2=-1)\to 0$.
\end{lemma}

\begin{lemma}\label{fact: main theorem: sign}
Suppose $w_1=w_2=w$. Then the estimator $\hro^{2,\text{raw}}$ constructed using $w\hx$ and $w\hy$ equals that constructed using $\hx$ and $\hy$.
\end{lemma}

Since we take $\hx$ and $\hy$ so as to satisfy $\hx^T\hSxy\hy>0$, by Lemma~\ref{lemma: main: sign lemma}, for sufficiently large $n$, $(w_1,w_2)$ equals either $(1,1)$ or $(-1,-1)$  with high probability. However, Fact~\ref{fact: main theorem: sign} implies that the estimator $\hro^{2,\text{raw}}$ constructed with $\hx$, $\hy$ and $-\hx$, $-\hy$ are the same. Hence, without loss of generality, we  assume that $w_1=w_2=1$, and thus by Lemma~\ref{corollary: rate of x and y},
\[\|\hx-x^0\|_1+\|\hy-y^0\|_1=O_p(s^{\kappa+1/2}\lambda),\quad \|\hx-x^0\|_2+\|\hy-y^0\|_2=O_p(s^{\kappa}\lambda).\]

Now we state another Lemma, which will be used to prove the asymptotic expansion of $\hro^{2,\text{raw}}$.
 \begin{lemma}\label{lemma: rho: variance of rhohat}
  Suppose $\mathcal{L}_{(1)}$ and $\mathcal{L}_{(2)}$ are as in \eqref{def: L}. Then it follows that
 \[-\mathcal L_1^T\Sxy y^0-\mathcal L_2^T\Syx x^0+  (x^0)^T(\hSxy-\hSxy)y^0=\sum_{i=1}^n\frac{(Z_i-E[Z_i])}{n}
  \]
  where $Z_i=-\rhk (X_i^Tx^0)^2-\rhk (Y_i^Ty^0)^2+2(X_i^Tx^0)(Y_i^Ty^0)$.
 \end{lemma}
 
Now note that
\begin{align}\label{intheorem: rho: first expansion}
   \MoveEqLeft \hx^T\hSxy\hy-\rhk^2\nn\\
    =&\ \hx^T\hSxy\hy-(x^0)^T\Sxy (y^0)\nn\\
    =&\ \hx^T\hSxy(\hy-y^0)+(\hx-x^0)^T\hSxy y^0+(x^0)^T(\hSxy-\Sxy)y^0\nn\\
    =&\ \hx^T\hSxy(\hy-y^0)+(\hx-x^0)^T\hSxy (y^0-\hy) \nn\\
    &\ +\hy^T\hSyx(\hx-x^0)+(x^0)^T(\hSxy-\Sxy)y^0\nn\\
    =&\ \hx^T\hSxy(\hy-\hdbi)+\hx^T\hSxy(\hdbi-y^0)+(\hx-x^0)^T\hSxy (y^0-\hy)\nn\\
    &\ +\hy^T\hSyx(\hx-\hdai)+ \hy^T\hSyx(\hdai-x^0)+(x^0)^T(\hSxy-\Sxy)y^0.
\end{align}
Since $w_1=1$, and $w_2=1$, Theorem~\ref{thm: for alpha} implies that
\begin{align}\label{intheorem: rate of rem}
    \text{rem}=\hdai-x^0+\mathcal{L}_{(1)}
\end{align}
satisfies $\|\text{rem}\|_\infty=o_p(s^{2\kappa}\lambda^2)$.
We can write
\begin{align*}
     \hy^T\hSyx(\hdai-x^0) =&\ (\hy-y^0)^T\hSyx(\hdai-x^0)-(y^0)^T\Syx\mathcal{L}_{(1)}-(y^0)^T(\hSyx-\Syx)\mathcal{L}_{(1)}\\
     &\ +(y^0)^T\hSyx\text{rem}
     \end{align*}
     which implies
\begin{align}\label{intheorem: rho: bound: 1}
 \MoveEqLeft \abs{ \hy^T\hSyx(\hdai-x^0)+ (y^0)^T\hSyx\mathcal{L}_{(1)}}\nn\\
 \leq &\ |(\hy-y^0)^T\hSyx\mathcal{L}_{(1)}|+|(\hy-y^0)^T\hSyx\text{rem}| +|(y^0)^T(\hSyx-\Syx)\mathcal{L}_{(1)}|+|(y^0)^T\hSyx\text{rem}|\nn\\
  \leq &\ \|\hy-y^0\|_1|\hSyx|_\infty(\|\mathcal L_1\|_\infty+\|\text{rem}\|_\infty)+\|y^0\|_1|\hSyx-\Syx|_{\infty}\|\mathcal{L}_{(1)}\|_\infty+\|y^0\|_1|\hSyx|_\infty\|\text{rem}\|_\infty.
\end{align}  
From Lemma~\ref{corollary: rate of x and y} it follows that 
$\|\hy-y^0\|_1$ is $O_p(s^{\kappa+1/2}\lambda)$.  Lemma~\ref{result: inf norm: dif } implies $|\hSxy-\Sxy|_\infty$ is $O_p(\lambda)$ and $|\hSxy|_\infty$ is $O_p(1)$, and \eqref{intheorem: rate of rem} indicates that $\|\text{rem}\|_{\infty}=O_p(s^{2\kappa}\lambda^2)$.  Lemma~\ref{lemma:norm:  l1 and l2 norm} entails that $\|y^0\|_1=O_p(s^{1/2})$. The definition of $\mathcal L_1$ in \eqref{def: L} implies that $\|\mathcal L_1\|_\infty$ is of the order  $O_p(\|(\hSx-\Sx)x^0\|_\infty)+O_p(\|(\hSy-\Sy)y^0\|_\infty)$.   Using Lemma~\ref{lemma:norm:  l1 and l2 norm} and  Lemma~\ref{result: inf norm: dif }, therefore, we can show that
$\|\mathcal L_1\|_\infty=O_p(\lambda)$. Using these rates in the bound derived in \eqref{intheorem: rho: bound: 1}, we obtain that
\begin{align*}
    \abs{ \hy^T\hSyx(\hdai-x^0)+ (y^0)^T\Syx\mathcal{L}_{(1)}}\leq &\ O_p(s^{3\kappa+1}\lambda^3)+ O_p(s^{\kappa+1/2}\lambda^2)+O_p(s^{1/2}\lambda^2)+O_p(s^{2\kappa+1/2}\lambda^2)\\
    =&\ O_p(s^{3\kappa+1}\lambda^3)+O_p(s^{2\kappa+1/2}\lambda^2).
\end{align*}
By Fact~\ref{fact: slambda goes to zero}, $s^{\kappa+1/2}\lambda\to 0$. Thus, the above bound is of order $O_p(s^{2\kappa+1/2}\lambda^2)$, which is $o_p(n^{-1/2})$ by our assumption on $s$.
By symmetry, we also have
\[\abs{ \hx^T\hSxy(\hdbi-y^0)+ (x^0)^T\Sxy\mathcal{L}_{(2)}}=o_p(n^{-1/2}),\]
where $\mathcal L_2$ is as defined in \eqref{def: L}.
Therefore, \eqref{intheorem: rho: first expansion} leads to
\begin{align}\label{intheorem: rho: second expansion}
  \MoveEqLeft  \hx^T\hSxy\hdbi+(\hdai)^T\hSxy\hy-\hx^T\hSxy\hy-\rhk^2\nn\\
    = &\ -(y^0)^T\Syx\mathcal{L}_{(1)}- (x^0)^T\Sxy\mathcal{L}_{(2)}+ (x^0)^T(\hSxy-\Sxy)y^0
    -(\hx-x^0)^T\hSxy (\hy-y^0)+o_p(n^{-1/2}).
\end{align}
We will show that the
 quadratic term is also $o_p(n^{-1/2})$. To that end, notice that
\begin{align*}
 \MoveEqLeft   \abs{(\hx-x^0)^T\hSxy (\hy-y^0)}\\
 \leq &\ \abs{(\hx-x^0)^T(\hSxy -\Sxy)(\hy-y^0)}+\abs{(\hx-x^0)^T\Sxy (\hy-y^0)}\\
 \leq &\ \|\hx-x^0\|_1|\hSxy-\Sxy|_\infty\|\hy-y^0\|_1+M\|\hx-x^0\|_2\|\hy-y^0\|_2
\end{align*}
by Assumption~\ref{assump: bounded eigenvalue}. From Lemma~\ref{corollary: rate of x and y} and Lemma~\ref{result: inf norm: dif } it follows that 
\[\|\hx-x^0\|_1|\hSxy-\Sxy|_\infty\|\hy-y^0\|_1=O_p(s^{2\kappa+1}\lambda^3).\]
By Fact~\ref{fact: slambda goes to zero}, $s\lambda\to 0$. Therefore, $s^{2\kappa+1}\lambda^3=o_p(s^{2\kappa}\lambda^2)$. Therefore,
\[\|\hx-x^0\|_1|\hSxy-\Sxy|_\infty\|\hy-y^0\|_1=o_p(s^{2\kappa}\lambda^2).\]
 Lemma~\ref{corollary: rate of x and y} implies, on the other hand, that
\[\|\hx-x^0\|_2\|\hy-y^0\|_2=O_p(s^{2\kappa}\lambda^2).\]
 Since
\[s^{2\kappa}\lambda^2\leq s^{2\kappa+1/2}\lambda^2=o(n^{-1/2}),\]
it follows that
\[\abs{(\hx-x^0)^T\hSxy (\hy-y^0)}=o_p(n^{-1/2}).\]
Now
\eqref{intheorem: rho: second expansion} indicates that
\begin{align*}
 \MoveEqLeft   \hx^T\hSxy\hdbi+(\hdai)^T\hSxy\hy-\hx^T\hSxy\hy-\rhk^2\\
    =&\ -(x^0)^T\Sxy \mathcal L_2-(y^0)^T\Syx\mathcal L_1+(x^0)^T(\hSxy-\Sxy)y^0\\
    \stackrel{(a)}{=}&\ \rhk\sum_{i=1}^n\frac{(Z_i-E[Z_i])}{n}
\end{align*}
where
\[Z_i=-\rhk(X_i^Tx^0)^2-\rhk(Y_i^Ty^0)^2+2(X_i^Tx^0)(Y_i^Ty^0).\]
Here (a) follows from Lemma~\ref{lemma: rho: variance of rhohat}.
Note that $Z_i$'s are $n$ independent copies of the random variable  $Z=-\rhk(X^Tx^0)^2-\rhk(Y^Ty^0)^2+2(X^Tx^0)(Y^Ty^0)$.
 If we can show that $EZ^4<\infty$, then $\xi$'s satisfy the Lyaponov's condition \citep[cf. Theorem 27.3 of ][]{billingsley2008}, which leads to
\begin{align}\label{intheorem: rho: convergence}
    \frac{\sum_{i=1}^n(Z_i-E[Z_i])}{(n\text{var}(Z))^{1/2}}\to_d N(0,1).
\end{align}
Now note that
\[E[Z^{4}]\leq C(E[(X_i^Tx^0)^8]+E[(Y_i^Ty^0)^8])\stackrel{(a)}{\leq}  C(\|x^0\|^8_2+\|y^0\|^8_2)\]
where (a) follows from  \eqref{intheorem: main: moment inequality}.
Hence, by Lemma~\ref{lemma:norm:  l1 and l2 norm}, $E [Z^4]<\infty$. Thus \eqref{intheorem: rho: convergence} holds with $\sigma^2_\rho=\text{var}(Z) $. Hence, first part of the proof follows.

The second part of the proof will be devoted towards obtaining the form of $\sigma_\rho^2$ when $X$ and $Y$ are multivariate Gaussian vectors. 
To that end, we note that
\begin{align*}
\sigma^2_\rho=\text{var}(Z)  =&\ \rhk^2\text{var}((X^Tx^0)^2)+\rhk^2\text{var}((Y^Ty^0)^2)+4\text{var}((X^Tx^0)(Y^Ty^0))+2\rhk^2\text{cov}((X^Tx^0)^2,(Y^Ty^0)^2)\\
&\ -4\rhk\text{cov}((X^Tx^0)^2,(X^Tx^0)(Y^Ty^0))-4\rhk\text{cov}((Y^Ty^0)^2,(X^Tx^0)(Y^Ty^0))
\end{align*}
\def\X{\en{\mathcal X}}
Let us denote $\X_1=X^Tx^0$ and $\X_2=Y^Ty^0$. Then we have
\[(\X_1,\X_2)\equiv (X^Tx^0,Y^Ty^0)\sim N_2\left(0,\begin{bmatrix}
   \rhk & \rhk^2\\
   \rhk^2 & \rhk
\end{bmatrix}\right).\]
Since $\X_1^2\sim \rhk\chi^2_1$, it follows that $\text{var}(\X_1^2)=2\rhk^2$. Similarly, $\text{var}(\X_2^2)=2\rhk^2$. On the other hand,
$\X_2|\X_1\sim N(\rhk \X_1,\rhk(1-\rhk^2)).$
Note that
\[Var(Z)=4\rhk^4+4\text{var}(\X_1\X_2)+2\rhk^2\text{cov}(\X_1^2,\X_2^2)-4\rhk\text{cov}(\X_1^2+\X_2^2,\X_1\X_2).\]
Noting $E[\X_1^2\X_2^2]=\rhk^2+2\rhk^4$ by Fact~\ref{fact: higher moments of bivariate normal}, we calculate
\[\text{cov}(\X_1^2,\X_2^2)=E[\X_1^2\X_2^2]-E[\X_1^2]E[\X_2^2]=(\rhk^2+2\rhk^4)-\rhk^2=2\rhk^4\]
and
\[\text{var}(\X_1\X_2)=E[\X_1^2\X_2^2]-E[\X_1\X_2]^2=\rhk^2+2\rhk^4-(\rhk^2)^2=\rhk^4+\rhk^2.\]
Fact~\ref{fact: higher moments of bivariate normal} also implies $E[\X_1^3\X_2]=3\rhk^3$, leading to
\[\text{cov}(\X_1^2,\X_1\X_2)=E[\X_1^3\X_2]-E[\X_1^2]E[\X_1\X_2]=3\rhk^3-\rhk^3=2\rhk^3.\]
Then the proof follows noting
\begin{align}
    \text{var}(Z)=&\ 4\rhk^4+(4\rhk^2+4\rhk^4)+4\rhk^6-16\rhk^4= 4\rhk^2(1-2\rhk^2+\rhk^4)=4\rhk^2(1-\rhk^2)^2.
\end{align}


\end{proof}
 
\section{Proofs of Supplement~\ref{app: COLAR}}
 \label{app: proof of colar}
 
\subsection{Proof of Theorem~\ref{thm: Chao Thm 4.2}}
We will prove the theorem only for $\ha$ because the proof for $\hb$ will follow similarly. In particular, we will show that 
\[\|\ha-\alk\|_2=\begin{cases}
O_p(s_U^{1/2}\lambda) & \text{if } r=1\\
O_p(s_U\lambda) & \text{o.w.}
\end{cases}
\]
and
\[\|(\ha-\alk)_{S_U}\|_1\leq s_U^{1/2}\|(\ha-\alk)\|_1\quad\text{and}\quad \|(\ha-\alk)_{S^c_U}\|_1=O_p(s_U\lambda).\]

First note that since $s\lambda\to 0$, $\lambda\to 0$. Then by \eqref{def:lambda}, $\log(p+q)=o(n)$ follows.
The proof of Theorem ~\ref{thm: Chao Thm 4.2} hinges on Theorem~\ref{theorem: COLAR: Theorem 4.1}, which we will prove later this section. Theorem~\ref{theorem: COLAR: Theorem 4.1} collects the rate of $\|\widehat F_n-F_0\|_F$.
\begin{theorem}\label{theorem: COLAR: Theorem 4.1}
Under the set-up of Theorem~\ref{thm: Chao Thm 4.2},
 $\|\widehat F_n-F_0\|_F=O_p(s\lambda).$
\end{theorem}
Theorem~\ref{theorem: COLAR: Theorem 4.1} is similar to that of Theorem 4.1 of \cite{gao2017}.

The importance of  Theorem~\ref{theorem: COLAR: Theorem 4.1} will be clear very soon. Since $\haz$ and $\hbz$ are the first pair of singular vectors of  $\hF$, we can find their rate of  convergence to $\alk$ and $\bk$, respectively, from the rate of convergence of $\hF$ using the Davis-Kahan sin $\theta$ theorem.
 We will use the version of Davis-Kahan Sin $\theta$ theorem given by Theorem 4 of \cite{yu2015} because it is suited for general $p\times q$ matrices. Since $\hbz$ and $\bk$ are the respective left singular vectors of $\hF$ and $F_0$, Theorem 4 of \cite{yu2015}  entails that
 \[\min_{s=\pm 1}\|s\hbz-\bk\|_2\leq C (2+\|\hF-F_0\|_F)\|\hF-F_0\|_F,\]
 where $C$ is a universal constant. Under our Assumption~\ref{assumption: sparsity}, Theorem~\ref{theorem: COLAR: Theorem 4.1} implies that $\|\hF-F_0\|_F$ is $o_p(1)$, which indicates
 \begin{equation}\label{intheorem: colar: beta rate}
     \min_{s=\pm 1}\|s\hbz-\bk\|_2\leq 3C\|\hF-F_0\|_F=O_p(s\lambda).
 \end{equation}
As a side result, we also obtain 
\begin{equation}\label{intheorem: colar: ub: hbz}
  \|\hbz\|_2\leq M+o_p(1)
\end{equation}
which follows from Lemma~\ref{lemma:norm:  l1 and l2 norm} since
$ |\|\hbz\|_2-\|\bk\|_2|=o_p(1)$.

  We define the quantity
\begin{equation}\label{def: the definition of cu}
    \tg=U\Lambda V^T\Sy\hbz\quad\text.
\end{equation}
Note that $\tg$ is dependent, through $\hbz$, only on the first part of the data. Thus $\tg$ is independent of $\hSxy^{(1)}$, $\hSx^{(1)}$, and $\hSy^{(1)}$ because the above-mentioned matrices are computed from the second part of the data.

Now we will present some key lemmas which will be useful in proving Theorem~\ref{thm: Chao Thm 4.2}. The proof of these lemmas can be found in Subsection~\ref{sec: proof of lemmas for Chao 4.2}.
We begin by noting some properties of $\tg$. 
\begin{lemma}\label{lemma: colar: cu properties}
Under the set up of Theorem \ref{thm: Chao Thm 4.2}, the vector $\tg$ defined in \eqref{def: the definition of cu} satisfy
\[\abs{(\tg)^T\Sx \tg-\rhk^2}=o_p(1)\]
Moreover, 
\[\rhk/(2M^{1/2})\leq \|\tg\|_2\leq 2\rhk M^{1/2},\]
where $M$ is as in Assumption~\ref{assump: bounded eigenvalue}.
\end{lemma}

Our next lemma establishes that $\ha=(\tx^T\Sx\tx)^{-1/2}\tx$ with high probability for large $n$.

\begin{lemma}\label{lemma: colar: thm 4.2: T1}
 Under the set up of Theorem \ref{thm: Chao Thm 4.2},
 \begin{align}\label{statement: lemma: colar: T1: diff}
     \abs{\tx^T(\hSx^{(1)}-\Sx)\tx}=O_p(s_U^{1/2}\lambda),
 \end{align}
   where $\tx$ is as  in Algorithm \ref{algo:Modified COLAR}.
  Also, the $\ha$ defined in \eqref{def: COLAR: ha} satisfies
  \begin{align}\label{statement: lemma: colar: T1: probability}
      P\lb \ha= (\tx^T\hSx\tx)^{-1/2}\tx\rb\to 1\quad\text{ as }n\to\infty,
  \end{align}
 and
 \begin{align}\label{statement: lemma: colar: T1: ha diff}
    \|\ha- (\tx^T\Sx\tx)^{-1/2}\tx\|_2=O_p(s_U^{1/2}\lambda). 
 \end{align}
  \end{lemma}
  
  The next lemma will be essential in bounding $\inf_{w\in\{\pm 1\}}\|w\ha-\alk\|_2$.
  \begin{lemma}\label{lemma: COLAR: last step where slambda comes from}
  Consider the set up of  of Theorem \ref{thm: Chao Thm 4.2}.
 Let us denote $\tU=\Sx^{1/2}U$ and $\tV=\Sy^{1/2}V$. Suppose $x\in\RR^p$ has unit norm  and  $y\in\RR^q$. 
\begin{itemize}
    \item[A.] If the rank of $\Lambda$, i.e. $r=1$, then
 \[\|P_{x}-P_{\tu_1}\|_F\leq \|P_{x}-P_{\tU\Lambda(\tV)^Ty}\|_F. \]
 \item[B.] Suppose in addition, $\inf_{w\in\{\pm 1\}}\|wy-\tv_1\|_2=O_p(s\lambda)$. Then for $r>1$,
 \[\|P_x-P_{\tu_1}\|_F\leq 5\|P_x-P_{\tU\Lambda(\tV)^Ty}\|_F+O_p(s\lambda). \]
\end{itemize}
\end{lemma}
Now that we have collected all the tools necessary, we are ready to start the main proof.

\begin{proof}[of Theorem~\ref{thm: Chao Thm 4.2}]
We denote $\Delta=\tx-\tg$. In the first step of the proof, we show that $\|\Delta\|_2$ is small. The second step is devoted towards showing  that if $\|\Delta\|_2$ is small, then  $\inf_{w\in\pm 1}\|w\ha-\alk\|_2$ is negligible as well.

 In the first step, we begin by deriving a bound on $tr(\Delta^T
\hSx^{(1)}\Delta)$. 
First, since $\tx$ is a solution to \eqref{COLAR: stage 2}, we have 
\[\tx^T\hSx^{(1)} \tx-2 \tx^T\hSxy^{(1)} \hbz +\lambda_2 \|\tx\|_1\leq (\tg)^T\hSx^{(1)}\tg-2(\tg)^T\hSxy^{(1)}\hbz+\lambda_2\|\tg\|_1.\]
Rearranging the terms give
\begin{align}\label{inlemma: rearrange}
    \Delta^T\hSx^{(1)}\Delta \leq &\  \lambda_2(\|\tg\|_1-\|\tx\|_1)+2\Delta^T\hSxy^{(1)}\hbz+2(\tg)^T(\hSx^{(1)}\tg-\hSx^{(1)}\tx)\nn\\
    =&\ \lambda_2(\underbrace{\|\tg\|_1-\|\tx\|_1}_{T_1})- \underbrace{2\Delta^T(\hSx^{(1)}\tg-\hSxy^{(1)}\hbz)}_{T_2}.
    \end{align}
For the second term $T_2$, using the definition of $\tg$, we have
\[\Sx \tg=\Sx U\Lambda V^T\Sy \hbz=\Sxy\hbz,\]
which implies
\begin{align}\label{intheorem: colar: 4.2: T2}
 |T_2/2|=| \Delta^T(\hSx^{(1)}\tg-\hSxy^{(1)}\hbz ) |\leq &\ \|\Delta\|_1\|\hSx^{(1)}\tg-\Sx\tg\|_\infty+\|\Delta\|_1\|(\Sxy-\hSxy^{(1)})\hbz\|_{\infty}.
\end{align}
Because $\hSx^{(1)}$ and $\tg$ are independent, 
Lemma~\ref{result: inf norm: dif } can be applied to show that
\begin{equation}\label{intheorem:  sara : cu}
    \|\hSx^{(1)}\tg-\Sx\tg\|_\infty
=O_p(\|\tg\|_2\lambda),
\end{equation}
 which is $O_p(\lambda)$ because $\|\tg\|_2=O_p(1)$ by Lemma~\ref{lemma: colar: cu properties}.
Similarly, because $\hbz$ is independent of $\hSxy^{(1)},$ using Lemma~\ref{result: inf norm: dif } again, we can show that 
\[\|(\Sxy-\hSxy^{(1)})\hbz\|_{\infty}=\|\hbz\|_2 O_p(\lambda)\]
which is $O_p(\lambda)$ because 
by \eqref{intheorem: colar: ub: hbz}, $\|\hbz\|_2=O_p(1)$.
The above, in conjunction with \eqref{intheorem: colar: 4.2: T2} and \eqref{intheorem:  sara : cu} imply that there exists $C>0$ such that
\begin{align}\label{intheorem: colar: T2}
   | T_2|\leq C_2\|\Delta\|_1\lambda
\end{align}
with high probability.


Recalling  that we denoted $S_U$  to be the indices of the non-zero rows of  $U$, we note \[\tg_{S_U^c}=(U\Lambda V^T\Sy\hbz)_{S_U^c}=U_{S_U^c}\Lambda V^T\Sy\hbz=0.\]
Therefore, the support of $\tg$ is not larger than $S_U$.
For $T_1$, it thus follows that
\[\|\tg\|_1-\|\tx\|_1=\|\tg_{S_U}\|_1-\|\tg_{S_U}+\Delta_{S_U}\|_1-\|\Delta_{S_U^c}\|_1\leq \|\Delta_{S_U}\|_1-\|\Delta_{S_U^c}\|_1.\]
Noting $\lambda_2=C\lambda$, we choose $C>C_2$. Then we have with probability tending to one,
\begin{align*}
    \Delta^T\hSx^{(1)}\Delta \leq &\ 
    C \lambda\slb \|\Delta_{S_U}\|_1-\|\Delta_{S_U^c}\|_1\srb +
    C_2\lambda\|\Delta\|_1\\
    =&\   C_2\lambda
  \lbs C/C_2 \slb \|\Delta_{S_U}\|_1-\|\Delta_{S_U^c}\|_1\srb+
    \|\Delta_{S_U}\|_1+\|\Delta_{S_U^c}\|_1\rbs\\
    =&  C_2 \lambda\lbs (1+C/C_2)\|\Delta_{S_U}\|_1 -(C/C_2-1)\|\Delta_{S_U^c}\|_1\rbs.
\end{align*}
Recalling we chose $C> C_2$, we have $C/C_2-1>0$. There are some important consequences of the above inequality. First, because $\hSx^{(1)}$ is non-negative definite, we obtain the cone condition
\begin{equation}\label{intheorem: cone condition}
    \|\Delta_{S_U^c}\|_1\leq \frac{C/C_2+1}{C/C_2-1}\|\Delta_{S_U}\|_1.
\end{equation}
Second, we derive
\begin{equation}\label{intheorem: colar: l1 and l2 of delta}
     \Delta^T\hSx^{(1)}\Delta \leq (C_2+C)  \lambda \|\Delta_{S_U}\|_1\stackrel{(a)}{\leq} (C_2+C) s_U^{1/2}\lambda \|\Delta_{S_U}\|_2
\end{equation}
where (a) follows by Cauchy Schwarz inequality.
Now we will show that the bound on $ \Delta^T\hSx^{(1)}\Delta$ induces a bound on $\|\Delta\|_2$. 

Let $I_1=\{i_1,\ldots, i_t\}$ be the index set of the $t$  elements with largest absolute values in $S_U^c$. Let us denote $\tilde S_U=S_U\cup I_1$. 
Note that
\begin{align*}
    \Delta^T\hSx^{(1)}\Delta\geq \Delta_{\tilde S_U}^T\hSx^{(1)}\Delta_{\tilde S_U}-\Delta_{\tilde S_U^c}^T\hSx^{(1)}\Delta_{\tilde S_U^c}.
\end{align*}
Lemma 6.5 of \citeauthor{gao2017} implies that 
\[\Delta_{\tilde S_U}^T\hSx^{(1)}\Delta_{\tilde S_U}\geq (M^{-1}-C((s_U+t)\log p/n)^{1/2})\|\Delta_{\tilde S_U}\|_2^2,\]
and
\[\Delta_{\tilde S_U}^T\hSx^{(1)}\Delta_{\tilde S_U^c}\leq (M+C(s_U\log p/n)^{1/2})\|\Delta_{\tilde S_U^c}\|_2^2.\]
Note that $s_U\log p/n=o_p(1)$ because $s_U\lambda\to 0$.
Now using the cone condition \eqref{intheorem: cone condition}, and 
proceeding like the Step 2 of the proof of  Theorem 4.2 of \citeauthor{gao2017}, we can show that
\begin{align}\label{inlemma: colar: l2 triangle inequality}
    \|\Delta_{\tilde S_U^c}\|_2\leq 3\frac{s_U}{t}\|\Delta_{\tilde S_U}\|_2.
\end{align}
Since the proof is identical to that of the Step 2 of  Theorem 4.2 of \citeauthor{gao2017}, it is skipped. 
When $s_U=o(p)$, we can take $t=c_1 s_U$ where $c_1$ is a large constant. Then $\|\Delta_{\tilde S_U^c}\|_2\leq 3\|\Delta_{\tilde S_U}\|_2/c_1$. Therefore, combining all the pieces above give us
\[\|\Delta_{\tilde S_U}\|_2^2( M^{-1}- 9M/c_1^2+o_p(1))\leq  \Delta^T\hSx^{(1)}\Delta.\]
When $c_1>3M$, the multiplicative constant with $\|\Delta_{\tilde S_U}\|_2^2$ is positive.
Therefore, using \eqref{intheorem: colar: l1 and l2 of delta} we obtain the following inequality: 
\[\|\Delta_{\tilde S_U}\|_2^2=O_p(s_U^{1/2}\lambda) \|\Delta_{S_U}\|_2.\]
Because $S_U\subset \tilde S_U$, the above implies $\|\Delta_{\tilde S_U}\|_2^2\leq O_p(s_U^{1/2}\lambda)\|\Delta_{\tilde S_U}\|_2$, 
which entails $\|\Delta_{\tilde S_U}\|_2$ is $O_p(s_U^{1/2}\lambda)$. 
Finally, \eqref{inlemma: colar: l2 triangle inequality} and the fact that $t=c_1s_u$ implies $\|\Delta_{\tilde S^c_U}\|_2$ is also $O_p(s_U^{1/2}\lambda)$. Since $\|\Delta\|_2^2$ equals $\|\Delta_{\tilde S_U}\|_2^2+\|\Delta_{\tilde S^c_U}\|_2^2$, we have
\begin{align}\label{intheorem: COLAR: Delta rate}
\|\Delta\|_2=   \|\tx-\tg\|_2=O_p((s_U\log(p+q)/n)^{1/2}).
   \end{align}


 Now we are ready to compute the rate of $\inf_{w\in\{\pm 1\} } \|w\ha-\alk\|_2$. To this end, note that
 \begin{align}\label{intheorem: ineq: ha alk}
  \inf_{w\in\{\pm 1\} } \|w\ha-\alk\|_2\leq \MoveEqLeft \underbrace{\|\ha- (\tx^T\Sx\tx)^{-1/2}\tx\|_2}_{T_1}+\underbrace{\inf_{w\in\{\pm 1\} }\|w(\tx^T\Sx\tx)^{-1/2}\tx-\alk\|_2}_{T_2}
 \end{align}
 
 Lemma~\ref{lemma: colar: thm 4.2: T1} shows that $T_1=O_p(s_U^{1/2}\lambda)$.
 To control the term $T_2$, first note that
 \begin{equation}\label{ineq: colar: T2}
     T_2\leq M^{1/2}\inf_{w\in\{\pm 1\}}\|w(\tx^T\Sx\tx)^{-1/2}\Sx^{1/2}\tx-\Sx^{1/2}\alk\|_2.
 \end{equation}
 Since the normalized vectors  $a=(\tx^T\Sx\tx)^{-1/2}\Sx^{1/2}\tx$ and $\tu_1=\Sx^{1/2}\alk$ have unit norm, they are easier to work with than  $(\tx^T\Sx\tx)^{-1/2}\tx$ and $\alk$. By Fact~\ref{fact: Chen 2020}, we then obtain that
 \begin{equation}\label{ineq: colar: PaPb}
      \inf_{w\in\{\pm 1\}}\|wa-\tu_1\|_2^2\leq \|P_a-P_{\tu_1}\|_F^2.
 \end{equation}
We will now use Lemma~\ref{lemma: COLAR: last step where slambda comes from} to bound $\|P_a-P_{\tu_1}\|_F^2$, and we will see that the rate of this term depends on the rank $r$. 
Before applying Lemma~\ref{lemma: COLAR: last step where slambda comes from}, we notice \eqref{intheorem: colar: beta rate} and Assumption~\ref{assump: bounded eigenvalue} imply
\[\|\Sy^{1/2}(w\hbz-\bk)\|_2\leq M^{1/2}O_p(s\lambda).\]
Therefore, we can take the $y$ in Lemma~\ref{lemma: COLAR: last step where slambda comes from} to be $\Sy^{1/2}\hbz$.

We first consider the case when $r=1$. An application of Lemma~\ref{lemma: COLAR: last step where slambda comes from} with $x=a$ and $y=\Sy^{1/2}\hbz$ then yields 
\[\|P_a-P_{\tu_1}\|_F^2\leq \|P_a-P_{\Sx^{1/2}U\Lambda V^T\Sy \hbz}\|_F^2. \]
However, 
\[\|P_a-P_{\Sx^{1/2}U\Lambda V^T\Sy \hbz}\|_F^2\stackrel{(a)}{\leq} \|a-\Sx^{1/2}U\Lambda V^T\Sy \hbz\|_2^2\stackrel{(b)}{\leq} M\|\tx-\tg\|_2^2,\]
where (a) follows from Fact~\ref{fact: Chen 2020} and (b) follows from the definition of $a$, $\tg$, and Assumption~\ref{assump: bounded eigenvalue}. The term $\|\tx-\tg\|_2^2$
 is $O_p(s_U\lambda^2)$ by \eqref{intheorem: COLAR: Delta rate}.
 Hence, \eqref{ineq: colar: T2} and \eqref{ineq: colar: PaPb} imply that when $r=1$,
 $T_2=O_p(s_U^{1/2}\lambda)$. Then \eqref{intheorem: ineq: ha alk} and Lemma~\ref{lemma: colar: thm 4.2: T1} entail that for $r=1$, $\inf_{w\in\{\pm 1\}}\|w\ha-\alk\|_2=O_p(s_U^{1/2}\lambda)$.

  Now consider $r>2$. Proceeding like the previous case, we apply Lemma~\ref{lemma: COLAR: last step where slambda comes from} with $x=a$ and $y=\Sy^{1/2}\hbz$ to obtain
  \[\|P_a-P_{\tu_1}\|_F\leq 5\|P_a-P_{\Sx^{1/2}U\Lambda V^T\Sy \hbz}\|_F+O_p(s\lambda).\]
  Since we just showed that 
\[ \|P_a-P_{\Sx^{1/2}U\Lambda V^T\Sy \hbz}\|_F=O_p(s_U^{1/2}\lambda),\]
the above implies $\|P_a-P_{\tu_1}\|_F=O_p(s_U\lambda)$. 

To infer on the $l_1$ error, first observe that $u^*_{S^c_U}=U_{S^c_U}\Lambda V^T\Sy\hbz=0$, where $S_U$ is denotes set of  indexes of the non-zero rows in $U$. Also, $(\alk)_{S_U^c}=0$ because $\alk$ is the first column of $U$.  
 By Lemma~\ref{lemma: colar: thm 4.2: T1}, we also have  $\ha=\tx(\tx^T(\hSx^{(1)})\tx)^{-1/2}$ with probability tending to one. 
 Therefore, with probability tending to one, 
 \[\|(\ha-\alk)_{S_U^c}\|_1=\|(\ha)_{S_U^c}\|_1=(\tx^T(\hSx^{(1)})\tx)^{-1/2}\|(\tx)_{S_U^c}\|_1=(\tx^T(\hSx^{(1)})\tx)^{-1/2}\|(\tx-u^*)_{S_U^c}\|_1.\]
Observe that \eqref{intheorem: cone condition} implies there exists $c>0$ so that
\[\|(\tx-u^*)_{S_U^c}\|_1=\|\Delta_{S_U^c}\|_1\leq c\|\Delta_{S_U}\|_1\stackrel{(a)}{\leq} cS_U^{1/2}\|\Delta_{S_U}\|_2\stackrel{(b)}{=}O_p(s_U\lambda),\]
 where (a) follows by Cauchy-Schwarz inequality and (b) follows because $\|\Delta\|_2=O_p(S_U^{1/2}\lambda)$ by \eqref{intheorem: COLAR: Delta rate}.  Moreover, equation \ref{inlemma: colar: lower bound: T1: quad term} of Lemma \ref{lemma: colar: thm 4.2: T1} implies $(\tx^T(\hSx^{(1)})\tx)^{-1/2}=O_p(1)$. Therefore,
 \[\|(\ha-\alk)_{s_U^c}\|_1=O_p(s_U\lambda).\]
 Also, by Cauchy-Schwarz inequality,
 \[\|(\ha-\alk)_{S_U}\|_1\leq \sqrt{s_U}\|(\ha-\alk)_{S_U}\|_2.\]
 Hence, the proof follows.
\end{proof}

 \subsection{Proof of Theorem~\ref{theorem: COLAR: Theorem 4.1} }
  We begin by introducing some new  notations. Let us define
\begin{equation}\label{def: COLAR: tu}
    \hU= U(U^T\hSx^{(0)} U)^{-1/2}\quad\text{and}\quad \hV=V(V^T\hSy^{(0)} V)^{-1/2},
\end{equation}
and denote $\ta =\hU_1$ and $\tb=\hV_1$. 
Several times we will use without stating the fact that $\ta^T\hSx^{(0)}\ta=\tb^T\hSy^{(0)}\tb=1$.
We also denote
\begin{equation}\label{def: COLAR: tilde Lambda}
    \overline{\Lambda}=(U^T\hSx^{(0)} U)^{1/2}\Lambda(V^T\hSy^{(0)} V)^{1/2}
\end{equation}
and $\tilde F_n=\ta\tb^T$.  For notational convenience, we define
\[\e_{n,u}=n^{-1/2}\lb s+\log\frac{ep}{s_x}\rb^{1/2},\quad \e_{n,v}=n^{-1/2}\lb s+\log\frac{e q}{s_y}\rb^{1/2}.\]
We denote
\begin{equation}\label{def: tsxy}
  \tSxy=\hSx^{(0)} U\Lambda V^T\hSy^{(0)}.  
\end{equation}
Finally, let $\Delta^{(F)}=\widehat F_n-\tilde F_n$.

Now we state some lemmas which will be required for the proof of Theorem~\ref{theorem: COLAR: Theorem 4.1}. These lemmas are proved in Subsection~\ref{sec: lemmas: colar 4.1}.
The first lemma shows that $\tF$ is a good approximation of $F_0$ because $\|\tF-F_0\|_F$ is small. 
\begin{lemma}\label{lemma: COLAR thm 1: delta suffices}
Under the set up of Theorem~\ref{theorem: COLAR: Theorem 4.1},
 $\|\tF-F_0\|_F=O_p(\e_{n,u}+\e_{n,v})$
\end{lemma}

The next Lemma shows that $\tF=\ta\tb^T$ is a feasible solution to the step 1 optimization problem in  Algorithm~\ref{COLAR: First stage}.
\begin{lemma}\label{lemma: COLAR 6.2}
  When $\tF$ exists, 
\[\|(\hSx^{(0)})^{1/2}\tF(\hSy^{(0)})^{1/2}\|_*=1\quad \text{ and }\quad \|(\hSx^{(0)})^{1/2}\tF(\hSy^{(0)})^{1/2}\|_{op}=1.\]
\end{lemma}

The next lemma exploits the convexity of the unpenalized version of \eqref{COLAR: First stage} at $F_0$ and establishes a strong convexity type result at $F_0$.
\begin{lemma}\label{lemma: COLAR 6.3}
 Let $A\in\mathcal O(p,r)$ and $G\in \mathcal O(q,r)$. Suppose $\tilde D\in\RR^{r\times r}$ and $D$ are two diagonal matrices in $\RR^{r\times r}$ which diagonal entries  $D_{11}\geq D_{22},\ldots,\geq D_{rr}\geq 0$. Further suppose $d_{12}=D_{11}-D_{22}>0$. Let $E=e_1e_1^T$. If $F$ satisfies $\|F\|_{op}\leq 1$ and $\|F\|_*\leq 1$, then
 \[\langle A\tilde DG^T, AEG^T-F\rangle\geq \frac{d_{12}}{2}\|AEG^T-F \|_F^2-\|\tilde D-D\|_{F}\|AEG^T-F \|_F.\]

 \end{lemma}
  Now we are ready to start the main proof.
\begin{proof}[of Theorem~\ref{theorem: COLAR: Theorem 4.1} ]
{ Lemma~\ref{lemma: COLAR 6.2} implies that $\tilde F_n\in\mathcal G$, that is $\tF$ is a feasible solution of \eqref{COLAR: First stage}. Therefore, 
\[\langle \hSyx^{(0)}, \tilde F_n\rangle-\lambda_1\|\tilde F_n\|_1\leq \langle \hSyx^{(0)}, \widehat F_n\rangle-\lambda_1\|\widehat F_n\|_1, \]
which leads to
\begin{align}\label{inthm: colar 1: ineq 1}
   -\langle \tSxy,\Delta^{(F)}\rangle \leq \lambda_1(\|\tF\|_1-\|\tF+\Delta^{(F)}\|_1)+\langle \hSxy^{(0)}-\tSxy,\Delta^{(F)}\rangle,
\end{align}
where $\tSxy$ is as defined in \eqref{def: tsxy}. Observe that
\begin{align*}
    \|\tF\|_1-\|\tF+\Delta^{(F)}\|_1=\|(\ta)_{S_x}(\tb)_{S_y}^T\|_1-\|(\ta)_{S_x}(\tb)_{S_y}^T+\Delta_{S_x,S_y}^{(F)}\|_1-\|\Delta^{(F)}_{(S_x,S_y)^c}\|_1
\end{align*}
where  $\Delta^{(F)}_{S_u,S_v}=(\Delta^{(F)}_{i,j})_{i\in S_x, j\in S_y}$ and $\Delta^{(F)}_{(S_u,S_v)^c}=(\Delta^{(F)}_{i,j})_{(i,j)\in (S_x\times S_y)^c}$.
Note that
\begin{align*}
 \MoveEqLeft  \|(\ta)_{S_x}(\tb)_{S_y}^T\|_1-\|(\ta)_{S_x}(\tb)_{S_y}^T+\Delta_{S_x,S_y}^{(F)}\|_1-\|\Delta^{(F)}_{(S_x,S_y)^c}\|_1\\
 \leq &\ \|\Delta_{S_x,S_y}^{(F)}\|_1-\|\Delta^{(F)}_{(S_x,S_y)^c}\|_1.
\end{align*}}
On the other hand, the second term on the right hand side of \eqref{inthm: colar 1: ineq 1} satisfies
\[\langle \hSxy^{(0)}-\tSxy,\Delta^{(F)}\rangle \leq |\hSxy^{(0)}-\tSxy|_{\infty}\|\Delta^{(F)}\|_1.\]
Therefore, for $\lambda_1\geq 2|(\hSxy^{(0)})-\tSxy|_{\infty}$, \eqref{inthm: colar 1: ineq 1} implies that
\begin{align}\label{inlemma: ub: Dxy}
    -\langle \tSxy, \Delta^{(F)}\rangle \leq \frac{3\lambda_1}{2}\|\Delta_{S_xS_y}\|_1-\frac{\lambda_1}{2}\|\Delta_{(S_xS_y)^c}\|_1.
\end{align}
We will provide a lower bound on $  -\langle \tSxy, \Delta^{(F)}\rangle$ using Lemma~\ref{lemma: COLAR 6.3}.
Denoting $E=e_1e_1^T$ and ${\delta=\|\tilde\Lambda-\Lambda\|_F}$,  we obtain 
\begin{align}\label{inlemma: lb: Dxy}
    -\langle \tSxy, \Delta^{(F)}\rangle=&\ \langle (\hSx^{(0)})^{1/2}U\Lambda V^T(\hSy^{(0)})^{1/2},(\hSx^{(0)})^{1/2}(\tF-\widehat F_n)(\hSy^{(0)})^{1/2}\rangle\nn\\
    =&\ \langle(\hSx^{(0)})^{1/2} \hU\tL \hV^T(\hSy^{(0)})^{1/2},(\hSx^{(0)})^{1/2}(\hU E\hV^T-\widehat F_n)(\hSy^{(0)})^{1/2}\rangle\nn\\
    \stackrel{(a)}{\geq} &\ \frac{\Lambda_1-\Lambda_2}{2}\|(\hSx^{(0)})^{1/2}(\tF-\widehat F_n)(\hSy^{(0)})^{1/2}\|_F^2-\delta\|(\hSx^{(0)})^{1/2}(\tF-\hF)(\hSy^{(0)})^{1/2}\|_F
\end{align}
where (a) follows by Lemma~\ref{lemma: COLAR 6.3}. 
Here we used the fact that $d_{12}=\Lambda_1-\Lambda_2>0$ by the  Assumption~\ref{assump: eigengap assumptions}. 
Now Fact~\ref{fact: frob norm operator norm} implies $\delta\leq \sqrt{r}\|\tilde\Lambda-\Lambda\|_{op}$, but Lemma 6.1 of \citeauthor{gao2017} entails that $\|\tilde\Lambda-\Lambda\|_{op}=O_p(s^{1/2}\lambda)$.
Because $r$ is less than the number of non-zero rows in $U$, and the number of non-zero rows is $s_U\leq s$, we can say $\delta=O_p(s\lambda)$.

 Combining the upper and lower bounds derived in \eqref{inlemma: ub: Dxy} and \eqref{inlemma: lb: Dxy}, and denoting $\Lambda_1-\Lambda_2$ by $d_{12}$, we obtain that
 \begin{align}\label{inlemma: ub: d12 frob}
     d_{12}\|(\hSx^{(0)})^{1/2}\Delta^{(F)}(\hSy^{(0)})^{1/2}\|_F^2\leq 3\lambda_1\|\Delta_{S_US_V}\|_1-\lambda_1\|\Delta_{(S_US_V)^c}\|_1+2\delta\|(\hSx^{(0)})^{1/2}\Delta^{(F)}(\hSy^{(0)})^{1/2}\|_F
 \end{align}
which leads to
\begin{align}\label{inlemma: ub: d12 frob weaker}
     d_{12}\|(\hSx^{(0)})^{1/2}\Delta^{(F)}(\hSy^{(0)})^{1/2}\|_F^2\leq 3\lambda_1\|\Delta_{S_US_V}\|_1+2\delta\|(\hSx^{(0)})^{1/2}\Delta^{(F)}(\hSy^{(0)})^{1/2}\|_F.
 \end{align}
 Solving the quadratic equation (see equation 53 of \citeauthor{gao2017}) gives 
 \begin{equation}\label{inlemma: colar: quad sol}
    \|(\hSx^{(0)})^{1/2}\Delta^{(F)}(\hSy^{(0)})^{1/2}\|_F^2\leq 6\lambda_1\|\Delta_{S_US_V}\|_1/d_{12}+4\delta^2/d_{12}^2. 
 \end{equation}
We derive two conclusions from \eqref{inlemma: colar: quad sol}. First, using $\|\Delta_{S_US_V}\|_1\leq  \sqrt{s_x s_y}\|\Delta_{S_US_V}\|_F$, we derive
\begin{equation}\label{inlemma: bound: frobenius norm}
    \|(\hSx^{(0)})^{1/2}\Delta^{(F)}(\hSy^{(0)})^{1/2}\|_F^2\leq 6\lambda_1 \sqrt{s_U s_V}\|\Delta_{S_US_V}\|_F/d_{12}+4\delta^2/d_{12}^2.
\end{equation}
 Second, noting $ax^2-bx$ achieves minima at $x=b/(2a)$, we obtain that
 \[d_{12}\|(\hSx^{(0)})^{1/2}\Delta^{(F)}(\hSy^{(0)})^{1/2}\|_F^2-2\delta\|(\hSx^{(0)})^{1/2}\Delta^{(F)}(\hSy^{(0)})^{1/2}\|_F\geq -\frac{\delta^2}{d_{12}}.\]
 Therefore, \eqref{inlemma: ub: d12 frob} yields  the generalized cone condition
\begin{equation}\label{inlemma: generalized cone condition}
 0\leq    3\|\Delta_{S_US_V}\|_1-\|\Delta_{(S_US_V)^c}\|_1+\frac{\delta^2}{\lambda_1 d_{12}}.
\end{equation}
 Although the constants in our inequality \eqref{inlemma: generalized cone condition} are a little bit sharper than the cone condition inequality (56) of \citeauthor{gao2017}, both cone conditions are equivalent.  By Lemma~\ref{result: inf norm: dif }, we have
 \[|(\hSxy^{(0)})-\Sxy|_{\infty}\leq C_1\lambda.\]
 Let us set $\lambda_1=C \lambda$ with $C\geq C_1$. The rest of the proof follows from  step 2 of the proof of Theorem 4.1 in \citeauthor{gao2017}, which indicates that for this choice of $C$, when  \eqref{inlemma: bound: frobenius norm} and  \eqref{inlemma: generalized cone condition} hold, and  $\delta=O_p(s\lambda)$,   there exists $C'>0$ so that
 \begin{align}\label{inlemma: ub: Delta}
     \|\Delta^{(F)}\|_F\leq C'(s_Us_V)^{1/2}\lambda_1/d_{12}.
 \end{align}
 with high probability.  The proof of the current theorem follows combining \eqref{inlemma: ub: Delta} and Lemma~\ref{lemma: COLAR thm 1: delta suffices}.
\end{proof}

\section{ Proof of Theorem~\ref{thm: nodewise Lasso theorem}}
\label{app: proof of NL}

\subsection{Proof of the main theorem}
\label{sec: proof of main theorem}
The proof relies on the proximity of $\hx$
  and $\hy$ to $x^0$ and $y^0$, respectively.
 Note that   Lemma~\ref{corollary: rate of x and y}  imply  $\|\hx\|_i=\|x^0\|_i+o_p(1)$ and $\|\hy\|_i=\|y^0\|_i+o_p(1)$ $(i=1,2)$
 because $s^{\kappa+1/2}\lambda=o(1)$ by Fact~\ref{fact: slambda goes to zero}. These facts will be used often times in proving our lemmas and claims.
  
 
We will begin by introducing some notations and stating some  lemmas.
First we state a lemma that gives bound on the maximum and minimum eigenvalues of $\Phi^0$ and $H^0$. The proof can be found in Subsection~\ref{sec: lemma: NL}.

\begin{lemma}\label{lemma: Phi and H knot }
Under Assumption~\ref{assump: bounded eigenvalue} and Assumption~\ref{assump: eigengap assumptions},
\begin{gather*}
\Lambda_{max}(H^0)\leq  8\rhk M,\nn\\
    \Lambda_{max}(\Phi^0)\leq 2^{-1}M/(\rhk-\Lambda_2),\nn\\
    \Lambda_{min}(H^0)\geq 2(\rhk-\Lambda_2)/M\nn\\
     \Lambda_{min}(\Phi^0)\geq (8\rhk M)^{-1}\nn\\
     \tau_j^2\geq 2(\rhk-\Lambda_2)/M\nn.
\end{gather*}
where
$\tau_j^0$ is as defined in \eqref{def: tau j knot}.
\end{lemma}

Recall the term $\eta_j^0$  defined in \eqref{def: eta j knot} in Supplement \ref{sec: asymp: NL}.
 For the time being, let us also denote
 \begin{equation}\label{def: Gamma knot}
     \Gamma^0_j=(-(\etak)_1,\ldots,-(\etak)_{j-1},1,-(\etak)_{j+1},\ldots,-(\etak)_{p+q}).
 \end{equation}
 
 Lemma~\ref{lemma: NL: claim 2 criteria lemma} establishes that $ \max_{1\leq j\leq p+q}\|\Gamma^0_j\|_2$ is bounded. The proof can be found in Subsection~\ref{sec: lemma: NL}.
\begin{lemma}\label{lemma: NL: claim 2 criteria lemma}
The $\Gamma^0_j$ defined in \eqref{def: Gamma knot} satisfies $ \max_{1\leq j\leq p+q}\|\Gamma^0_j\|_0=O(s)$. Moreover, 
there exists $C>0$ so that
$ \max_{1\leq j\leq p+q}\|\Gamma^0_j\|_2\leq C$.
\end{lemma}

We denote the sample version of $\eta_j^0$ to be $\widehat\eta_j$, which is given by the step NL1 in Algorithm \ref{algo: nodewise lasso}  when $A=\widehat H_n(\hx,\hy)$, where $\hx=|\hro|^{1/2}\ha$ and $\hb=|\hro|^{1/2}\hb$.
Let us denote $\Delta(j)=\etak-\heta$.
Recall from \eqref{def: widehat gamma j} in  Algorithm~\ref{algo: nodewise lasso} also that
\[\widehat\Gamma_j=(-(\heta)_1,\ldots,-(\heta)_{j-1},1,-(\heta)_{j+1},\ldots,-(\heta)_{p+q}).\]
Let us denote
\begin{equation}\label{def: Delta j}
    \Delta_{\Gamma}(j)=\widehat\Gamma_j-\Gamma^0_j.
\end{equation}
We will now state a key lemma for the proof of Theorem~\ref{thm: nodewise Lasso theorem}, which is proved in Subsection~\ref{sec: nodewise lasso: proof: key lemmas}.
\begin{lemma}\label{lemma: NL: the cross term lemma}
Under the set up of Theorem~\ref{thm: nodewise Lasso theorem}, we can find $C>0$ so that  the  followings hold with high probability for all sufficiently large $p$, $q$, and $n$:
\begin{align*}
 \abs{\Delta_{\Gamma}(j)^T(\widehat{H}_n(\hx,\hy)-H^0)\Delta(j)}\leq C\lambda(\|\Delta(j)\|_1+s^{1/2}\|\Delta_{\Gamma,1}(j)\|_2+s^{\kappa}\|\Delta_{\Gamma}(j)\|_2^2) \quad(j=1,\ldots,p+q),
\end{align*}
\[ \abs{\Delta_{\Gamma}(j)^T(\{\widehat{H}_n(\hx,\hy)\}-H^0)\Gamma^0_j}\leq  C(\lambda\|\Delta(j)\|_1+s^{\kappa}\lambda\|\Delta(j)\|_2)\quad(j=1,\ldots,p+q),\]
and
\[ \abs{\Delta_{\Gamma}(j)^T(\{\widehat{H}_n(\hx,\hy)\}-H^0)e_j}\leq  C(\lambda\|\Delta(j)\|_1+s^{\kappa}\lambda\|\Delta(j)\|_2)\quad(j=1,\ldots,p+q).\]
\end{lemma}

Lemma~\ref{lemma: NL: tau: cross term lemma}, which is proved in Subsection~\ref{sec: nodewise lasso: proof: key lemmas}, conveys a similar result.
\begin{lemma}\label{lemma: NL: tau: cross term lemma}
Under the set up of Theorem~\ref{thm: nodewise Lasso theorem}, we can find $C>0$ so that for  sufficiently large $p$, $q$ and $n$, the  following holds with high probability:
\begin{equation}\label{inlemma: NL: tau: term 2 claim}
\max_{1\leq j\leq p+q}    \abs{e^T_j(\widehat H_n(\hx,\hy)-H^0){\Gamma}^0_j}=O_p(s^{\kappa}\lambda).
\end{equation}
\end{lemma}

 Now we will start the proof of Theorem~\ref{thm: nodewise Lasso theorem}.
The proof has  two main steps.
The  first step establishes the proximity between $\eta_j^0$ and $\heta$. In the second step, we establish that $|\widehat\tau^2_j-(\tau^0_j)^2|=O_p(s^{\kappa}\lambda)$. Then using Lemma~\ref{Lemma: linear algebra}, we show that  $\Phi^0_j$ and $(\hf)_j$ are close.
Now we state and prove a lemma which establishes that the $l_1$ and $l_2$ norms of $\heta-\etak$ are small.
\begin{lemma}\label{lemma: NL: sara lemma 4}
\[\max_{1\leq j\leq p+q}\|\heta-\etak\|_2=O_p(s^{\kappa}\lambda)\quad\text{and}\quad
\max_{1\leq j\leq p+q}\|\heta-\etak\|_1=O_p(s^{\kappa+1/2}\lambda).\]
\end{lemma}

\begin{proof}[of Lemma~\ref{lemma: NL: sara lemma 4}]
We will denote
\[L(\eta)=\eta^TH^0_{-j,-j}\eta+
H^0_{j,j}-2\eta^TH^0_{-j,j},\quad \eta\in\RR^{p+q}.
\]
Observe that
\[\dot{L}(\eta)=2H^0_{-j,-j}\eta-2H^0_{-j,j},\quad \eta\in\RR^{p+q}. \]
The sample version of $L(\eta)$ writes as 
\begin{equation}\label{def: Ln}
    L_n(\eta)=\eta^T\{\widehat H_n(\hx,\hy)\}_{-j,-j}\eta+
\widehat H_n(\hx,\hy)_{j,j}-2\eta^T\{\widehat H_n(\hx,\hy)\}_{-j,j},\quad\eta\in\RR^{p+q}.
\end{equation}
 We can show that
\[\dot{L}_n(\eta)=2\{\widehat{H}_n(\hx,\hy)\}_{-j,-j}\eta-\{\widehat{H}_n(\hx,\hy)\}_{-j,j}\quad \eta\in\RR^{p+q}. \]
Note that \eqref{def: Ln} is also the unpenalized objective function of \eqref{opt: nodewise Lasso}.

Since $B_j\geq \|\etak\|_1$, $\etak$ is in the feasible region of \eqref{opt: nodewise Lasso}, where $\heta$ is a stationary point of \eqref{opt: nodewise Lasso}. Because  \eqref{opt: nodewise Lasso} is a convex program, the following inequality holds:
\begin{equation}\label{inlemma: NL: stationary}
    (\dot{L}_n(\heta)+\lnlm_j\widehat Z_n)^T(\etak-\heta)\geq 0,
\end{equation}
where $\widehat Z_n$ is the subdifferential of the $l_1$ norm evaluated at $\heta$. On the other hand, since $L$ is a quadratic function in $\eta$,
\begin{equation}\label{inlemma: NL: L}
    L(\etak)-L(\heta)= \dot{L}(\heta)^T\Delta(j)+\frac{1}{2}\Delta(j)^TH^0_{-j,-j}\Delta(j).
\end{equation}
Using Lemma~\ref{lemma: Phi and H knot } we obtain that
\begin{align}\label{inlemma: NL: strong convexity of L}
    \Lambda_{min}(H^0_{-j,-j})\geq \Lambda_{min}(H^0)\geq 2(\rhk-\Lambda_2),
\end{align}
which is positive by 
Assumption~\ref{assump: eigengap assumptions}, which indicates $L:\RR^{p+q-1}\mapsto\RR$ is strongly convex at $\etak$ with positive definite Hessian $H^0_{-j,-j}$. Therefore \eqref{inlemma: NL: L} leads to
\begin{align}\label{inlemma: NL: L ineq}
    L(\heta)-L(\etak)\leq -\dot{L}(\heta)^T\Delta(j).
\end{align}
Recalling  $\Delta(j)=\etak-\heta$, and adding  \eqref{inlemma: NL: stationary} and \eqref{inlemma: NL: L ineq}, we obtain an upper bound of $L(\heta)-L(\etak)$:
\[L(\heta)-L(\etak)\leq (\dot{L}_n(\heta)-\dot{L}(\heta))^T\Delta(j)+
\lnlm_j\widehat Z_n^T\Delta(j).\]
We can also find a lower bound on $L(\heta)-L(\etak)$.

Let us define $C_M=(\rhk-\Lambda_2)/M$. Equation \ref{def: L} and \ref{inlemma: NL: strong convexity of L} indicate that
\[L(\eta)-L(\etak)\geq \dot{L}(\etak)+ C_M\|\eta-\etak\|^2_2.\]
 By the definition of $\etak$ in \eqref{def: eta j knot}, it is the minimizer of $L$, i.e. $ \dot{L}(\etak)=0$. Therefore, 
\[L(\eta)-L(\etak)\geq C_M\|\eta-\etak\|^2_2.\]
Hence,
\[C_M\|\Delta(j)\|_2^2 \leq L(\eta)-L(\etak)\leq  (\dot{L}_n(\heta)-\dot{L}(\heta))^T\Delta(j)+
\lnlm_j\widehat Z_n^T\Delta(j).\]
Now let us denote $S=s(\etak)$.
By definition of $\widehat Z_n$, we have $\widehat Z_n^T\heta=\|\heta\|_1$ and $\widehat Z_n^T\etak\leq \|\etak\|_1$, yielding
\begin{align*}
  \MoveEqLeft  C_M\|\Delta(j)\|_2^2-(\dot{L}_n(\heta)-\dot{L}(\heta))^T\Delta(j)\\
    \leq  &\ 
\lnlm_j(\|\etak\|_1-\|\heta\|_1)\\
= & \ \lnlm_j(\|(\etak)_S\|_1-\|\Delta(j)+\etak\|_1)\\
 =&\ \lnlm_j\{\|(\etak)_S\|_1-\|\Delta(j)_S+(\etak)_S\|_1-\|\Delta(j)_{S^c}\|_1\}\\
    \leq &\ \lnlm_j(\|\Delta(j)_S\|_1-\|\Delta(j)_{S^c}\|_1).
\end{align*}
Thus
\begin{align}\label{inlemma: NL: delta inequality}
   C_M\|\Delta(j)\|_2^2\leq  \lnlm_j(\|\Delta(j)_S\|_1-\|\Delta(j)_{S^c}\|_1)+(\dot{L}_n(\heta)-\dot{L}(\heta))^T\Delta(j).
\end{align}
Our next step is to find the rate of decay of the cross-term $(\dot{L}_n(\heta)-\dot{L}(\heta))^T\Delta(j)$, which equals
\begin{align*}
   \MoveEqLeft \Delta(j)^T(\dot{L}_n(\heta)-\dot{L}(\heta)) \\
  = &\ 2 \Delta(j)^T(\widehat{H}_n(\hx,\hy)_{-j,-j}-H^0_{-j,-j})\heta+ \Delta(j)^T(\widehat{H}_n(\hx,\hy)_{-j,j}-H^0_{-j,j})\\
  \stackrel{(a)}{=}&\ 2\Delta_{\Gamma}(j)^T(\widehat{H}_n(\hx,\hy)-H^0)\widehat\Gamma_j+ \Delta_{\Gamma}(j)^T(\widehat{H}_n(\hx,\hy)-H^0)e_j\\
  =&\ 2\Delta_{\Gamma}(j)^T(\widehat{H}_n(\hx,\hy)-H^0)\Delta_{\Gamma}(j)+
  2\Delta_{\Gamma}(j)^T(\widehat{H}_n(\hx,\hy)-H^0)\Gamma^0_j+\Delta_{\Gamma}(j)^T(\widehat{H}_n(\hx,\hy)\H^0)e_j.
\end{align*}
where $(a)$ follows because $\Delta_{\Gamma}(j)_j=0$.
The above decomposition, combined with
Lemma~\ref{lemma: NL: the cross term lemma}, indicates that there exists  a large positive constant $C$ such that  the following holds with high probability for large $p$, $q$, and $n$:
\begin{align*}
 \Delta(j)^T(\dot{L}_n(\heta)-\dot{L}(\heta)) \leq &\ C\lambda\slb\|\Delta(j)\|_1+s^{\kappa}(\|\Delta(j)\|_2+\|\Delta(j)\|_2^2)\srb\quad (j=1,\ldots,p+q).
\end{align*}
Lemma~\ref{lemma: NL: the cross term lemma} also had some $s^{1/2}$ terms, which we ignored because $s^{\kappa}$ is greater than $s^{1/2}$ since $\kappa\geq 1/2$ by Condition~\ref{cond: preliminary estimator}.
Since $\kappa\leq 1$, and $s\lambda=o(1)$ by Fact~\ref{fact: slambda goes to zero} for sufficiently large $p$, $q$, and $n$, \eqref{inlemma: NL: delta inequality} implies
\begin{align}\label{inlemma: NL: cone condition preparation}
\frac{C_M}{2}\|\Delta(j)\|^2_2\leq (\lnlm_j+C\lambda)\|\Delta(j)_S\|_1-(\lnlm_j-C\lambda)\|\Delta(j)_{S^c}\|_1+Cs^{\kappa}\lambda \|\Delta(j)\|_2\quad(j=1,\ldots,p+q).
\end{align}
Suppose $\lnlm_j=C_1\lambda$ where $C_1>C$. 
Equation \ref{inlemma: NL: cone condition preparation} then leads to some important consequences. First, note that
\[\|\Delta(j)_S\|_1\leq s^{1/2}\|\Delta(j)_S\|_2\leq s^{\kappa}\|\Delta(j)_S\|_2\]
because $\kappa\geq 1/2$. Using the above inequality,  \ref{inlemma: NL: cone condition preparation} reduces to
\[\|\Delta(j)\|^2_2\leq2(2C+C_1)s^{1/2}\lambda\|\Delta(j)\|_2/C_M\quad(j=1,\ldots,p+q),\]
which implies $\max_{1\leq j\leq p+q}\|\Delta(j)\|_2=O(s^{\kappa}\lambda)$. 
Using the rate of $\|\Delta(j)\|_2$, from equation \ref{inlemma: NL: cone condition preparation}, we conclude that there exists $C'>0$ so that
\begin{align}\label{inlemma: NL: cone condition of delta}
  (C_1-C)\lambda\|\Delta(j)_{S^c}\|_1\leq (C_1+C)\lambda\|\Delta(j)_{S}\|_1+C' s^{2\kappa}\lambda^2\quad(j=1,\ldots,p+q)  
\end{align}
with high probability for sufficiently large $n$. Since $\|\Delta(j)_{S}\|_1\leq s^{1/2}\|\Delta(j)_{S}\|_2$, the above implies
$\max_{1\leq j\leq p+q}\|\Delta(j)\|_1=O(s^{\kappa+1/2}\lambda+s^{2\kappa}\lambda^2)$.
Now because $s\leq 1$,
\[s^{2\kappa}\lambda^2= s^{\kappa+1/2}\lambda (s^{\kappa-1/2}\lambda)\leq s^{\kappa+1/2}\lambda s^{1/2}\lambda=o(s^{\kappa+1/2}\lambda) \]
because $s^{1/2}\lambda=o(1)$ by Fact~\ref{fact: slambda goes to zero}.
Hence, the proof follows.
\end{proof}

Our next step is to find the rate of convergence of $\widehat{\tau}_j^2$ defined in \eqref{def: widehat tau j}.

\begin{lemma}\label{lemma: NL: tau hat}
Under the set-up of Theorem~\ref{thm: nodewise Lasso theorem}, $\widehat{\tau}_j$ implies
\[\max_{1\leq j\leq p+q}|\widehat{\tau}_j^2-(\tau^0)^2_j|=O_p(s^{\kappa}\lambda).\]
\end{lemma}

\begin{proof}[of Lemma~\ref{lemma: NL: tau hat} ]
By \eqref{KKT: NL: consequences}, we have
$\widehat \tau_j^2=\widehat H_n(\hx,\hy)_j^T\widehat{\Gamma}_j$. Also, \eqref{eq: inverse columns} implies $\Gamma^0_j/(\tau^0_j)^2=\Phi^0_j$. Noting
$(\Phi^0_j)^TH^0_j=1$, We derive the relation $(\tau^0_j)^2=(H^0_j)^T\Gamma^0_j$. Therefore we can write
\begin{align*}
  |\widehat{\tau}_j^2-(\tau^0)^2_j|=&\ \abs{e^T_j\slb H_n(\hx,\hy)\widehat{\Gamma}_j-H^0\Gamma^0_j\srb}\\ 
  \leq &\ \abs{e^T_j( \widehat H_n(\hx,\hy)-H^0)\widehat{\Gamma}_j}+\abs{e^T_jH^0(\widehat{\Gamma}_j-\Gamma^0_j)}\\
  \leq &\ \abs{e^T_j( \widehat H_n(\hx,\hy)-H^0)\Delta_{\Gamma}(j)}+\abs{e^T_j(\widehat H_n(\hx,\hy)-H^0){\Gamma}^0_j}+\abs{e^T_jH^0(\widehat{\Gamma}_j-\Gamma^0_j)}
\end{align*}
where $\Delta_{\Gamma}(j)=\widehat\Gamma_j-\Gamma^0_j$.
Lemma~\ref{lemma: NL: tau: cross term lemma} implies
\begin{equation*}
\max_{1\leq j\leq p+q}    \abs{e^T_j(\widehat H_n(\hx,\hy)-H^0){\Gamma}^0_j}=O_p(s^{\kappa}\lambda).
\end{equation*}
By Lemma~\ref{lemma: NL: the cross term lemma} and Lemma~\ref{lemma: NL: sara lemma 4},
\[\max_{1\leq j\leq p+q}\abs{e^T_j( \widehat H_n(\hx,\hy)-H^0)\Delta_{\Gamma}(j)}=O_p((s^{\kappa+1/2}+s^{2\kappa})\lambda^2)\stackrel{(a)}{=}O_p(s^{2\kappa}\lambda^2)\]
where (a) follows because $\kappa\geq 1/2$.
Since $s^{\kappa}\lambda\leq s\lambda=o(1)$ by  Fact~\ref{fact: slambda goes to zero}, we have
\[\max_{1\leq j\leq p+q}\abs{e^T_j( \widehat H_n(\hx,\hy)-H^0)\Delta_{\Gamma}(j)}=o_p(s^\kappa\lambda).\]
On the other hand, there exist positive constants $C$ and $C'$ so that
\[\max_{1\leq j\leq p+q}\abs{e^T_jH^0(\widehat{\Gamma}_j-\Gamma^0_j)}\stackrel{(a)}{\leq} \max_{1\leq j\leq p+q} C\|\widehat\Gamma_j-\Gamma^0_j\|_2\leq C' \max_{1\leq j\leq p+q}\|\heta-\etak\|_2\stackrel{(b)}{=}O_p(s^{\kappa}\lambda) \]
where (a) and (b) follow from Lemma~\ref{lemma: Phi and H knot } and Lemma~\ref{lemma: NL: sara lemma 4}, respectively.
Hence, the proof follows.
\end{proof}

The definition of $\widehat\Gamma_j$ in \eqref{def: widehat gamma j} implies $(\hf)_j=\widehat{\tau}_j^{-2}\widehat\Gamma_j$
and \eqref{eq: inverse columns} implies $\Phi^0_j=(\tau^0_j)^{-2}\Gamma^0_j$. 
Hence, for $i=1,2$, we have
\begin{equation}\label{intheorem: nodwise phi js}
    \|(\hf)_j-\Phi^0_j\|_i=\|\widehat{\tau}_j^{-2}(\widehat\Gamma_j-\Gamma^0_j)\|_i+|\widehat{\tau}_j^{-2}-(\tau_j^0)^{-2}|\|\Gamma^0_j\|_i.
\end{equation}
Recall from Lemma~\ref{lemma: Phi and H knot } 
that $\min_{1\leq j\leq p+q}(\tau_j^0)^{2}$ is bounded below by a positive constant, say $C>0$. Writing
\[\widehat{\tau}_j^{-2}-(\tau_j^0)^{-2}=\frac{|\widehat{\tau}_j^{2}-(\tau^0_j)^2|}{\widehat{\tau}_j^{2}(\tau^0_j)^2}\leq \frac{|\widehat{\tau}_j^{2}-(\tau^0_j)^2|}{(\tau^0_j)^2\slb (\tau^0_j)^2-|\widehat{\tau}_j^{2}-(\tau^0_j)^2|\srb},\]
we thus obtain
\[\max_{1\leq j\leq p+q}|\widehat{\tau}_j^{-2}-(\tau_j^0)^{-2}|\leq \frac{\max_{1\leq j\leq p+q} |\widehat{\tau}_j^{2}-(\tau^0_j)^2|}{C\slb C-\max_{1\leq j\leq p+q} |\widehat{\tau}_j^{2}-(\tau^0_j)^2|\srb},\]
which is $O_p(s^{\kappa}\lambda)$ by 
 Lemma~\ref{lemma: NL: tau hat}.
As a corollary,
\[\max_{1\leq j\leq p+q}|\widehat{\tau}_j^{-2}|\leq \max_{1\leq j\leq p+q}|(\tau^0_j)^{-2}|+O_p(s^{\kappa}\lambda)\leq 2/C.\]
Noting
\begin{equation}\label{inlemma: NL: gamma j gamma knot j}
   \|\widehat\Gamma_j-\Gamma^0_j\|_i=\|\heta-\etak\|_i \quad(i=1,2, j=1,\ldots,p+q),
\end{equation}
\eqref{intheorem: nodwise phi js} can be used to obtain
\[\max_{1\leq j\leq p+q}\|(\hf)_j-\Phi^0_j\|_2\leq 2\|\heta-\etak\|_2/C+O_p(s^{1/2}\lambda)\|\Gamma^0_j\|_2\quad (j=1,\ldots,p+q),\]
which is $O_p(s^{\kappa}\lambda)$ since $|(\hf)_j-\Phi^0_j\|_2=O_p(s^{\kappa}\lambda)$ by  Lemma~\ref{lemma: NL: sara lemma 4} and $\max_{1\leq j\leq p+q}\|\Gamma^0_j\|_2=O_p(1)$ by Lemma~\ref{lemma: NL: claim 2 criteria lemma}.

For the $l_1$-error, \eqref{intheorem: nodwise phi js} and \eqref{inlemma: NL: gamma j gamma knot j} yield
\[\max_{1\leq j\leq p+q}\|(\hf)_j-\Phi^0_j\|_1\leq 2\|\heta-\etak\|_1/C+O_p(s^{\kappa}\lambda)\|\Gamma^0_j\|_1,\]
which is $O_p(s^{\kappa+1/2}\lambda)$ since $\|\heta-\etak\|_1=O_p(s^{\kappa+1/2}\lambda)$ by  Lemma~\ref{lemma: NL: sara lemma 4}, and
\[\max_{1\leq j\leq p+q}\|\Gamma^0_j\|_1\stackrel{(a)}{\leq} O(s^{1/2}) \max_{1\leq j\leq p+q}\|\Gamma^0_j\|_2\stackrel{(b)}{=}O(s^{1/2})\]
where (a) follows because 
\[\max_{1\leq j\leq p+q}\|\Gamma^0_j\|_0=\max_{1\leq j\leq p+q}\|\etak\|_0+1,\]
which is $O(s^{1/2})$
by Assumption~\ref{assump:Phi 0}, and (b) follows by Lemma~\ref{lemma: NL: claim 2 criteria lemma}.
Thus the proof of Theorem~\ref{thm: nodewise Lasso theorem} follows.

\subsection{Proof of the  key lemmas for Theorem~\ref{thm: nodewise Lasso theorem}}
\label{sec: nodewise lasso: proof: key lemmas}

\begin{proof}[of Lemma~\ref{lemma: NL: the cross term lemma}]
 Our first step is to find an expression for $\Delta_{\Gamma}(j)^T(\widehat{H}-H^0)z$ for a general $z\in\RR^{p\times q}$, and then use this expression to find the rates for the special cases when $z=\Delta_{\Gamma}(j)$, $\Gamma^0_j$, or $e_j$. Now let us introduce some new notations. 
Let $\Delta_{\Gamma}(j)=(\Delta_{\Gamma,1}(j),\Delta_{\Gamma,2}(j))$ and $z=(z_1, z_2)$ where $\Delta_{\Gamma,1}(j),z_1\in\RR^p$ and $\Delta_{\Gamma,2}(j),z_2\in\RR^q$. Also for the sake of simplicity, we denote $\widehat H_n=\widehat H(\hx,\hy)$.
 For $A=\widehat H_n$ and $H^0$, let us partition $A$ into
\[A=\begin{bmatrix}
A_{11} & A_{12}\\
A_{21} & A_{22}
\end{bmatrix},\quad A_{11}\in\RR^{p\times p}, A_{12}\in\RR^{p\times q}, A_{21}\in\RR^{q\times p}, A_{22}\in\RR^{q\times q}.\]

Using these new notations, we write
\begin{align*}
   \Delta_{\Gamma}(j)^T(\widehat{H}-H^0)z= &\ \Delta_{\Gamma,1}(j)^T(\widehat{H}_{11}-H^0_{11})z_1+\Delta_{\Gamma,1}(j)^T(\widehat{H}_{12}-H^0_{12})z_2\\
   &\ +\Delta(j)_{2}^T(\widehat{H}_{21}-H^0_{21})z_1+\Delta(j)_{2}^T(\widehat{H}_{22}-H^0_{22})z_2.
\end{align*}

Now observe $\Delta_{\Gamma,1}(j)^T(\widehat{H}_{11}-H^0_{11})z_1$ can be further decomposed into
\begin{align*}
  \MoveEqLeft \Delta_{\Gamma,1}(j)^T(\widehat{H}_{11}-H^0_{11})z_1\\
  =&\ 2\Delta(j)_{1}^T\slb (\hx^T\hSx\hx)\hSx +2\hSx\hx\hx^T\hSx-\rhk\Sx-2\Sx x^0(x^0)^T\Sx\srb z_1\\
    =&\ 2\underbrace{\slb \hx^T\hSx\hx-\rhk\srb \Delta_{\Gamma,1}(j)^T(\hSx-\Sx) z_1}_{T_1(z;j)}+2\underbrace{\slb \hx^T\hSx\hx-\rhk\srb \Delta_{\Gamma,1}(j)^T\Sx z_1}_{T_2(z;j)}\\
    &\ +2\underbrace{\rhk\Delta_{\Gamma,1}(j)^T(\hSx-\Sx)z_1}_{T_3(z;j)}
     +4\underbrace{\Delta_{\Gamma,1}(j)^T(\hSx-\Sx)\hx\hx^T(\hSx-\Sx) z_1}_{T_4(z;j)}\\
     &\ +4\underbrace{\Delta_{\Gamma,1}(j)^T(\hSx-\Sx)\hx\hx^T\Sx z_1}_{T_5(z;j)}+4\underbrace{\Delta_{\Gamma,1}(j)^T\Sx\hx\hx^T(\hSx-\Sx) z_1}_{T_6(z;j)}\\
    &\ + 4\underbrace{\Delta_{\Gamma,1}(j)^T\Sx(\hx\hx^T-x^0(x^0)^T)\Sx z_1}_{T_7(z;j)}.
\end{align*}
Also,
\begin{align*}
 \Delta_{\Gamma,1}(j)^T(\widehat{H}_{12}-H^0_{12})z_2=&\    2\underbrace{\Delta(j)_{1}^T(\Sxy-\hSxy)z_2}_{T_8(z;j)}.
\end{align*}
To find the rate of $ \Delta_{\Gamma}(j)^T(\widehat{H}-H^0)z$ for any $z$, it suffices to look at the rate of $\Delta(j)_{1}^T(\widehat{H}_{11}-H^0_{11})z_1$ and $\Delta(j)_{1}^T(\widehat{H}_{12}-H^0_{12})z_2$ only because the calculations for the other two terms will be similar. Hence, it is sufficient to find the rate of $\sum_{i=1}^8T_i(z,j)$ when $z=\Delta(j)$, $e_j$, and $\Gamma^0_j$.

First, let us consider the case when $z=\Delta_{\Gamma}(j)$. Claim~\ref{claim: T of delta} and the above decomposition implies 
\[ \Delta_{\Gamma}(j)^T(\widehat{H}-H^0) \Delta_{\Gamma}(j)=O_p(\lambda)\|\Delta(j)\|_1+O_p(s^{1/2}\lambda)\|\Delta(j)\|_2+O_p(s^{1/2}\lambda)\|\Delta(j)\|_2^2\]
uniformly across $j=1,\ldots,p+q$.
\begin{claim}\label{claim: T of delta}
Under the set up of Theorem~\ref{thm: nodewise Lasso theorem}, there exists $C>0$ so that
\begin{gather*}
 \sum_{i=1}^8 |T_i(\Delta_{\Gamma}(j);j)|\leq C\slb s^{\kappa}\lambda\|\Delta(j)\|_2^2+\lambda\|\Delta(j)\|_1+s^{1/2}\lambda\|\Delta(j)\|_2\srb \quad(j=1,\ldots,p+q)
\end{gather*}
with high probability for $p$, $q$, and $n$.
\end{claim}
The proof of Claim~\ref{claim: T of delta} can be found in Supplement\ref{sec: lemma: NL}.
Claim~\ref{claim: NL: T of z} handles the case when $z=e_j$ or $\Gamma_j^0$. The proof of Claim~\ref{claim: NL: T of z} can be found in Supplement\ref{sec: lemma: NL}.

\begin{claim}\label{claim: NL: T of z}
Suppose $z$ is a fixed vector in $\RR^{p+q}$ such that   $\|z\|_{0}\leq C_1 s$ for some  $C_1>0$. 
Then we can find $C>0$, depending on $C_1$, but not depending on the particular $z$,  so that
 \begin{align*}
   \sum_{i=1}^8|T_i(z;j)|\leq C\|z\|_2(\lambda\|\Delta(j)\|_1+s^{\kappa}\lambda\|\Delta(j)\|_2)\quad (j=1,\ldots, p+q)
\end{align*}
with high probability for sufficiently large $n$.
\end{claim}

That $e_j$ satisfies the criteria of Claim~\ref{claim: NL: T of z} is immediate because $\|e_j\|_0=1$. Condition~\ref{assumption: precision matrix} implies $\Gamma_j^0$'s also satisfy the criteria of Claim~\ref{claim: NL: T of z}.  Lemma~\ref{lemma: NL: claim 2 criteria lemma}, on the other hand, implies that  $\|\Gamma_j^0\|_2$'s are uniformly bounded over $j$'s.
 Lemma~\ref{lemma: NL: claim 2 criteria lemma} and Claim~\ref{claim: NL: T of z}, therefore, establish that there exists an $C>0$ so that
\begin{align}\label{inlemma: gamma j and ej}
\sum_{i=1}^8(|T_i(\Gamma^0_j;j)|+|T_i(e_j;j)|)\leq C(\lambda\|\Delta(j)\|_1+s^{\kappa}\lambda\|\Delta(j)\|_2)\quad(j=1,\ldots,p+q)
\end{align}
with high probability for all sufficiently large $p$, $q$, and $n$. The proof follows combining the above result with Claim~\ref{claim: T of delta} because $\kappa\geq 1/2$,.
\end{proof}

\begin{proof}[of Lemma~\ref{lemma: NL: tau: cross term lemma} ]
Without loss of generality, we assume $1\leq j\leq p$.   The proof follows in identical way if $p+1\leq j\leq p+q$. Let us denote $z=\Gamma^0_j$. We will use the notations $z_1$, $z_2$, $\widehat H_{11}$, $\widehat{H}_{12}$, $H^0_{11}$, and $H^0_{12}$ developed in the proof of Lemma~\ref{lemma: NL: the cross term lemma} for partitioning the matrices and the vectors. Denote by $\tilde e_j$ the first $p$ elements of $e_j$. 
Proceeding in the same way as in Lemma~\ref{lemma: NL: the cross term lemma}, we see that
\begin{align*}
   \MoveEqLeft \tilde e_j^T(\widehat H_n(\hx,\hy)-H^0)z\\
   =&\ \tilde e_j^T(\widehat{H}_{11}-H^0_{11})z_1+\tilde e_j^T(\widehat{H}_{12}-H^0_{12})z_2\\
   =&\ 2\tilde e^T_j\slb (\hx^T\hSx\hx)\hSx +2\hSx\hx\hx^T\hSx-\rhk\Sx-2\Sx x^0(x^0)^T\Sx\srb z_1\\
   &\ + \tilde e_j^T(\widehat{H}_{12}-H^0_{12})z_2\\
    =&\ 2\underbrace{\slb \hx^T\hSx\hx-\rhk\srb \tilde e^T_j(\hSx-\Sx) z_1}_{\mathcal{T}_1(z;j)}+2\underbrace{\slb \hx^T\hSx\hx-\rhk\srb\tilde e^T_j\Sx z_1}_{\mathcal T_2(z;j)}\\
    &\ +2\underbrace{\rhk\tilde e^T_j(\hSx-\Sx)z_1}_{\mathcal T_3(z;j)}
     +4\underbrace{\tilde e^T_j(\hSx-\Sx)\hx\hx^T(\hSx-\Sx) z_1}_{\mathcal T_4(z;j)}\\
     &\ +4\underbrace{\tilde e^T_j(\hSx-\Sx)\hx\hx^T\Sx z_1}_{\mathcal T_5(z;j)}+4\underbrace{\tilde e^T_j\Sx\hx\hx^T(\hSx-\Sx) z_1}_{\mathcal T_6(z;j)}\\
    &\ + 4\underbrace{\tilde e^T_j\Sx(\hx\hx^T-x^0(x^0)^T)\Sx z_1}_{\mathcal T_7(z;j)}+2\underbrace{\tilde e_j^T(\hSxy-\Sxy)z_2}_{\mathcal T_8(z;j)}.
\end{align*}
Since $z=\Gamma^0_j$ satisfies the conditions of Claim~\ref{claim: NL: T of z}, we can use results derived in the proof of Claim~\ref{claim: NL: T of z} for our $z$. In particular, we will use the bounds in\eqref{inclaim: z: fact 3} and \eqref{inclaim: z: fact 4} Also, we will  develop below some new inequalities to bound the $\mathcal T(z;j)$'s.

Note that by Lemma~\ref{lemma: additional: l1 norm of x and l2 norm of z knot}, for  large $C$, \begin{equation}\label{inlemma: tau cross lemma: fact 1} 
  \max_{1\leq j\leq p+q} \abs{ \tilde e_j^T(\hSx-\Sx)z_1}\leq C\|z_1\|_2\|\tilde e_j\|_1\lambda\leq C\|z\|_2\lambda
\end{equation}
with high probability as $n,p\to\infty$.
On the other hand, Assumption~\ref{assump: bounded eigenvalue} implies 
\begin{equation}\label{inlemma: tau cross lemma: fact 2} 
   \max_{1\leq j\leq p+q} \abs{ \tilde e_j^T\Sx z_1}\leq C\|z\|_2. 
\end{equation}
Since $\min_{w\in\{\pm 1\}}\|w\hx-x^0\|_1=o_p(1)$ byLemma~\ref{corollary: rate of x and y} and $\|x^0\|_2=O(1)$, Lemma~\ref{lemma: additional: l1 norm of x and l2 norm of z knot} implies 
\begin{equation}\label{inlemma: tau cross lemma: fact 3} 
  \max_{1\leq j\leq p+q} \abs{ \tilde e_j^T(\hSx-\Sx)\hx}= \max_{1\leq j\leq p+q}\min_{w\in\{\pm 1\}} \abs{ \tilde e_j^T(\hSx-\Sx)w\hx}\leq \max_{1\leq j\leq p+q}  C\|\tilde e_j\|_1\lambda= C\lambda
\end{equation}
with high probability for sufficiently large $n$,  $p$, and $q$.

Now \eqref{inlemma: T delta: fact 1} and \eqref{inlemma: tau cross lemma: fact 1}  imply
\[ \max_{1\leq j\leq p+q}|\mathcal T_1(z;j)|= \max_{1\leq j\leq p+q}\abs{(\hx^T\hSx\hx-\rhk) \tilde e^T_j(\hSx-\Sx) z_1}\leq C s^\kappa\lambda^2\|z\|_2,\]
where combining \eqref{inlemma: T delta: fact 1} with \eqref{inlemma: tau cross lemma: fact 2}  yields
\[ \max_{1\leq j\leq p+q}|\mathcal T_2(z;j)|\leq C s^{\kappa}\lambda\|z\|_2, \]
and
and \eqref{inlemma: tau cross lemma: fact 1} leads to
\[ \max_{1\leq j\leq p+q}|\mathcal T_3(z;j)|\leq C \lambda\|z\|_2\]
with high probability for sufficiently large $p$, $q$, and $n$.
 Similar results hold for $\mathcal T_4(z;j)$, $\mathcal T_5(z;j)$  noting
\[ \max_{1\leq j\leq p+q}|\mathcal T_4(z;j)|\leq Cs^{1/2}\lambda^2 \|z\|_2\]
with high probability 
by \eqref{inlemma: tau cross lemma: fact 3}  and
\eqref{inclaim: z: fact 3}, and 
\[ \max_{1\leq j\leq p+q}|\mathcal T_5(z;j)|\leq Cs^{1/2}\lambda \|z\|_2\]
with high probability
by \eqref{inlemma: tau cross lemma: fact 3}  and
\eqref{inclaim: z: fact 4}. Equation \ref{inlemma: tau cross lemma: fact 3} and \eqref{inclaim: z: fact 3} imply
\[ \max_{1\leq j\leq p+q}|\mathcal T_6(z;j)|\leq Cs^{1/2}\lambda \|z\|_2\]
with high probability.
Finally, $C$ can be chosen so large such that for sufficiently large $p$, $q$, and $n$,
\[ \max_{1\leq j\leq p+q}|\mathcal T_7(z;j)|\leq C\|z\|_2\|\hx\hx^T-(x^0)(x^0)^T\|_F\stackrel{(a)}{\leq} Cs^{\kappa}\lambda\|z\|_2\]
with high probability, where (a) follows   by \eqref{inlemma: projection matrix}.
Lemma~\ref{lemma: additional: l1 norm of x and l2 norm of z knot} implies
\[\max_{1\leq j\leq p+q}|\mathcal T_8(z;j)|\leq C\|e_j\|_1\lambda\|z\|_2= C\lambda\|z\|_2.\]
Because
\[\max(\lambda,s^{1/2}\lambda,s^{\kappa}\lambda, s^{1/2}\lambda^2, s^{\kappa}\lambda^2)=s^{\kappa}\lambda,\]
we have
\[\sum_{i=1}^8|\mathcal T_i(z;j)|\leq s^{\kappa}\lambda\|z\|_2\quad(j=1,\ldots,p+q).\]
When $z=\Gamma^0_j$, the above leads to
\[\sum_{i=1}^8|\mathcal T_i(\Gamma_j^0,j)|\leq s^{\kappa}\lambda\|\Gamma_j^0\|_2\quad(j=1,\ldots,p+q),\]
implying
\[|e_j^T(\widehat H_n(\hx,\hy)-H^0)\Gamma^0_j|\leq s^{\kappa}\lambda\|\Gamma_j^0\|_2\quad(j=1,\ldots,p+q).\]
Hence, the result follows by Lemma~\ref{lemma: NL: claim 2 criteria lemma}.
\end{proof}

 \section{Proof of Technical Lemmas}

 \subsection{Proof of  technical lemmas for Theorem~\ref{corrolary: main: rho}}
  \label{sec: proof of rho lemmas}
 
 Our next lemma, which gives the form of $\Phi^0$, is required for obtaining the form of $\sigma_\rho^2$.
\begin{lemma}\label{lemma: NL: form of Phi knot}
Suppose $\Phi^0=(H^0)^{-1}$, where $H^0$ is as defined in \eqref{def: H knot}. Then
\[\Phi^0=(2\rhk)^{-1}\begin{bmatrix}
     UO_4U^T+\Sx^{-1} & UO_3V^T\\
     VO_3U^T & VO_4V^T+\Sy^{-1}
  \end{bmatrix},\]
where
    \[O_3=\text{Diag}(1/8, \rhk\Lambda_2/(\rhk^2-\Lambda_2^2),\ldots,\rhk\Lambda_r/(\rhk^2-\Lambda_r^2))\in\RR^{r\times r},\]
    and
    \[O_4=O_1-I_r=\text{Diag}\slb -5/8,\Lambda_2^2/(\rhk^2-\Lambda^2_2),\ldots,\Lambda_r^2/(\rhk^2-\Lambda_r^2)\srb.\]
 \end{lemma}  
  
  \begin{proof}[of Lemma~\ref{lemma: NL: form of Phi knot}]
  Let us denote $\tu_i=\Sx^{1/2}u_i$ $(i=1,\ldots,r)$ and 
and $\tv_i=\Sy^{1/2}v_i$  $(i=1,\ldots, r)$. Letting
 \[D=\text{Diag}(\Sx^{1/2},\Sy^{1/2}),\quad A=\begin{bmatrix}
      I_p+2\tu_1\tu_1^T & -\Sx^{1/2}U\Lambda V^T\Sy^{1/2}/\rhk\\
      -\Sy^{1/2}V\Lambda U^T\Sx^{1/2}/\rhk & I_q+2\tv_1\tv_1^T
      \end{bmatrix},\]
      and using \eqref{def: Hessian general}, we obtain that $H^0=2\rhk DAD$.
   If $A$ is invertible, then $\Phi^0=(2\rhk)^{-1}D^{-1}A^{-1}D^{-1}$. We will now show that $A$ is invertible, and find its inverse.
   
   To that end, first we introduce some notations.  Since the columns of $\tU$ are orthogonal, we can extend $\tU=[\tu_1,\ldots,\tu_r]$ to $\tU_*=[\tu_1,\ldots,\tu_p]$ so that $\tU_*\tU_*^T=\tU_*^T\tU_*=I_p$. Similarly, we can extend $\tV$ to a basis $\tV_*=[\tv_1,\ldots,\tv_q]$. Now if we let
   \[\Lambda_*=\begin{bmatrix}
   \Lambda_{r\times r} & 0_{r\times(q-r) }\\ 
   0_{(p-r)\times r } & 0_{(p-r)\times(q-r) }
   \end{bmatrix},\]
   then it follows that
   \[\tU\Lambda \tV^T=\tU_*\Lambda_*\tV^T_*,\quad I_p=\tU_*\tU_*^T,\quad I_q=\tV_*\tV_*^T.\]
   Let us denote
   \[F=\text{Diag}(3,\underbrace{1,\ldots, 1}_{p-1\text{ times }})\quad\text{ and }\quad G=\text{Diag}(3,\underbrace{1,\ldots, 1}_{q-1\text{ times }}).\]
   Note that
   \[  \tU_*F\tU^T_*=I_p+2\tu_1\tu_1^T,\quad \tU_*G\tU^T_*=I_q+2\tv_1\tv_1^T.\]
   Therefore, $A$ can be written as
   \begin{align*}
       A=\begin{bmatrix}
       \tU_*F\tU^T_* & -\tU_*(\Lambda_*/\rhk)\tV^T_*\\
       -\tV_*(\Lambda^T_*/\rhk)\tU^T_* & \tV_*G\tV^T_*
       \end{bmatrix}.
   \end{align*}
   Further simplification of $A$ is possible. To that end,
we define
\begin{equation}\label{def: inlemma: variance: D2}
   D_2=\text{Diag}(\tU_*,\tV_*)\quad\text{and}\quad J=\begin{bmatrix}
F & -\Lambda_*/\rhk\\
-\Lambda^T_*/\rhk & G
\end{bmatrix}. 
\end{equation}
It is easy to see that $A=D_2JD_2^T$. Because $F$, $G$, and the corresponding Schur components $F-\Lambda_* G^{-1} \Lambda^T_*/\rhk^2$ and $G-\Lambda_*^T F^{-1} \Lambda_*/\rhk^2$ are diagonal, they are invertible. Therefore, $J$ is also invertible, which implies
$A^{-1} =(D_2^T)^{-1}J^{-1}D_2^{-1}$ where
$D_2^{-1}=\text{Diag}(\tU^T_*,\tV^T_*)=D_2^T$.
Thus $A^{-1} =D_2J^{-1}D_2^{T}$.
To find $J^{-1}$, we first write it in a block matrix form:
\[J^{-1}=\begin{bmatrix}
J^{11} & J^{12}\\
(J^{12})^T & J^{22}
\end{bmatrix}.\]
Here we used the fact that $J^{-1}$ is symmetric which follows since $J$ is symmetric. Now using the formula for block matrix inversion, we obtain that
\begin{gather*}
    J^{11}=(F-\Lambda_* G^{-1} \Lambda^T_*/\rhk^2)^{-1},\\
    J^{12}=J^{11}\Lambda_* G^{-1}/\rhk,\\
    J^{22}=(G-\Lambda_*^T F^{-1} \Lambda_*/\rhk^2)^{-1}.
\end{gather*}
Now we compute that
\begin{align*}
    (J^{11})^{-1}=&\ \text{Diag}(
    3-\rhk^2/(3\rhk^2), 1-\Lambda_2^2/\rhk^2,\ldots, 1-\Lambda_r^2/\rhk^2,\underbrace{ 1,\ldots,1}_{p-r\text{ times}}
    )\\
    =&\ \text{Diag}(
    8/3, 1-\Lambda_2^2/\rhk^2,\ldots, 1-\Lambda_r^2/\rhk^2,\underbrace{ 1,\ldots,1}_{p-r\text{ times}}
    )
\end{align*}
Therefore,
\[J^{11}=\text{Diag}(3/8,\rhk^2/(\rhk^2-\Lambda_2^2),\ldots,\rhk^2/(\rhk^2-\Lambda_r^2),\underbrace{ 1,\ldots,1}_{p-r\text{ times}}).\]
Letting
\[O_1=\text{Diag}(3/8,\rhk^2/(\rhk^2-\Lambda_2^2),\ldots,\rhk^2/(\rhk^2-\Lambda_r^2))\in \RR^{r\times r},\quad\text{we obtain}\quad J^{11}=
\begin{bmatrix}
   O_1 & 0\\
    0 & I_{p-r}
    \end{bmatrix}.\]
Similarly, observe that
\[\Lambda_*G^{-1}/\rhk=\begin{bmatrix}
   O_2 & 0\\
   0 & 0
\end{bmatrix}, \quad\text{where}\quad O_2=\text{Diag}(1/3,\Lambda_2/\rhk,\ldots,\Lambda_r/\rhk)\in\RR^{r\times r}.\]
Thus
\begin{align*}
    J^{12}= \begin{bmatrix}
   O_1 & 0\\
    0 & I_{p-r}
    \end{bmatrix}\begin{bmatrix}
   O_2 & 0\\
    0 & 0
    \end{bmatrix}= \begin{bmatrix}
  O_3 & 0\\
   0 & 0
    \end{bmatrix},
\end{align*}
where
\[ O_3=O_1O_2=\text{Diag}\slb 1/8, \rhk\Lambda_2/(\rhk^2-\Lambda_2^2),\ldots,\rhk\Lambda_r/(\rhk^2-\Lambda_r^2)\srb.\] 

By symmetry,
\[J^{-1}=\begin{bmatrix}
   O_1 & 0 & O_3 & 0\\
   0 & I_{p-r} & 0 & 0\\
   O_3 & 0 & O_1 & 0\\
   0 & 0 & 0 & I_{q-r}
\end{bmatrix}.\]
Hence using \eqref{def: inlemma: variance: D2}, we obtain that
\begin{align*}
    \Phi^0=&\ (2\rhk)^{-1}D^{-1}D_2J^{-1}D_2^{T}D^{-1}\\
    = &\ (2\rhk)^{-1}\begin{bmatrix}
       \Sx^{-1/2}U_*&0\\
       0 & \Sy^{-1/2} V_*
    \end{bmatrix}\begin{bmatrix}
   O_1 & 0 & O_3 & 0\\
   0 & I_{p-r} & 0 & 0\\
   O_3 & 0 & O_1 & 0\\
   0 & 0 & 0 & I_{q-r}
\end{bmatrix}\begin{bmatrix}
       \tU_*^T\Sx^{-1/2} &0\\
       0 & \tV_*^T\Sy^{-1/2} 
    \end{bmatrix}.
\end{align*}
If we denote $\tU_{-r}=[\tu_{r+1},\ldots,\tu_p]$, and $\tV_{-r}=[\tv_{r+1},\ldots,\tv_q]$, then it follows that
\begin{align*}
    \Phi^0=&\ (2\rhk)^{-1}\begin{bmatrix}
       \Sx^{-1/2}\tU & \Sx^{-1/2}\tU_{-r} & 0 & 0\\
       0 & 0 & \Sy^{-1/2}\tV & \Sy^{-1/2}\tV_{-r}
    \end{bmatrix}\begin{bmatrix}
   O_1 & 0 & O_3 & 0\\
   0 & I_{p-r} & 0 & 0\\
   O_3 & 0 & O_1 & 0\\
   0 & 0 & 0 & I_{q-r}
\end{bmatrix}\begin{bmatrix}
       \tU^T\Sx^{-1/2} &0\\
         \tU_{-r}^T\Sx^{-1/2} &0\\
       0 & \tV^T\Sy^{-1/2} \\
        0 & \tV_{-r}^T\Sy^{-1/2} 
    \end{bmatrix}\\
    =&\ (2\rhk)^{-1}\begin{bmatrix}
       \Sx^{-1/2}\tU O_1 & \Sx^{-1/2}\tU_{-r} & \Sx^{-1/2}\tU O_3 & 0\\
       \Sy^{-1/2}\tV O_3 & 0 & \Sy^{-1/2}\tV O_1 & \Sy^{-1/2}\tV_{-r}
    \end{bmatrix}\begin{bmatrix}
       \tU^T\Sx^{-1/2} &0\\
         \tU_{-r}^T\Sx^{-1/2} &0\\
       0 & \tV^T\Sy^{-1/2} \\
        0 & \tV_{-r}^T\Sy^{-1/2} 
    \end{bmatrix}\\
    =&\ (2\rhk)^{-1} \begin{bmatrix}
        \Sx^{-1/2}\tU O_1\tU^T\Sx^{-1/2}+\Sx^{-1/2}\tU_{-r}\tU_{-r}^T\Sx^{-1/2} & \Sx^{-1/2}\tU O_3\tV^T\Sy^{-1/2}\\
        \Sy^{-1/2}\tV O_3\tU^T\Sx^{-1/2} &  \Sy^{-1/2}\tV O_1\tV^T\Sy^{-1/2}+\Sy^{-1/2}\tV_{-r}\tV^T_{-r}\Sy^{-1/2}
    \end{bmatrix}\\
   \stackrel{(a)}{=}&\ (2\rhk)^{-1}D^{-1}\begin{bmatrix}
       \tU O_1\tU^T+I_p-\tU\tU^T & \tU O_3\tV^T\\
       \tV O_3\tU^T & \tV O_1\tV^T+I_q-\tV\tV^T
    \end{bmatrix}D^{-1}\\
    \stackrel{(b)}{=}&\ (2\rhk)^{-1}D^{-1}\begin{bmatrix}
       \tU O_4\tU^T+I_p & \tU O_3\tV^T\\
       \tV O_3\tU^T & \tV O_4\tV^T+I_q
    \end{bmatrix}D^{-1}
\end{align*}
where (a) follows because $\tU\tU^T=I_p-\tU_{-r}\tU_{-r}^T$,  $\tV\tV^T=I_q-\tV_{-r}\tV_{-r}^T$, and in (b), we used the notation
\[O_4=O_1-I_r=\text{Diag}\slb -5/8,\Lambda_2^2/(\rhk^2-\Lambda^2_2),\ldots,\Lambda_r^2/(\rhk^2-\Lambda_r^2)\srb.\]
Since $\Sx^{-1/2}\tU=U$ and $\Sy^{-1/2}\tV=V$, we have
  \[\Phi^0=(2\rhk)^{-1}\begin{bmatrix}
     UO_4U^T+\Sx^{-1} & UO_3V^T\\
     VO_3U^T & VO_4V^T+\Sy^{-1}
  \end{bmatrix}\]
  \end{proof}
  \begin{proof}[of Lemma~\ref{lemma: rho: variance of rhohat}]
  First we will find the expression of $\mathcal L_1$. To that end, using Lemma~\ref{lemma: NL: form of Phi knot}, we calculate that
  \begin{align*}
\MoveEqLeft 2\Phi^0\begin{bmatrix}
   \rhk(\hSx-\Sx)x^0-(\hSxy-\Sxy)y^0+\slbs(x^0)^T(\hSx-\Sx)x^0\srbs \Sx x^0\\
    \rhk(\hSy-\Sy)y^0-(\hSyx-\Syx)x^0+\slbs(y^0)^T(\hSy-\Sy)y^0\srbs \Sy y^0\\
    \end{bmatrix}\\
    =&\ 2(2\rhk)^{-1}\begin{bmatrix}
       UO_4U^T+\Sx^{-1} & UO_3V^T\\
       VO_3U^T & VO_4V^T+\Sy^{-1}
    \end{bmatrix}\begin{bmatrix}
   \rhk(\hSx-\Sx)x^0-(\hSxy-\Sxy)y^0\\
   +\slbs(x^0)^T(\hSx-\Sx)x^0\srbs \Sx x^0\\
   \\
    \rhk(\hSy-\Sy)y^0-(\hSyx-\Syx)x^0\\
    +\slbs(y^0)^T(\hSy-\Sy)y^0\srbs \Sy y^0\\
    \end{bmatrix}.
  \end{align*}
  Thus
  \begin{align*}
      \mathcal L_1=&\ UO_4U^T(\hSx-\Sx)x^0+\Sx^{-1}(\hSx-\Sx)x^0-\rhk^{-1}UO_4U^T(\hSxy-\Sxy)y^0\\
      &\ -\rhk^{-1}\Sx^{-1}(\hSxy-\Sxy)y^0+\rhk^{-1}\slbs(x^0)^T(\hSx-\Sx)x^0\srbs UO_4U^T\Sx x^0\\
      &\ +\rhk^{-1}\slbs(x^0)^T(\hSx-\Sx)x^0\srbs x^0+ UO_3V^T(\hSy-\Sy)y^0\\
      &\ -\rhk^{-1}UO_3V^T(\hSyx-\Syx)x^0+\rhk^{-1}\slbs(y^0)^T(\hSy-\Sy)y^0\srbs UO_3V^T\Sy y^0.
  \end{align*}
 Noting $\Sxy y^0=\rhk\Sx x^0$, we deduce
  \begin{align*}
    \mathcal L_1^T\Sxy y^0
      =&\ \underbrace{\rhk(x^0)^T(\hSx-\Sx)UO_4U^T\Sx x^0}_{T_1}+\underbrace{\rhk(x^0)^T(\hSx-\Sx) x^0}_{T_2}\\
    &\  -\underbrace{(y^0)^T(\hSyx-\Syx)UO_4U^T\Sx x^0}_{T_3}-\underbrace{(y^0)^T(\hSyx-\Syx) x^0}_{T_4}\\
    &\ +\underbrace{\slbs(x^0)^T(\hSx-\Sx)x^0\srbs(x^0)^T\Sx UO_4U^T \Sx x^0}_{T_5}+ \underbrace{\slbs(x^0)^T(\hSx-\Sx)x^0\srbs (x^0)^T\Sx x^0}_{T_6}\\
    &\ +\underbrace{\rhk(y^0)^T(\hSy-\Sy)VO_3U^T\Sx x^0}_{T_7}-\underbrace{(x^0)^T(\hSxy-\Sxy)VO_3U^T\Sx x^0}_{T_8}\\
    &\ +\underbrace{\slbs(y^0)^T(\hSy-\Sy)y^0\srbs (y^0)^T\Sy VO_3U^T\Sx x^0}_{T_9}.
  \end{align*}
  Note that
  \begin{align*}
   T_1=\rhk(x^0)^T(\hSx-\Sx)UO_4U^T\Sx x^0=   \rhk^{3/2}(x^0)^T(\hSx-\Sx)UO_4e_1=
   \rhk (O_4)_{11}(x^0)^T(\hSx-\Sx)x^0.
  \end{align*}
  Similarly, we can show that
   \begin{align*}
      T_3=(O_4)_{11}(y^0)^T(\hSyx-\Syx) x^0,
  \end{align*}
   \begin{align*}
      T_5=\rhk(O_4)_{11}(x^0)^T(\hSx-\Sx)x^0,
  \end{align*}
   \begin{align*}
      T_6=\rhk (x^0)^T(\hSx-\Sx)x^0,
  \end{align*}
   \begin{align*}
      T_7=\rhk (O_3)_{11}(y^0)^T(\hSy-\Sy)y^0,
  \end{align*}
   \begin{align*}
      T_8=(O_3)_{11} (x^0)^T(\hSxy-\Sxy)y^0,
  \end{align*}
   \begin{align*}
      T_9=\rhk (O_3)_{11}(y^0)^T(\hSy-\Sy)y^0.
  \end{align*}
  Therefore, 
  \begin{align*}
      \mathcal L_1^T\Sxy y^0= &\ 2\rhk\{1+ (O_4)_{11}\} (x^0)^T(\hSx-\Sx)x^0+ 2\rhk(O_3)_{11}(y^0)^T(\hSy-\Sy)y^0\\
      &\ -\{(O_4)_{11}+(O_3)_{11}+1\}(x^0)^T(\hSxy-\Sxy) y^0.
  \end{align*}
By symmetry,
  \begin{align*}
      \mathcal L_2^T\Syx x^0= &\ 2\rhk\{1+ (O_4)_{11}\} (y^0)^T(\hSy-\Sy)y^0+ 2\rhk(O_3)_{11}(x^0)^T(\hSx-\Sx)x^0\\
      &\ -\{(O_4)_{11}+(O_3)_{11}+1\}(x^0)^T(\hSxy-\Sxy) y^0.
  \end{align*}
  Noting
  \begin{gather*}
      2(1+(O_4)_{11}+(O_3)_{11})=2(1-5/8+1/8)=1,
  \end{gather*}
  we obtain
  \begin{align*}
  \MoveEqLeft  -\mathcal L_1^T\Sxy y^0+     \mathcal -L_2^T\Syx x^0+  (x^0)^T(\hSxy-\hSxy)y^0\\
    =&\ -\rhk\{(x^0)^T(\hSx-\Sx)x^0-\rhk(y^0)^T(\hSy-\Sy)y^0\}+2 (x^0)^T(\hSxy-\hSxy)y^0\\
    =&\ \rhk\sum_{i=1}^n\frac{(Z_i-E[Z_i])}{n}
  \end{align*}
  \[\text{where}\quad Z_i=-\rhk(X_i^Tx^0)^2-\rhk(Y_i^Ty^0)^2+2(X_i^Tx^0)(Y_i^Ty^0).\]
  \end{proof}

\begin{proof}[of Lemma~\ref{lemma: main: sign lemma}]
Suppose $w_1=1$ but $w_2=-1$. Fix $\e>0$.    Lemma~\ref{corollary: rate of x and y} implies that if $n$ is large enough, then 
\[ \|\hx-x^0\|_1+\|\hy+y^0\|_1<\e,\quad \|\hx-x^0\|_2+\|\hy+y^0\|_2<\e  \]
with high probability. 
Proceeding as in the proof of Lemma~\ref{corollary: rate of rho}, we can then show that if $n$ is sufficiently large, then
\begin{align*}
    |\hx^T\hSxy\hy +x^0 \Sxy y^0|<C\e,
\end{align*}
where $C$ is an absolute constant. Therefore, 
$\hx^T\hSxy\hy<-\rhk^2+C\e$. Taking  $\e=\rhk^2/(2C)$, we can therefore show that
\[\limsup_n P(w_1=1,w_2=-1)\leq \limsup_n P(\hx^T\hSxy\hy<-\rhk^2/2) .\]
 However, $\hx^T\hSxy\hy>0$ for all $n$.  Therefore,  $P(w_1=1,w_2=-1)\to 0$. Similarly we can show that  $P(w_1=-1,w_2=1)\to 0$, and the proof follows.
\end{proof}
  
   \begin{proof}[of Lemma~\ref{fact: main theorem: sign}
]
Let us define $\hx^*=w_1\hx$ and $\hy^*=w_2\hy$. Suppose $(\hdai,\hdbi)$ is the de-biased estimator  constructed using $\hx$ and $\hy$. Since $\hf$ does not depend on the sign of $\hx$ and $\hy$, \eqref{def: hat dh general} and \eqref{def: de-biased estimators} indicate that if $w_1=w_2=w$, the de-biased estimators constructed using $\hx^*$ and $\hy^*$ equal $w\hdai$ and $w\hdbi$, respectively. Therefore,
the estimator $\hro^{2,\text{raw}}$ constructed using $\hx^*$ and $\hy^*$ equals
\begin{align*}
    \MoveEqLeft w_1\hx^T\hSxy w_2\hdbi+(w_1\hdai)^T\hSxy w_2\hy-w_1\hx^T\hSxy w_2\hy\\
    =&\ w^2\slb \hx^T\hSxy \hdbi+(\hdai)^T\hSxy w_2\hy-\hx^T\hSxy w_2\hy\srb\\
    =&\ \hx^T\hSxy \hdbi+(\hdai)^T\hSxy w_2\hy-\hx^T\hSxy w_2\hy,
\end{align*}
which is the $\hro^{2,\text{raw}}$ constructed using $\hx$ and $\hy$.
\end{proof}

\subsection{Proof of  technical lemmas for Theorem~\ref{thm: Chao Thm 4.2}}
\label{sec: proof of lemmas for Chao 4.2}
\begin{proof}[of Lemma~\ref{lemma: colar: cu properties}]
We will first establish that $(\tg)^T\Sx\tg-\rhk^2$ is $o_p(1)$. 
To that end, first we derive the expression of $(\tg)^T\Sx\tg$. Note that
\begin{align*}
 \MoveEqLeft (\tg)^T\Sx\tg
 = (\hbz)^T\Sy V\Lambda U^T\Sx U\Lambda V^T\Sy\hbz=(\hbz)^T\Sy V\Lambda^2 V^T\Sy\hbz
\end{align*}
because $U^T\Sx U=I_r$. Now let us denote
\[w=\argmin_{w'\in\{\pm 1\}}\|w'\hbz-\bk\|_2.\] Now 
\begin{align*}
  (\hbz)^T\Sy V\Lambda^2 V^T\Sy\hbz
 = \sum_{i=1}^r\Lambda_i^2((\hbz)^T\Sy v_i)^2
 = \sum_{i=1}^r\Lambda_i^2(w(\hbz)^T\Sy v_i)^2.
 \end{align*}
We have thus obtained 
 \begin{align}\label{inlemma: colar: tg property }
 \MoveEqLeft\abs{(\tg)^T\Sx\tg-\rhk^2}\nn\\
= &\   \abs{\sum_{i=1}^r\Lambda_i^2(\{\bk+w\hbz-\bk\}^T\Sy v_i)^2-\rhk^2 }\nn\\
=&\ \abs{\sum_{i=1}^r\Lambda_i^2\lbt (\bk^T\Sy v_i)^2 +\{(w\hbz-\bk)^T\Sy v_i\}^2+ 2(w\hbz-\bk)^T\Sy v_i(\bk^T\Sy v_i)\rbt-\rhk^2}.
 \end{align}
Because $v_1=\bk$ and $\Lambda_1=\rhk$, we have $\bk^T\Sy v_i=0$ for $i=2,\ldots, r$, leading to
\[\sum_{i=1}^r\Lambda_i^2 (\bk^T\Sy v_i)^2-\rhk^2=0.\]
Also, Cauchy Schwarz inequality implies that 
\begin{align*}
 \MoveEqLeft\sum_{i=1}^r\Lambda_i^2  \{(w\hbz-\bk)^T\Sy v_i\}^2\\
 =& (w\hbz-\bk)^T \Sy\lb\sum_{i=1}^r \Lambda_i^2v_iv_i^T\rb\Sy (w\hbz-\bk)\\
 =&\ (w\hbz-\bk)^T  \Sy V\Lambda^2 V^T\Sy (w\hbz-\bk)\\
 \leq &\ \|\Sy\|_{op}^2\|V\|_{op}^2\|\|\Lambda\|_{op}^2\|w\hbz-\bk\|_2^2\\
 \leq &\ M^2\rhk^2 \| w\hbz-\bk\|^2_2 
\end{align*}
by Assumption~\ref{assump: bounded eigenvalue}. Since $\| w\hbz-\bk\|^2_2=O_p(s^2\lambda^2)$ by \eqref{intheorem: colar: beta rate}, 
we have
\[\sum_{i=1}^r\Lambda_i^2  \{(w\hbz-\bk)^T\Sy v_i\}^2=O_p(s^2\lambda^2).\]
Finally, because $\bk^T\Sy v_i=0$ for $i\geq 2$,
\[\sum_{i=1}^r\Lambda_i^2\abs{(w\hbz-\bk)^T\Sy v_i(\bk^T\Sy v_i)}=\rhk^2\abs{(w\hbz-\bk)^T\Sy \bk}\leq M^{1/2}\|w\hbz-\bk\|_2,\]
where the last step follows from Cauchy Schwarz inequality, Assumption~\ref{assump: bounded eigenvalue} and the fact that $\bk^T\Sy\bk=1$. The right hand side of the above display is $o_p(1)$ by \eqref{intheorem: colar: beta rate}. Thus we have established that the right hand side of \eqref{inlemma: colar: tg property } is $o_p(1)$. Assumption~\ref{assump: bounded eigenvalue} then implies that 
\[\rhk/M^{1/2}-o_p(1)\leq \|\tg\|_2\leq \rhk M^{1/2}+o_p(1),\]
which completes the proof.
\end{proof}

 \def\tr{\en{\text{tr}}}

\begin{proof}[of Lemma~\ref{lemma: colar: thm 4.2: T1} ]
To show \eqref{statement: lemma: colar: T1: probability},
 we first bound the difference
 \begin{align*}
     \|(\tx^T\hSx^{(1)}\tx)^{-1/2}\tx- (\tx^T\Sx\tx)^{-1/2}\tx\|_2=&\ \|\tx\|_2\abs{ (\tx^T\Sx\tx)^{-1/2}-(\tx^T\hSx^{(1)}\tx)^{-1/2}}\\
     \leq &\ (\|\tx-u^*\|_2+\|u^*\|_2)\frac{ \abs{\tx^T(\hSx^{(1)}-\Sx)\tx}}{(\tx^T\Sx\tx)^{1/2}+(\tx^T\hSx^{(1)}\tx)^{1/2}}\\
     \leq &\ (\|\tx-u^*\|_2+\|u^*\|_2)\frac{ \abs{\tx^T(\hSx^{(1)}-\Sx)\tx}}{(\tx^T\Sx\tx)^{1/2}}.
 \end{align*}
 Now by Lemma~\ref{lemma: colar: cu properties}, $\|\cu\|_2=O_p(1)$. Also by \eqref{intheorem: COLAR: Delta rate}, the difference term $\|\Delta\|_2=\|\tx-u^*\|_2$ is $O_p(s_U^{1/2}\lambda)$, which is $o_p(1)$ because $s_U^{1/2}\lambda\to 0$. Also, Assumption~\ref{assump: bounded eigenvalue} implies
  $\tx^T\Sx\tx\geq \|\tx\|_2/M$. Since $\|\tx-u^*\|_2=o_p(1)$,  Lemma~\ref{lemma: colar: cu properties} implies that $(\tx^T\Sx\tx)^{-1/2}=O_p(1)$. Hence, we have derived that
  \begin{align}\label{inlemma: colar: quad}
      \|(\tx^T\hSx^{(1)}\tx)^{-1/2}\tx- (\tx^T\Sx\tx)^{-1/2}\tx\|_2=O_p(1)\abs{\tx^T(\hSx^{(1)}-\Sx)\tx}.
  \end{align}
 Since $\Delta=\tx-u^*$, we obtain
  \begin{align*}
    \MoveEqLeft  \abs{\tx^T(\hSx^{(1)}-\Sx)\tx-(u^*)^T((\hSx^{(1)})-\Sx)u^*} \\
      \leq &\ \abs{\Delta^T(\hSx^{(1)}-\Sx)\Delta} +2\abs{\Delta^T(\hSx-\Sx)u^*}
  \end{align*}
  From Lemma~\ref{lem: quadratic form for sx} and the cone condition \ref{intheorem: cone condition} it follows that
  \[\abs{\Delta^T(\hSx^{(1)}-\Sx)\Delta}\leq s_U\|\Delta\|_2^2 O_p(\lambda).\]
 From \eqref{intheorem: COLAR: Delta rate} it follows that $\|\Delta\|^2_2=O_p(s_U\lambda^2)$. Therefore,
  \[\abs{\Delta^T(\hSx^{(1)}-\Sx)\Delta}=O_p(s_U^2\lambda^3),\]
  which is $o_p(\lambda)$ since $s_U\lambda\to 0$.
  On the other hand
  \[\abs{\Delta^T(\hSx^{(1)}-\Sx)u^*}\leq \|\Delta\|_1\|(\hSx^{(1)}-\Sx)u^*\|_\infty\leq \|\Delta\|_1 \|u^*\|_2O_p(\lambda)\]
  where the last inequality follows from Lemma~\ref{result: inf norm: dif } because $u^*$ only depends on the first sample part, which is independent of $\hSx^{(1)}$. On the other hand, \eqref{intheorem: cone condition} implies that
  \[\|\Delta\|_1=O_p(s_U\lambda).\]
 Since $\|u^*\|_2=O_p(1)$, by Lemma~\ref{lemma: colar: cu properties},
 \[\abs{\Delta^T(\hSx^{(1)}-\Sx)u^*}=O_p(s_U\lambda^2),\]
where the last term is $o_p(\lambda)$ because $s_U\lambda\to 0$.
 
 Combining all the pieces, we obtain that
 \begin{align*}
     \abs{\tx^T(\hSx^{(1)}-\Sx)\tx-(u^*)^T(\hSx^{(1)}-\Sx)u^*}=o_P(\lambda).
 \end{align*}
 Now note that $u^*_{S_U^c}=U_{S_U^c}\Lambda V^T\Sy\hbz=0$. Therefore 
 \begin{align*}
    (u^*)^T(\hSx^{(1)}-\Sx)u^* = &\ (u^*)^T(\hSx^{(1)}-\Sx)_{S_U\times S_U}u^*\\
    \leq &\ \|(\hSx^{(1)}-\Sx)_{S_U\times S_U}\|_{op}\|u^*\|_2^2\\
    =&\ O_p((s_U\log (p)/n)^{1/2})\|u^*\|_2^2
 \end{align*}
 by Theorem 5.31 of \cite{vershynin2010} \citep[see also Lemma 12 of][]{gao2015}. Because  $\|u^*\|_2=O_p(1)$,
 it follows that
 \begin{align}\label{inlemma: COLAR: quad 2}
    \abs{\tx^T(\hSx^{(1)}-\Sx)\tx}=O_p((s_U\log (p)/n)^{1/2}). 
 \end{align}
 Note that \eqref{statement: lemma: colar: T1: diff} follows from \eqref{inlemma: COLAR: quad 2} because $\log p/n\leq \lambda^2$. Since $S\lambda\to 0$ by our assumption on $s$, \eqref{inlemma: COLAR: quad 2} implies
  $\tx^T\hSx^{(1)}\tx=\tx^T\Sx\tx+o_p(1)$. 
Hence, by Assumption~\ref{assump: bounded eigenvalue}, 
\[\tx^T\hSx^{(1)}\tx\geq \|\tx\|_2^2/M+o_p(1).\]
Noting 
\[\|\tx\|_2\geq \|u^*\|_2-\|\Delta\|_2 =\|u^*\|_2+o_p(1),\] 
and using 
Lemma~\ref{lemma: colar: cu properties}, we find that
\begin{equation}\label{inlemma: colar: lower bound: T1: quad term}
    \tx^T\hSx^{(1)}\tx> \rhk/(2M^2)+o_p(1).
\end{equation}
Hence \eqref{def: COLAR: ha} implies as $n\to\infty$,  $\ha=\tx(\tx^T\hSx^{(1)}\tx)^{-1/2}$ with probability tending to one. Hence \eqref{statement: lemma: colar: T1: probability} is proved. 
 This fact implies, with high probability, 
 \[\|\ha-(\hx^T\Sx\hx)^{-1/2}\hx\|_2=\|(\hx^T\hSx\hx)^{-1/2}\hx-(\hx^T\Sx\hx)^{-1/2}\hx\|_2,\]
 which, by 
 \eqref{inlemma: colar: quad} and \eqref{inlemma: COLAR: quad 2}, is $O_p(s_U^{1/2}\lambda)$.
 Thus \eqref{statement: lemma: colar: T1: ha diff} follows, which completes the proof.
 \end{proof}
 
 \begin{proof}[of Lemma~\ref{lemma: COLAR: last step where slambda comes from}]
 For sake of simplicity, we denote $z=\tU\Lambda(\tV)^Ty$. 
 Note that Fact~\ref{fact: frob norm of proj matrix} implies
 \begin{gather*}
     \|P_{x}-P_{\tu_1}\|^2_F=2-2\tr(P_xP_{\tu_1})\quad\text{and}\quad
     \|P_x-P_z\|^2_F=2-2\tr(P_xP_z).
 \end{gather*}

 
Therefore,
\begin{equation}\label{inlemma: COLAR: projection matrices and trace}
    \|P_x-P_{\tu_1}\|_F^2-\|P_x-P_z\|_2^2=2\tr(P_xP_z)- 2\tr(P_xP_{\tu_1}).
\end{equation}
 Because $\tu_1$ and $x$ have unit norm,
 \begin{align*}
     \tr(P_xP_{\tu_1})=\tr(xx^T\tu_1\tu_1^T)={(x^T\tu_1)^2}.
 \end{align*}
 Also since $\tU^T\tU=I_r$, we have 
 \begin{align}\label{inlemma: COLAR: the bound on PxPz}
   \tr(P_xP_z)=&\ \frac{\tr \slb xx^T\tU\Lambda V^T yy^TV\Lambda \tU^T\srb}{x^Tx\slb y^T\tV\Lambda^2\tV^Ty\srb }
   = \frac{(x^T\tU\Lambda\tV^Ty)^2}{ \|\Lambda V^Ty\|_2^2}
  \leq c{\|x^T\tU\|_2^2}
 \end{align}
 by Cauchy-Schwarz inequality. Therefore when $r=1$, $\tU=\tu_1$, and we have 
 $tr(P_xP_z)\leq tr(P_xP_{\tu_1})$, which,  combined with \eqref{inlemma: COLAR: projection matrices and trace}, implies that
 $\|P_x-P_{\tu_1}\|_F^2\leq \|P_x-P_z\|_F^2$. This completes the proof of part A of the current lemma. 
 
    Now we turn our attention to part B of the current lemma. When $r\geq 2$, using \eqref{inlemma: COLAR: the bound on PxPz}, we obtain that
 \[\tr(P_xP_z)\leq{\|x^T\tU\|_2^2}={\sum_{i=1}^r(x^T\tu_i)^2}=tr(P_xP_{\tu_1})+{\sum_{i=2}^r(x^T\tu_i)^2},\]
 implying that for any $w\in\{\pm 1\}$,
 \[\tr(P_xP_z)-\tr(P_xP_{\tu_1})\leq {\sum_{i=2}^r(wx^T\tu_i)^2}.\]
 Therefore using \eqref{inlemma: COLAR: projection matrices and trace}, we can write
 \[\|P_x-P_{\tu_1}\|_F^2-\|P_x-P_z\|_F^2\leq 2\inf_{w\in\{\pm 1\}}{\sum_{i=2}^r(wx^T\tu_i)^2}.\]
 Letting $\tilde z=z/\|z\|_2$, and noting $\tu_1^T\tu_i=0$ for $i=2,\ldots,r$, 
 the term on the right hand side of \eqref{inlemma: COLAR: projection matrices and trace} can be bounded since
 \begin{align*}
   2\inf_{w\in\{\pm 1\}}{\sum_{i=2}^r(wx^T\tu_i)^2} =&\  2\inf_{w,w'\in\{\pm 1\}}{\sum_{i=2}^r\slb (wx-z)^T\tu_i+(z-w'\tu_1)^T\tu_i\srb^2}\\
   \leq &\   4\inf_{w,w'\in\{\pm 1\}}\sum_{i=2}^r\lbs ((wx-\tz)^T\tu_i)^2+((w'\tz-\tu_1)^T\tu_i)^2\rbs\\
   =&\ 4\inf_{w\in\{\pm 1\}}\sum_{i=2}^r ((wx-\tz)^T\tu_i)^2+4\inf_{w'\in\{\pm 1\}}\sum_{i=2}^r((w'\tz-\tu_1)^T\tu_i)^2\\
   \leq &\  4\inf_{w\in\{\pm 1\}} \|wx-\tz\|_2^2+4\inf_{w\in\{\pm 1\}}\|w\tz-\tu_1\|_2^2
 \end{align*}
 where the last step follows because $\tu_i$'s are orthogonal vectors. Since $x$ and $\tu_1$ have unit norm, by 
Fact~\ref{fact: Chen 2020},
\[\inf_{w\in\{\pm 1\}}{\|wx-\tz\|_2^2}\leq \|P_x-P_{\tz}\|_F^2=\|P_x-P_z\|_F^2\]
because $P_{\tz}=P_z$.
Therefore,
\begin{equation}\label{inlemma: COLAR: 5PxPz}
    \|P_x-P_{\tu_1}\|_F^2\leq 5\|P_x-P_z\|^2_F+4\inf_{w\in\{\pm 1\}}\|w\tz-\tu_1\|_2^2.
\end{equation}
Next we will bound  $\inf_{w'\in\{\pm 1\}}\|w\tz-\tu_1\|_2$ using the rate of decay of $\|wy-\tv_1\|_2$. To this end, we first show that $\|z\|_2$ is asymptotically equivalent to $\rhk$.
Noting $z=\tU\Lambda\tV^Ty$, for any $w\in\{\pm 1\}$, we have
\begin{align*}
    \|z\|^2-\rhk^2=&\ y^T\tV\Lambda^2\tV ^Ty-\tv_1^T\tV\Lambda^2\tV^T\tv_1\\
    =&\ wy^T\tV\Lambda^2\tV ^Tyw-\tv_1^T\tV\Lambda^2\tV^T\tv_1\\
    =&\ (wy-\tv_1)^T\tV\Lambda^2\tV^T wy +(wy-\tv_1)^T\tV\Lambda^2\tV^T\tv_1\\
    =&\ (wy-\tv_1)^T\tV\Lambda^2\tV^T (wy-\tv_1)+2(wy-\tv_1)^T\tV\Lambda^2\tV^T\tv_1,
\end{align*}
which implies
\begin{align*}
    |  \|z\|^2-\rhk^2|\leq &\ \|\tV\Lambda^2\tV^T\|_{op}\inf_{w\in\{\pm 1\}}(\|wy-\tv_1\|_2^2+2\|\tv_1\|_2\|wy-\tv_1\|_2)\\
    =&\ \rhk^2O_p(s\lambda)
\end{align*}
because $\inf_{w}\|wy-\tv_1\|_2=O_p(s\lambda)$  and $s\lambda\to 0$ by our assumption.  
Therefore, it also follows that
\[\abs{\|z\|_2^{-1}-\rhk^{-1}}=\frac{\abs{\|z\|_2-\rhk}}{\|z\|_2\rhk}\leq \frac{\abs{\|z\|^2_2-\rhk^2}}{\|z\|_2\rhk(\|z\|_2+\rhk)}\leq \frac{O_p(s\lambda)}{\rhk^2(\rhk-O_p(s\lambda)) },\]
which is $O_p(s\lambda)$ because $s\lambda\to 0$ and $\rhk>0$. Hence,
\begin{align*}
   \inf_{w\in\{\pm 1\}} \|w\tz-\tu_1\|_2= &\ \inf_{w\in\{\pm 1\}}\|\tU\Lambda\tV^T(\|z\|^{-1}wy)-\tU\Lambda \tV^T(\rhk^{-1}\tv_1)\|_F\\
    \leq &\ \inf_{w\in\{\pm 1\}}\|\tU\Lambda\tV^T\|_{op}\norm{\|z\|_2^{-1}wy-\rhk^{-1}\tv_1}_2\\
   = & \ \rhk\inf_{w\in\{\pm 1\}} \norm{ \|z\|_2^{-1}wy-\rhk^{-1}\tv_1}_2\\
    \leq &\ \|z\|_2^{-1}\rhk\inf_{w\in\{\pm 1\}}\|wy-\tv_1\|_2+(\|z\|_2^{-1}-\rhk^{-1})\|\tv_1\|_2,
\end{align*}
which is $O_p(s\lambda)$ since $\inf_{w\in\{\pm 1\}}\|wy-\tv_1\|_2=O_p(s\lambda)$  by our assumption and we just showed that $\|z\|_2^{-1}=\rhk^{-1}+O_p(s\lambda)$.
The proof then follows noting \eqref{inlemma: COLAR: 5PxPz} implies
\[\|P_x-P_{\tu_1}\|_F^2\leq 5\|P_x-P_z\|^2_F+O_p(s\lambda).\]



\end{proof}

 
 \subsection{Proof of technical lemmas for Theorem~\ref{theorem: COLAR: Theorem 4.1}}
\label{sec: lemmas: colar 4.1}

\begin{proof}[of Lemma~\ref{lemma: COLAR thm 1: delta suffices}]
Since the eigenvalues of $\Sx$ and $\Sy$ are bounded below by Assumption~\ref{assump: bounded eigenvalue}, it suffices to prove that
\[\|\Sx^{1/2}(\tF-F_0)\Sy^{1/2}\|_F=O_p(\e_{n,u}+\e_{n,v}).\]
To that end, note that
\begin{align*}
    \|\Sx^{1/2}(\tF-F_0)\Sy^{1/2}\|_F= &\ \|\Sx^{1/2}(\ta\tb^T-\alk\bk^T)\Sy^{1/2}\|_F\\
    \leq &\ \|\Sx^{1/2}\ta(\tb-\bk)^T\Sy^{1/2}\|_F+\|\Sx^{1/2}(\ta-\alk)\bk^T\Sy^{1/2}\|_F\\
    \leq &\ \|\Sx^{1/2}\ta\|_2\|\Sy^{1/2}(\tb-\bk)\|_2+\|\Sy^{1/2}\bk\|_2\|\Sx^{1/2}(\ta-\alk)\|_2\\
    \leq &\ (\|\Sx^{1/2}(\ta-\alk)\|_2+\|\Sx^{1/2}\alk\|_2) \|\Sy^{1/2}(\tb-\bk)\|_2\\
    &\ +\|\Sx^{1/2}(\ta-\alk)\|_2\\
    =&\ (\|\Sx^{1/2}(\ta-\alk)\|_2+ 1) \|\Sy^{1/2}(\tb-\bk)\|_2 +\|\Sx^{1/2}(\ta-\alk)\|_2
\end{align*}
Because $\ta=\tU_1$ and $\alk=U_1$, we can write
\begin{align*}
\|\Sx^{1/2}(\ta-\alk)\|_2 =  \|\Sx^{1/2}(\tU-U)e_1\|_2\leq \|\Sx^{1/2}(\tU-U)\|_{op}
\end{align*}
which is $O_p(\e_{n,u})$ by Lemma 6.1 of \cite{gao2017}. Similarly, we can show that $\|\Sy^{1/2}(\tb-\bk)\|_2$ is $O_p(\e_{n,v})$, which completes the proof.
\end{proof}

\begin{proof}[of Lemma~\ref{lemma: COLAR 6.2} ]
Consider $A=(\hSx^{(0)})^{1/2}\tF(\hSy^{(0)})^{1/2}$.
Since $\tF=\ta\tb^T$, $\|(\hSx^{(0)})^{1/2}\ta\|_2=1$, and  $\|(\hSy^{(0)})^{1/2}\tb\|_2=1$,
    \[\|A\|_{op}=\|(\hSx^{(0)})^{1/2}\tF(\hSx^{(0)})^{1/2}\|_{op}\leq \|(\hSx^{(0)})^{1/2}\ta\|_2\|(\hSy^{(0)})^{1/2}\tb\|_2=1.\]
   Also, by definition of  operator norm, we have
    \[\|A\|_{op}\geq ((\hSx^{(0)})^{1/2}\ta)^T A ((\hSy^{(0)})^{1/2}\tb)=1.\]
    Therefore, $\|A\|_{op}=1$.
    Second, 
    \[A^TA=(\hSy^{(0)})^{1/2}\tb\ta^T\hSx^{(0)}\ta\tb^T(\hSy^{(0)})^{1/2}=(\hSy^{(0)})^{1/2}\tb\tb^T(\hSy^{(0)})^{1/2}.\]
    Therefore,
    \[tr(A^TA)=tr(\tb^T\hSy^{(0)}\tb)=\tb^T\hSy^{(0)}\tb=1.\]
   Hence, $A$ has only one non-zero singular value, which is one. Thus $\|A\|_*=1$.
\end{proof}

\begin{proof}[of Lemma~\ref{lemma: COLAR 6.3}]
 First note that 
 \begin{align*}
 \MoveEqLeft\abs{ \langle A  (\tilde D-D)G^T, AEG^T-F\rangle} \\
 =&\ \abs{\tr(G (\tilde D-D)^T A^T(A E G^T-F))}\\
 \stackrel{(a)}{\leq} &\ \|A (\tilde D-D)G^T\|_{F}\|AEG^T-F\|_F\\
 \stackrel{(b)}{=}&\ \|\tilde D-D\|_{F}\|AEG^T-F\|_F
 \end{align*}
 where (a) follows because by Cauchy Schwarz inequality and (b) follows 
 because the Frobenius norm is unitarily invariant \citep[cf. p. 26][]{chen2020} and $A$ and $G$ are unitary matrices.
 Therefore
 \begin{align}\label{inlemma: colar: thm 4.1: first inequality}
     \langle A  \tilde D G^T, AEG^T-F\rangle \geq \langle A  D G^T, AEG^T-F\rangle-\|\tilde D-D\|_F\|AEG^T-F\|_F.
 \end{align}
 Let us denote $c_i=A_i^TFG_i$ $(i=1,\ldots,r)$. We will first show that
 $\|AEG^T-F \|_F^2\leq 2(1-c_1)$ and then we will show that $\langle A  D G^T, AEG^T-F\rangle\geq d_{12}(1-c_1)$, from which, the proof will show.
 For the upper bound on $\|AEG^T-F \|_F^2$, notice that
 \begin{align}\label{inlemma: colar: thm 4.1: second inequality}
 \|AEG^T-F \|_F^2
  =&\ \tr\slb( AEG^T-F )^T( AEG^T-F )\srb\nn\\
  =&\ \tr(GE^2G^T)+\|F\|_F^2-2\tr(F^TAEG^T)\nn\\
 \stackrel{(a)}{\leq}  & \tr(E)+\|F\|_*^2-2Tr(EA^TFG)\nn\\
\stackrel{ (b)}{\leq }&\ 2(1-c_1).
 \end{align}
 Here (a) follows because nuclear norm is greater than the Frobenius norm,  and $E^2=E$ and $ G\in \mathcal O(q,r)$. Also (b) follows because (i) $\tr(E)=\tr(e_1e_1^T)=1$, (ii) $\|F\|_{*}\leq 1$ by our assumption on $F$, and (iii) $tr(EA^TFG)=tr(e_1^TA^TFGe_1)=A_1^TFG_1=c_1$. We have used the relation $\tr(AB)=\tr(BA)$ here. 
 
 Now we will establish the lower bound  $\langle A  D G^T, AEG^T-F\rangle\geq d_{12}(1-c_1)$. To that end, first note that
 \begin{align}\label{inlemma: colar: Thm 4.1: matrices }
     \langle ADG^T, AEG^T-F\rangle =&\ \tr(GDA^TAEG^T)-tr(GDA^TF)
 \end{align}
 It follows that because $A$ and $G$ are unitary, the first term of \eqref{inlemma: colar: Thm 4.1: matrices } 
 \begin{align}\label{inlemma: colar: Thm 4.1: D11 } 
     \tr(GDA^TAEG^T)=\tr(DEG^TG)=\tr(DE)=\tr(e_1^TDe_1)=D_{11}.
 \end{align}
 We will bound the second term of \eqref{inlemma: colar: Thm 4.1: matrices } by $D_{22}$.To that end, recalling
  $d_{12}=D_{11}-D_{22}$, and 
  denoting $D'=Diag(0,D_{22},D_{22},D_{33},\ldots,D_{rr})$,
  we write $D=d_{12}e_1e_1^T+D'$. Hence,
 \begin{align}\label{inlemma: colar: Thm 4.1: trace GDAF } 
    \tr(GDA^TF)=&\ d_{12}\tr(Ge_1e_1^TA^TF)+\tr(GD'A^TF)\nn\\
    =&\ d_{12}\tr(e_1^TA^TFGe_1)+\tr(GD'A^TF)\nn\\
    =&\ d_{12}c_1+\tr(GD'A^TF)
 \end{align}
 Consider an SVD $U_1\Lambda' V_1^T$ of $F$, which means $U_1\in \mathcal O(p,r)$, $V_1\in\mathcal O(q,r)$ and $\Lambda'=diag(\Lambda'_1,\ldots,\Lambda_r)$ is the diagonal matrix whose diagonal entries are the singular values of $F$. Then
 \begin{align*}
  \tr(GD'A^TF)=&\   \tr(GD'A^TU_1\Lambda' V_1^T)=\tr(V_1^TGD'A^TU_1\Lambda')\\
  = &\ \sum_{i=1}^r e_i^TV_1^TGD'A^TU_1\Lambda'e_i\\
 =&\ \sum_{i=1}^{r} (V_1^TGD'A^TU_1e_i)^T\Lambda_i'e_i\\
 =&\ \sum_{i=1}^{r} e_i^TU_1^TAD'G^TV_1e_i\Lambda'_i\\
 \leq &\ \sup_{1\leq i\leq r}|e_i^TU_1^TAD'G^TV_1e_i|\sum_{i=1}^r\Lambda'_i\\
\stackrel{(a)}{ \leq} &\ \|U_1^TAD'G^TV_1\|_{op}\|F\|_{*}\\
 \end{align*}
 Here (a) uses the fact that  $\|F\|_*=\sum_{i=1}^r\Lambda'_i$ is the sum of the singular values of $F$. Since $U_1$, $V_1$, $A$, and $G$ are unitary matrices, and $\|F\|_*\leq 1$ by our assumption, the above calculations lead to
 \begin{equation}\label{inlemma: colar: Thm 4.1: D22}
     \tr(GD'A^TF)\leq \|U_1^TAD'G^TV_1\|_{op}\|F\|_{*}\leq \|D'\|_{op}\leq D_{22}
 \end{equation}
  by definition of $D'$.
 Combining \eqref{inlemma: colar: Thm 4.1: matrices }, \eqref{inlemma: colar: Thm 4.1: D11 }, \eqref{inlemma: colar: Thm 4.1: trace GDAF }, and \eqref{inlemma: colar: Thm 4.1: D22}, we obtain
 \[  \langle ADG^T, AEG^T-F\rangle = tr(GDA^TAEG^T)-tr(GDA^TF)\geq D_{11}-d_{12}c_1-D_{22}=d_{12}(1-c_1),\]
 which, in conjunction with \eqref{inlemma: colar: thm 4.1: first inequality} and \eqref{inlemma: colar: thm 4.1: second inequality}, completes the proof. 
 \end{proof}

 
 \subsection{Proof of technical lemmas and claims for Supplement~\ref{App: nodewise lasso}}
\label{sec: lemma: NL}

 
\begin{proof}[of Lemma~\ref{lemma: Phi and H knot }]
Lemma~\ref{lemma: Hessian positive definite} implies $\Lambda_{max}(\Phi^0)=(\Lambda_{min}(H^0))^{-1}\leq 2^{-1}M/(\rhk-\Lambda_2)$.
Let us denote $\tu_i=\Sx^{1/2}u_i$
and $\tv_i=\Sy^{1/2}v_i$  $(i=1,\ldots, r)$.
Recall from \eqref{def: Hessian general} in the proof of Lemma~\ref{lemma: Hessian positive definite}  that
\[H^0= 2\rhk D\underbrace{\begin{bmatrix}
      I_p+2\tu_1\tu_1^T & -\Sx^{-1/2}\Sxy\Sy^{-1/2}/\rhk\\
      -\Sy^{-1/2}\Syx\Sx^{-1/2}/\rhk & I_q+2\tv_1\tv_1^T
      \end{bmatrix}}_{A}D\]
   where $D=Diag(\Sx^{1/2},\Sy^{1/2})$. Note that 
      $\Lambda_{max}(H^0)\leq 2\rhk\|D\|^2_{op}\|A\|_{op}$.
     From the proof of Lemma~\ref{lemma: Hessian positive definite} it follows that follows that $\|D\|_{op}\leq M^{1/2}$ and $\|A\|_{op}=4 $. Therefore,
      $\Lambda_{max}(H^0)\leq 8\rhk M$ and  $\Lambda_{min}(H^0)\geq (8\rhk M)^{-1}$.
    For $\tau_j^2$, note that
      \[(\tau_j^2)^{-1}=\Phi^0_{j,j}\leq \|\Phi^0\|_{op}\leq 2^{-1}M/(\rhk-\Lambda_2),\]
    which implies
      \[\tau_j^2\geq 2(\rhk-\Lambda_2)/M.\]
\end{proof}

\begin{proof}[of Lemma~\ref{lemma: NL: claim 2 criteria lemma}]
 That
 $\max_{1\leq j\leq p+q}\|\Gamma^0_j\|_0$ is $O(s)$ follows from Condition~\ref{assumption: precision matrix}. For the $l_2$-norm, note that
\[\max_{1\leq j\leq p+q}\|\Gamma^0_j\|^2_2=1+\max_{1\leq j\leq p+q}\|\etak\|_2^2.\]
Equation \ref{def: eta j knot} implies
\[\|\etak\|_2\leq \|H^0_{-j,-j}\|_{op}\|H_{-j,j}^0\|_2.\]
Since $\|H^0_{-j,-j}\|_{op}\leq \|H^0\|_{op}$
and
\[\|H_{-j,j}^0\|^2_2\leq \|H^0_j\|_2^2=\|H^0 e_j\|_2^2\leq \|H^0\|_{op}^2,\]
the proof follows from
by Lemma~\ref{lemma: Phi and H knot }.
\end{proof}

  \begin{proof}[of Claim~\ref{claim: T of delta}]
From the definition of $\Delta_{\Gamma,1}(j)$ and $\Delta_{\Gamma}(j)$, it follows  that   $\sup_{i\in\{1,2\}}\|\Delta_{\Gamma,i}(j)\|_p\leq \|\Delta_\Gamma(j)\|_k$ for $k=1,2$. Since the $j$th element of $\Delta_\Gamma(j)=0$, 
\[\|\Delta_\Gamma(j)\|_k^k=\|\Delta(j)\|_k^k,\]
which indicates that there exists absolute constant $C$ so that
\begin{equation}\label{inclaim: delta relations}
    \max_{i\in\{1,2\}}\|\Delta(j)_{\Gamma,i}\|_k\leq C\|\Delta(j)\|_k,\quad (k=1,2,\ j=1:(p+q)).
\end{equation}
The above relation will be used often times without stating throughout the proof.

We  make note of some facts first. First,
\begin{equation}\label{inlemma: T delta: fact 1}
    |\hx^T\hSx\hx-\rhk|=O_p(s^{\kappa}\lambda)
\end{equation}
by Lemma~\ref{corollary: colar: normalization of hx}. Next, we want to derive a bound on $\Delta_{\Gamma,1}^T(\hSx-\Sx)\Delta_{\Gamma,1}$ using Lemma~\ref{lemma: Quadratic:  bounded}. To apply this lemma, we have to show that $\|\Delta_{\Gamma,1}(j)\|_1=O(s^{1/2})$.  To that end, we show that both $\max_{1\leq j\leq p+q}\|\Gamma_j^0\|_1$ and $\max_{1\leq j\leq p+q}\|\widehat\Gamma_j^0\|_1$ are $O_p(s^{1/2})$. Because $\|\Gamma_j^0\|_1=1+\|\etak\|_1$, Assumption~\ref{assump:Phi 0} implies $\max_{1\leq j\leq p+q}\|\widehat\Gamma_j^0\|_1=O(s^{1/2})$.  On the other hand, $\widehat\Gamma_j$ is a solution to \eqref{opt: nodewise Lasso}, and therefore $\|\widehat \Gamma_j^0\|_1\leq B_j$. Since we have taken $B_j\leq C_Ts^{1/2}$,  we also have $\max_{1\leq j\leq p+q}\|\widehat\Gamma_j\|_1\leq C_Ts^{1/2}$. Thus \begin{equation}\label{inlemma: NL: weak bound on delta j}
\max_{1\leq j\leq p+q}\|\Delta_{\Gamma}(j)\|_1
 \leq C's^{1/2}.
\end{equation}
Hence,  Lemma~\ref{lemma: Quadratic:  bounded} implies that there exists constant $C>0$ depending only on $C'$ and the distribution of $X$ so that so that for large $p$ and $n$,
\begin{equation}\label{inlemma: T delta: fact: quad}
   \abs{ \Delta_{\Gamma,1}(j)^T(\hSx-\Sx)\Delta_{\Gamma,1}(j)}\leq C(s^{1/2}\lambda\|\Delta_{\Gamma,1}(j)\|_2^2+\lambda\|\Delta(j)\|_1)\quad (j=1:\ldots,p+q). 
\end{equation}
with high probability.

Third, by Lemma~\ref{lemma: additional: l1 norm of x and l2 norm of z knot} and Lemma~\ref{lemma:norm:  l1 and l2 norm}, the following holds with high probability for a constant $C$ again not depending on $j$
:
\begin{align}\label{inlemma: T Delta: fact 3}
  \abs{ \hx^T(\hSx-\Sx)\Delta_{\Gamma,1}(j)}\leq C \|\Delta_{\Gamma,1}(j)\|_1\|x^0\|_2\lambda\quad (j=1:\ldots,p+q).
\end{align}
Fourth, Assumption~\ref{assump: bounded eigenvalue} implies that 
\begin{align}\label{inlemma: T delta: fact: some bound}
   |\hx^T\Sx\Delta_{\Gamma,1}(j)|\leq M\|\hx\|_2\|\Delta_{\Gamma,1}(j)\|_2\stackrel{(a)}{\leq} MC\|\Delta_{\Gamma,1}(j)\|_2\quad (j=1,\ldots,p+q)
\end{align}
with high probability where (a) follows by Lemma~\ref{corollary: rate of x and y} and Lemma~\ref{lemma:norm:  l1 and l2 norm}. 
Finally, by Assumption~\ref{assump: bounded eigenvalue},
\begin{equation}\label{inlemma: NL: fact 5}
    \max_{1\leq j\leq p+q}\Delta_{\Gamma,1}(j)^T\Sx\Delta_{\Gamma,1}(j)/\|\Delta_{\Gamma,1}(j)\|_2^2\leq\|\Sx\|_{op}= M .
\end{equation}

For the rest of the lemma, the constant $C$ does not depend on $j$.
Using \eqref{inlemma: T delta: fact 1} and \eqref{inlemma: T delta: fact: quad}, we can a find $C$ so that
\begin{align*}
 |T_1(\Delta(j);j)|=&\ 2\abs{(\hx^T\hSx\hx-\rhk)\Delta_{\Gamma,1}(j)^T(\hSx-\Sx)\Delta_{\Gamma,1}(j)}\\
   &\ \leq C\slb s^{\kappa+1/2}\lambda^2\|\Delta_{\Gamma,1}(j)\|_2^2+s^{\kappa}\lambda^2\|\Delta_{\Gamma,1}(j)\|_1\srb\quad (j=1:\ldots,p+q). 
\end{align*}
with high probability for sufficiently large $n$, $p$, and $q$.
Using \eqref{inlemma: T delta: fact 1} and \eqref{inlemma: NL: fact 5}, for $T_2(\Delta(j);j)$, we can choose $C$ to be so large such that for sufficiently large $n$, $p$, and $q$,
\begin{align*}
   |T_2(\Delta(j);j)|=&\ 2\abs{(\hx^T\hSx\hx-\rhk)\Delta_{\Gamma,1}(j)^T\Sx\Delta_{\Gamma,1}(j)}\\
   \leq &\ Cs^{\kappa}\lambda\|\Delta_{\Gamma,1}(j)\|_2^2\quad (j=1:\ldots,p+q)
\end{align*}
with high probability.
Also, by \eqref{inlemma: T delta: fact: quad}, we can obtain a large enough $C$ so that
\begin{align*}
    |T_3(\Delta(j);j)|=&\ 2\abs{\Delta_{\Gamma,1}(j)^T(\hSx-\Sx)\Delta_{\Gamma,1}(j)}\\
    \leq  &\ C(s^{1/2}\lambda\|\Delta_{\Gamma,1}(j)\|_2^2+\lambda\|\Delta_{\Gamma,1}(j)\|_1)\quad (j=1:\ldots,p+q)
\end{align*}
with high probability for large $p$, $q$, and $n$.
Next, observe that \eqref{inlemma: T Delta: fact 3} implies
\[T_4(\Delta(j);j)=4\slb\hx^T(\hSx-\Sx)\Delta_{\Gamma,1}(j)\srb^2\leq C\lambda^2\|x^0\|_2^2\|\Delta_{\Gamma,1}(j)\|^2_1\quad (j=1:\ldots,p+q)\]
with high probability for large $p$, $q$, and $n$. 
 Since $\max_{1\leq j\leq p+q}\|\Delta_{\Gamma,1}(j)\|_1=O_p(s^{1/2})$ by \eqref{inlemma: NL: weak bound on delta j} and $s\lambda\to 0$ by Fact~\ref{fact: slambda goes to zero},
$T_4(\Delta(j);j)=\lambda\|\Delta_{\Gamma,1}(j)\|_1o_p(1)$ uniformly across the $j$'s.

Now note that \[|T_5(\Delta(j);j)|=|T_6(\Delta_{\Gamma,1}(j))|=4\abs{\Delta_{\Gamma,1}(j)^T(\hSx-\Sx)\hx \hx^T\Sx\Delta_{\Gamma,1}(j)}.\]
Using  \eqref{inlemma: T Delta: fact 3} and  \eqref{inlemma: T delta: fact: some bound}, we find that for $j=1,\ldots,p+q$,
\[|T_5(\Delta(j);j)|=|T_6(\Delta(j);j)|\leq C \lambda\|\Delta_{\Gamma,1}(j)\|_2\|\Delta_{\Gamma,1}(j)\|_1\]
for some $C>0$ for large $p$, $q$, and $n$.
For $T_7(\Delta_j;j)$, we observe that
\[|T_7(\Delta(j);j)|=4\Delta_{\Gamma,1}(j)^T\Sx(\hx\hx^T-\xk(\xk)^T)\Sx\Delta_{\Gamma,1}(j)\leq M^2\|\Delta_{\Gamma,1}(j)\|_2^2\|\hx\hx^T-x^0(x^0)^T\|_F,\]
where by Fact~\ref{fact: Projection matrix} and Lemma~\ref{corollary: rate of x and y},
\begin{equation}\label{inlemma: projection matrix}
    \|\hx\hx^T-x^0(x^0)^T\|_F\leq \frac{\inf_{w\in\{\pm 1\}}\|w\hx-x^0\|_2}{\|x^0\|_2}=O_p(s^{\kappa}\lambda).
\end{equation}
Hence, uniformly over $j=1,\ldots, p+q$,
\[|T_7(\Delta(j);j)|=O_p(s^{\kappa}\lambda)\|\Delta_{\Gamma, 1}(j)\|_2^2.\]
Lemma~\ref{Additional lemma: Quadratic:  bounded: corollary} implies that the following holds uniformly over all $j=1,\ldots, p+q$,
\[T_8(\Delta(j);j)=\Delta_{\Gamma,1}(j)^T(\Sxy-\hSxy)\Delta_{\Gamma,2}(j)=O_p\slb s^{1/2}\lambda\|\Delta_{\Gamma}(j)\|_2^2+\lambda\|\Delta_{\Gamma}(j)\|_1\srb.\]
Combining the above pieces, and using the fact that $\|\Delta_{\Gamma}(j)\|_i=\|\Delta(j)\|_i$ for $i\in\NN$, we conclude that there exists a $C>0$ not depending on $j$ such that the following holds with  high probability:
\begin{align*}
 \sum_{i=1}^8 |T_i(\Delta(j);j)|\leq &\ C\slbt \{ (s+s^{\kappa+1/2})\lambda+s^{1/2}+s^\kappa\}\lambda\|\Delta(j)\|_2^2\\
 &\ +(s^{\kappa}\lambda^2+\lambda)\|\Delta(j)\|_1+\lambda\|\Delta(j)\|_1\|\Delta(j)\|_2\srbt\quad (j=1,\ldots,p+q).
\end{align*}
Now since $\kappa\in[1/2,1]$ by Condition~\ref{cond: preliminary estimator}, $s^{\kappa}=O(s)$.
Because $s\lambda=o(1)$ by Fact~\ref{fact: slambda goes to zero},
\[\{(s+s^{\kappa+1/2})\lambda+(s^{1/2}+s^\kappa)\}\lambda=O(s^{\kappa}\lambda),\]
\[s^{\kappa}\lambda^2+\lambda\leq \lambda(s\lambda+1)=O(\lambda).\]
Finally, noting $\max_{1\leq j\leq p+q}\|\Delta(j)\|_1=O_p(s^{1/2})$ by \eqref{inlemma: NL: weak bound on delta j}, we have
\[\lambda\|\Delta(j)\|_1\|\Delta(j)\|_2=O_p(s^{1/2}\lambda\|\Delta(j)\|_2).\]
Thus $C$ can be chosen so that for all sufficiently large $p$, $q$, $n$,
\[ \sum_{i=1}^8 |T_i(\Delta_{\Gamma}(j);j)|\leq C\slb s^{\kappa}\lambda\|\Delta(j)\|_2^2+\lambda\|\Delta(j)\|_1+s^{1/2}\lambda\|\Delta(j)\|_2\srb \quad(j=1,\ldots,p+q),\]
with high probability, which completes the proof.
\end{proof}

\begin{proof}[of Claim~\ref{claim: NL: T of z}]
The proof techniques of the current claim will be similar to that of Claim~\ref{claim: T of delta}. 
We will make often use of the following relation stated in  \eqref{inclaim: delta relations} of Claim~\ref{claim: T of delta}:
\[\max_{i\in\{1,2\}}\|\Delta(j)_{\Gamma,i}\|_k\leq C\|\Delta(j)\|_k,\quad (k=1,2,\ j=1:(p+q)).
\]
First, using Lemma~\ref{lemma: additional: l1 norm of x and l2 norm of z knot} we find that there exists $C>0$ so that
\begin{equation}\label{inclaim: z: fact 1}
    \abs{\Delta_{\Gamma,1}(j)^T(\hSx-\Sx)z_1}\leq C\|\Delta_{\Gamma,1}(j)\|_1\|z\|_2\lambda\quad (j=1,\ldots,p)
\end{equation}
with high probability for sufficiently large $n$. 
Second,
from Assumption~\ref{assump: bounded eigenvalue} it follows that 
\begin{equation}\label{inclaim: z: fact 2}
    \abs{\Delta_{\Gamma,1}(j)^T\Sx z_1}\leq M\|z\|_2\|\Delta_{\Gamma,1}(j)\|_2.
\end{equation}
Third, noting that Lemma~\ref{corollary: rate of x and y} implies $\inf_{w\in\{\pm 1\}}\|w\hx-x^0\|_1=O_p(s^{\kappa+1/2}\lambda)$, and $s^{\kappa+1/2}\lambda\to 0$ by Fact~\ref{fact: slambda goes to zero}, we can find a large enough $C>0$ so that with high probability,
\begin{equation}\label{inclaim: z: fact 3}
    \abs{\hx^T(\hSx-\Sx)z_1}= \inf_{w\in\{\pm 1\}}  \abs{w\hx^T(\hSx-\Sx)z_1}\stackrel{(a)}{\leq} \|x^0\|_1\|z\|_2\lambda\stackrel{(b)}{\leq }C s^{1/2}\|z\|_2\lambda.
\end{equation}
where (a) follows from Lemma~\ref{lemma: additional: l1 norm of x and l2 norm of z knot}
and (b) follows because
\[\|x^0\|_1=\rhk^{1/2} \|\alk\|_1\leq s^{1/2}\|\alk\|_2\]
which is $O(s^{1/2})$ by Lemma~\ref{lemma:norm:  l1 and l2 norm}.
On the other hand, by Assumption~\ref{assump: bounded eigenvalue},
\begin{equation*}
    \abs{\hx^T\Sx z_1}\leq M\|\hx\|_2\|z\|_2.
\end{equation*}
 Lemma~\ref{lemma:norm:  l1 and l2 norm}, Fact~\ref{fact: slambda goes to zero}, and Lemma~\ref{corollary: rate of x and y}  yield $\|\hx\|_2=O_p(1)$, implying
\begin{equation}\label{inclaim: z: fact 4}
    \abs{\hx^T\Sx z_1}\leq C\|z\|_2
\end{equation}
with high probability for sufficiently large $C$, $p$, $q$, and $n$.
For the rest of the proof, $C$ should be understood as a  large constant whose value changes from line to line.
Using \eqref{inlemma: T delta: fact 1} and \eqref{inclaim: z: fact 1} we obtain that for sufficiently large $p$, $q$, and $n$:
\begin{align*}
    |T_1(z;j)|=&\ \abs{(\hx^T\hSx\hx-\rhk)\Delta_{\Gamma,1}(j)^T(\hSx-\Sx)z_1}\\
    \leq &\ Cs^{\kappa}\lambda^2\|\Delta_{\Gamma,1}(j)\|_1\|z\|_2\quad (j=1,\ldots, p+q)
\end{align*}
with high probability. 
Similarly, \eqref{inlemma: T delta: fact 1}, when combined with \eqref{inclaim: z: fact 2}, leads to
\begin{align*}
    |T_2(z;j)|=&\ \abs{(\hx^T\hSx\hx-\rhk)\Delta_{\Gamma,1}(j)^T\Sx z_1}\leq C s^{\kappa}\lambda\|\Delta_{\Gamma,1}(j)\|_2\|z\|_2 \quad (j=1,\ldots, p+q),
\end{align*}
whereas \eqref{inclaim: z: fact 1} implies
\[|T_3(z;j)|=\abs{\rhk\Delta_{\Gamma,1}(j)^T(\hSx-\Sx)z_1}\leq C\lambda\|\Delta_{\Gamma,1}(j)\|_1\|z\|_2  \quad (j=1,\ldots, p+q).\]
We use the bounds in  \eqref{inlemma: T Delta: fact 3} and \eqref{inclaim: z: fact 3} to obtain
\begin{align*}
    |T_4(z;j)|=&\ \abs{\rhk\Delta_{\Gamma,1}(j)^T(\hSx-\Sx)\hx\hx^T(\hSx-\Sx)z_1}\\
\leq &\ C s^{1/2}\lambda^2\|\Delta(j)\|_1\|z\|_2\quad (j=1,\ldots, p+q),
\end{align*}
and use \eqref{inlemma: T Delta: fact 3} and \eqref{inclaim: z: fact 4} to show
\[|T_5(z;j)|=\abs{\Delta_{\Gamma,1}(j)^T(\hSx-\Sx)\hx\hx^T\Sx z_1}\leq C\lambda\|\Delta(j)\|_1\|z\|_2\quad (j=1,\ldots, p+q)\]
with high probability for sufficiently large $n$.
Similarly, \eqref{inlemma: T delta: fact: some bound} and  \eqref{inclaim: z: fact 3} and jointly imply that
\[|T_6(z;j)|=\abs{z_1^T(\hSx-\Sx)\hx\hx^T\Sx \Delta_{\Gamma,1}(j)}\leq Cs^{1/2}\lambda\|z\|_2\|\Delta(j)\|_2 \quad (j=1,\ldots, p+q)\]
with high probability for sufficiently large $n$.
Finally, 
\[|T_7(z;j)|=\abs{\Delta_{\Gamma,1}^T\Sx(\hx\hx^T-x^0(x^0)^T)\Sx z_1}\leq M^2\|\Delta(j)\|_2\|z\|_2\|\hx\hx^T-x^0(x^0)^T\|_F,\]
where the Frobenius norm is $O_p(s^{\kappa}\lambda)$ by \eqref{inlemma: projection matrix}.
Therefore,
\[|T_7(z;j)|\leq Cs^{\kappa}\lambda\|\Delta(j)\|_2\|z\|_2\quad (j=1,\ldots, p+q)\]
with high probability for sufficiently large $n$.
Lemma~\ref{lemma: additional: l1 norm of x and l2 norm of z knot} implies 
\[|T_8(z;j)|=\abs{\Delta_{\Gamma,1}(j)^T(\hSxy-\Sxy)z_2}\leq \lambda\|\Delta(j)\|_1\|z\|_2\quad (j=1,\ldots, p+q)\]
with high probability for sufficiently large $n$.
Combining the above pieces leads to
\begin{align*}
    \sum_{i=1}^8|T_i(z;j)|\leq C\|z\|_2\lambda((1+s^{1/2}\lambda+s^{\kappa}\lambda)\|\Delta(j)\|_1+(s^{1/2}+s^{\kappa})\|\Delta(j)\|_2)\quad (j=1,\ldots, p+q)
\end{align*}
with high probability for sufficiently large $n$.
By Condition~\ref{cond: preliminary estimator}, $\kappa\in[1/2,1]$, and Fact~\ref{fact: slambda goes to zero} implies $s^{\kappa+1/2}\lambda=o(1)$. Therefore 
\begin{align*}
    \sum_{i=1}^8|T_i(z;j)|\leq C\|z\|_2\lambda(\|\Delta(j)\|_1+s^{\kappa}\|\Delta(j)\|_2)\quad (j=1,\ldots, p+q).
\end{align*}
\end{proof}

  \section{ Proof of the Lemmas and Facts  in Supplement~\ref{sec: aditional lemmas}}
\label{sec: proof of lemmas in fact}

\begin{proof}[of Lemma~\ref{corollary: rate of rho}]
Let 
\[w_1^*=\inf_{w\in\{\pm 1\}}\|w\ha-\alk\|_2\quad\text{and}\quad w_2^*=\inf_{w\in\{\pm 1\}}\|w\hb-\bk\|_2.\]
Let us define
\[\hro^*=\dfrac{(w_1^*\ha)^T\hSxy(w_2^*\hb)}{\sqrt{(w_1^*\ha)^T\hSx(w_1^*\ha)}\sqrt{(w_2^*\hb)^T\hSy(w_2^*\hb)}}.\]
Since $\rhk>0$, we have $\|\hro|-\rhk|\leq |\hro^*-\rhk|$. Thus it suffices to prove the result for $\hro^*$. 
First we show that the rate of $\hro^*$ is mainly controlled by the numerator because the denominator converge to $1$ in probability. For the sake of simplicity, we will assume that $w_1^*=1$ and $w_2^*=1$. The proof for the other cases will be identical.

Simple algebra shows that the numerator is bounded above by
\begin{align*}
\MoveEqLeft|\ha^T\hSxy\hb-\alk^T\Sxy\beta| \\
=&\ \bl (\ha-\alk)^T\hSxy\hb+\alk(\hSxy-\Sxy)\hb+\alk^T\Sxy(\hb-\bk)\bl\\
\leq &\  \bl (\ha-\alk)^T\hSxy\hb\bl+\bl \alk(\hSxy-\Sxy)\hb\bl+\bl\alk^T\Sxy(\hb-\bk)\bl\\
\end{align*}
The first term can be bounded since
\begin{align}\label{incor: colar: hro: first term}
  |(\ha-\alk)^T\hSxy\hb|\leq &\  \abs{(\ha-\alk)^T(\hSxy-\Sxy)\hb}+\abs{(\ha-\alk)^T\Sxy\hb } \nn\\
  \stackrel{(a)}{\leq} &\ \|\ha-\alk\|_1|\hSxy-\Sxy|_{\infty}\|\hb\|_2+M\|\ha-\alk\|_2 \|\hb\|_2\nn\\
  \stackrel{(b)}{=}&\ O_p(s^{\kappa+1/2}\lambda^2)+MO_p(s^{\kappa}\lambda)
\end{align}
where in step (a), we used Assumption~\ref{assump: bounded eigenvalue} and (b) uses Condition~\ref{cond: preliminary estimator}, with $\kappa$ as defined in Condition~\ref{cond: preliminary estimator}. 
  For the second term, note that Lemma~\ref{lemma:norm:  l1 and l2 norm} and  Lemma~\ref{lemma: additional: l1 norm of x and l2 norm of z knot}  imply
\begin{align}\label{incor: colar: hro: second  term}
 \bl \alk(\hSxy-\Sxy)\hb\bl
\leq &\ \|\bk\|_{2}\|\alk\|_1O_p(\lambda)
=O_p(s^{1/2}\lambda).
\end{align}
For the third term, using Lemma~\ref{lemma:norm:  l1 and l2 norm} and Condition~\ref{cond: preliminary estimator}, 
we have
\begin{align}\label{incor: colar: hro: third term}
|\alk^T\Sxy(\hb-\bk)|\leq  \|\alk\|_2\|\Sxy\|_{op}\|\hb-\bk\|_2
= O_p( M^{3/2} s^{\kappa}\lambda).
\end{align}
Therefore, using the expansion of $|\ha^T\Sxy\hb-\alpha^T\Sxy\beta|$, and combining \eqref{incor: colar: hro: first term}, \eqref{incor: colar: hro: second  term}, and \eqref{incor: colar: hro: third term}, we have 
\[|\ha^T\hSxy\hb-\alpha^T\Sxy\beta|=O_p(s^\kappa\lambda)+O_p(s^{1/2}\lambda)+O_p(s^{\kappa+1/2}\lambda).\]
Because $\kappa\in\{1/2,1\}$, and  $s\lambda\to 0$ by Fact~\ref{fact: slambda goes to zero}, the above term is $O_p(s^{\kappa}\lambda)$.
This settles the case for the numerator of $\hro$, i.e.
\begin{equation}\label{incor: colar: rho: numerator}
  |\ha^T\hSxy\hb-\alpha^T\Sxy\beta|=  O_p(s^{\kappa}\lambda)
\end{equation}
For the denominator, it suffices to show that
\begin{align}\label{inlemma: rate ha hsx ha}
\ha^T\hSx\ha=1+o_p(1)
\end{align}
since the proof for $\hb^T\hSy\hb$ will be similar. To this end, proceeding as before, we decompose
\begin{align}\label{incor: colar: rho: hahsxha}
\MoveEqLeft |\ha^T\hSx\ha-\alk^T\Sx\alk|
\leq  |(\ha-\alk)^T\hSx\ha|+|\alk^T(\hSx-\Sx)\ha|+|\alk^T\Sx(\ha-\alk)|
\end{align}
Proceeding in a similar way as we did while proving \eqref{incor: colar: hro: first term}, we can show that
\[ |(\ha-\alk)^T\hSx\ha|=O_p(s^{\kappa}\lambda).\]
The second term can be controlled in the same way as \eqref{incor: colar: hro: second  term}, to yield
\[|\alk^T(\hSx-\Sx)\hb|=O_p(s^{1/2}\lambda).\]
For the third term, using Lemma~\ref{lemma:norm:  l1 and l2 norm} and Condition~\ref{cond: preliminary estimator}, we obtain that 
\begin{align*}
|\alk^T\Sx(\hb-\bk)|\leq M\|\alk\|_2\|\hb-\bk\|_2=O_p( M^{3/2} s^{\kappa}\lambda).
\end{align*}
Therefore, \eqref{incor: colar: rho: hahsxha} implies
\[|\ha^T\hSx\ha-\alk^T\Sx\alk|=O_p(s^{\kappa}\lambda)+O_p(s^{\kappa+1/2}\lambda)+O_p(s^{1/2}\lambda).\]
Since  $s\lambda\to 0$ by Fact~\ref{fact: slambda goes to zero},  and $\kappa\in\{1/2,1\}$, $|\ha^T\hSx\ha-\alk^T\Sx\alk|=o_p(1)$, which, combined with \eqref{incor: colar: rho: numerator}, implies
 $|\hro^*-\rhk|=O_p(s^{\kappa}\lambda)$, and hence, the proof follows.
\end{proof}

\begin{proof}[of Lemma~\ref{corollary: rate of x and y}]
Using Lemma~\ref{corollary: rate of rho} we derive
\begin{align*}
  \|\hro|^{1/2}-\rhk^{1/2}|=\frac{\|\hro|-\rhk|}{|\hro|^{1/2}+\rhk^{1/2}}\leq \frac{\|\hro|-\rhk|}{\rhk^{1/2}}= \rhk^{-1/2}O_p(s^{\kappa}\lambda),  \end{align*}
which is $O_p(s^{\kappa}\lambda)$. 
Next,
\begin{align*}
  \inf_{w\in\{\pm 1\}}\|w|\hro|^{1/2}\ha-\rhk^{1/2}\alk\|_2
   \leq &\ \abs{|\hro|^{1/2}-\rhk^{1/2}}\|\ha\|_2+
   \rhk^{1/2}\inf_{w\in\{\pm 1\}}\|w\ha-\alk\|_2\\
   =&\ O_p(s^{\kappa}\lambda)O_p(1)+ \rhk^{1/2}O_p(s^{\kappa}\lambda)
\end{align*}
by Condition~\ref{cond: preliminary estimator}.
Also,
\begin{align*}
 \MoveEqLeft    \inf_{w\in\{\pm 1\}}\|w|\hro|^{1/2}\ha-\rhk^{1/2}\alk\|_1\\
   \leq &\ \abs{|\hro|^{1/2}-\rhk^{1/2}}\|\ha\|_1+
   \inf_{w\in\{\pm 1\}} \rhk^{1/2}\|w\ha-\alk\|_1\\
   =&\ O_p(s^{\kappa}\lambda)\slbs\|\alk\|_1+O_p(s^{\kappa+1/2}\lambda)\srbs+ \rhk^{1/2}O_p(s^{\kappa+1/2}\lambda)
\end{align*}
by Condition~\ref{cond: preliminary estimator}.
Now $\|\alk\|_1\leq s^{1/2}\|\alk\|_2=O_p(s^{1/2})$ by Cauchy Schwarz inequality and Lemma~\ref{lemma:norm:  l1 and l2 norm}. On the other hand, Fact~\ref{fact: slambda goes to zero}
 implies $O_p(s^{\kappa+1/2}\lambda)=o_p(1)$, which leads to
 \[\inf_{w\in\{\pm 1\}}\\|\hro|^{1/2}w\ha-(\rhk)^{1/2}\alk\|_1=O_p(s^{\kappa+1/2}\lambda).\]
 Since similar results hold for $|\hro|^{1/2}\hb$ as well, the proof follows.
 \end{proof}
 
 \begin{proof}[of Fact~\ref{fact: slambda goes to zero}]
Because $s^\kappa\lambda=n^{-1/4}o(1)$, we have $s^\kappa=n^{-1/4}\lambda^{-1}o(1)$, which leads to
\begin{align*}
    s=&\ \frac{n^{1/(4\kappa) }}{(\log(p+q))^{1/(2\kappa)}}o(1),
\end{align*}
which implies
\begin{align*}
    s^{\kappa+1/2}=&\ \frac{n^{\frac{\kappa+1/2}{4\kappa} }}{(\log(p+q))^{\frac{\kappa+1/2}{2\kappa}}}o(1),
\end{align*}
and
\[s^{\kappa+1/2}\lambda=\frac{n^{\frac{1/2-\kappa}{4\kappa}}}{(\log(p+q))^{\frac{1}{4\kappa}}}o(1).\]
Suppose $\kappa>1/2$. Then \[s^{\kappa+1/2}\lambda=n^{\frac{1/2-\kappa}{4\kappa}}o(1)=o(1).\]
Now consider the case when $\kappa=1/2$. Then $s^{\kappa+1/2}\lambda=o((\log(p+q))^{-1/(4\kappa)})$, which is $o(1)$. Because $\kappa\geq 1/2$, we have $s\lambda\leq s^{\kappa+1/2}\lambda$. Therefore $s\lambda=o(1)$ also follows.
\end{proof}
 
 \begin{proof}[of Fact~\ref{fact: Projection matrix}]
Let $x'=x/\|x\|_2$ and $y'=y/\|y\|_2$. Then 
\[\|P_x-P_y\|^2_F=\|P_{x'}-P_{y'}\|^2_2\leq 2\inf_{w\in\{\pm 1\}}\|wx'-y'\|_2^2\]
where the last equality follows by Fact~\ref{fact: Chen 2020}.
Note that
\[\|wx'-y'\|_2\leq \|(wx-y)\|_2/\|x\|_2+\|y\|_2(\|x\|_2^{-1}-\|y\|_2^{-1})\]
where for any $s\in\{\pm 1\}$,
\begin{align*}
 \abs{ \|x\|_2^{-1}-\|y\|_2^{-1}}=\frac{\abs{\|wx\|_2-\|y\|_2}}{\|x\|_2\|y\|_2} \leq\frac{\|wx-y\|_2}{\|x\|_2\|y\|_2}.
\end{align*}
Thus,
\begin{align*}
  \|wx'-y'\|_2\leq  2\|wx-y\|_2\|x\|_2^{-1}.
\end{align*}
Similarly we can show that
\begin{align*}
  \|wx'-y'\|_2\leq  2\|wx-y\|_2\|y\|_2^{-1}.
\end{align*}
Hence, the proof follows.
\end{proof}

\begin{proof}[of Lemma~\ref{result: inf norm: dif }]
 From Lemma 7 of \cite{jankova2018}, it  follows that for sufficiently large $n$, 
\[\|(\hSx-\Sx)v\|_{\infty}\leq C\|v\|_2 \lambda\]
with high probability for some $C$ depending only on the subgaussian parameter of $X$. Setting $v=e_i $ $(i=1,\ldots,p)$,  it then follows that
\[|\hSxy-\Sxy|_\infty=\sup_{1\leq i\leq p}\|(\hSxy-\Sxy)e_i\|_\infty\leq C \lambda\]
with high probability.
The above could also be proved directly using Bernstein inequality.

Thus it remains to show that
\[\|(\hSxy-\Sxy)v\|_{\infty}\leq C\|v\|_2 \lambda.\]
To that end, note that 
\begin{align*}
|(\hSxy-\Sxy)v\|_{\infty}\leq \left |\left|\begin{bmatrix}
\hSx-\Sx & \hSxy-\Sxy\\
\hSyx-\Syx & \hSy-\Sy
\end{bmatrix}
\begin{bmatrix}
0\\
v
\end{bmatrix}\right|\right|_{\infty}\leq \|v\|_2 C\lambda,
\end{align*}
where $C$ depends only on the subgaussian parameter of the vector $(X,Y)$.
Thus the proof follows.
\end{proof}

\begin{proof}[of Lemma~\ref{lemma: Quadratic:  bounded}]
This lemma  follows as a corollary to Lemma 10 of \cite{jankova2018}, which indicates that there exist $C_1$ and $C_2$ depending only on the sub-gaussian parameter of $X$ so that
\[|\z^T(\hSx-\Sx)\z|\leq C_1\lambda \|\z\|_1+C_2\lb \|\z\|_1^2\|\z\|_2^2\lambda^2+\|\z\|_1\|\z\|_2^2\lambda\rb\]
with high probability for sufficiently large $n,p$, and $q$.
Now since
$\|\widehat z_n\|_1\leq s^{1/2}$, $\|\z\|_1\lambda \leq s^{1/2}\lambda=o_p(1)$. Thus, \[\|\z\|_1\lambda(1+\|\z\|_1\lambda)\leq 2\|\z\|_1\lambda\]
and hence the result follows.
\end{proof}

\begin{proof}[of lemma~\ref{lemma: Quadratic:  bounded}]
Let us denote $x=(\widehat z_n, \widehat w_n)$. Then writing $t=\widehat z_n^T(\hSxy-\Sxy)\widehat w_n$ we note that
\[2t=x^T(\widehat \Sigma_n-\Sigma)x-\widehat z_n^T(\hSx-\Sx)\widehat z_n-\widehat w_n^T(\hSy-\Sy)\widehat w_n.\]
 Lemma~\ref{lemma: Quadratic:  bounded} indicates that there exist $C$ depending only on the subgaussian parameter of $X$ so that
 \[\abs{\widehat z_n^T(\hSx-\Sx)\widehat z_n}+\abs{\widehat w_n^T(\hSy-\Sy)\widehat w_n}\leq C\lambda\slb s^{1/2}(\|\widehat z_n\|_2^2+\|\widehat w_n\|_2^2)+(\|\widehat z_n\|_1+\|\widehat w_n\|_1)\srb\]
 with high probability as $n,p,q\to\infty$.
Since $[X^T Y^T]^T$ is a sub-Gaussian matrix, Lemma~\ref{lemma: Quadratic:  bounded} can be applied to the term $x^T(\widehat \Sigma_n-\Sigma)x$ as well, and the result follows.
\end{proof}

\begin{proof}[of Lemma~\ref{lemma: additional: l1 norm of x and l2 norm of z knot}]
We have
\begin{align*}
    |x^T(\hSx-\Sx)\widehat z_n|\leq &\  |x^T(\hSx-\Sx)(\widehat z_n-z_0)|+|x^T(\hSx-\Sx)z_0|\\
   \leq &\ \|x\|_1|\hSx-\Sx|_\infty\|\widehat z_n-z_0\|_1+\|x\|_1\|(\hSx-\Sx)z_0\|_\infty\\
    =&\ C\slb\|x\|_1 \lambda o_p(1)+\|x\|_1\|z_0\|_2 \lambda\srb
\end{align*}
with high probability for large $n$ for some $C>0$ depending only  the subgaussian parameter of $X$ by by Lemma~\ref{result: inf norm: dif } and Lemma~\ref{result: inf norm: dif }. 
Therefore 
\[ |x^T(\hSx-\Sx)\widehat z_n|\leq C\|z_0\|_2\|x\|_1\lambda\]
with high probability as $n,p,q\to\infty$.
\end{proof}

\begin{proof}[of Fact~\ref{fact: variance of quadratic terms}]
  Note that 
  \[(a^TX,b^TX)\sim N_2\left(0,\begin{bmatrix}
     a^T\Sx a & a^T\Sx b\\
     b^T\Sx a & b^T\Sx b
  \end{bmatrix}\right), (z^TX,\ d^TY)\sim N_2\left(0,\begin{bmatrix}
     z^T\Sx z & z^T\Sxy d\\
     d^T\Syx z & d^T\Sy d
  \end{bmatrix}\right) \]
  Therefore, Fact~\ref{fact: higher moments of bivariate normal} implies that
  \[\text{var}(a^TXX^Tb)= (a^T\Sx a)( b^T\Sx b)+(a^T\Sx b)^2\]
  and
  \[\text{var}(z^TXY^Td)=(z^T\Sx z)(d^T\Sy d)+(z^T\Sxy d)^2.\]
  \end{proof}
  
  \begin{proof}[of fact~\ref{fact: sub-exponential norms}]
Suppose $a\in\RR^p$ and $b\in\RR^q$. Since $X$ and $Y$ are sub-Gaussian random vectors, $a^TX$ and $b^TY$ are sub-Gaussian random variables.  Therefore Lemma 2.7.5 of \cite{vershynin2018} implies that $\|a^TXY^Tb\|_{\psi_1}\leq \|a^TX\|_{\psi_2}\|b^TY\|_{\psi_2}$. By definition of the sub-Gaussian norm $\|X\|_{\psi_2}$ of a random vector $X\in\RR^p$ \citep[cf. Definition 3.4.1][]{vershynin2018}, we have 
$\|a^TX\|_{\psi_2}\leq \|a\|_2\|X\|_{\psi_2}$ for any  $a\in\RR^p$.
Therefore,
\[\|a^TXY^Tb\|_{\psi_1}\leq \|a\|_2\|b\|_2\|X\|_{\psi_2}\|Y\|_{\psi_2}.\]
Similarly we can show that  $a,c\in\RR^p$ satisfy
\[\|a^TXX^Tc\|_{\psi_1}\leq \|a\|_2\|c\|_2\|X\|_{\psi_2}^2.\]
\end{proof}

 \begin{proof}[of Lemma~\ref{lemma: variance lemma facts}]
 \def\cov{\en{\text{cov}}}
 This lemma follows by straightforward calculation. Note that
 \begin{align*}
    \text{var}(T)=&\ \text{var}(a^TXX^Tb)+\text{var}(c^TYY^Td)+\text{var}(z^TXY^Td)+\text{var}(b^TXY^T\gamma)\\
    &\  +2\cov(a^TXX^Tb,c^TYY^Td)-2\cov(a^TXX^Tb,z^TXY^Td) -2\cov(a^TXX^Tb,b^TXY^T\gamma)\\
     &\ -2\cov(c^TYY^Td,z^TXY^Td)-2\cov(c^TYY^Td,b^TXY^T\gamma)+2\cov(z^TXY^Td,b^TXY^T\gamma)
 \end{align*}
 First, we will find the variance of $a^TXX^Tb$.  To that end, note that Fact~\ref{fact: variance of quadratic terms} implies
 \begin{align*}
  \MoveEqLeft \text{var}(a^TXX^Tb)+\text{var}(c^TYY^Td)+\text{var}(z^TXY^Td)+\text{var}(b^TXY^T\gamma)\\
  =&\ (a^T\Sx a)(b^T\Sx b)+(a^T\Sx b)^2+(c^T\Sx c)(d^T\Sy d)+(c^T\Sy d)^2\\
  &\ +(z^T\Sx z)(d^T\Sy d)+(z^T\Sxy d)^2+(b^T\Sx b)(\gamma^T\Sy\gamma)+(b^T\Sxy\gamma)^2.
 \end{align*}
 Now note that
 \begin{align*}
  \cov(a^TXX^Tb,c^TYY^Td)
     =&\ E[a^TXX^Tbc^TYY^Td]-E[a^TXX^Tb]E[c^TYY^Td]\\
     =&\ E\slbt a^TXX^Tbc^TE[YY^T|X]d\srbt -a^T\Sx b c^T\Sy d\\
     \stackrel{(a)}{=}&\ E\slbt a^TXX^Tbc^T(\Sy -\Syx\Sx^{-1}\Sxy+\Syx\Sx^{-1}XX^T\Sx^{-1}\Sxy)d\srbt -a^T\Sx b c^T\Sy d\\
     =&\ (a^T\Sx b)c^T(\Sy-\Syx\Sx^{-1}\Sxy)d
     + E\slbt a^TXX^Tb (c^T\Syx\Sx^{-1}XX^T\Sx^{-1}\Sxy d)\srbt\\
     &\ -a^T\Sx b c^T\Sy d\\
     \stackrel{(b)}{=}&\ -a^T\Sx bc^T\Syx\Sx^{-1}\Sxy d +a^T\Sx b c^T\Syx \Sx^{-1}\Sxy d\\
     &\ +a^T\Sxy c  b^T\Sxy d+a^T\Sxy d b^T\Sxy c\\
     =&\ a^T\Sxy c  b^T\Sxy d+a^T\Sxy d b^T\Sxy c,
 \end{align*}
 where in step (a), we used the fact that
 \[E[YY^T|X]=Var(Y|X)+E[Y|X]E[Y|X]^T,\]
 and 
 \begin{align}\label{inlemma: main: conditional distribution of Y}
     Y|X\sim N(\Syx \Sx^{-1}X,\  \Sy-\Syx\Sx^{-1}\Sxy),
 \end{align}
 and in step
 (b), we used Fact~\ref{fact: higher moments of bivariate normal}.
 On the other hand,
 \begin{align*}
     \cov(a^TXX^Tb,z^TXY^Td)=&\ E\slbt a^TXX^Tbz^TXY^Td\srbt-a^T\Sx b z^T\Sxy d\\
     =&\ E\slbt a^TXX^Tbz^TXE[Y|X]^Td\srbt-a^T\Sx b z^T\Sxy d\\
     \stackrel{(a)}{=}&\ E\slbt a^TXX^Tbz^TX(\Syx\Sx^{-1}X)^Td\srbt-a^T\Sx b z^T\Sxy d\\
     =&\ E\slbt a^TXX^Tbz^TXX^T\Sx^{-1}\Sxy d\srbt-a^T\Sx b z^T\Sxy d\\
     \stackrel{(b)}{=}&\ \slb a^T\Sx bz^T\Sxy d+a^T\Sx zb^T\Sxy d+a^T\Sxy db^T\Sx z\srb-a^T\Sx b z^T\Sxy d\\
     =&\ a^T\Sx zb^T\Sxy d+a^T\Sxy db^T\Sx z
 \end{align*}
 where (a) follows from \eqref{inlemma: main: conditional distribution of Y} and (b) follows from Fact~\ref{fact: higher moments of bivariate normal}.
 Similarly, we can show that
 \begin{gather*}
     \cov(a^TXX^Tb,b^TXY^T\gamma)=a^T\Sx bb^T\Sxy \gamma+a^T\Sxy \gamma b^T\Sx b\\
     \cov(c^TYY^Td,z^TXY^Td)=c^T\Syx zd^T\Sy d+c^T\Sy dd^T\Syx z\\
     \cov(c^TYY^Td,b^TXY^T\gamma)=c^T\Syx b d^T\Sy \gamma+ c^T\Sy\gamma d^T\Syx b
 \end{gather*}
 Finally,
 \begin{align*}
    \MoveEqLeft \cov(z^T XY^T d,b^TXY^T\gamma)\\
     =&\ E[z^T XY^T db^TXY^T\gamma]-z^T\Sxy d b^T\Sxy \gamma\\
     =&\ E\slbt z^TXX^Tb d^TE[YY^T|X]\gamma\srbt-z^T\Sxy d b^T\Sxy \gamma\\
    \stackrel{(a)}{=}&\ E\slbt z^TXX^Tb d^T\slb \Sy-\Syx\Sx^{-1}\Sxy+\Syx\Sx^{-1}XX^T\Sx^{-1}\Sxy)\gamma\srbt-z^T\Sxy d b^T\Sxy \gamma\\
     =&\ E\slbt z^TXX^Tb d^T \Sy\gamma\srbt -E\slbt z^TXX^Tb d^T\Syx\Sx^{-1}\Sxy\gamma\srbt\\
     &\ +E\slbt z^TXX^Tb d^T\Syx\Sx^{-1}XX^T\Sx^{-1}\Sxy\gamma\srbt-z^T\Sxy d b^T\Sxy \gamma\\
     \stackrel{(b)}{=}&\ z^T\Sx bd^T\Sy\gamma-z^T\Sx bd^T\Syx\Sx^{-1}\Sxy\gamma + \slb z^T\Sx b d^T\Syx \Sx^{-1}\Sxy\gamma\\
     &\ +z^T\Sxy\gamma b^T\Sxy d+z^T\Sxy db^T\Sxy\gamma\srb-z^T\Sxy d b^T\Sxy \gamma\\
     =&\ z^T\Sx bd^T\Sy\gamma+z^T\Sxy\gamma b^T\Sxy d,
 \end{align*}
 where (a) follows from \eqref{inlemma: main: conditional distribution of Y} and (b) follows from Fact~\ref{fact: higher moments of bivariate normal}. Thus $\text{var}(T)$ equals
 \begin{align*}
  \MoveEqLeft  (a^T\Sx a)(b^T\Sx b)+(a^T\Sx b)^2+(c^T\Sx c)(d^T\Sy d)+(c^T\Sy d)^2\\
  &\ +(z^T\Sx z)(d^T\Sy d)+(z^T\Sxy d)^2+(b^T\Sx b)(\gamma^T\Sy\gamma)+(b^T\Sxy\gamma)^2\\
  &\ +2(a^T\Sxy c)(b^T\Sxy d)+2(a^T\Sxy d)(b^T\Sxy c)+2(z^T\Sx b)(d^T\Sy\gamma)+2(z^T\Sxy\gamma) (b^T\Sxy d)\\
  &\ -2 (a^T\Sx z)(b^T\Sxy d)-2(a^T\Sxy d)(b^T\Sx z)-2(a^T\Sx b)(b^T\Sxy \gamma)-2(a^T\Sxy \gamma)( b^T\Sx b)\\
  &\ -2(c^T\Syx z)(d^T\Sy d)-2(c^T\Sy d)(d^T\Syx z)-2(c^T\Syx b)( d^T\Sy \gamma)-2(c^T\Sy\gamma) (d^T\Syx b)
 \end{align*}
\end{proof}

\end{document}